\documentclass[oneside,11pt]{amsart}
\usepackage{amssymb,amsfonts,amsmath,amsthm}
\usepackage{amscd}
\usepackage{amssymb}
\usepackage{tikz-cd}
\usepackage{amsthm}
\usepackage{galois}
\usepackage{mathtools}
\usepackage[all,arc]{xy}
\usepackage{enumerate}
\usepackage{mathrsfs}
\usepackage[toc,page]{appendix}
\usepackage[left=3cm, right=3cm, bottom=3cm]{geometry}
\usepackage{graphicx}
\usepackage{tabularx}
\usepackage{url}
\usepackage{color}
\usepackage{tikz-cd}  
\usepackage{mathdots} 
\usepackage{romannum} 
\usepackage[colorlinks]{hyperref} 
\hypersetup{
bookmarksnumbered,
pdfstartview={FitH},
breaklinks=true,
linkcolor=blue,
urlcolor=blue,
citecolor=blue,
bookmarksdepth=2
}
\usepackage{float} 
\usepackage{comment}
\usepackage{microtype}

\usepackage[toc,page]{appendix}

\newtheorem{thm}{Theorem}[section]

\newtheorem*{thm*}{Theorem}
\newtheorem*{cor*}{Corollary}
\newtheorem*{prop*}{Proposition}
\newtheorem{cor}[thm]{Corollary}
\newtheorem{prop}[thm]{Proposition}
\newtheorem{lem}[thm]{Lemma}

\theoremstyle{definition}
\newtheorem{defn}[thm]{Definition}

\newtheorem*{notn*}{Notation}

\theoremstyle{remark}
\newtheorem{rem}[thm]{Remark}

\newtheorem*{idea*}{Idea}

\newcommand{\midwd}{\,\middle\vert\,}

\newcommand\smvee{\raise0.9ex\hbox{$\scriptscriptstyle\vee$}}
\makeatletter
\let\origsection\section
\renewcommand\section{\@ifstar{\starsection}{\nostarsection}}

\newcommand\nostarsection[1]
{\sectionprelude\origsection{#1}\sectionpostlude}

\newcommand\starsection[1]
{\sectionprelude\origsection*{#1}\sectionpostlude}

\newcommand\sectionprelude{%
  \vspace{1em}
}

\newcommand\sectionpostlude{%
  \vspace{1em}
}

\newcommand*\mybackmatter{%
\startcontents
\phantomsection
\addcontentsline{toc}{part, thebibliography}{}%
\@endpart}

\makeatother

\makeatletter
\let\c@equation\c@thm
\makeatother
\numberwithin{thm}{section}
\numberwithin{equation}{section}

\title[Symmetric spaces for groups over involutive algebras]{Symmetric spaces for groups over involutive algebras and applications to Higgs bundles}

\author{Pengfei Huang, Georgios Kydonakis, Eugen Rogozinnikov, and Anna Wienhard}

\begin{document}

\pagenumbering{arabic}

\begin{abstract}

We study symplectic groups and indefinite orthogonal groups over involutive, possibly noncommutative, algebras $(A, \sigma)$. In the case when the algebra $(A, \sigma)$ is Hermitian, or the complexification $(A_{\mathbb{C}}, \sigma_{\mathbb{C}})$ of a Hermitian involutive algebra, one can identify maximal compact subgroups of such groups, and consider their associated Riemannian symmetric spaces. This new perspective allows for the realization of various geometric models for the symmetric space. We describe explicitly the complexified tangent space for each of the models, as well as the diffeomorphisms between them and their differentials.

In the second part of the article, we give a number of applications of this theory. 
The geometric realizations of the Riemannian symmetric spaces described in the first part provide new geometric interpretations of Higgs bundle data that can be used for the study of fundamental group representations into symplectic or into indefinite orthogonal groups over Hermitian involutive algebras. 
We give an exact component count for the moduli spaces of $\mathrm{Sp}_2(A_{\mathbb{C}}, \sigma_{\mathbb{C}})$-Higgs bundles and of $\mathrm{O}(A_{\mathbb{C}}, \sigma_{\mathbb{C}})$-Higgs bundles, using the topology of the corresponding
maximal compact subgroups rather than Morse--Bott theory techniques. 
Furthermore, we use the noncommutative symmetric-space models to construct a factorization of the Hitchin morphism for $\mathrm{Sp}_2(A_{\mathbb C},\sigma_{\mathbb C})$-Higgs bundles, together with analogous factorizations for the real groups
$\mathrm{Sp}_2(A,\sigma)$ and $\mathrm{O}_{(1,1)}(A,\sigma)$. These factorizations are induced by quadratic norm maps from the corresponding tangent models to Jordan-algebraic targets and
pass through intermediate affine GIT quotients. As a consequence, they reduce the algebraic complexity required in order to characterize the Hitchin base explicitly.
\end{abstract}

\maketitle

\flushbottom

\tableofcontents

\section{Introduction}

The general theory of Lie groups and their Lie algebras over noncommutative rings presents itself as a highly non-trivial generalization of the extensively studied commutative case. For instance, the non-canonical choice for a notion of the determinant of a matrix in $\mathrm{GL}_n(R)$, when $R$ is noncommutative, prevents an appropriate definition for the group $\mathrm{SL}_n(R)$ as a subgroup of $\mathrm{GL}_n(R)$ (see \cite{BR1}, \cite{BR2}). However, for a possibly noncommutative algebra with an anti-involution $(A, \sigma)$ symplectic groups $\mathrm{Sp}_2(A,\sigma)$ over $(A, \sigma)$ have been introduced in \cite{ABRRW}, \cite{Rog}. Here we further extend this work and provide applications to the theory of Higgs bundles. 

The motivation to realize classical Lie groups as symplectic or orthogonal Lie groups over noncommutative rings arose from the generalizations of total positivity introduced in \cite{GuiWie}. We expect the noncommutative perspective on Higgs bundles, we introduce here, to be fruitful when using Higgs bundle methods to study the space of positive representations of the fundamental groups of closed surfaces.

For an involutive algebra $(A,\sigma)$ and a standard symplectic form $\omega$ on $A^2$, the \emph{symplectic group} $\mathrm{Sp}_2(A,\sigma)$ over $(A, \sigma)$ is defined to be the automorphism group $\mathrm{Aut}(\omega)$. Similarly, if $\omega$ is the standard indefinite orthogonal form on $A^2$, then the \emph{indefinite orthogonal group} $\mathrm{O}_{(1,1)}(A,\sigma)$ over $(A, \sigma)$ is defined to be the automorphism group $\mathrm{Aut}(\omega)$. Several classical groups can be realized this way, including  the groups $\mathrm{Sp}_{2n}(\mathbb{R})$, $\mathrm{O}(n,n)$, $\mathrm{U}(n,n)$, $\mathrm{Sp}_{2n}(\mathbb{C})$, $\mathrm{O}_{2n}(\mathbb{C})$, $\mathrm{SO}^*(4n)$, $\mathrm{Sp}(n,n)$, $\mathrm{GL}_{2n}(\mathbb{C})$, among others.

In the case when the involutive algebra $(A, \sigma)$ is Hermitian, or the complexification $(A_{\mathbb{C}}, \sigma_{\mathbb{C}})$ of a Hermitian involutive algebra, one can identify maximal compact subgroups of symplectic and of indefinite orthogonal groups, and consider Cartan decompositions of their Lie algebras. For Hermitian algebras, the associated Riemannian symmetric spaces of such groups are considered. In fact, the Shilov boundary of the symmetric space is well-defined, thus the symmetric space of most classical tube-type Hermitian Lie groups can be realized in this way. In \cite{ABRRW}, \cite{Rog},  various different geometric interpretations of the symmetric space $\mathcal{X}_{G}$ of the groups $G=\mathrm{Sp}_{2}(A,\sigma )$ and  $G=\mathrm{Sp}_2(A_{\mathbb{C}}, \sigma_{\mathbb{C}})$ have been given. 
In fact, these models for $\mathcal{X}_{\mathrm{Sp}_{2}(A,\sigma)}$ generalize the various models of the hyperbolic plane viewed as the symmetric space associated to the group $\mathrm{SL}_2( \mathbb{R})=\mathrm{Sp}_2( \mathbb{R})$. Analogously, the models obtained for the symmetric space  $\mathcal{X}_{\mathrm{Sp}_{2}(A_{\mathbb{C}},\sigma_{\mathbb{C}} )}$ of the complexified Hermitian algebra $(A_{\mathbb{C}},\sigma_{\mathbb{C}})$ generalize models of the three-dimensional hyperbolic space considered as the symmetric space for the group $\mathrm{SL}_2( \mathbb{C})=\mathrm{Sp}_2( \mathbb{C})$.

In this work we review these models and provide new models for the Riemannian symmetric spaces $\mathcal{X}_{\mathrm{O}_{(1,1)}(A,\sigma )}$, $\mathcal{X}_{A^{\times}}$, $\mathcal{X}_{\mathrm{O}(A_{\mathbb{C}},\sigma_{\mathbb{C}} )}$  as well, where $A^{\times}$ is the group of all invertible elements of $A$, and $\mathrm{O}(A_{\mathbb{C}},\sigma_{\mathbb{C}}):=\{ a\in A_{\mathbb{C}}^{\times}\mid \sigma_{\mathbb{C}}(a)=a^{-1}\}$ is the \emph{orthogonal group} of $A_{\mathbb{C}}$. The spaces $\mathcal{X}_G$ all give  \emph{models} for the symmetric space $G/K$, where $K$ is a maximal compact subgroup of $G$, in the sense that the group $G$ acts smoothly and transitively on $\mathcal{X}$ by diffeomorphisms, and the stabilizer $\mathrm{Stab}_{G}(p_0) \cong K$, for any $p_0 \in \mathcal{X}$; then the map $gK \mapsto g(p_0)$ gives a diffeomorphism between $G/K$ and $\mathcal{X}$. 
We describe in this article explicitly the (complexified) tangent space for each of the models, as well as the differential for each of the diffeomorphisms between them. This allows us to give applications of the noncommutative perspective to the theory of Higgs bundles. 

We obtain geometric interpretations of $G$-Higgs bundles using the particular geometry of the models of the symmetric space of the group $G$. Since all the groups we consider here are Lie groups of non-compact type, their associated Riemannian symmetric spaces are non-positively curved. 
Thus, given a reductive representation $\rho\colon \pi_1(X) \to G$, where $X$ is a compact Riemann surface, there is a $\rho$-equivariant map 
$f_{\mathcal{X}}\colon \tilde{X} \to \mathcal{X}$, from the universal cover $\tilde{X}$ of $X$ to $\mathcal{X}$ \cite{Corlette, Sagman}, which is harmonic. 
The pull-back bundle
$f^{*}_{\mathcal{X}} T^{\mathbb{C}}\mathcal{X}$ of the complexified tangent bundle, is a flat principal bundle. Thus, a reductive fundamental group representation gives rise to a $G$-Higgs bundle, which is a pair $(\mathcal{E}, \varphi)$, where $\mathcal{E}:=(E, \bar{\partial}_{E})$ is a holomorphic principal bundle over $X$ and $\varphi$ is a $(1,0)$-form with values in $f^{*}_{\mathcal{X}}T^{\mathbb{C}}\mathcal{X}$, such that $\bar{\partial}_{E} \varphi=0$. The different incarnations of the symmetric space $\mathcal{X}$ described in the first part of the article thus give new geometric models for $G$-Higgs bundles in the case when the group $G$ is a Hermitian Lie group of tube type or its complexification, thus offering a new viewpoint for the study of such Higgs bundles. 

As an application we give an exact component count for the moduli spaces of polystable $G$-Higgs bundles, for groups of the form $G=\mathrm{Sp}_2(A_{\mathbb{C}}, \sigma_{\mathbb{C}})$ and $G=\mathrm{O}(A_{\mathbb{C}}, \sigma_{\mathbb{C}})$, without using the standard Morse--Bott theory techniques. 

As a final application, we show that this noncommutative perspective allows us to construct a factorization of the Hitchin morphism for moduli spaces of polystable $\mathrm{Sp}_2(A_{\mathbb C},\sigma_{\mathbb C})$-Higgs bundles. We also obtain analogous factorizations for the real groups $\mathrm{Sp}_2(A,\sigma)$ and $\mathrm{O}_{(1,1)}(A,\sigma)$. These factorizations are governed by quadratic norm maps on the corresponding tangent models of the symmetric space. After passing to invariant theory, the norm maps induce morphisms between intermediate affine GIT quotients, which are then identified with the usual Hitchin bases.

\vspace{3mm}

The article is divided into two main parts. In the first part, which spans Sections \ref{sec:symplectic and indefinite}-\ref{sec:models for compact}, we explicitly describe the models of the Riemannian symmetric spaces for the indefinite orthogonal groups of the form $\mathrm{O}_{(1,1)}(A, \sigma)$, $A^{\times}$, $\mathrm{O}(A_{\mathbb{C}}, \sigma_{\mathbb{C}})$, as well as for the symplectic groups $\mathrm{Sp}_{2}(A, \sigma)$ and $\mathrm{Sp}_{2}(A_{\mathbb{C}}, \sigma_{\mathbb{C}})$. We describe their (complexified) tangent spaces and the diffeomorphisms between them in detail. This analysis is of independent interest for the development of the theory of Lie groups over noncommutative rings and for this reason is presented separately from the rest of this work. 

In the second part, covering Sections \ref{sec:Higgs-Data}-\ref{sec:Hitchin morphism},
 we apply this theory to nonabelian Hodge theory. In particular, we give the geometric description  of polystable $G$-Higgs bundles  with respect to each model of the symmetric space studied in Part 1. For a considerable number of Lie groups, we describe the Higgs bundle data explicitly. 
 
 We count the number of connected components of the moduli spaces of polystable $G$-Higgs bundles in the cases when the group is of the form $\mathrm{Sp}_2(A_{\mathbb{C}}, \sigma_{\mathbb{C}})$ or $\mathrm{O}(A_{\mathbb{C}}, \sigma_{\mathbb{C}})$ (Section~\ref{sec:components}).  
 
In Section \ref{sec:Hitchin morphism} we introduce a factorization of the Hitchin morphism for moduli spaces of polystable
$\mathrm{Sp}_2(A_{\mathbb C},\sigma_{\mathbb C})$-Higgs bundles, and analogous factorizations for the real groups $\mathrm{Sp}_2(A,\sigma)$ and $\mathrm{O}_{(1,1)}(A,\sigma)$. Central to this result is the identification of the intermediate GIT quotient space with the Hitchin base. This identification is mediated by a certain formally real Jordan algebra, which provides the underlying algebraic structure for our symmetric space models. 

This perspective offers two primary advantages. First, it provides a geometric dictionary for the Hitchin base, interpreting abstract holomorphic differentials as invariants obtained from quadratic norm maps on the corresponding Jordan-algebraic targets. Second, this factorization serves as a method of dimensional compression. Indeed, the target spaces of these norm maps are the complexifications of natural formally real Jordan algebras: the algebra $A_{\mathbb H}^{\sigma_1}$ of quaternionic Hermitian elements in the complex symplectic case, the algebra $A_{\mathbb C}^{\bar\sigma_{\mathbb C}}$ of complex Hermitian elements in the real symplectic case, and the algebra $A^\sigma$ of symmetric elements in the indefinite
orthogonal case. In the classical matrix examples of interest, these complexified Jordan algebras have dimensions strictly smaller than 
the complexified tangent models of the symmetric spaces in which
the Higgs fields originally take their values. Consequently, the intermediate GIT spaces reduce the algebraic complexity of describing the Hitchin base explicitly and enable a more efficient and explicit characterization of the image of the
Hitchin morphism.

 \vspace{5mm}

\textbf{Acknowledgements}.
We would like to warmly thank Arkady Berenstein, Qiongling Li, Vladimir Retakh and Nathaniel Sagman for useful discussions. P. H. and G. K. are grateful to the Tianyuan Mathematics Research Center for its hospitality and for providing a fertile working environment. P. H. and E. R. would like to express deep gratitude to the Institut f\"ur Mathematik, Universit\"at Heidelberg, and the Max Planck Institute for Mathematics in the Sciences for their kind hospitality and support. P. H. was partially supported by the Deutsche Forschungsgemeinschaft (DFG, Projektnummer 547382045). G. K. was supported by the Alexander von Humboldt Foundation and the Scientific Committee of the University of Patras through the program ``Medicos''. E. R. was supported by a postdoc scholarship of the German Academic Exchange Service (DAAD) and the KIAS Individual Grant (MG100901) at Korea Institute for Advanced Study.
P. H., E. R., and A. W. received funding from the European Research Council (ERC) under the European Union’s Horizon 2020 research and innovation programme (grant agreement No 101018839). A. W. thanks the Hector Fellow Academy for support. 

\part{Models of the symmetric space for groups over involutive algebras}

\section{Symplectic and indefinite orthogonal groups over involutive algebras}\label{sec:symplectic and indefinite}
In this first Section, we set the main definitions regarding involutive algebras and consider the families of Lie groups over Hermitian algebras that will be our main objects of study. We also demonstrate a considerable number of examples of classical Lie groups that can be seen as Lie groups over Hermitian involutive algebras.

\subsection{Involutive algebras and Hermitian algebras} 

We recall the definition of involutive and Hermitian algebras, and review some of their important properties. For a more detailed discussion, we refer to~\cite[Section~2]{ABRRW}.

Let $A$ be a unital associative, possibly noncommutative, finite-dimensional semisimple algebra over $\mathbb{R}$.

\begin{defn}\label{def:antiinv}
An \emph{anti-involution} on $A$ is an $\mathbb{R}$-linear map $\sigma\colon A\to A$ such that
\begin{itemize}
\item $\sigma(ab)=\sigma(b)\sigma(a)$;
\item $\sigma^2=\mathrm{Id}$.
\end{itemize}
An \emph{involutive $\mathbb{R}$-algebra} is a pair $(A,\sigma)$, where $A$ is an $\mathbb{R}$-algebra as above and $\sigma$ is an~anti-involution on $A$.
\end{defn}

\begin{rem}
Sometimes in the literature, the maps that satisfy Definition~\ref{def:antiinv} are called just \emph{involutions}. We add the prefix ``anti'' in order to emphasize that they exchange factors.
\end{rem}

\begin{defn} An element $a\in A$ is called \emph{$\sigma$-normal} if $\sigma(a)a=a\sigma(a)$; it is called \emph{$\sigma$-symmetric} if $\sigma(a)=a$. Lastly, $a$ is called \emph{$\sigma$-anti-symmetric} if $\sigma(a)=-a$. We denote:
$$A^{\sigma}:=\mathrm{Fix}_A(\sigma)=\{a\in A\mid \sigma(a)=a\},$$
$$A^{-\sigma}:=\mathrm{Fix}_A(-\sigma)=\{a\in A\mid \sigma(a)=-a\}.$$
\end{defn}

\begin{rem}
    The space $A^{-\sigma}$ is a Lie algebra with respect to the commutator. The space $A^{\sigma}$ is a triple system with respect to the commutator and a Jordan algebra with respect to the anti-commutator.
\end{rem}

\begin{rem}
Since the algebra $A$ is unital, we always have the canonical copy of $\mathbb{R}$ in the center of $A$, namely, $\mathbb{R}\cdot 1$, where $1$ is the unit of $A$. We always identify $\mathbb{R}\cdot 1$ with $\mathbb{R}$. Moreover, since $\sigma$ is linear, for all $k\in \mathbb{R}$, $\sigma(k\cdot 1)=k\sigma(1)=k\cdot 1$, in other words, $\mathbb{R}\cdot 1\subseteq A^\sigma$.
\end{rem}

We denote by $A^\times$ or $\mathrm{GL}_1(A)$ the group of all invertible elements of $A$. If $V \subset A$ is a vector subspace, we denote by $V^\times = A^\times \cap V$
the set of invertible elements in $V$.

\begin{defn}
The closed subgroup 
$$\mathrm{O}_1{(A,\sigma)}:=\mathrm{O}{(A,\sigma)}:=\{a\in A^\times\mid \sigma(a)=a^{-1}\}$$
of $A^\times$ is called the \emph{orthogonal group} of $A$. The Lie algebra of $\mathrm{O}{(A,\sigma)}$ agrees with $A^{-\sigma}$.
\end{defn}

\begin{defn}
\label{df:cone}
Let $(A,\sigma)$ be an involutive algebra. We define the sets of \emph{$\sigma$-positive} and  \emph{$\sigma$-negative} elements by
$$A^\sigma_+:=\left\{ a^2  \mid a\in (A^\sigma)^\times\right\},\;A^\sigma_-:=-A^\sigma_+,$$
and the set of \emph{$\sigma$-non-negative} and \emph{$\sigma$-non-positive} elements by
$$A^\sigma_{\geq 0}:=\left\{ a^2 \mid a\in A^\sigma\right\},\;A^\sigma_{\leq 0}:=-A^\sigma_{\geq 0}.$$
\end{defn}

The following definition of a Hermitian algebra will be one of the key notions in this paper.

\begin{defn}\label{Herm_A}
An involutive algebra $(A,\sigma)$ is called \emph{Hermitian} if for all $x,y\in A^\sigma$, the relation $x^2+y^2=0$ implies $x=y=0$. For an anti-involution $\sigma$ on an algebra $A$, we say that $\sigma$ is \emph{Hermitian} if $(A,\sigma)$ is a Hermitian algebra. 
\end{defn}

For further considerations the following definition will be important:
\begin{defn}
Let $V$ be an $\mathbb{R}$-vector space. A subset $\Omega\subseteq V$ is called a \emph{cone} in $V$ if $\lambda x\in \Omega$, for all $\lambda\in (0,\infty)$, $x\in\Omega$. A cone $\Omega$ is called \emph{convex} if $x+y\in\Omega$, for all $x,y\in\Omega$. Lastly, a cone $\Omega$ is called \emph{proper} if it does not contain affine lines.
\end{defn}

\begin{rem}[{\cite[Section~2]{ABRRW}}]
Let $(A,\sigma)$ be a Hermitian algebra. The spaces $A^\sigma_{\geq 0}$ and $A^\sigma_+$ are proper convex cones in $A^\sigma$. The space $A^\sigma_{\geq 0}$ is the topological closure of $A^\sigma_+$, and $A^\sigma_+$ is, in turn, the interior of $A^\sigma_{\geq 0}$ in $A^\sigma$. In this case, the Jordan algebra $A^\sigma$ is formally real,  the group $\mathrm{O}{(A,\sigma)}$ is compact, and it acts on $A^\sigma$ by conjugation preserving $A^\sigma_+$ and $A^\sigma_{\geq 0}$.
\end{rem}

\begin{defn}
Let $x\in A$ and $a\in A^\times$. The element $ax\sigma(a)$ is called \emph{congruent} to $x$ by $a$. The group $A^\times$ acts on $A$ by \emph{congruence} preserving $A^\sigma$ and $A^\times$.
\end{defn}

Notice that the action of $A^\times$ by congruence on $A^\sigma$ restricted to $\mathrm{O}{(A,\sigma)}$ agrees with the action by conjugation of $\mathrm{O}{(A,\sigma)}$ on $A^\sigma$.

\begin{rem}[{\cite[Corollary~2.65]{ABRRW}}]
Let $(A,\sigma)$ be a Hermitian algebra. The action of $A^\times$ by congruence preserves $A^\sigma_{\geq 0}$ and $A^\sigma_{+}$. 
\end{rem}

\subsection{Conjugations on algebras}\label{sec:conjugations}

In this subsection, we introduce a convenient notation for conjugations in algebras with several imaginary units, i.e. elements that square to $-1$.

Let $A$ be an extension of a real algebra $A_\mathbb R$ by adding imaginary units $i_1,\dots,i_k$, for some $k\in\mathbb N$, that commute with elements of $A_\mathbb R$ and for every $m,l\in\{1,\dots,k\}$, they satisfy $i_m i_l=\pm i_l i_m$. Moreover, we assume that the imaginary units are independent, i.e. there are no algebraic relations between them but the relations that follow from the commutation relations and relations $i_m^2=-1$, for all $m\in\{1,\dots,k\}$. For simplicity, we also assume $A$ to be associative, since in the sequel only this case will be needed.

Every element of $A$ can be written uniquely as a linear combination 
$$\sum_{j_1,\dots j_k=0}^1 a_{j_1,\dots,j_k} i_1^{j_1}\dots i_k^{j_k},$$ 
where all $a_{j_1,\dots,j_k}\in A_\mathbb R$. 

Notice that for every imaginary unit $i_m$, every element of $A$ can be uniquely written as $a_1+a_2i_m$, where $a_1,a_2\in A'$, for $A'$ the subalgebra of $A$ generated by $A_\mathbb R$ and all $i_s$, for $s\neq m$.  We define the conjugation $\theta_{i_m}\colon A\to A$ as follows: 
$$\theta_{i_m}(a_1+a_2i_m)=a_1-a_2i_m.$$ 
This is an automorphism of $A$ that preserves $A'=\mathrm{Fix}_A(\theta_{i_m})$, which we call the \emph{$i_m$-real locus}.

If $i_m$ is central in $A$, then the conjugation $\theta_{i_m}$ is never inner. However, if there exists $i_s$ such that $i_mi_s=-i_si_m$, then  $\theta_{i_m}(x)$ might be inner. For instance, if $A$ is the quaternionic extension of $A_\mathbb R$ by $i$ and $j$, then $\theta_i(x)=jxj^{-1}$, for all $x\in A$.

Notice that $A$ and $A_\mathbb R$ are given as algebras, yet there are many different ways to choose an independent collection of imaginary units inside $A$, so that $A$ is an extension of $A_\mathbb R$ in the sense of the discussion above. In this case, for every imaginary unit $i_m$, the involution $\theta_{i_m}$ is not determined by only $i_m$, but also by all other imaginary units in the chosen independent collection. For instance, to obtain the quaternionic extension, we can extend $A_\mathbb R$ by $i$ and $j$ or by $i$ and $k=ij$. In both ways, we obtain the same algebra $A=A_\mathbb R\otimes_{\mathbb R} \mathbb H$. However, in the first case it is $\theta_i(x)=jxj^{-1}$, and in the second case $\theta_i(x)=kxk^{-1}$, for $x\in A$, which are different maps. 

Finally, if $(A_\mathbb R,\sigma)$ is an involutive algebra and $A$ is an extension of $A_\mathbb R$ by commuting imaginary units $i_1,\dots,i_k$, there is a unique extension $\sigma_\mathbb C$ of $\sigma$ to $A$ such that $\sigma_\mathbb C(i_j)=i_j$ for all $j\in\{1,\dots, k\}$. However, there are also several extensions of $\sigma$ to $A$ which inverses the signs of some imaginary units. We denote by $\sigma^{i_{j_1},\dots, i_{j_s}}$ the extension of $\sigma$ such that $\sigma(i_{j_r})=-i_{j_r}$, for $r\in\{1,\dots, s\}$, and $\sigma(i_m)=i_m$, for all other imaginary units. If $k=1$, we denote $\bar\sigma_\mathbb C:=\sigma^{i_1}$.

\subsection{Sesquilinear forms on \texorpdfstring{$A$}{A}-modules and their groups of symmetries}\label{sec:sesq_forms}

Let $A$ be an involutive algebra and $V$ be a right $A$-module. An element $v\in V$ is called \emph{weakly regular} if the map $a\mapsto va$ is a homomorphism of right $A$-modules $A$ and $vA$. In this case, the right $A$-module $vA$ is called a \emph{weak $A$-line}. The space of all weak $A$-lines is denoted by $\mathbb P(V)$.

For $l\in\mathbb P(V)$, the tangent space at $l$ can be described as follows:
\begin{equation}\label{eq:projective.tangent}
    T_l\mathbb P(V)=\{[v,w]\mid l=vA,\; w\in V/l\},    
\end{equation}
where $[v,w]$ denotes the equivalence class under the following equivalence relation: $(v,w)\sim (v',w')$, if $v'=va$ and $w'=wa$, for some $a\in A^\times$. The space $T_l\mathbb{P}(V)$ is a right $A$-module. Indeed, we define 
$$
[v,w_1]a_1+[v,w_2]a_2:=[v,w_1a_1+w_2a_2],
$$ 
for $vA=l$ and $w_1,w_2\in w\in V/l$ with $a_1,a_2\in A$. It is easy to see that this definition does not depend on the choice of representatives. Sometimes it is convenient to choose a generator $v$ of $l$ and identify $T_l\mathbb P(V)$ with $V/l$. However, this identification is non-canonical.

We can also identify $T_l\mathbb P(V)$ with the space of $A$-linear operators $\mathrm{Hom}(l,V/l)$. For a tangent vector $[v,w]$, the corresponding operator maps $va\in l$ to $wa\in V/l$, for all $a\in A$. This identification is canonical. Therefore, we will sometimes identify (slightly abusing the notation) $[v,w]$ with an $A$-linear map $l\to V/l$ mapping $v\mapsto w$.

Let $(A,\sigma)$ be an involutive algebra. We begin with the following:

\begin{defn}\label{osp}
A \emph{$\sigma$-sesquilinear form} $\omega$ on a right $A$-module $V$ is a map
$$\omega\colon V\times V\to A,$$
such that for all $x,y,z\in V$ and for all $r_1,r_2\in A$:
\begin{align*} \omega(x+y,z) & =\omega(x,z)+\omega(y,z),\\
\omega(x,y+z) & =\omega(x,y)+\omega(x,z),\\
\omega(xr_1,yr_2) & =\sigma(r_1)\omega(x,y)r_2.
\end{align*}

\noindent We denote by $$\mathrm{Aut}(\omega):=\{f\in\mathrm{Aut}(V)\mid \text{ for every } x,y\in V, \omega(f(x),f(y))=\omega(x,y)\},$$ the \emph{group of symmetries} of $\omega$. We also define the corresponding Lie algebra:
$$\mathrm{End}(\omega):=\{f\in\mathrm{End}(V)\mid \text{ for every } x,y\in V,\omega(f(x),y)+\omega(x,f(y))=0\}$$
with the usual Lie bracket $[f,g]=fg-gf$.
\end{defn}

A form $\omega$ is called \emph{$\sigma$-skew symmetric} if $\omega(x,y)=-\sigma(\omega(y,x))$, for all $x,y\in V$.
It is called \emph{$\sigma$-symmetric} if $\omega(x,y)=\sigma(\omega(y,x))$, for all $x,y\in V$. If $(A,\sigma)$ is Hermitian, $\omega$ is called a \emph{$\sigma$-inner product} if it is $\sigma$-symmetric and for all regular elements $x\in V$, $\omega(x,x)\in A^\sigma_{\geq 0}$ and $\omega(x,x)=0$ if and only if $x=0$. Notice that the space of $\sigma$-inner products is an open proper convex cone inside the space of $\sigma$-symmetric forms.

\begin{prop}
    Let $\omega$ be a $\sigma$-inner product. An element $x\in V$ is weakly regular if and only if $\omega(x,x)\in A_+^\sigma$.  
\end{prop}

\begin{proof}
    Assume, $\omega(x,x)$ is not invertible, then by~\cite[Proposition~2.27]{ABRRW}, there exists $0
    \neq c\in A$ such that 
    $$0=\sigma(c)\omega(x,x)c=\omega(xc,xc),$$ 
    i.e., $xc=0$. This contradicts to the fact that the map $A\to xA$, $a\mapsto xa$ is an isomorphism.

    Assume $x$ is not weakly regular. Therefore, there exists $0 \neq c\in A$ such that $xc=0$. Therefore,  $0=\omega(x,xc)=\omega(x,x)c$, thus $\omega(x,x)$ is not invertible. 
\end{proof}

We now set $V=A^2$. We view $V$ as the set of columns and endow it with the structure of a right $A$-module. The anti-involution $\sigma$ extends to an involutive map on $A^2$ and on the space $\mathrm{Mat}_2(A)$ of $2\times2$-matrices with coefficients in $A$ componentwisely.

\begin{defn}
We make the following definitions:
    \begin{enumerate}
        \item A pair $(x,y)$, for $x,y\in A^2$, is called a \emph{basis} of $A^2$ if the map $A^2\to A^2$, $(a,b)\mapsto xa+yb$ is an isomorphism of right $A$-modules.
        
        \item An element $x\in A^2$ is called \emph{regular} if there exists $y\in A^2$ such that $(x,y)$ is a basis of $A^2$.

        \item A subset $l\subseteq A^2$ is called an \emph{$A$-line} (or just a line) if $l=xA$ for a regular $x\in A^2$.

        \item Two regular elements $x,y\in A^2$ are called \emph{linearly independent} if $(x,y)$ is a basis of $A^2$.
        
        \item Two lines $l,m$ are called \emph{transverse} if $l=xA$, $m=yA$ for linearly independent $x,y\in A^2$.

        \item A form $\omega$ is called \emph{non-degenerate} if for every regular $x\in V$ there exists $y\in V$ such that $\omega(x,y)\in A^\times$.
        
        \item An element $x\in A^2$ is called \emph{$\omega$-isotropic} (or just isotropic if $\omega$ is fixed) if $\omega(x,x)=0$. The set of all $\omega$-isotropic regular elements of $(A^2,\omega)$ is denoted by $\mathrm{Is}(\omega)$.
        
        \item A line $l$ is called \emph{$\omega$-isotropic} if $l=xA$, for a regular $\omega$-isotropic $x\in A^2$. The set of all isotropic lines of $(A^2,\omega)$ is denoted by $\mathbb{P}\mathrm{Is}(\omega)$.
    \end{enumerate}
\end{defn}

A regular element is clearly weakly regular. Moreover, if $x$ is regular and $y$ is weakly regular and there exists $a_0\in A$ such that $y=xa_0$, then $a_0\in A^\times$, i.e. $y$ is regular. Indeed, otherwise there exists $0 \neq c\in A$ such that $a_0c=0$ and the map $a\mapsto ya$ cannot be bijective. In particular, the space of lines is embedded into the space of weak lines. Since $A^\times$ is dense in $A$ (cf.~\cite[Proposition~2.30]{ABRRW}), every element in $A^2$ is a limit of a~sequence of regular elements, and therefore, the closure of the space of lines in the space of weak lines is the entire space of weak lines. However, if $A$ admits a Hermitian anti-involution, the space of all lines is compact (cf.~\cite[Corollary~4.2]{ABRRW}). This means that in this case the space of weak lines agrees with the space of lines. From now on, we only consider algebras that admit a Hermitian anti-involution and, therefore, we will not distinguish between lines and weak lines as well as between regular and weakly regular elements of $A^2$. The space of all (weak) lines of $A^2$ is denoted by $\mathbb P(A^2)$. The description of the tangent space of $T_l\mathbb P(A^2)$ for $l\in \mathbb P(A^2)$ agrees with~\eqref{eq:projective.tangent}.

Now we compute the tangent spaces of $\mathrm{Is}(\omega)$ and $\mathbb{P}\mathrm{Is}(\omega)$. First, we introduce some notation: Let $x\in \mathrm{Is}(\omega)$; we denote 
$$x^{\perp_{\omega}}:=\{y\in A^2\mid \omega(x,y)+\omega(y,x)=0\}.$$
Note that, in general, $x^{\perp_{\omega}}$ is not an $A$-module, but it is an $\mathbb R$-vector subspace of $A^2$ and $xA\subseteq x^{\perp_{\omega}}$. Moreover, if $a\in A^\times$, then $(xa)^{\perp_{\omega}}=x^{\perp_{\omega}}a$.

\begin{rem}\label{prop:tangent_skew}Given the definitions above, one immediately sees that 
if $\omega$ is $\sigma$-skew symmetric, and $x\in \mathrm{Is}(\omega)$, then 
$$
x^{\perp_{\omega}}=\left\{ y\in A^{2}\mid {{\omega }}(y,x)\in {A^{{\sigma}}}  \right\}.
$$ 
On the other hand, if $\omega$ is $\sigma$-symmetric, and $x\in \mathrm{Is}(\omega)$, then 
$$
x^{\perp_{\omega}}=\left\{ y\in A^{2}\mid {{\omega }}(y,x)\in {A^{{-\sigma}}}  \right\}.
$$   
\end{rem}

\begin{prop}\label{prop:tangent_gen_omega}
    Let $\omega$ be a $\sigma$-sesquilinear form and $x\in \mathrm{Is}(\omega)$ and $l=xA$. Then, the tangent spaces of $\mathrm{Is}(\omega)$ and $\mathbb{P}\mathrm{Is}(\omega)$ are described by
\begin{align}
\begin{aligned}
T_x\mathrm{Is}\left(\omega\right) & =x^{\perp_\omega};\\
T_l\mathbb{P}\mathrm{Is}\left(\omega\right) & =\{Q\in\mathrm{Hom}(l,A^2/l)\mid \omega(Q(x),x)+\omega(x,Q(x))=0\}\\
&=\{[v,w] \mid  vA=l,\;w\in v^{\perp_\omega}/l\}\\ 
&\cong x^{\perp_{\omega}}/l.
\end{aligned}
\end{align}
The group $\mathrm{Aut}(\omega)$ acts on $T\mathrm{Is}(\omega)$ and on $T\mathbb P\mathrm{Is}(\omega)$ by linear transformations: For $g\in \mathrm{Aut}(\omega)$ and for $(x,v)\in T\mathrm{Is}(\omega)$, then 
\[g.(x,v)=(gx,gv),\]
and 
\[g.(xA,v+xA)=(g(x)A,g(v)+g(x)A)\]
respectively.
\end{prop}

\begin{proof}
Let $x(t)\in \mathrm{Is}\left( {\omega } \right)$ be a smooth path. We have 
\[
 \omega(x(t),x(t))=0,
\]
thus the derivation by $t$ gives 
\[
\omega(\dot{x}(t),x(t))+\omega(x(t),\dot{x}(t))=0.
\]
\noindent This implies the tangent space of the space of isotropic elements is given by	
\begin{align*}
    {{T}_{x}}\mathrm{Is}\left( {{\omega }} \right)&=\left\{ y\in A^{2}\mid \omega(y,x) + \omega(x,y)=0  \right\}\\
    &=x^{\perp_{\omega}}.
\end{align*}
For the projection $\varphi\colon\mathrm{Is}\left( {{\omega }} \right)\to \mathbb{P}\mathrm{Is}\left( {{\omega }} \right)$, we have the differential 
$$d\varphi\colon T\mathrm{Is}\left( {{\omega }} \right)\to T\mathbb{P}\mathrm{Is}(\omega)$$ 
and the kernel
\begin{align*}
\ker d_x\varphi & =\left\{ y\in {{T}_{x}}\mathrm{Is}(\omega)\mid d_x\varphi(y)=0 \right\}\\
& =\left\{ y\in {{T}_{x}}\mathrm{Is}\left( {{\omega }} \right) \mid yA\subseteq x{{A}} \right\}\\
& =l.
\end{align*}
Therefore, we obtain that 
$$
{{T}_l}\mathbb{P}\mathrm{Is}\left( {{\omega }} \right)= \{(x,y)\mid xA=l,\; y\in A^2/l,\; \omega(x,y)+\omega(y,x)=0\}/A^\times,
$$ 
where $A^\times$ acts by simultaneous right multiplication in both  components. As before, we can see elements of ${{T}_l}\mathbb{P}\mathrm{Is}(\omega)$ as $A$-linear maps $Q\in\mathrm{Hom}(l,A^2/l)$ such that $$
\omega(Q(x),x)+\omega(x,Q(x))=0.
$$ 
If a generator $x$ of $l$ is fixed, we can identify this space with the quotient 
\begin{equation*}
    \{ y\in A^{2}\mid \omega(y,x)+\omega(x,y)=0\}/xA.\qedhere
\end{equation*}
\end{proof}

One can describe  the tangent space ${{T}_l}\mathbb{P}\mathrm{Is}\left( {{\omega }} \right)$ more explicitly in the cases when $\omega$ is $\sigma$-symmetric or $\sigma$-skew symmetric:
\begin{cor}\label{cor:tangent_skew_invariant}
We have the following description of the tangent space of the space of all isotropic lines:
\begin{enumerate}
    \item If $\omega$ is $\sigma$-skew symmetric, then for $l\in \mathbb P\mathrm{Is}(\omega)$, the tangent space ${{T}_l}\mathbb{P}\mathrm{Is}\left( {{\omega }} \right)$ is isomorphic to
    $$\{L\in\mathrm{Hom}(l,A^2/l)\mid \omega(L(x),x)\in A^\sigma\}.$$
    \item If $\omega$ is $\sigma$-symmetric, then for $l\in \mathbb P\mathrm{Is}(\omega)$, the tangent space ${{T}_l}\mathbb{P}\mathrm{Is}\left( {{\omega }} \right)$ is isomorphic to
    $$\{L\in\mathrm{Hom}(l,A^2/l)\mid \omega(L(x),x)\in A^{-\sigma}\}.$$
\end{enumerate}    
\end{cor}

\begin{rem}
    Since the description of the tangent space provided in Proposition~\ref{prop:tangent_gen_omega} is independent of the generator of the line $l$, we may identify the tangent space $T_l\mathbb{P}\mathrm{Is}\left( {{\omega }} \right)$ with the space 
    \[\{Q\in\mathrm{Hom}(l,A^2/l)\mid \omega(Q(x),x)+\omega(x,Q(x))=0\}.\]
\end{rem}

\subsection{Duality of \texorpdfstring{$A$}{A}-modules}\label{sec:duality}

Let $A$ be a real algebra and let $V$ be a right $A$-module. A map $f\colon V\to A$ is called \emph{right $A$-linear} if $f(v_1a_1+v_2a_2)=f(v_1)a_1+f(v_2)a_2$, for all $v_1,v_2\in V$, $a_1,a_2\in A$. We denote by $V^*$ the space of all right $A$-linear maps on $V$. This is naturally a left $A$-module. Indeed, for $f\in V^*$ and $a\in A$, $af\colon V\to A$ is defined as $(af)(v)=a(f(v))$. We call $V^*$ the right dual module to $V$.

Similarly, for a left $A$-module $V$, the left dual module ${}^*V$ is well-defined and is a right $A$-module.

In the following consideration, we assume $V$ to be a right $A$-module. However, a similar consideration holds for left $A$-modules.

If $(A,\sigma)$ is an involutive algebra, then every right $A$-module $V$ can be turned into a left $A$-module and vice versa as follows: if $V$ is a right $A$-module, for $a\in A$ and $v\in V$, we define the left $\sigma$-twisted multiplication $a\cdot v:=v\sigma(a)$. We denote by $\sigma(V)$ the corresponding left $A$-module. However, $V$ with the right and $\sigma$-twisted left $A$-multiplications is in general not an $A$-bimodule because in general it is $a\cdot(va')=va'\sigma(a)\neq v\sigma(a)a'=(a\cdot v)a'$. 

\begin{rem}
    We always use $\cdot$ to emphasize the $\sigma$-twisted multiplication, whereas we omit the multiplication sign for the multiplication in $V$.
\end{rem}

The anti-involution $\sigma$ induces maps $V^*\to {}^*\sigma(V)$ and ${}^*\sigma(V)\to V^*$ mapping $f\in V^*\cup {}^*\sigma(V)$ to $\sigma\circ f$. Indeed, if $f\in V^*$ then $\sigma\circ f$ is additive and for every $v\in\sigma(V)$, $a\in A$
$$(\sigma\circ f)(a\cdot v)=(\sigma\circ f)(v\sigma(a))=\sigma(f(v)\sigma(a))=a(\sigma\circ f(v)),$$
i.e. $\sigma\circ f\in {}^*\sigma(V)$.

Further, every $A$-morphism between two left $A$-modules $f\in\mathrm{Hom}_A(V,W)$ naturally determines an element in $\mathrm{Hom}_A(\sigma(V),\sigma(W))$ which by a slight abuse of notation will be also denoted by $f$. Indeed, for $v\in\sigma(V)$, $a\in A$ we have:
$$f(a\cdot v)=f(v\sigma(a))=f(v)\sigma(a)=a\cdot f(v).$$
In other words, a map $f\colon V\to V$ preserves the structure of a right $A$-module on $V$ if and only if it preserves the induced structure of a $\sigma$-twisted left $A$-module on $V$. 

\begin{prop}
    Let $V=A^2$ and $\omega$ be a non-degenerate $\sigma$-sesquilinear form. Let $\ell_1$ and $\ell_2$ be two transverse $\omega$-isotropic $A$-lines in $V$. Then $\omega$ provides an isomorphism of right $A$-modules:
    $$\begin{array}{rcl}
    \Phi_\omega\colon \ell_j & \to & {}^*\sigma(\ell_i)\\
    v & \mapsto & v_\omega:=\omega(\cdot,v)
    \end{array}$$
    for $i,j\in\{1,2\}$ and $i\neq j$, i.e., for $v\in\ell_j$ and $w\in\sigma(\ell_i)$, $v_\omega(w):=\omega(w,v)$. 
\end{prop}

\begin{proof}
    First, the map $\Phi_\omega$ is well-defined: for $v\in\ell_j$, the map $v_\omega\in {}^*\sigma(\ell_i)$ because for $w_1,w_2\in\sigma(\ell_i)$ and $a_1,a_2\in A$ we have
    $$v_\omega(a_1\cdot w_1+a_2\cdot w_2)=\omega(w_1\sigma(a_1)+w_2\sigma(a_2),v)=a_1\omega(w_1,v)+a_2\omega(w_2,v)=a_1v_\omega(w_1)+a_2v_\omega(w_2).$$
    Further, the map $\Phi_\omega$ is additive and for $a\in A$ and $w\in \sigma(\ell_i)$ we have $$(va)_\omega(w)=\omega(w,va)=\omega(w,v)a=v_\omega(w)a.$$
    Thus, $\Phi_\omega$ is a homomorphism. Finally, since $\omega$ is non-degenerate and $\ell_1$, $\ell_2$ are isotropic, the kernel of $\Phi_\omega$ is trivial.     
\end{proof}
 
In this case, we say that \emph{$\omega$ identifies $\ell_j$ as dual to $\ell_i$}.

\subsection{Groups over Hermitian algebras} 

In the sequel of the article, we will be only considering algebras $A$ admitting a Hermitian anti-involution $\sigma$.
From now on, we assume 
\[\omega(x,y):=\sigma(x)^t\Omega y,\] with either $\Omega=\begin{pmatrix}0 & 1 \\ -1 & 0\end{pmatrix}$ or $\Omega=\begin{pmatrix}0 & 1 \\ 1 & 0\end{pmatrix}$, for $x,y\in A^2$. In the first case, the form $\omega$ is called the \emph{standard symplectic form} on $A^2$; in the second case, it is called the \emph{standard indefinite orthogonal form} on $A^2$.

\begin{rem}
    More generally, we say that a $\sigma$-sesquilinear form $\omega$ on a right $A$-module $V$ is \emph{$\sigma$-symplectic}, resp. \emph{$\sigma$-indefinite orthogonal}, if there is a basis $v=(v_1,v_2)$ of $V$ such that the matrix $(\omega(v_i,v_j))_{i,j}$ agrees with $\begin{pmatrix}0 & 1 \\ -1 & 0\end{pmatrix}$, resp. $\begin{pmatrix}0 & 1 \\ 1 & 0\end{pmatrix}$.
\end{rem}

\begin{defn}
Let $(A,\sigma)$ be an involutive algebra. If $\omega$ is the standard symplectic form, then the group $\mathrm{Sp}_2(A,\sigma):=\mathrm{Aut}(\omega)$ is called the \emph{symplectic group} $\mathrm{Sp}_2$ over $(A,\sigma)$.
    If $\omega$ is the standard indefinite orthogonal form, then the group $\mathrm{O}_{(1,1)}(A,\sigma):=\mathrm{Aut}(\omega)$ is called the \emph{indefinite orthogonal group} $\mathrm{O}_{(1,1)}$ over $(A,\sigma)$.
\end{defn}

We have
\begin{align}\label{eq:groups}
    \begin{aligned}
    \mathrm{Sp}_2(A,\sigma)&=\left\{\begin{pmatrix}
    a & b \\
    c & d
    \end{pmatrix}\in\mathrm{GL}_2(A)\midwd \sigma(a)c,\,\sigma(b)d\in A^\sigma,\,\sigma(a)d-\sigma(c)b=1\right\}, \text{ and }\\[2pt]
    \mathrm{O}_{(1,1)}(A,\sigma)&=\left\{\begin{pmatrix}
    a & b \\
    c & d
    \end{pmatrix}\in \mathrm{GL}_2(A)\midwd \sigma(a)c,\,\sigma(b)d\in A^{-\sigma},\,\sigma(a)d+\sigma(c)b=1\right\}.
    \end{aligned}
\end{align}
We can determine the Lie algebras $\mathfrak{sp}_2(A,\sigma)$ of $\mathrm{Sp}_2(A,\sigma)$ and $\mathfrak o_{(1,1)}(A,\sigma)$ of $\mathrm O_{(1,1)}(A,\sigma)$:
\begin{align}\label{eq:Lie_algebras}
    \begin{aligned}
	    \mathfrak{sp}_2(A,\sigma)& =\left\{g\in \mathrm{Mat}_2(A)\midwd \sigma(g)^t\begin{pmatrix}
        0 & 1 \\
        -1 & 0
      \end{pmatrix}+ \begin{pmatrix}
        0 & 1 \\
        -1 & 0
      \end{pmatrix} g=0\right\}\\
    & =\left\{\begin{pmatrix}
		x & z \\
		y & -\sigma(x)
	\end{pmatrix}\midwd x\in A,\;y,z\in A^\sigma\right\}, \text{ and }\\[2pt]
        \mathfrak o_{(1,1)}(A,\sigma)& =\left\{g\in \mathrm{Mat}_2(A)\midwd \sigma(g)^t\begin{pmatrix}
        0 & 1 \\
        1 & 0
      \end{pmatrix}+ \begin{pmatrix}
        0 & 1 \\
        1 & 0
      \end{pmatrix} g=0\right\}\\
      & =\left\{\begin{pmatrix}
		x & z \\
		y & -\sigma(x)
	\end{pmatrix}\midwd x\in A,\;y,z\in A^{-\sigma}\right\}.
    \end{aligned}
\end{align}

Let now $(A,\sigma)$ be Hermitian. We consider the subgroup 
$$\mathrm{O}_2(A,\sigma):=\{g\in\mathrm{GL}_2(A)\mid \sigma(g)^tg=\mathrm{Id}\},$$ 
which is compact. Then according to {\cite[Proposition 3.18, Theorem 3.19]{ABRRW}}, one has that the groups 
\begin{align}\label{max_cpt_real}
    \begin{aligned}
    \mathrm{KSp}_2(A,\sigma):=&\mathrm{O}_2(A,\sigma)\cap\mathrm{Sp}_2(A,\sigma)\\
    =&\left\{\begin{pmatrix}
    a&b\\
    -b&a
    \end{pmatrix}\in\mathrm{GL}_2(A)\midwd \sigma(a)a+\sigma(b)b=1,\ \sigma(a)b\in A^{\sigma}\right\}, \text{ and }\\[2pt]
    \mathrm{KO}_{(1,1)}(A,\sigma):=&\mathrm{O}_2(A,\sigma)\cap\mathrm{O}_{(1,1)}(A,\sigma)\\
    =&\left\{\begin{pmatrix}
    a&b\\
    b&a
    \end{pmatrix}\in\mathrm{GL}_2(A)\midwd \sigma(a)a+\sigma(b)b=1,\ \sigma(a)b\in A^{-\sigma}\right\}
\end{aligned}
\end{align}
are  maximal compact subgroups of $\mathrm{Sp}_2(A,\sigma)$, resp. of $\mathrm{O}_{(1,1)}(A,\sigma)$. From \eqref{eq:Lie_algebras} and \eqref{max_cpt_real}, we obtain their Lie algebras
\begin{align}\label{eq2.2}
    \begin{aligned}
    \mathfrak{ksp}_2(A,\sigma):=&\left\{\begin{pmatrix}
    a&b\\
    -b&a
    \end{pmatrix}\in\mathrm{Mat}_2(A)\midwd \sigma(a)=-a,\ \sigma(b)=b\right\}, \text{ and }\\[2pt]
    \mathfrak{ko}_{(1,1)}(A,\sigma):=&\left\{\begin{pmatrix}
    a&b\\
    b&a
    \end{pmatrix}\in\mathrm{Mat}_2(A)\midwd \sigma(a)=-a,\ \sigma(b)=-b\right\}.
    \end{aligned}
\end{align}
This gives the Cartan decompositions of the Lie algebras $\mathfrak{sp}_2(A,\sigma)$ and $\mathfrak{o}_{(1,1)}(A,\sigma)$ as
\begin{align}\label{Cartan_real}
\begin{aligned}
    \mathfrak{sp}_2(A,\sigma)= & \,\,\mathfrak{ksp}_2(A,\sigma)\oplus\mathfrak{msp}_2(A,\sigma), \text{ and }\\
    \mathfrak{o}_{(1,1)}(A,\sigma)=& \,\,\mathfrak{ko}_{(1,1)}(A,\sigma)\oplus\mathfrak{mo}_{(1,1)}(A,\sigma),
\end{aligned}
\end{align}
where
\begin{align}\label{eq1.6}
\begin{aligned}
    \mathfrak{msp}_2(A,\sigma):=&\left\{\begin{pmatrix}
    a & b\\
    b & -a
    \end{pmatrix}\in\mathrm{Mat}_2(A)\midwd \sigma(a)=a,\ \sigma(b)=b\right\}, \text{ and }\\[2pt]
    \mathfrak{mo}_{(1,1)}(A,\sigma):=&\left\{\begin{pmatrix}
    a & b\\
    -b & -a
    \end{pmatrix}\in\mathrm{Mat}_2(A)\midwd \sigma(a)=a,\ \sigma(b)=-b\right\}.
\end{aligned}
\end{align}

\subsection{Groups over complexified Hermitian algebras}\label{group_complexified_hermitian}

Let $(A,\sigma)$ be a semisimple Hermitian $\mathbb{R}$-algebra as above, and let $A_\mathbb{C}$ be the complexification of $A$, together with the $\mathbb{C}$-linear and $\mathbb{C}$-anti-linear anti-involutions $\sigma_\mathbb{C}$ and $\bar\sigma_{\mathbb{C}}$, respectively. Then,  $(A_\mathbb{C},\bar\sigma_\mathbb{C})$ is a Hermitian and semisimple $\mathbb{C}$-algebra (see \cite[Corollary 2.79]{ABRRW}). 

Further, consider the group $\mathrm{Sp}_2(A_\mathbb{C},\sigma_\mathbb{C})$. By {\cite[Proposition 3.22, Theorem 3.23]{ABRRW}}, the group 
\begin{align}\label{max_cpt_cplx}
    \begin{aligned}
    K^c :=&\mathrm{KSp}^c_2(A_\mathbb C,\sigma_\mathbb C)\\
    :=&\ \mathrm{O}_2(A_\mathbb{C},\bar \sigma_\mathbb{C})\cap\mathrm{Sp}_2(A_\mathbb{C},\sigma_\mathbb{C})\\
    =&\left\{\begin{pmatrix}
    a&b\\
    -\bar b&\bar a
    \end{pmatrix}
    \in\mathrm{GL}_2(A_\mathbb{C})\midwd \bar\sigma_\mathbb{C}(a)a+\sigma_\mathbb{C}(b)\bar b=1,\ \bar\sigma_\mathbb{C}(a)b-\sigma_\mathbb{C}(b)\bar a=0\right\}\\
    = & \left\{\begin{pmatrix}
    a & b\\
    -\bar b & \bar a
    \end{pmatrix}    \in\mathrm{GL}_2(A_\mathbb{C})
    \midwd a\bar\sigma_\mathbb{C}(a)+b\bar\sigma_\mathbb{C}(b)=1,\ b\sigma_\mathbb{C}(a)-a\sigma_\mathbb{C}(b)=0\right\}
\end{aligned}
\end{align}
is a maximal compact subgroup of $\mathrm{Sp}_2(A_\mathbb{C},\sigma_\mathbb{C})$, where 
$$\mathrm{O}_2(A_\mathbb{C},\bar \sigma_\mathbb{C}):=\{M\in\mathrm{Mat}_2(A_\mathbb{C})\mid \bar \sigma_\mathbb{C}(M)^tM=\mathrm{Id}\}.$$ Note that, we are adopting the notation $K^c$ here, in contrast to the notation $K$ for the maximal compact subgroup of $\mathrm{Sp}_2(A,\sigma)$. From \eqref{eq:Lie_algebras} and \eqref{max_cpt_cplx}, we obtain the following expression of the Lie algebra of $K^c$
\begin{align}\label{eq1.8}
    \mathfrak{k}^c:=\mathfrak{ksp}^c_2(A_\mathbb C,\sigma_\mathbb C):=&\ \mathrm{Lie}(K^c)\\
    =& \left\{\begin{pmatrix}
    a&b\\
    -\bar b&\bar a
    \end{pmatrix}\in\mathrm{Mat}_2(A_\mathbb{C})\midwd \bar\sigma_\mathbb{C}(a)=-a,\ \sigma_\mathbb{C}(b)=b\right\}.
\end{align}
This gives the Cartan decomposition of the Lie algebra $\mathfrak{sp}_2(A_\mathbb{C},\sigma_\mathbb{C})$ as
\begin{align}\label{Cartan_cplx}
    \mathfrak{sp}_2(A_\mathbb{C},\sigma_\mathbb{C})=\mathfrak{k}^c\oplus\mathfrak{m}^c,
\end{align}
where
\begin{align}\label{eq1.10}
    \mathfrak{m}^c=\left\{\begin{pmatrix}
    a&b\\
    \bar b&-\bar a
    \end{pmatrix}\in\mathrm{Mat}_2(A_\mathbb{C})\midwd \bar\sigma_\mathbb{C}(a)=a,\ \sigma_\mathbb{C}(b)=b\right\}.
\end{align}
We should stress here the fact that the decomposition \eqref{Cartan_cplx} is never the complexification of the decomposition \eqref{Cartan_real}. The spaces $\mathfrak m^c$ and $\mathfrak k^c$ are $\mathbb R$-vector spaces, but not over $\mathbb C$. 

Finally, we discuss the group $\mathrm{O}_{(1,1)}(A_\mathbb C,\sigma_{\mathbb C})$. This group is isomorphic to $\mathrm{O}(A'_\mathbb C,\sigma^t_{\mathbb C})$ where $A'=\mathrm{Mat}_2(A)$ and $\sigma^t$ is the transposition composed with componentwise application of $\sigma$. This isomorphism is given by the following base change: 
\[e_1:=\frac{\exp(-i\pi/4)}{\sqrt{2}i}(1,i)^t, \quad  e_2:=\frac{\exp(-i\pi/4)}{\sqrt{2}i}(i,1)^t,\] 
in other words, for 
 $T=\frac{\exp(-i\pi/4)}{\sqrt{2}i}\begin{pmatrix}
    1 & i \\
    i & 1
\end{pmatrix}$, then
$$T^{-1}\mathrm{O}_{(1,1)}(A_\mathbb C,\sigma_{\mathbb C})T=\mathrm{O}(A'_\mathbb C,\sigma^t_{\mathbb C}).$$
Its maximal compact subgroup is $\mathrm{O}(A',\sigma^t)=\mathrm{O}(A'_\mathbb C,\sigma^t_{\mathbb C})\cap \mathrm{O}(A'_\mathbb C,\bar\sigma^t_{\mathbb C})$.

\subsection{Incarnations of \texorpdfstring{$\mathrm{Sp}_{2}(A,\sigma)$ and $\mathrm{O}_{(1,1)}(A,\sigma)$}{}}\label{sec:incarnations}

The elements 
$$e_1=\frac{1}{\sqrt{2}}(1,-1)^t \quad \text{and} \quad  e_2=\frac{1}{\sqrt{2}}(1,1)^t$$
form a basis of $A^2$. In this basis, the standard indefinite orthogonal form $\omega$ is represented by the matrix $\Omega=\begin{pmatrix}
    -1 & 0 \\
    0 & 1
\end{pmatrix}$. This change of basis conjugates the group $\mathrm{O}_{(1,1)}(A,\sigma)$ by the matrix $R=\frac{1}{\sqrt{2}}\begin{pmatrix}
    1 & 1 \\
    -1 & 1
\end{pmatrix}$.

Then, one sees
\begin{align}\label{eq:conj_orth}
\begin{aligned}
    R^{-1}\mathrm{O}_{(1,1)}(A,\sigma)R&=\left\{\begin{pmatrix}
    a & b \\
    c & d
    \end{pmatrix}\midwd
    \begin{aligned}
    &\ \sigma(a)b-\sigma(c)d=0,\\
    &\ \sigma(a)a-\sigma(c)c=\sigma(d)d-\sigma(b)b=1
    \end{aligned}\right\}\\[2pt]
    & \subseteq \mathrm{GL}_2(A),\\[2pt]
    R^{-1}\mathrm{KO}_{(1,1)}(A,\sigma)R&=\left\{\begin{pmatrix}
    a & 0 \\
    0 & d
    \end{pmatrix}\midwd a,d\in \mathrm{O}{(A,\sigma)}\right\}\\[2pt]
    & \cong \mathrm O{(A,\sigma)}\times \mathrm O{(A,\sigma)}, \text{ and }\\[2pt]
    R^{-1}\mathfrak{mo}_{(1,1)}(A,\sigma)R&=\left\{\begin{pmatrix}
    0 & b \\
    \sigma(b) & 0
    \end{pmatrix}\midwd b\in A \right\} \cong A.
\end{aligned}
\end{align}

From the above descriptions, we now obtain the following:
\begin{prop}\label{prop:O_(1,1)-connected}
    Let $(A,\sigma)$ be Hermitian. If $A^\times$ has $n$ connected components, then $\mathrm{O}_{(1,1)}(A,\sigma)$ has $n^2$ connected components. In particular, $\mathrm{O}_{(1,1)}(A,\sigma)$ is connected if and only if $A^\times$ is connected.
\end{prop}

\begin{proof}
    A maximal compact subgroup of $A^\times$ is $\mathrm{O}{(A,\sigma)}$ which has the same number of connected components as $A^\times$ because it is a strong deformation retract of $A^\times$. A maximal compact subgroup of $\mathrm{O}_{(1,1)}(A,\sigma)$ is isomorphic to $\mathrm{O}{(A,\sigma)}\times \mathrm{O}{(A,\sigma)}$ which has $n^2$ connected components. Thus, $\mathrm{O}_{(1,1)}(A,\sigma)$ has the same number of connected components.
\end{proof}

Similarly, after the complexification of $A^2$, elements $e_1=\frac{1}{\sqrt{2}}(1,i)^t$, $e_2=\frac{1}{\sqrt{2}}(i,1)^t$ form a basis of $A_\mathbb C^2$. In this basis, the complex linear extension $\omega_{\mathbb C}$ of the standard symplectic form $\omega$ is represented by the matrix $\Omega=\begin{pmatrix}
    0 & 1 \\
    -1 & 0
\end{pmatrix}$. This change of basis conjugates the group $\mathrm{Sp}_{2}(A,\sigma)$ by the matrix $T=\frac{1}{\sqrt{2}}\begin{pmatrix}
    1 & i \\
    i & 1
\end{pmatrix}$. In particular, a maximal compact subgroup conjugates as follows:

\begin{align}\label{eq:conj_symp_real}
\begin{aligned}
     T^{-1}\mathrm{KSp}_{2}(A,\sigma)T&=\left\{\begin{pmatrix}
    u & 0 \\
    0 & \bar u
    \end{pmatrix}\midwd u\in \mathrm{O}{(A_\mathbb C,\bar\sigma_{\mathbb C})}\right\}\cong \mathrm O{(A_\mathbb C,\bar\sigma_{\mathbb C})}.
\end{aligned}
\end{align}

We then have the following:
\begin{prop}\label{prop:Sp_2-connected}
    Let $(A,\sigma)$ be Hermitian. The group $\mathrm{Sp}_{2}(A,\sigma)$ is connected.
\end{prop}

\begin{proof}
    By~\cite[Corollary~2.7.24]{Rog}, $\mathrm{O}{(A_\mathbb C,\bar\sigma_\mathbb C)}$ is connected. A maximal compact subgroup of $\mathrm{Sp}_{2}(A,\sigma)$ is isomorphic to $\mathrm{O}{(A_\mathbb C,\bar\sigma_\mathbb C)}$. Thus $\mathrm{Sp}_{2}(A,\sigma)$ is connected because $\mathrm{O}{(A_\mathbb C,\bar\sigma_\mathbb C)}$ is a strong deformation retract of $\mathrm{Sp}_{2}(A,\sigma)$.
\end{proof}

Finally, when we quaternionify $A^2$, the elements $e_1=\frac{1}{\sqrt{2}}(1,j)^t$, $e_2=\frac{1}{\sqrt{2}}(j,1)^t$ form a basis of $A_\mathbb H^2$. In this basis, the quaternionic $\sigma_0$-sesquilinear extension $\omega_{\mathbb H}$ of the standard symplectic form $\omega$ is represented by the matrix $\Omega=\begin{pmatrix}
    0 & 1 \\
    -1 & 0
\end{pmatrix}$. This change of basis conjugates the group $\mathrm{Sp}_{2}(A_\mathbb C,\sigma_\mathbb C)$ by the matrix $Q=\frac{1}{\sqrt{2}}\begin{pmatrix}
    1 & j \\
    j & 1
\end{pmatrix}$. In particular, a maximal compact subgroup conjugates as follows:

\begin{align}\label{eq:conj_symp_cpx}
\begin{aligned}
    Q^{-1}\mathrm{KSp}^c_{2}(A_\mathbb C,\sigma_\mathbb C)Q&=\left\{\begin{pmatrix}
    u & 0 \\
    0 & (ij)u(ij)^{-1}
    \end{pmatrix}\midwd u\in \mathrm{O}{(A_\mathbb H,\sigma_1)}\right\}\cong \mathrm{O}{(A_\mathbb H,\sigma_1)}.
\end{aligned}
\end{align}

Now we assume the involutive algebra $(A,\sigma)$ is Hermitian and semisimple. From the Artin--Wedderburn Theorem, it follows that
\begin{equation}\label{Artin_wed_A}
A \cong \bigoplus_{i=1}^p \mathrm{Mat}_{l_i}\mathbb{R} \oplus \bigoplus_{j=1}^q \mathrm{Mat}_{m_j}\mathbb{C} \oplus \bigoplus_{k=1}^r \mathrm{Mat}_{n_k}\mathbb{H}.
\end{equation}

We then have the following:

\begin{thm}\label{prop:max_semisimple_simply_con}
Let $(A, \sigma)$ be a Hermitian semisimple algebra, and consider the integers $q$ and $r$ from \eqref{Artin_wed_A}. For the group $\mathrm{Sp}_2(A_\mathbb{C}, \sigma_{\mathbb{C}})$, we have the following statements: 
\begin{enumerate}
\item The number of connected components of the group $\mathrm{Sp}_2(A_\mathbb{C}, \sigma_{\mathbb{C}})$ is $2^r$.
\item For $\mathrm{Sp}_2(A_\mathbb{C}, \sigma_{\mathbb{C}})_0$, the connected component of the identity of the group $\mathrm{Sp}_2(A_\mathbb{C}, \sigma_{\mathbb{C}})$, its fundamental group is described by
\[\pi_1\left(\mathrm{Sp}_2(A_\mathbb{C}, \sigma_{\mathbb{C}})_0\right)\cong \mathbb{Z}^q\times (\mathbb{Z}/2\mathbb{Z})^r.\] 
\item For $\mathrm{Sp}_2(A_\mathbb{C}, \sigma_{\mathbb{C}})_0^{ss}$, the maximal semisimple subgroup of $\mathrm{Sp}_2(A_\mathbb{C}, \sigma_{\mathbb{C}})_0$, its fundamental group is described by
\[\pi_1\left(\mathrm{Sp}_2(A_\mathbb{C}, \sigma_{\mathbb{C}})_0^{ss}\right)\cong  (\mathbb{Z}/2\mathbb{Z})^r.\]
\end{enumerate}
\end{thm}

\begin{proof}
(1) It follows from \eqref{eq:conj_symp_cpx} that a maximal compact subgroup of $\mathrm{Sp}_2(A_\mathbb{C}, \sigma_{\mathbb{C}})$ is the group $\mathrm{O}(A_{\mathbb{H}}, \sigma_1)$. At the same time, the group $\mathrm{O}(A_{\mathbb{H}}, \sigma_1)$ is a maximal compact subgroup of the group $A^{\times}_{\mathbb{H}}$, which is the group of invertible elements of the quaternionic algebra $A_{\mathbb{H}}$, which is semisimple because $(A_{\mathbb{H}},\sigma_1)$ is Hermitian and semisimple (see \cite[Corollary 2.93]{ABRRW}).

\noindent From \eqref{Artin_wed_A}, we now get
\[A_{\mathbb{H}} \cong \bigoplus_{i=1}^p \mathrm{Mat}_{l_i}\mathbb{H} \oplus \bigoplus_{j=1}^q \mathrm{Mat}_{2m_j}\mathbb{C} \oplus \bigoplus_{k=1}^r \mathrm{Mat}_{4n_k}\mathbb{R},\]
where we used the isomorphisms 
\[\mathrm{Mat}_{m_j}\mathbb{H}\otimes_{\mathbb{R}}\mathbb{C} \cong \mathrm{Mat}_{2m_j}\mathbb{C} \quad \text{and} \quad 
\mathrm{Mat}_{n_k}\mathbb{H}\otimes_{\mathbb{R}}\mathbb{H} \cong \mathrm{Mat}_{4n_k}\mathbb{R},\] for any factor $j$ and any factor $k$ respectively 
(see \cite[Appendix B.1]{ABRRW}).
A maximal compact subgroup of $A_{\mathbb{H}}^{\times}$ is then more explicitly described by the  group 
\[\mathrm{O}(A_{\mathbb{H}}, \sigma_1) \cong \prod_{i=1}^p \mathrm{Sp}(l_i) \times \prod_{j=1}^q \mathrm{U}(2m_j) \times\prod_{k=1}^r \mathrm{O}(4n_k).\]
Then, as topological spaces, $A_{\mathbb{H}}^{\times}$ and $\mathrm{Sp}_2(A_\mathbb{C}, \sigma_{\mathbb{C}})$ are homotopy equivalent to the space $\mathrm{O}(A_{\mathbb{H}}, \sigma_1)$. Therefore, the number of connected components of the group $\mathrm{Sp}_2(A_\mathbb{C}, \sigma_{\mathbb{C}})$ is equal to $2^r$.

(2) The preceding argument applies step by step for the case of the component of the identity $\mathrm{Sp}_2(A_\mathbb{C}, \sigma_{\mathbb{C}})_0$, which is homotopy equivalent to its maximal compact subgroup \[\mathrm{O}(A_{\mathbb{H}}, \sigma_1)_0 \cong \prod_{i=1}^p \mathrm{Sp}(l_i) \times \prod_{j=1}^q \mathrm{U}(2m_j) \times\prod_{k=1}^r \mathrm{SO}(4n_k).\]  
We then analogously get 
\begin{align*}
\pi_1\left( \mathrm{Sp}_2(A_\mathbb{C}, \sigma_{\mathbb{C}})_0\right) & \cong \pi_1\left(\mathrm{O}(A_{\mathbb{H}}, \sigma_1)_0\right)\\
& \cong \pi_1\left( \prod_{i=1}^p \mathrm{Sp}(l_i) \times \prod_{j=1}^q \mathrm{U}(2m_j) \times\prod_{k=1}^r \mathrm{SO}(4n_k)\right)\\
& \cong \mathbb{Z}^q \times (\mathbb{Z}/2\mathbb{Z})^r.
\end{align*}

(3) For the third part of the theorem, we similarly get that the maximal semisimple subgroup $\mathrm{Sp}_2(A_\mathbb{C}, \sigma_{\mathbb{C}})_0^{ss}$ is homotopy equivalent to its maximal compact subgroup \[\mathrm{O}(A_{\mathbb{H}}, \sigma_1)_0^{ss} \cong \prod_{i=1}^p \mathrm{Sp}(l_i) \times \prod_{j=1}^q \mathrm{SU}(2m_j) \times\prod_{k=1}^r \mathrm{SO}(4n_k),\]  
and so we analogously get 
\begin{align*}
\pi_1\left( \mathrm{Sp}_2(A_\mathbb{C}, \sigma_{\mathbb{C}})_0^{ss}\right) & \cong \pi_1\left(\mathrm{O}(A_{\mathbb{H}}, \sigma_1)_0^{ss}\right)\\
& \cong \pi_1\left( \prod_{i=1}^p \mathrm{Sp}(l_i) \times \prod_{j=1}^q \mathrm{SU}(2m_j) \times\prod_{k=1}^r \mathrm{SO}(4n_k)\right)\\
& \cong (\mathbb{Z}/2\mathbb{Z})^r. \qedhere
\end{align*}
\end{proof}

\begin{thm}\label{prop:max_semisimple_simply_con_O}
Let $(A, \sigma)$ be a Hermitian semisimple algebra, and consider the integers $p$ and $q$ from \eqref{Artin_wed_A}, and denote by $p_2:=\#\{ i\mid l_i=2\}$ and $p_{\geq3}:=\#\{i\mid l_i\geq3\}$. For the group $\mathrm{O}(A_\mathbb{C}, \sigma_{\mathbb{C}})$, we have the following statements: 
\begin{enumerate}
\item The number of connected components of the group $\mathrm{O}(A_\mathbb{C}, \sigma_{\mathbb{C}})$ is $2^p$.
\item For $\mathrm{O}(A_\mathbb{C}, \sigma_{\mathbb{C}})_0$, the connected component of the identity of the group $\mathrm{O}(A_\mathbb{C}, \sigma_{\mathbb{C}})$, its fundamental group is described by
\[\pi_1\left(\mathrm{O}(A_\mathbb{C}, \sigma_{\mathbb{C}})_0\right)\cong (\mathbb{Z}/2\mathbb{Z})^{p_{\geq3}}\times \mathbb{Z}^{q+p_2}.\] 
\item For $\mathrm{O}(A_\mathbb{C}, \sigma_{\mathbb{C}})_0^{ss}$, the maximal semisimple subgroup of $\mathrm{O}(A_\mathbb{C}, \sigma_{\mathbb{C}})_0$, its fundamental group is described by
\[\pi_1\left(\mathrm{O}(A_\mathbb{C}, \sigma_{\mathbb{C}})_0^{ss}\right)\cong  (\mathbb{Z}/2\mathbb{Z})^{p_{\geq3}}.\]
\end{enumerate}
\end{thm}

\begin{proof}
(1) A maximal compact subgroup of $\mathrm{O}(A_\mathbb{C}, \sigma_{\mathbb{C}})$ is the group $\mathrm{O}(A, \sigma)$. At the same time, the group $\mathrm{O}(A, \sigma)$ is a maximal compact subgroup of the group $A^{\times}$, which is the group of invertible elements of the algebra $A$, which is semisimple.

\noindent A maximal compact subgroup of $A^{\times}$ is then more explicitly described by the  group 
\[\mathrm{O}(A, \sigma) \cong \prod_{i=1}^p \mathrm{O}(l_i) \times \prod_{j=1}^q \mathrm{U}(m_j) \times\prod_{k=1}^r \mathrm{Sp}(n_k).\]
Then, as topological spaces, $A^{\times}$ and $\mathrm{O}(A_\mathbb{C}, \sigma_{\mathbb{C}})$ are homotopy equivalent to the space $\mathrm{O}(A, \sigma)$. Therefore, the number of connected components of the group $\mathrm{O}(A_\mathbb{C}, \sigma_{\mathbb{C}})$ is equal to $2^p$.

(2) The preceding argument applies step by step for the case of the component of the identity $\mathrm{O}(A_\mathbb{C}, \sigma_{\mathbb{C}})_0$, which is homotopy equivalent to its maximal compact subgroup 
\[\mathrm{O}(A, \sigma)_0 \cong \prod_{i=1}^p \mathrm{SO}(l_i) \times \prod_{j=1}^q \mathrm{U}(m_j) \times\prod_{k=1}^r \mathrm{Sp}(n_k).\]  
We then analogously get 
\begin{align*}
\pi_1\left( \mathrm{O}(A_\mathbb{C}, \sigma_{\mathbb{C}})_0\right) & \cong \pi_1\left(\mathrm{O}(A, \sigma)_0\right)\\
& \cong \pi_1\left( \prod_{i=1}^p \mathrm{SO}(l_i) \times \prod_{j=1}^q \mathrm{U}(m_j) \times\prod_{k=1}^r \mathrm{Sp}(n_k)\right)\\
& \cong (\mathbb{Z}/2\mathbb{Z})^{p_{\geq3}}\times \mathbb{Z}^{q+p_2},
\end{align*}
where the last identification is due to the following fact
\begin{align*}
    \pi_1(\mathrm{SO}(l_i))\cong\left\{
   \begin{aligned}
    &0,\hspace{39pt} l_i=1,\\
    &\mathbb{Z},\hspace{38pt} l_i=2,\\
    &\mathbb{Z}/2\mathbb{Z},\hspace{20pt} l_i\geq3.
    \end{aligned}
    \right.
\end{align*}

(3) For the third part of the theorem, we similarly get that the maximal semisimple subgroup $\mathrm{O}(A_\mathbb{C}, \sigma_{\mathbb{C}})_0^{ss}$ is homotopy equivalent to its maximal compact subgroup \[\mathrm{O}(A, \sigma)_0^{ss} \cong \prod_{i=1}^p \mathrm{SO}(l_i) \times \prod_{j=1}^q \mathrm{SU}(m_j) \times\prod_{k=1}^r \mathrm{Sp}(n_k),\]  
and so we analogously get 
\begin{align*}
\pi_1\left( \mathrm{O}(A_\mathbb{C}, \sigma_{\mathbb{C}})_0^{ss}\right) & \cong \pi_1\left(\mathrm{O}(A, \sigma)_0^{ss}\right)\\
& \cong \pi_1\left( \prod_{i=1}^p \mathrm{SO}(l_i) \times \prod_{j=1}^q \mathrm{SU}(m_j) \times\prod_{k=1}^r \mathrm{Sp}(n_k)\right)\\
& \cong (\mathbb{Z}/2\mathbb{Z})^{p_{\geq3}}. \qedhere
\end{align*}
\end{proof}

\begin{prop}\label{prop:incarn}
    Let $(A,\sigma)$ be a real involutive algebra. Then, the following assertions hold:
    \begin{enumerate}    
        \item The group $\mathrm{O}_{(1,1)}(A_\mathbb C,\bar\sigma_\mathbb C)$ is isomorphic to 
        $\mathrm{Sp}_{2}(A_\mathbb C,\bar\sigma_\mathbb C)$. Moreover,
        \begin{align*}
        \mathrm{Sp}_2(A,\sigma)&=\mathrm{Sp}_{2}(A_\mathbb C,\sigma_\mathbb C)\cap \mathrm{Sp}_{2}(A_\mathbb C,\bar\sigma_\mathbb C), \text{ and }\\
         T^{-1}\mathrm{Sp}_2(A,\sigma)T&=\mathrm{Sp}_{2}(A_\mathbb C,\sigma_\mathbb C)\cap R^{-1}\mathrm{O}_{(1,1)}(A_\mathbb C,\bar\sigma_\mathbb C)R,
        \end{align*}
        where $T=\frac{1}{\sqrt{2}}\begin{pmatrix}
        1 & i \\
        i & 1 \end{pmatrix}$ and $R=\frac{1}{\sqrt{2}}\begin{pmatrix}
        1 & 1 \\
        -1 & 1 \end{pmatrix}$.
        \item The group $\mathrm{O}_{(1,1)}(A_\mathbb H,\sigma_1)$ is isomorphic to $\mathrm{Sp}_{2}(A_\mathbb H,\sigma_0)$.
        \item The group $\mathrm{O}_{(1,1)}(A_\mathbb H,\sigma_0)$ is isomorphic to $\mathrm{Sp}_{2}(A_\mathbb H,\sigma_1)$.
        \item $\mathrm{Sp}_{2}(A_{\mathbb C\{i\}},\sigma_\mathbb C)=\mathrm{Sp}_{2}(A_\mathbb H,\sigma_0)\cap \mathrm{Sp}_{2}(A_\mathbb H,\sigma_3)$, where $\sigma_3(x+yj)={\sigma}_{\mathbb{C}}(x)-\bar{\sigma}_{\mathbb{C}}(y)j$, for $x,y\in A_{\mathbb C\{i\}}$.    
    \end{enumerate}
\end{prop}

\begin{proof}
(1). In the basis $e_1:=(1,0)^t$, $e_2:=(0,i)^t$ of $A_\mathbb C^2$, the standard indefinite form $\omega$ is represented by the matrix $\Omega=\begin{pmatrix}
    0 & i \\
    -i & 0
\end{pmatrix}$. This change of basis conjugates the group $\mathrm{O}_{(1,1)}(A_\mathbb C,\bar\sigma_\mathbb C)$ by the matrix $S=\frac{1}{\sqrt{2}}\begin{pmatrix}
    1 & 0 \\
    0 & i
\end{pmatrix}$.
The form $\Omega$ is proportional to the standard symplectic form. Therefore, the conjugated group $S^{-1}\mathrm{O}_{(1,1)}(A_\mathbb C,\bar\sigma_\mathbb C)S$ agrees with $\mathrm{Sp}_{2}(A_\mathbb C,\bar\sigma_\mathbb C)$.  Moreover, if $g\in \mathrm{Sp}_{2}(A_\mathbb C,\bar\sigma_\mathbb C)\cap \mathrm{Sp}_{2}(A_\mathbb C,\sigma_\mathbb C)$, then $\sigma_\mathbb C(g)=\bar\sigma_\mathbb C(g)$, i.e. $g\in\mathrm{Mat}_2(A)$.
\vspace{2mm}

\noindent (2)--(4). We consider the following indefinite form on $A_{\mathbb{H}}^{2}$:
\[h(x,y):={{\sigma }_{1}}{{(x)}^{t}}\begin{pmatrix}
   0 & j  \\
   -j & 0  \\
\end{pmatrix}y.\]
We now check that the form $h$ is ${{\sigma }_{1}}$--sesquilinear. Indeed, one sees: 
\begin{align*}
h(y,x)  & =\sigma_1(y)^t\begin{pmatrix}
   0 & j  \\
   -j & 0  \\
\end{pmatrix}x\\
 & =\sigma_1\left( \sigma_1(x)^t\begin{pmatrix}
   0 & j  \\
   -j & 0  \\
\end{pmatrix}y \right)\\
& =\sigma_1\left( h(x,y) \right).
\end{align*}
Then, with respect to the basis $e_1:=\left(1,0\right)^t$, $e_2:=\left(0,-j\right)^t$, we have $h\left( {{e}_{1}},{{e}_{1}} \right)=h\left( {{e}_{2}},{{e}_{2}} \right)=0$, and $h\left( {{e}_{1}},{{e}_{2}} \right)=1$, that is, $h$ is indefinite. Thus, 
$$S^{-1}\mathrm{Aut}(h)S=\mathrm{O}_{(1,1)}(A_\mathbb H,\sigma_1),$$
for $S=\begin{pmatrix}
    1 & 0 \\
    0 & -j 
\end{pmatrix}$.

 \noindent Further, $h=j\hat\omega_\mathbb H$, where $\hat\omega_\mathbb H(x,y)=\sigma_3(x)\begin{pmatrix}
   0 & 1  \\
   -1 & 0  \\
\end{pmatrix} y$ where $\sigma_3(x+yj)={\sigma}_{\mathbb{C}}(x)-\bar{\sigma}_{\mathbb{C}}(y)j$. This means that $\mathrm{Sp}_2(A_\mathbb H,\sigma_3)=\mathrm{Aut}(h)$. Finally, the involution $\theta_i\colon \mathbb H\to\mathbb H$, $q\mapsto i^{-1}qi$, for $q\in\mathbb H$ fixing $\mathbb C\{i\}\subset\mathbb H$, induces an~involution $\theta_j\colon\mathrm{Mat}_2(A_\mathbb H)\to \mathrm{Mat}_2(A_\mathbb H)$ fixing $\mathrm{Mat}_2(A_{\mathbb C\{i\}})$ such that $\theta_j\circ\sigma_3\circ\theta_j=\sigma_0$. Thus, 
$$
\theta_j\mathrm{Sp}_2(A_\mathbb H,\sigma_3)\theta_j=\mathrm{Sp}_2(A_\mathbb H,\sigma_0).$$

\noindent  Similarly, because $(ij)^{-1}\sigma_0(g)ij=\sigma_1(g)$, we obtain $\mathrm{Sp}_2(A_\mathbb H,\sigma_1)=\mathrm{Aut}(\hat h)$, where 
$$\hat h(x,y)=\sigma_0(x)^t\begin{pmatrix}
    0 & ij \\
    -ij & 0
\end{pmatrix}y.$$
In the basis $e_1=(1,0)^t,\;e_2=(0,ij)^t$, this form agrees with the standard indefinite form.

 \noindent Finally, if $g\in \mathrm{Sp}_{2}(A_\mathbb H,\sigma_0)\cap \mathrm{Sp}_{2}(A_\mathbb H,\sigma_3)$, then $\sigma_3(g)=\sigma_0(g)$, i.e. $g\in\mathrm{Mat}_2(A_{\mathbb C\{i\}})$.

\end{proof}

\subsection{Examples}\label{sec:real_examples}

The following involutive algebras provide examples of Hermitian algebras:
\begin{enumerate}
    \item Let $A=\mathrm{Mat}_n(\mathbb{R})$ be the space of real $n\times n$-matrices and $\sigma(r):=r^t$ for $r\in A$. Then $A^\sigma=\mathrm{Sym}_n(\mathbb{R})$ is the space of symmetric matrices. The algebra $(A,\sigma)$ is Hermitian with $A^{\sigma}_+=\mathrm{Sym}^+_n(\mathbb{R})$, the space of real symmetric positive definite matrices. Then,
    \begin{align*}\mathrm{Sp}_2(A,\sigma) & \cong \mathrm{Sp}_{2n}(\mathbb R),\\
    \mathrm{O}_{(1,1)}(A,\sigma) & \cong\mathrm{O}(n,n).
    \end{align*}

    \item Let $A=\mathrm{Mat}_n(\mathbb{C})$ be the space of complex $n\times n$-matrices. We consider two anti-involutions $\sigma(r)=r^t$ and $\bar\sigma(r):=\bar r^t$, for $r\in A$. Then $A^{\bar\sigma}=\mathrm{Herm}_n(\mathbb{C})$ is the space of complex Hermitian matrices. The algebra $(A,\bar\sigma)$ is Hermitian with $A^{\bar\sigma}_+=\mathrm{Herm}^+_n(\mathbb{C})$, the space of complex Hermitian positive definite matrices. In this case, 
    \begin{equation}\label{eq:U(n,n)}
    \mathrm{Sp}_2(A,\bar\sigma)\cong\mathrm{O}_{(1,1)}(A,\bar\sigma)\cong\mathrm{U}(n,n).\end{equation}
    Notice that this group is not semisimple, but it has a compact center, namely $\mathrm U(1)$, so it does not affect the symmetric space. Further, the algebra $(A,\sigma)$ is not Hermitian, and 
    \begin{align*}
    \mathrm{Sp}_2(A,\sigma) & \cong\mathrm{Sp}_{2n}(\mathbb C),\\
    \mathrm{O}_{(1,1)}(A,\sigma) & \cong\mathrm{O}_{2n}(\mathbb C).
    \end{align*}
    These groups are the complexifications of the groups from Example (1).
    \vspace{2mm}
    
    \item Let $A=\mathrm{Mat}_n(\mathbb{H})$ be the space of quaternionic $n\times n$-matrices. We consider two anti-involutions 
    \begin{align*}
        \sigma_0(r_0+r_1j)&=r_0^t+\bar r_1^tj,\ \text{and}\\
        \sigma_1(r_0+r_1j)&=\bar r_0^t- r_1^tj,
    \end{align*} 
    where $r_0,r_1\in\mathrm{Mat}_n(\mathbb C)$. Then $\sigma_0$ is not Hermitian, $\sigma_1$ is Hermitian, and we have $A^{\sigma_1}=\mathrm{Herm}_n(\mathbb{H})$ is the space of quaternionic Hermitian matrices, $A^{\sigma_1}_+=\mathrm{Herm}^+_n(\mathbb{H})$ is the space of quaternionic Hermitian positive definite matrices, and we obtain the following isomorphisms:
    $$\mathrm{Sp}_2(A,\sigma_1)\cong\mathrm{O}_{(1,1)}(A,\sigma_0)\cong\mathrm{SO}^*(4n),$$    \begin{equation}\label{eq:Sp(n,n)}
    \mathrm{O}_{(1,1)}(A,\sigma_1)\cong\mathrm{Sp}_2(A,\sigma_0)\cong\mathrm{Sp}(n,n).
    \end{equation}
    
    \item Let $A:=\mathrm{Mat}_n(\mathbb C\{I\})$ with the Hermitian anti-involution $\sigma^I$ given by transposition and complex conjugation. Let $A_{\mathbb C\{i\}}=\mathrm{Mat}_n(\mathbb C\{I\})\otimes_\mathbb R\mathbb C\{i\}$. We extend the anti-involution $\sigma^I$ in $\mathbb C\{i\}$-linear way and (slightly abusing the notation) denote it again by $\sigma^I:=\sigma^I\otimes\mathrm{Id}$. Then,
    $$\mathrm{Sp}_2(A_{\mathbb C\{i\}},\sigma^I)\cong\mathrm{O}_{(1,1)}(A_{\mathbb C\{i\}},\sigma^I)\cong\mathrm{GL}_{2n}(\mathbb C).$$
    This group is the complexification of the group $\mathrm{U}(n,n)$ (cf. Example (2)).
    \vspace{2mm}
        
    \item Let $A:=\mathrm{Mat}_n(\mathbb H\{I,J,K\})$ with the Hermitian anti-involution $\sigma_1$ given by transposition and quaternionic conjugation. Let $A_{\mathbb{C}\{i\}}=\mathrm{Mat}_n(\mathbb H\{I,J,K\})\otimes_\mathbb R\mathbb C\{i\}.$ We extend the anti-involutions $\sigma_0$ and $\sigma_1$ in $\mathbb C\{i\}$-linear way and (slightly abusing the notation) denote them again by $\sigma_0:=\sigma_0\otimes\mathrm{Id}$ and $\sigma_1:=\sigma_1\otimes\mathrm{Id}$. Then
    \begin{align*}
    \mathrm{Sp}_2(A_{\mathbb C\{i\}},\sigma_1)&\cong\mathrm{O}_{(1,1)}(A_{\mathbb C\{i\}},\sigma_0)\cong\mathrm{O}_{4n}(\mathbb C), \text{ and }\\[1pt]
    \mathrm{Sp}_2(A_{\mathbb C\{i\}},\sigma_0)&\cong\mathrm{O}_{(1,1)}(A_{\mathbb C\{i\}},\sigma_1)\cong\mathrm{Sp}_{4n}(\mathbb C).
    \end{align*}
    These groups are the complexifications of the groups from Example (3).

\end{enumerate}

\section{Models for the symmetric space of \texorpdfstring{$\mathrm{O}_{(1,1)}(A,\sigma)$}{} and their tangent spaces}\label{sec:ort_models}

In this section, we assume $(A,\sigma)$ to be a real Hermitian algebra. Let $G:=\mathrm{O}_{(1,1)}(A,\sigma)$ and $K:=\mathrm{KO}_{(1,1)}(A,\sigma)$ the maximal compact subgroup of $G$. We construct several models of the symmetric space $\mathcal X_{G}:=G/K$. 

\subsection{Space of indefinite involutive operators}\label{sec:indefinite_O}

We call an $A$-linear operator $J\colon A^2\to A^2$ \emph{involutive} if $J^2=\mathrm{Id}$. An involutive operator $J$ is called \emph{indefinite} if there are two $A$-lines $l_J^+$ and $l_J^-$ such that $J|_{l_J^+}=\mathrm{Id}$ and $J|_{l_J^-}=-\mathrm{Id}$. These two lines are necessarily transverse. Each $A$-linear map $J$ on $A^2$ defines a $\sigma$-sesquilinear form $h_J\colon A^2\times A^2\to A$ by
$$h_J(\bullet,\bullet):=h(J(\bullet),\bullet),$$ 
where $h$ is the standard indefinite orthogonal form on $A^2$ introduced in Section~\ref{sec:sesq_forms}. An indefinite involutive operator is called \emph{compatible} with $h$ if $h_J$ is non-degenerate.

We consider the following space: 
\begin{align*}
\mathfrak{C}_G := &\; \{J\text{ is indefinite involutive on $A^2$} \mid h_J\text{ is a $\sigma$-inner product}\}\\
 = &\; \{J\text{ is indefinite involutive on $A^2$} \mid h(x,x)\in A^\sigma_+\text{ for all regular $x\in l_J^+$}\}.
\end{align*}
\begin{rem}
Notice that for all $J\in\mathfrak{C}_G$, it is $h(x,x)\in A^\sigma_-$, for all $x\in l_J^-=( l_J^+)^{\perp_h}$. Moreover, for every $J\in\mathfrak{C}_G$, it is also $J\in G$.
\end{rem}

The group $G$ acts on $\mathfrak{C}_G$ transitively by conjugation. Indeed, let $J\in \mathfrak{C}_G$, $x\in l^+_J$, $y\in l^-_J$. We can choose $x,y$ such that $-h(x,x)=h(y,y)=1$. Then $(x,y)$ is an $h$-orthonormal basis of $A^2$. Since $G$ acts transitively on the space of all $h$-orthonormal bases, there exists $g\in G$ such that 
\[g.(x,y)=\left(\frac{1}{\sqrt{2}}(1,-1)^t,\frac{1}{\sqrt{2}}(1,1)^t\right).\] 
It then follows that $gJg^{-1}=J_0$ where  $J_0=\begin{pmatrix}
0 & 1\\
1 & 0
\end{pmatrix}\in\mathfrak{C}_G$, thus the action is transitive. Furthermore, the stabilizer $J_0$ is $K$, and so $\mathfrak{C}_G$ is a model of the symmetric space $G/K$. 

The topological closure of the space $\mathfrak{C}_G$ is 
\begin{align*}
\overline{\mathfrak{C}}_G=\left\{J\colon A^2\to A^2\midwd
\begin{aligned}
&\ J\text{ is indefinite involutive, and }\\
&\ h(x,x)\in A^\sigma_{\geq 0} \text{ for all regular $x\in l_J^+$}
\end{aligned}
\right\}.
\end{align*}
Note that this space is not compact. For example, the sequence $J_n=\begin{pmatrix}
    0 & n \\
    n^{-1}  & 0
\end{pmatrix}\in\mathfrak{C}_G$ 
does not have a convergent subsequence as $n\to\infty$. In the next section we will construct a compactification of $G/K$ using embeddings of it as an open precompact domain.

\begin{prop}\label{prop:tangent_sio}
    The tangent space $T_J\mathfrak{C}_G$ at $J\in \mathfrak{C}_G$ is described by
    $$T_J\mathfrak{C}_G=\{L\colon A^2\to A^2\mid \text{$L$ is $A$-linear, and }LJ+JL=0\}.$$
    The group $G$ acts on $T\mathfrak{C}_G$  by conjugation, i.e. 
    $$
    g.(J,L)=(gJg^{-1},gLg^{-1}).
    $$ 
    In particular, 
    $$
    T_{gJg^{-1}}\mathfrak{C}_G = gT_J\mathfrak{C}_Gg^{-1}.
    $$
\end{prop}

\begin{proof}
    Let $J(t)\in\mathfrak{C}_G$ be a smooth curve defined on an interval $I\ni 0$ and $J(0)=J$. Since $J(t)^2=\mathrm{Id}$ for all $t\in I$, the derivation of this equality at $t=0$ is $\dot J(0) J(0)+ J(0)\dot J(0)=0$.

    \noindent The claim about the action on the tangent bundle is trivial. 
\end{proof}

Let $J\in\mathfrak{C}_G$. The action of $J$ on $A^2$ extends naturally to the action on $A_\mathbb C^2$. The complex linear extension of this operator is denoted by $J_{\mathbb C}$. The complexification of the tangent space of $\mathfrak C_G$ at $J$ is given by
$$T_J^\mathbb C\mathfrak{C}_G=\{L'\colon A_\mathbb C^2\to A_\mathbb C^2\mid \text{$L'$ is $A_\mathbb C$-linear, and }L'J_\mathbb C+J_\mathbb CL'=0\}.$$

\subsection{Projective space model}\label{sec:projective_O}

We consider the spaces
\begin{align*}
    \mathfrak{P}_G^\pm:=\left\{xA\midwd\; x\in A^{2}\text{ regular such that }h(x,x)\in A^\sigma_{\pm} \right\},
\end{align*}
where $h$ is the standard indefinite orthogonal form. By~\cite[Section~5.3]{ABRRW}, $G$ acts transitively  by linear transformations on $\mathfrak{P}_G^\pm$. The stabilizers of lines $(1,1)^tA\in\mathfrak{P}_G^+$ and $(1,-1)^tA\in\mathfrak{P}_G^-$ agree with the group $K$. In particular, these spaces are models of the symmetric space $G/K$.

There exists a natural homeomorphism between $\mathfrak{P}_G^+$ and $\mathfrak{P}_G^-$, which maps $l\in\mathfrak{P}_G^+$ to $l^{\perp_h}\in \mathfrak{P}_G^-$ and vice versa. These two lines are always transverse and, therefore, $l\oplus l^{\perp_h}= A^2$. 

Notice that $\mathfrak{P}_G^\pm$ are open domains in $\mathbb {P}(A^2)=\{xA\mid x\text{ regular}\}$. In particular, the tangent space at $l\in\mathfrak{P}_G^\pm$ can be seen as the space of $A$-linear operators 
$$T_l\mathfrak{P}_G^\pm=\mathrm{Hom}(l,A^2/l)\cong \mathrm{Hom}(l,l^{\perp_h}).$$

The topological closure of $\mathfrak{P}_G^\pm$ in the space $\mathbb {P}(A^2)$ of $A$-lines is
$$\overline{\mathfrak{P}}_G^\pm=\left\{xA\midwd\; x\in A^{2}\text{ regular such that }\pm h(x,x)\in A^\sigma_{\geq 0} \right\};$$
in each case, this space is compact because $\mathbb{P}(A^2)$ is compact. The boundaries $\partial\mathfrak{P}_G^\pm$ contain a common compact locus 
\[\mathrm{Is}( h):=\left\{xA\midwd\; x\in A^{2}\text{ regular such that } h(x,x)=0 \right\}\] 
of isotropic lines.

The maps 
\begin{equation}\label{eq:ort_CP}
    F_{\mathfrak{C}_G,\mathfrak{P}_G^\pm}\colon \mathfrak{C}_G\to\mathfrak{P}_G^\pm    
\end{equation}
associating to an operator $J\in\mathfrak{C}_G$ the eigenline $l_J^\pm\in\mathfrak{P}_G^\pm$ provide diffeomorphisms between $\mathfrak{C}_G$ and $\mathfrak{P}_G^\pm$ which are equivariant with respect to the actions of $G$ on $\mathfrak{C}_G$ and $\mathfrak{P}_G^\pm$.

\begin{prop}
    Let $J\in\mathfrak{C}_G$, $L\in T_J\mathfrak{C}_G$ and $x\in l_J^\pm$ be a generator of $l_J^\pm$. Then 
    $$
    d_JF_{\mathfrak{C}_G,\mathfrak{P}_G^\pm}(L)=[x,\pm\frac{1}{2}L(x)+l_J^\pm]=\pm\frac{1}{2}[L]\in T_{l_J^\pm}\mathfrak{P}_G^\pm,
    $$
    where $[L]\colon l_J^\pm\to A^2/l_J^\pm$ is the natural projection of the linear operator $L|_{l_J^\pm}\colon l_J^\pm\to A^2$.
\end{prop}

\begin{proof}
    Let $J(t)$ be a smooth path in $\mathfrak{C}_G$ such that $J(0)=J$ and $\dot J(0)=L$. Let $x(t)$ be a smooth path in $A^2$ such that $x(t)$ generates $l^\pm_{J(t)}$ and $x(0)=x$. Then $J(t)x(t)=\pm x(t)$, for all $t$. Differentiating this equality at $t=0$, we obtain 
    \begin{align*}
    L(x)+J(\dot x(0)) & =\pm\dot x(0), \text{ and }\\
    (\pm\mathrm{Id}-J)\dot x(0) & =L(x).
    \end{align*}
    Notice that, if $\dot x(0)$ satisfies the equality above, then for any $a\in A$, $\dot x(0)+xa$ satisfies it too. Now it is enough to check that $\pm\frac{1}{2}L(x)$ satisfies the equation above: 
    \begin{align*}
        (\pm\mathrm{Id}-J)\left(\pm\frac{1}{2}L(x)\right)&=\frac{1}{2}(L(x)\mp JL(x))\\
        &=\frac{1}{2}(L(x)\pm LJ(x))\\
        &=L(x).
    \end{align*}
    This shows that 
    $$d_JF_{\mathfrak{C}_G,\mathfrak{P}_G^\pm}(L)=[x,\pm\frac{1}{2}L(x)+l_J^\pm]=\pm\frac{1}{2}[L].$$
\end{proof}

Let $l\in\mathfrak{P}_G^\pm$. Its complexification is denoted by $l_\mathbb C\subset A_\mathbb C^2$. The complexification of the tangent space of $\mathfrak P_G^\pm$ at $l$ is given by the space of all $A_\mathbb C$-linear morphisms:
\begin{equation}\label{eq:o(1,1)-comp.tang}
T_l^\mathbb C\mathfrak{P}_G^\pm=\mathrm{Hom}(l_\mathbb C,A_\mathbb C^2/l_\mathbb C)\cong \mathrm{Hom}(l_\mathbb C,l_\mathbb C^{\perp_{ h_\mathbb C}}).    
\end{equation}

The spaces $\mathfrak{P}_G^\pm$ are open domains in  the projective space $\mathbb{P}(A^2)$ of all $A$-lines. In the following we will construct affine charts for these spaces.

\subsection{Half-space models}\label{sec:halfspace_O}
We obtain a half-space model by considering the spaces $\mathfrak{P}_G^\pm$ in an appropriate affine chart. For this we fix two transverse isotropic lines $\ell,\ell'\in \mathrm{Is}( h)$ spanned by the vectors $e_1=(1,0)^t$ and $e_2=(0,1)^t$ respectively. Every $l\in \mathfrak{P}_G^\pm$ is transverse to $\ell$. Therefore, it can be seen as a map $L\colon\ell'\to\ell$, such that for $v\in\ell'$, $L(v):=w$ where $w\in\ell$ is the unique element of $\ell$ such that $v+w\in l$. In the basis $(e_1,e_2)$, $L$ can be identified with an element of $a_l\in A$ such that $L(e_2)= e_1a_l$. Due to the condition $ h(x,x)\in A^\sigma_\pm$, for all regular $x\in l$, we obtain 
\[h(e_1+e_2a,e_1+e_2a)=\sigma(a)+a\in A^\sigma_\pm.\] Thus we can write affine charts for $\mathfrak{P}_G^\pm$ as follows:
\begin{align*}
    \mathfrak{U}_G^\pm:=&\ \{a\in A\mid \sigma(a)+a\in A^\sigma_\pm\}\\
    =&\ \{a_0+a_1\mid a_0\in A^\sigma_\pm,\;a_1\in A^{-\sigma}\}\\
    \cong&\ A^\sigma_\pm\oplus A^{-\sigma}.
\end{align*}
This is an unbounded open subspace in $A$. In particular, the tangent space at every $a\in \mathfrak{U}_G^\pm$ is isomorphic to $A$. Moreover, every element of $\mathfrak{U}_G^\pm$ is invertible. Indeed, if $a\in A$ is not invertible, then there exists $0\neq b\in A$ such that $ab=0$. Then,  
$$
\sigma(b)(\sigma(a)+a)b=\sigma(ab)b+\sigma(b)ab=0,
$$
which contradicts the fact that $\sigma(a)+a\in A^\sigma_\pm$.

We denote the map
\begin{equation}\label{eq:ort_PU}
F_{\mathfrak{P}_G^\pm,\mathfrak{U}_G^\pm}\colon\mathfrak{P}_G^\pm\to\mathfrak{U}_G^\pm,
\end{equation} 
mapping $l\mapsto a_l$. Notice, if $x=e_1x_1+e_2x_2$ and $l=xA$, then $a_l=x_1x_2^{-1}$. This map does not extend to the boundary of $\mathfrak{P}_G^\pm$. However, it is possible to extend it to the subset of $\mathrm{Is}( h)$ which is transverse to $\ell$. Its image under $F_{\mathfrak{P}_G^\pm,\mathfrak{U}_G^\pm}$ is $A^{-\sigma}$.

The group $G$ acts on $\mathfrak{U}_G^\pm$ by M\"obius transformations as follows: For $x\in \mathfrak{U}_G^\pm$, and for $g:=\begin{pmatrix}
    a & b\\
    c & d
\end{pmatrix}\in G$, then \[g.x=(ax+b)(cx+d)^{-1}.\] 
This action is transitive and equivariant with respect to the map $F_{\mathfrak{P}_G^\pm,\mathfrak{U}_G^\pm}$ and the linear action of $G$ on $\mathfrak{P}_G^\pm$. The stabilizer of the point $\pm 1\in \mathfrak{U}_G^\pm$ agrees with $K$. As $\mathfrak{U}_G^\pm$ are open in $A$, the tangent space $T_x\mathfrak{U}_G^\pm=A$. In particular, $\mathfrak{U}_G^\pm$ are models of the symmetric space $G/K$.

Since $\mathfrak U_G^\pm$ is an open subspace of $A$, then for every $x\in \mathfrak U_G^\pm$, $T_x\mathfrak U_G^\pm = A$. The action of $G$ extends naturally to the action on $T\mathfrak U_G^\pm=\mathfrak U_G^\pm\times A$. For $g$ as above and $(x,v)\in \mathfrak U_G^\pm\times A^{\sigma_\mathbb C}_\mathbb C$, the action is given by 
\begin{equation}\label{eq:diff_moebius_U_ortog}
    g.(x,v)=\left((ax+b)(cx+d)^{-1}, av(cx+d)^{-1}-(ax+b)(cx+d)^{-1}cv(cx+d)^{-1}\right).  
\end{equation}

For every $x\in\mathfrak U_G^+$, the stabilizer $\mathrm{Stab}_G(x)$ is conjugated to the stabilizer of $x=1$ which is $K$. We are now going to understand the action of $K$ and its Lie algebra $\mathfrak k$ on $T_1\mathfrak U_G^+= A$. For $g=\begin{pmatrix}
    a & b \\
    b & a
\end{pmatrix}\in K$ and $v\in T_1\mathfrak U_G^+$, then
$$g.v
=(a+b)v(b+a)^{-1}
=(a+b)v(\sigma(a)+\sigma(b)),$$
because $(b+a)^{-1}=\sigma(a)+\sigma(b)$.

The map $F_{\mathfrak{P}_G^\pm,\mathfrak{U}_G^\pm}$ is a diffeomorphism which is equivariant with respect to the action of $G$. Its~differential can be computed immediately:

\begin{prop}
    Let $x=e_1x_1+e_2x_2$ be such that $l=xA\in\mathfrak{P}_G^\pm$, and $v=[x,e_1v_1+e_2v_2+l]\in T_l\mathfrak{P}_G^\pm$. Then 
    $$d_{l}F_{\mathfrak{P}_G^\pm,\mathfrak{U}_G^\pm}(v)=(v_1-F_{\mathfrak{P}_G^\pm,\mathfrak{U}_G^\pm}(l) v_2)x_2^{-1}.$$
\end{prop} 

\subsection{Precompact model}\label{sec:precompact_O}
Taking a different affine chart for $\mathfrak{P}_G^+$, we obtain a bounded model, which we call the precompact model. 

For this, we fix one line $\ell'\in \mathfrak{P}_G^-$ spanned by the vector $e_1=\frac{1}{\sqrt{2}}(1,-1)^t$. Its orthogonal complement $\ell\in \mathfrak{P}_G^+$ is spanned by $e_2=\frac{1}{\sqrt{2}}(1,1)^t$. Every $l\in \overline{\mathfrak{P}}_G^+$ is transverse to $\ell'$. Therefore, it can be seen as a map $L\colon\ell\to\ell'$, such that for $v\in \ell$, $L(v):=w$, where $w\in\ell'$ is the unique element of $\ell'$ such that $v+w\in l$. In the basis $(e_1,e_2)$, $L$ can be identified with an element of $a_l\in A$ such that $L(e_2)=a_l e_1$. Because of the condition $ h(x,x)\in A^\sigma_+$, for all regular $x\in l$, we obtain \[ h(e_1a_l+e_2,e_1a_l+e_2)=1-\sigma(a_l)a_l\in A^\sigma_+.\] 
So we can write an affine chart for  $\overline{\mathfrak{P}}_G^+$ as follows:
$$\overline{\mathfrak{B}}_G:=\{a\in A\mid 1-\sigma(a)a\in A^\sigma_{\geq 0}\}.$$
This is a compact subspace in $A$ with interior 
$$\mathfrak{B}_G:=\{a\in A\mid 1-\sigma(a)a\in A^\sigma_+\}.$$ 
In particular, the tangent space at every $a\in \mathfrak{B}_G$ is isomorphic to $A$. 

We denote the map 
\begin{equation}\label{eq:ort_PB}
    F_{\mathfrak{P}_G^+,\mathfrak{B}_G}\colon\overline{\mathfrak{P}}_G^+\to\overline{\mathfrak{B}}_G,
\end{equation}
mapping $l\mapsto a_l$. Notice, if $x=e_1x_1+e_2x_2$ and $l=xA$, then $a_l=x_1x_2^{-1}$. Under this map ${\mathfrak{P}_G}^+$ maps to $\mathfrak{B}_G$ and $\mathrm{Is}( h)$ to $\mathrm O{(A,\sigma)}$.

The group $G':=R^{-1}GR$ (cf.~\eqref{eq:conj_orth}) acts on $\overline{\mathfrak{B}}_G$ by M\"obius transformations as follows: For $x\in \mathfrak{B}_G$, and $g:=\begin{pmatrix}
    a & b\\
    c & d
\end{pmatrix}\in G'$, then 
\[g.x=(ax+b)(cx+d)^{-1}.\] 
The restriction of this action to $\mathfrak{B}_G$ is transitive and equivariant with respect to the map $F_{\mathfrak{P}_G^+,\mathfrak{B}_G}$ and the linear action of $G$ on $\mathfrak{P}_G^+$. The stabilizer of the point $0\in \mathfrak{B}_G$ agrees with $K':=R^{-1}KR$ (cf.~\eqref{eq:conj_orth}). In~particular, $\mathfrak{B}_G$ is a model of the symmetric space $G'/K'$ which is isometric to $G/K$ because $G$ and $G'$ are conjugate. 

Since $\mathfrak{B}_G$ is an open subspace of $A$, then for every $x\in \mathfrak{B}_G$, $T_x\mathfrak{B}_G = A$. The complexified tangent space at $x\in \mathfrak{B}_G$ is then $T^\mathbb C_x\mathfrak{B}_G = A_\mathbb C$.

The action of $G'$ extends naturally to the action on the tangent bundle $T\mathfrak{B}_G=\mathfrak{B}_G\times A$. For $g=\begin{pmatrix}
        a & b \\
        c & d
    \end{pmatrix}\in G'$
and $(x,v)\in \mathfrak{B}_G\times A$, the action is given by
$$g.(x,v)=\left((ax+b)(cx+d)^{-1}, av(cx+d)^{-1}-(ax+b)(cx+d)^{-1}cv(cx+d)^{-1}\right).$$

For every $x\in\mathfrak B_G$, the stabilizer $\mathrm{Stab}_{G'}(x)$ is conjugate to the stabilizer of $x=0$ which is $K'$. We are now going to understand the action of $K'$ on $T_0\mathfrak B_G= A$. For $g=\begin{pmatrix}
    a & 0 \\
    0 & d
\end{pmatrix}\in K'$ and $v\in T_0\mathfrak B_G$, then
$$g.v=avd^{-1}.$$

The map $F_{\mathfrak{P}_G^+,\mathfrak{B}_G}$ is a diffeomorphism which is equivariant with respect to the action of $G$. Its~differential can be computed immediately:

\begin{prop}
    Let $x=e_1x_1+e_2x_2$ be such that $l=xA\in\mathfrak{P}_G^+$, and $v=e_1v_1+e_2v_2+A\in T_{l}\mathfrak{P}_G^+$. Then, 
    $$d_{l}F_{\mathfrak{P}_G^+,\mathfrak{B}_G}(v)=(v_1-F_{\mathfrak{P}_G^+,\mathfrak{B}_G}(l) v_2)x_2^{-1}.$$
\end{prop}

\begin{rem}
    The isometric map $F_{\mathfrak{U}_G^+,\mathfrak{B}_G}$ from the upper half-space model to the precompact model can be described explicitly by the M\"obius transformation given by the matrix $R^{-1}=\begin{pmatrix}
        1 & -1 \\
        1 & 1
    \end{pmatrix}$, i.e. $x\mapsto (x-1)(x+1)^{-1}$. This map agrees with the composition $F_{\mathfrak{P}_G^+,\mathfrak{B}_G}\circ F_{\mathfrak{P}_G^+,\mathfrak{U}_G^+}^{-1}$, and, therefore, is $G$-equivariant.
\end{rem}

\subsection{Riemannian metric on the models}

In terms of the upper half-space model, it is particularly easy to write the Riemannian metric on the Riemannian symmetric space of $\mathrm{O}_{(1,1)}(A,\sigma)$. Namely, for $z=x+y\in \mathfrak U^+_G$ where $x\in A^{-\sigma}$, $y\in A^\sigma_+$, and $v\in T_z\mathfrak U^+_G=A$, the associated norm of this metric is given by:
\begin{align*}
    g_z(v,v):=&\mathrm{tr}\left(y^{-\frac{1}{2}} vy^{-1}vy^{-\frac{1}{2}}\right).    
\end{align*}
This metric is $\mathrm{O}_{(1,1)}(A,\sigma)$-invariant. Using the polarization, one can also write the expression of $g_z(v,w)$, for $v,w\in T_z\mathfrak U^+_G$. The isomorphisms between models allow us to express this metric in terms of all four models of the symmetric space of $\mathrm{O}_{(1,1)}(A,\sigma)$.

Notice that if $z=1$, then the metric becomes particularly easy: 
$$g_1(v,w)=\mathrm{tr}(v\circ w).$$

\subsection{Example: \texorpdfstring{$A=\mathrm{Mat}_n(\mathbb R)$}{}}\label{sec:exam_A_real_mat_ortog}
    We now present all four models, their tangent spaces and the differential of the diffeomorphisms between them explicitly in a certain important case. Let $A=\mathrm{Mat}_n(\mathbb R)$ and $\sigma$ be the transposition. Then $G:=\mathrm{O}_{(1,1)}(A,\sigma)$ is isomorphic to the real indefinite orthogonal group $\mathrm{O}(n,n)$ which is the group of isometries of the vector space $\mathbb R^{2n}$ equipped with the following symmetric bilinear form  $\Omega=\begin{pmatrix} 0 & \mathrm{Id} \\ \mathrm{Id} & 0 \end{pmatrix}$ and $\mathrm{Id}$ denotes the $n\times n$ identity matrix.

    The space of indefinite operators $\mathfrak{C}_G$ can be seen as the space of all linear operators $J\colon \mathbb R^{2n}\to\mathbb R^{2n}$ such that $J^2=\mathrm{Id}$ such that there are two $n$-dimensional subspaces $V_+$ and $V_-$ of $\mathbb R^{2n}$ such that $J|_{V_\pm}=\pm\mathrm{Id}$ and $J^t\Omega$ is a symmetric positive definite matrix. Then the tangent space $T_J\mathfrak{C}_G$ can be identified with the space of all linear maps $L\colon \mathbb R^{2n}\to\mathbb R^{2n}$ such that $JL=-LJ$ and $L^t\Omega$ is a symmetric matrix.

    The projective space model $\mathfrak{P}_G^\pm$ in this case can be seen as the space of all $n$-dimensional subspaces $V\subset \mathbb R^{2n}$ and the bilinear form $\omega|_V$ is positive-, resp. negative-definite. Their compactifications agree with the spaces of all $n$-dimensional subspaces $V\subset\mathbb R^{2n}$ such that the form $\omega|_V$ is non-negative, resp. non-positive definite.  The tangent space $T_V\mathfrak{P}_G^\pm$ can be then identified with the space of all linear maps $Q\colon V\to \mathbb{R}^{2n}/V$. 

    The half-space models $\mathfrak U_G^\pm$ in this case are the spaces 
    $$\mathfrak U_G^\pm=\{M\in\mathrm{Mat}_n(\mathbb R)\mid M+M^t \in\mathrm{Sym}^\pm_n(\mathbb R)\},$$
    with tangent spaces $T_M\mathfrak U_G^\pm=\mathrm{Mat}_n(\mathbb R)$, for every $M\in\mathfrak U_G^\pm$.

    Finally, the precompact model $\mathfrak{B}_G$ in this case is described by
    $$\mathfrak{B}_G=\{M\in\mathrm{Mat}_n(\mathbb R)\mid \mathrm{Id}-M^tM \in\mathrm{Sym}^+_n(\mathbb R)\}.$$
    Its compactification is
    $$\overline{\mathfrak{B}}_G=\{M\in\mathrm{Mat}_n(\mathbb R)\mid \mathrm{Id}-M^tM \in\mathrm{Sym}^{\geq 0}_n(\mathbb R)\};$$
    The tangent space at every point $M\in\mathfrak B_G$ is then $T_M\mathfrak B_G=\mathrm{Mat}_n(\mathbb R)$.
    
    The map $F_{\mathfrak{C}_G,\mathfrak{P}_G^\pm}$ maps an operator $J\colon \mathbb R^{2n}\to \mathbb R^{2n}$ to the subspace $V_J^\pm$ which is the $(\pm 1)$-eigenspace of $J$. If $M\colon \mathbb R^{2n}\to \mathbb R^{2n}$ is a tangent vector at $J$, then 
    \[d_JF_{\mathfrak{C}_G,\mathfrak{P}_G^\pm}(M)=\pm\frac{1}{2}[M],\] 
    where $[M]\colon V_J^\pm\to \mathbb R^{2n}/V_J^\pm$ which is the natural projection of $M|_{V_J^\pm}$.

    Let now $V\in\mathfrak{P}_G^\pm$ and $Q\colon V\to \mathbb R^{2n}/V$ be a tangent vector at $V$. Then $F_{\mathfrak{P}_G^\pm,\mathfrak U_G^\pm}$ maps $V$ to the following symmetric matrix. We first fix a basis of $V$ and write it as a $2n\times n$-matrix $A$; this matrix contains an upper $(n\times n)$-block $A_1$ and a lower one, say $A_2$, which is invertible because $V\in\mathfrak{P}_G^\pm$. Then $F_{\mathfrak{P}_G^\pm,\mathfrak U_G^\pm}(V)=A_1A_2^{-1}$ and this expression is clearly independent of the chosen basis. Applying $Q$ to the vectors of the basis, we obtain a family of $2n\times n$ matrices $\{B+Ax\mid x\in\mathrm{Mat}_n(\mathbb R)\}$. We choose one of them, say for $x=0$ we obtain $B$. Then we obtain in a similar vain two $n\times n$ matrices $B_1$ and $B_2$ which are upper and lower submatrices of $B$. Thus, 
    \[d_VF_{\mathfrak{P}_G^\pm,\mathfrak U_G^\pm}(Q)=(B_1-A_1A_2^{-1}B_2)A_2^{-1}.\] 
    An easy calculation shows that this expression does not depend on any of the choices made.

    One similarly describes the map $F_{\mathfrak{P}_G^+,\mathfrak{B}_G}$. Let $T:=\frac{1}{\sqrt{2}}\begin{pmatrix}
    \mathrm{Id} & \mathrm{Id}\\
    -\mathrm{Id} & \mathrm{Id}
    \end{pmatrix}$. We denote $A':=TA$, $B':=TB$, where $A$ and $B$ are the matrices introduced above. Let $A'_1$ and $A'_2$ be the lower and upper $n\times n$ submatrices of $A'$ respectively, and $B'_1$ and $B'_2$ be the lower and the upper $n\times n$ submatrices of $B'$ respectively. Then, $F_{\mathfrak{P}_G^+,\mathfrak{B}_G}(V)=A_1'(A_2')^{-1}$ and \[d_VF_{\mathfrak{P}_G^+,\mathfrak B_G}(Q)=(B'_1-A'_1(A'_2)^{-1}B'_2)(A'_2)^{-1}.\]

\section{Models for the symmetric space of \texorpdfstring{$A^\times$}{} and their tangent spaces}\label{sec:A_models}

In this section, we assume $(A,\sigma)$ to be a real Hermitian algebra. In Section~\ref{sec:ort_models}, we discussed models of the symmetric space of $\mathrm{O}_{(1,1)}(A,\sigma)$. This group contains as a subgroup the group $G:=A^\times$ embedded as the set of diagonal matrices with $g$ and $\sigma(g)^{-1}$ on the diagonal, where $g\in G$. Notice that $G$ is a reductive Lie group with a maximal compact subgroup $K:=\mathrm{O}(A,\sigma)$. In this section, we discuss models of the symmetric space of $A^\times$. 

\begin{rem}
    The center of $A^\times$ is always nontrivial and  noncompact because $\mathbb R^\times\cdot 1\subseteq A^\times$ is always central. Therefore, the symmetric space of $A^\times$ is always reducible. 
\end{rem}

We denote by $\hat G$ (resp. $\hat K$) the corresponding diagonal copy of $G$ (resp. $K$) in $\mathrm{O}_{(1,1)}(A,\sigma)$. For $g\in G$, we denote $\hat g:=\begin{pmatrix}
    g & 0 \\
    0 & \sigma(g)^{-1}
\end{pmatrix}\in \hat G$.

As $A^\times$ is a subgroup of $\mathrm{O}_{(1,1)}(A,\sigma)$, the models of the symmetric space of $A^\times$ will be subspaces of corresponding models of the symmetric space of $\mathrm{O}_{(1,1)}(A,\sigma)$. Therefore, we will not write explicitly the maps between models and the Riemannian metric because they are simply the restrictions of the corresponding ones from Section~\ref{sec:ort_models}.

\subsection{Space of indefinite involutive operators}\label{sec:indefinite_A}

Let $\omega$ be the standard symplectic form on $A^2$ and $h$ be the standard indefinite orthogonal form on $A^2$.

An $A$-linear operator $J\colon A^2\to A^2$ is called \emph{almost symplectic} if 
$$
\omega(J(x),J(y))=-\omega(x,y).
$$
We consider the following space: 
\begin{align*}
\mathfrak{C}_G :=\{J\colon A^2\to A^2 \mid J\text{ is almost symplectic}\}\cap \mathfrak{C}_{\mathrm{O}_{(1,1)}(A,\sigma)}.
\end{align*}
It consists of all indefinite involutive almost symplectic operators.

\begin{thm}\label{thm:involaopermodel_A}
    The group $\hat G$ acts on $\mathfrak{C}_G$ by conjugation. This action is transitive and the stabilizer of $J_0=\begin{pmatrix}
        0 & 1 \\
        1 & 0
    \end{pmatrix}$
    agrees with $\hat K$.
\end{thm}

\begin{proof}
    Every indefinite involutive operator $J$ is an element of $\mathrm{O}_{(1,1)}(A,\sigma)$. The condition to be almost symplectic implies that $J\in\begin{pmatrix}
        0 & 1 \\
        1 & 0
    \end{pmatrix}\mathrm{Sp}_2(A,\sigma)$.    
    These two conditions imply that $J=\begin{pmatrix}
        0 & a \\
        \sigma(a)^{-1} & 0
    \end{pmatrix}$, for $a\in G$.
    The condition $h(J(x),x)\in A^\sigma_+$, for all regular $x\in A^2$, implies that $a\in A^\sigma_+$.

\noindent Since for $g=G$, $\hat gJ_0\hat g^{-1}=\begin{pmatrix}
        0 & g\sigma(g) \\
        \sigma(g)^{-1}g^{-1} & 0
    \end{pmatrix}$, the action of $G$ on $\mathfrak{C}_G$ is transitive because every element in $A^\sigma_+$ can be written as $g\sigma(g)$, for some $g\in A^\times$ (cf.~\cite[Corollary 2.65]{ABRRW}). The stabilizer of $J_0$ is clearly $\hat K$.
\end{proof}

For an operator $J=\begin{pmatrix}
        0 & a \\
        a^{-1} & 0
    \end{pmatrix}$,  with $a\in A^\sigma_+$, we can explicitly describe its eigenlines, namely, \[l_J^+=(a,1)^tA \quad \text{ and } \quad l_J^-=(a,-1)^tA.\]  

\begin{prop}\label{prop:tangent_sio_A}
    The tangent space $T_J\mathfrak{C}_G$ at $J\in \mathfrak{C}_G$ is described by
    \begin{align*}
    T_J\mathfrak{C}_G=\left\{L\colon A^2\to A^2\midwd 
    \begin{aligned}
    &\ \text{$L$ is $A$-linear,}\\
    &\ LJ+JL=0,\text{ and}\\
    &\ \omega(L(\cdot),\cdot)+\omega(\cdot,L(\cdot))=0
    \end{aligned}
    \right\}.
    \end{align*}
    The group $\hat G$ acts on $T\mathfrak{C}_G$ by conjugation, i.e. for $g\in \hat G$,  
    $$
    g.(J,L)=(gJg^{-1},gLg^{-1}).
    $$ 
    In particular, 
    $$
    T_{gJg^{-1}}\mathfrak{C}_G = gT_J\mathfrak{C}_Gg^{-1}.
    $$
\end{prop}

\begin{proof}
    Let $J(t)\in\mathfrak{C}_G$ be a smooth curve defined on an interval $I\ni 0$ and $J(0)=J$. Since $J(t)^2=\mathrm{Id}$ for all $t\in I$, the derivation of this equality at $t=0$ is $\dot J(0) J(0)+ J(0)\dot J(0)=0$. Further, the differentiation of the equality $\omega(J(t)(x),J(t)(y))=-\omega(x,y)$ at $t=0$, gives the condition $\omega(\dot J(0)(x),y)+\omega(x,\dot J(0)(y))=0$, for all $x,y\in A^2$.

    \noindent The claim about the action on the tangent bundle is trivial. 
\end{proof}

The tangent space at $J=\begin{pmatrix}
        0 & a \\
        a^{-1} & 0
    \end{pmatrix}$, for $a\in A^\sigma_+$, can be  explicitly described, namely:
$$T_J\mathfrak C_G=\left\{\begin{pmatrix}
    0 & -axa \\
    x & 0
\end{pmatrix}\midwd x\in A^\sigma\right\}.$$
The action of $J$ on $A^2$ extends naturally to the action on $A_\mathbb C^2$. The complex linear extension of this operator is denoted by $J_{\mathbb C}$. The complexification of the tangent space of $\mathfrak C_G$ at $J$ is given by:
\begin{align*}
    T_J^\mathbb C\mathfrak{C}_G  & =\left\{L'\colon A_\mathbb C^2\to A_\mathbb C^2\midwd
    \begin{aligned}
        &\ L' \text{ is } A_\mathbb C\text{-linear},\\
        &\ L'J_\mathbb C+J_\mathbb CL'=0,\text{ and }\\
        &\ \omega_\mathbb C(L'(\cdot),\cdot)+\omega_\mathbb C(\cdot,L'(\cdot))=0
    \end{aligned}
    \right\},  \\
    & = \left\{\begin{pmatrix}
    0 & -axa \\
    x & 0
\end{pmatrix}\midwd x\in A_\mathbb C^\sigma\right\},  
\end{align*}
where $\omega_\mathbb C$ is the complex linear extension of $\omega$.

\subsection{Projective space model}\label{sec:projective_A}

We consider the spaces
\begin{align*}
    \mathfrak{P}_G^\pm:=&\;\left\{xA\midwd\; x\in A^{2}\text{ regular such that }h(x,x)\in A^\sigma_{\pm}\text{ and }\omega(x,x)=0 \right\}\\
    =&\;\mathfrak{P}_{\mathrm{O}_{(1,1)(A,\sigma)}}^\pm\cap\mathbb P\mathrm{Is}(\omega),
\end{align*}
where $h$ is the standard indefinite orthogonal form and $\omega$ is the standard symplectic form. This is the image of $\mathfrak{C}_{G}$ under the map~\eqref{eq:ort_CP}. Therefore, it is a model of the symmetric space of $G$. The group $\hat G$ acts by linear transformations on $\mathfrak{P}_{G}^\pm$.

Notice that $\mathfrak{P}_G^\pm$ are open domains in $\mathbb {P}\mathrm{Is}(\omega)$. In particular, the tangent space at $l\in\mathfrak{P}_G^\pm$ can be seen as the space of $A$-linear operators 
\begin{align*}
    T_l\mathfrak{P}_G^\pm & =\{Q\in\mathrm{Hom}(l,A^2/l)\mid \text{for all $x\in l$, }\omega(Q(x),x)\in A^\sigma\}\\
    & =\{Q\in\mathrm{Hom}(l,l)\mid \text{for all $x\in l$, } h(Q(x),x)\in A^\sigma\}.
\end{align*}
The last equality is given by the fact that using the non-degenerate forms $\omega$ and $h$, we can identify $l \to A^2/l$ as follows: for $x\in l$ we associate $y+l\in A^2/l$ such that $h(x,\cdot)=\omega(y+l,\cdot)=\omega(y,\cdot)$ as $A$-valued forms on $l$.

The complexification of the tangent space at $l\in\mathfrak{P}_G^\pm$ is
\begin{align*}
    T_l^\mathbb C\mathfrak{P}_G^\pm & =\{Q\in\mathrm{Hom}(l_\mathbb C,A_\mathbb C^2/l_\mathbb C)\mid \text{for all $x\in l_\mathbb C$, }\omega_\mathbb C(Q(x),x)\in A_\mathbb C^{\sigma_\mathbb C}\}\\
    &=\{Q\in\mathrm{Hom}(l_\mathbb C,l_\mathbb C)\mid \text{for all $x\in l_\mathbb C$, } h_\mathbb C(Q(x),x)\in A_\mathbb C^{\sigma_\mathbb C}\},    
\end{align*}
where $l_\mathbb C=l\otimes_\mathbb R\mathbb C$ and $h_\mathbb C$ is the complex linear extension of $h$.

The topological closure of $\mathfrak{P}_G^\pm$ in the space $\mathbb {P}(A^2)$ of $A$-lines is
$$\overline{\mathfrak{P}}_G^\pm=\left\{xA\midwd\; x\in A^{2}\text{ regular such that }\pm h(x,x)\in A^\sigma_{\geq 0}\text{ and }\omega(x,x)=0 \right\};$$
in each case, this space is compact because $\mathbb{P}\mathrm{Is}(\omega)$ is compact. The boundaries $\partial\mathfrak{P}_G^\pm$ contain a common compact locus 
$$
\mathrm{Is}(\omega)\cap \mathrm{Is}(h):=\left\{xA\midwd\; x\in A^{2}\text{ regular such that }\omega(x,x)=h(x,x)=0 \right\}
$$ 
of $h$- and $\omega$-isotropic lines.

\subsection{Half-space models}\label{sec:halfspace_A}

The image of $\mathfrak{P}_{G}^\pm$ under the map~\eqref{eq:ort_PU} is the space $\mathfrak{U}_G^\pm:= A^\sigma_\pm$. Therefore, it is a model of the symmetric space of $G$. The group $G$ acts by congruence on $\mathfrak{U}_{G}^\pm$, i.e. $g.a\mapsto ga\sigma(g)$ for $g\in G$ and $a\in \mathfrak{U}_{G}^\pm$.

This is an unbounded open subspace in $A^\sigma$. In particular, the tangent space at every $a\in \mathfrak{U}_G^\pm$ is isomorphic to $A^\sigma$, and the complexification of the tangent space is isomorphic to $A_\mathbb C^{\sigma_\mathbb C}$.

\subsection{Precompact model}\label{sec:precompact_A}

The image of $\mathfrak{P}_{G}^+$ under the map~\eqref{eq:ort_PB} is the space 
$$\mathfrak{B}_G:= \{a\in A^\sigma\mid 1-a^2\in A^\sigma_+\}.$$
Therefore, it is a model of the symmetric space of $G$. 
This is a precompact open subspace in $A^\sigma$. In particular, the tangent space at every $a\in \mathfrak{B}_G$ is isomorphic to $A^\sigma$, and the complexification of the tangent space is isomorphic to $A_\mathbb C^{\sigma_\mathbb C}$. 

The compact closure of $\mathfrak{B}_G$ is
$$\overline{\mathfrak{B}}_G:= \{a\in A^\sigma\mid 1-a^2\in A^\sigma_{\geq 0}\},$$
and the group $R^{-1}\hat G R$ acts by M\"obius transformations on $\overline{\mathfrak{B}}_{G}$, where $R=\frac{1}{\sqrt{2}}\begin{pmatrix} 1& 1\\
-1 & 1\end{pmatrix}$.

\section{Models for the symmetric space of \texorpdfstring{$\mathrm{O}(A_\mathbb C,\sigma_\mathbb C)$}{} and their tangent spaces}\label{sec:OC_models}

In this section, we assume $(A,\sigma)$ to be a real Hermitian algebra. Let $G:=\mathrm{O}(A_\mathbb C,\sigma_\mathbb C)$ which is a noncompact reductive Lie group with a maximal compact subgroup $K:=\mathrm{O}(A,\sigma)$. This group can be embedded into the groups $\mathrm{O}_{(1,1)}(A,\sigma)$ seen as $\mathrm{Aut}(h)$ for $h(x,y):=\sigma(x)^t\begin{pmatrix}
    0 & 1 \\
    1 & 0
\end{pmatrix}y$ and $\mathrm{Aut}(b)$, for $b(x,y):=\sigma(x)^t\begin{pmatrix}
    -1 & 0 \\
    0 & 1
\end{pmatrix}y$,
where $x,y\in A^2$, as follows: For an element $g=g_0
+g_1i\in G$, where $g_0,g_1\in A$, we associate a matrix $\hat g=\begin{pmatrix}
    g_0 & g_1\\
    -g_1 & g_0
\end{pmatrix}$. The direct computation shows that $\hat g$ preserves both forms $h$ and $b$. We denote $\hat G:=\{\hat g\mid g\in G\}$. In fact, $\hat G=\mathrm{Aut}(h)\cap \mathrm{Aut}(b)$. The maximal compact subgroup $\hat K$ of $\hat G$ corresponding to $K$ agrees with the set of diagonal matrices in $\Hat G$. 

\begin{rem}
    Since $A^\times$ and $\mathrm{O}(A_\mathbb C,\sigma_\mathbb C)$ have isomorphic maximal compact subgroups, namely the group $\mathrm{O}(A,\sigma)$, then they have the same number of connected components.
\end{rem}

As $\mathrm{O}(A_\mathbb C,\sigma_\mathbb C)$ is a subgroup of $\mathrm{O}_{(1,1)}(A,\sigma)$, the models of the symmetric space of $\mathrm{O}(A_\mathbb C,\sigma_\mathbb C)$ will be subspaces of corresponding models of the symmetric space of $\mathrm{O}_{(1,1)}(A,\sigma)$. Therefore, we will not write explicitly the maps between models and the Riemannian metric because they are simply the restrictions of the corresponding ones from Section~\ref{sec:ort_models}.

\subsection{Space of indefinite involutive operators}\label{sec:indefinite_OC}

We consider the following space: 
\begin{align*}
\mathfrak{C}_G :=\{J\colon A^2\to A^2 \mid b(J(\cdot),J(\cdot))=-b(\cdot,\cdot)\}\cap \mathfrak{C}_{\mathrm{Aut}(h)}.
\end{align*}
This means more precisely that, for all $x,y\in A^2$,  $b(J(x),J(y))=-b(x,y)$ and if $x$ is regular, then $h(J(x),x)\in A^{\sigma}_+$. The following theorem can be proven similarly to Theorem~\ref{thm:involaopermodel_A}.

\begin{thm}\label{thm:involaopermodel_OC}
    The group $\hat G$ acts on $\mathfrak{C}_G$ by conjugation. This action is transitive and the stabilizer of $J_0=\begin{pmatrix}
        0 & 1 \\
        1 & 0
    \end{pmatrix}$
    agrees with $\hat K$.
\end{thm}

The following proposition is now immediate:

\begin{prop}\label{prop:tangent_sio_OC}
    The tangent space $T_J\mathfrak{C}_G$ at $J\in \mathfrak{C}_G$ is described as the following set of $A$-linear operators:
    $$T_J\mathfrak{C}_G=\{L\colon A^2\to A^2\mid LJ+JL=0,\text{ and } b(L(\cdot),\cdot)+b(\cdot,L(\cdot))=0\}.$$
    The group $\hat G$ acts on $T\mathfrak{C}_G$ by conjugation, i.e. for $g\in \hat G$,  
    $$
    g.(J,L)=(gJg^{-1},gLg^{-1}).
    $$ 
    In particular, 
    $$
    T_{gJg^{-1}}\mathfrak{C}_G = gT_J\mathfrak{C}_Gg^{-1}.
    $$
\end{prop}

The action of $J$ on $A^2$ extends naturally to the action on $A_\mathbb{C}$. The complex linear extension of this operator is denoted by $J_{\mathbb C}$. The complexification of the tangent space of $\mathfrak C_G$ at $J$ is given by:
\begin{align*}
    T_J^\mathbb C\mathfrak{C}_G & =\{L'\colon A_\mathbb{C}^2\to A_\mathbb{C}^2\mid L'J_\mathbb C+J_\mathbb CL'=0,\text{ and } b_\mathbb C(L'(\cdot),\cdot)+b_\mathbb C(\cdot,L'(\cdot))=0\},
\end{align*}
where $b_\mathbb C$ is the complex linear extension of $b$.

\subsection{Projective space model}\label{sec:projective_OC}

We consider the spaces
\begin{align*}
    \mathfrak{P}_G^\pm:=&\;\left\{xA\midwd\; x\in A^{2}\text{ regular such that }h(x,x)\in A^{\sigma}_{\pm},\text{ and }b(x,x)=0 \right\}\\
    =&\;\mathfrak{P}_{\mathrm{Aut}(h)}^\pm\cap\mathbb P\mathrm{Is}(b).
\end{align*}
This is the image of $\mathfrak{C}_{G}$ under the map~\eqref{eq:ort_CP}. Therefore, it is a model of the symmetric space of $G$. The group $\hat G$ acts by linear transformations on $\mathfrak{P}_{G}^\pm$.

Notice that $\mathfrak{P}_G^\pm$ are open domains in $ \mathbb{P}\mathrm{Is}(b)$. In particular, the tangent space at $l\in\mathfrak{P}_G^\pm$ can be seen as the space of $A$-linear operators 
\begin{align*}
    T_l\mathfrak{P}_G^\pm&=\{Q\in\mathrm{Hom}(l,A^2/l)\mid \text{for all $x\in l$, }b(Q(x),x)\in A^{-\sigma}\}\\
    &=\{Q\in\mathrm{Hom}(l,l)\mid \text{for all $x\in l$, }h(Q(x),x)\in A^{-\sigma}\}.    
\end{align*}
The identification $l\to A^2/l$ is given again using the non-degenerate forms $h$ and $b$ (in exactly the same way as in the case of the group $A^\times$). Its complexification is 
\begin{align}\label{eq:O-proj-compl.tang}
    T^\mathbb C_l\mathfrak{P}_G^\pm&=\{Q\in\mathrm{Hom}(l_\mathbb C,A_{\mathbb C}^2/l_\mathbb C)\mid \text{for all $x\in l_\mathbb C^2$, }b_\mathbb C(Q(x),x)\in A_\mathbb C^{-\sigma_\mathbb C}\}\\
    &=\{Q\in\mathrm{Hom}(l_\mathbb C,l_\mathbb C)\mid \text{for all $x\in l_\mathbb C^2$, } h_\mathbb C(Q(x),x)\in A_\mathbb C^{-\sigma_\mathbb C}\}.    
\end{align}

\noindent The topological closure of $\mathfrak{P}_G^\pm$ in the space $\mathbb {P}(A^2)$ of $A$-lines is
$$\overline{\mathfrak{P}}_G^\pm=\left\{xA\midwd\; x\in A^{2}\text{ regular such that }\pm h(x,x)\in A^\sigma_{\geq 0},\text{ and }b(x,x)=0 \right\};$$
in each case, this space is compact because $\mathbb{P}\mathrm{Is}(b)$ is compact. The boundaries $\partial\mathfrak{P}_G^\pm$ contain a common compact locus 
\[\mathrm{Is}(b)\cap\mathrm{Is}(h):=\left\{xA\midwd\; x\in A^{2}\text{ regular such that }h(x,x)=b(x,x)=0 \right\}\] of $h$- and $b$-isotropic lines.

\subsection{Half-space models}\label{sec:halfspace_OC}

The image of $\mathfrak{P}_{G}^\pm$ under the map~\eqref{eq:ort_PU} is the space 
\begin{align*}
    \mathfrak{U}_G^\pm:=&\ \{a\in A\mid \sigma(a)+ a\in A^\sigma_\pm\text{ and }\sigma(a)a=1\}\\
    =& \ (A^{\sigma}_\pm\oplus A^{-\sigma})\cap \mathrm{O}(A,\sigma).
\end{align*}
Therefore, it is a model of the symmetric space of $G$. The group $\hat G$ acts by M\"obius transformations on $\mathfrak{U}_{G}^\pm$.

Unlike half-space models from previous sections, this space is not an open domain in a vector space. However, this is an open subspace in $\mathrm{O}(A,\sigma)$ which is compact. In particular, $ \mathfrak{U}_G^\pm$ is precompact. Its compact closure is
\begin{align*}
    \overline{\mathfrak{U}}_G^\pm:=&\ \{a\in A\mid \pm(\sigma(a)+ a)\in A^\sigma_{\geq 0}\text{ and }\sigma(a)a=1\}\\
    =&\ (\pm A^{\sigma}_{\geq 0}\oplus A^{-\sigma})\cap \mathrm{O}(A,\sigma).
\end{align*}
The boundaries $\partial \mathfrak{U}_G^\pm$ contain a common compact locus
$$\{a\in A^{-\sigma}\mid a^2=-1\}=A^{-\sigma}\cap \mathrm{O}(A,\sigma).$$

\noindent Further, the tangent space at $a\in \mathfrak{U}_G^\pm$ is
$$
T_a\mathfrak{U}_G^\pm=T_a\mathrm{O}(A,\sigma)=aA^{-\sigma}=\{v\in A \mid a^{-1}v\in A^{-\sigma}\},
$$
with complexification
$$
T_a^\mathbb C\mathfrak{U}_G^\pm=\{v\in A_{\mathbb{C}} \mid a^{-1}v\in A_{\mathbb{C}}^{-\sigma_\mathbb C}\}=aA_{\mathbb{C}}^{-\sigma_\mathbb C}.
$$

\subsection{Precompact model}\label{sec:precompact_OC}

The image of $\mathfrak{P}_{G}^+$ under the map~\eqref{eq:ort_PB} is the space 
$$\mathfrak{B}_G:= \{a\in A^{-\sigma}\mid 1+a^2\in A^\sigma_+\}.$$
Therefore, it is a model of the symmetric space of $G$. 

This is a precompact open subspace in $A^{-\sigma}$. In particular, the tangent space at every $a\in \mathfrak{B}_G$ is $A^{-\sigma}$. Its complexification is $A^{-\sigma}\otimes_\mathbb R\mathbb C=A_\mathbb C^{-\sigma_{\mathbb C}}$.

The compact closure of $\mathfrak{B}_G$ is
$$\overline{\mathfrak{B}}_G:= \{a\in A^{-\sigma}\mid 1+a^2\in A^\sigma_{\geq 0}\}.$$
The group $R^{-1}\hat G R$ acts by M\"obius transformations on $\overline{\mathfrak{B}}_{G}$, where $R=\frac{1}{\sqrt{2}}\begin{pmatrix} 1& 1\\
-1 & 1\end{pmatrix}$. However, the direct computation shows that, up to a scalar multiple, the matrix $R$ conjugates the form $h$ to the form $b$ and vice versa. Thus, $R^{-1}\hat G R=\hat G$.

\section{Models for the symmetric space of \texorpdfstring{$\mathrm{Sp}_2(A, \sigma)$}{} and their tangent spaces}\label{sec:real_models}

As in the previous section, we assume $(A,\sigma)$ to be a real Hermitian algebra. Let $G:=\mathrm{Sp}_2(A,\sigma)$ and $K:=\mathrm{KSp}_2(A,\sigma)$ the maximal compact subgroup of $G$. We will introduce several models for the symmetric space $\mathcal{X}_{G}:=G/K$; for more details, we refer the reader to {\cite[Section 5]{ABRRW}}. We provide here an explicit description of the tangent space for each of these models, as well as the description of the differential of the various diffeomorphisms between the models.

\subsection{Space of complex structures}\label{sec_cplx-str-mod}
A \emph{complex structure} on $A^2$ is an $A$-linear map $J\colon A^2\to A^2$ such that $J^2=-\mathrm{Id}$. Each $A$-linear map $J$ on $A^2$ defines a $\sigma$-sesquilinear form $h_J\colon A^2\times A^2\to A$ by
$$
h_J(\bullet,\bullet):=\omega(J(\bullet),\bullet),
$$ 
where $\omega$ is the standard symplectic form on $A^2$ introduced in Section~\ref{sec:sesq_forms}. 

\begin{defn}
    The \emph{space of complex structures} is defined as the space
\begin{align*}
    \mathfrak{C}_G:=\{J\text{ is a complex structure on }A^2 \mid h_J\text{ is a $\sigma$-inner product}\}.
\end{align*}
\end{defn}
In particular, $J_0\in\mathfrak{C}_G$, for $J_0(x,y):=(y,-x)$, the standard complex structure on $A^2$. The group $G$ acts on $\mathfrak{C}_G$ by conjugation, i.e. $J\xrightarrow{g} gJg^{-1}$, smoothly and transitively, and the stabilizer of $J_0$ agrees with $K$. Thus $\mathfrak C_G$ is a model of the symmetric space $G/K$ (cf. \cite[Theorem 6.7]{ABRRW}). 

\begin{rem}
    For every $J\in\mathfrak{C}_G$, we have that $J\in G$.
\end{rem}

\begin{prop}\label{prop:tangent_scs}
    The tangent space $T_J\mathfrak{C}_G$ at $J\in \mathfrak{C}_G$ is described by
    $$T_J\mathfrak{C}_G=\left\{L\colon A^2\to A^2\mid \text{$L$ is $A$-linear, $h_L$ is $\sigma$-symmetric and }LJ+JL=0\right\}.$$
    The group $\mathrm{Sp}_2(A,\sigma)$ acts on $T\mathfrak{C}_G$  by conjugation, i.e. $$g.(J,L)=(gJg^{-1},gLg^{-1}).$$ 
    In particular, 
    $$T_{gJg^{-1}}\mathfrak{C}_G = gT_J\mathfrak{C}_Gg^{-1}.$$
\end{prop}

\begin{proof}
    Let $J(t)\in\mathfrak{C}_G$ be a smooth curve defined on an interval $I\ni 0$ and $J(0)=J$. Since $J(t)^2=-\mathrm{Id}$ for all $t\in I$, the derivation of this equality at $t=0$ is $\dot J(0) J(0)+ J(0)\dot J(0)=0$.
  
    \noindent The condition that $h_{J(t)}$ is a $\sigma$-inner product induces that $h_{\dot J(0)}$ is $\sigma$-symmetric.

    \noindent The claim about the action on the tangent bundle is trivial. 
\end{proof}

Since $J$ is a symplectic operator, then 
\[\omega((1,0),(1,0))=\omega(J(1,0),J(1,0))=0.\] 
We denote 
\[r:=\omega(J(1,0),(1,0))=h_J((1,0),(1,0))\in A^\sigma_+,\] which is in particular an invertible element. This implies that $\{(1,0),J(1,0)\}$ is a basis of $A^2$ and allows us to consider the following isomorphism of right $A$-modules:
\[\begin{matrix}
    f_J\colon & A^2 & \to & A_\mathbb C\\
           & (a,0)+J(b,0) & \mapsto & a+ib.
\end{matrix}\]
In other words, $f_J$ allows us to equip the right $A$-module $A^2$ with a structure of $\mathbb C$-algebra.

Further, for $L\in T_J\mathfrak{C}_G$, we denote by 
$$\hat L:= f_J\circ L \circ f_J^{-1}\colon A_\mathbb C\to A_\mathbb C.$$ 
The condition $LJ=-JL$ is equivalent to the condition that $\hat L$ is anti-linear, i.e. $\hat L(z_1z_2)=\hat L(z_1)\bar z_2$. In particular, $\hat L$ is well-defined if we fix the value $\hat L(1)=:l=l_1+il_2\in A_\mathbb C$, where $l_1,l_2\in A$.

Finally, the condition that $h_L$ is $\sigma$-symmetric is equivalent to the condition that $rl\in A_\mathbb C^{\sigma_\mathbb{C}}$. Indeed, identifying $A^2$ with $A_\mathbb C$ using $f_J$, we can write:

$$h_L(1,1)=\omega(l_1+il_2,1)=\sigma(l_2)r\in A^\sigma,$$
$$h_L(i,i)=\omega(l_2-il_1,i)=-\sigma(l_2)r\in A^\sigma,$$
$$h_L(1,i)=\omega(l_1+il_2,i)=-\sigma(l_1)r,$$
$$h_L(i,1)=\omega(l_2-il_1,1)=-\sigma(l_1)r=\sigma(h_L(1,i)).$$
This is equivalent to having $rl_1\in A^\sigma$ and $rl_2\in A^\sigma$, in other words, $rl\in A^{\sigma_{\mathbb C}}_\mathbb C$.

Thus, the tangent space $T_J\mathfrak{C}_G$ can be identified with $A_\mathbb C^{\sigma_\mathbb C}$ as follows:
\begin{equation}\label{tangent_complex_simplified}
T_J\mathfrak{C}_G=\{l\in A_\mathbb C\mid rl\in A_\mathbb{C}^{\sigma_\mathbb C}\}=r^{-1}A_\mathbb{C}^{\sigma_\mathbb C}\cong A_\mathbb{C}^{\sigma_\mathbb C}.
\end{equation}

The complex structures model $\mathfrak{C}_G$ identifies with $G/K$ as follows:
\begin{align*}
F_{\mathfrak{C}_G} \colon G/K&\to\mathfrak{C}_G,\\
gK& \mapsto g(J_0)=gJ_0g^{-1}
\end{align*}
where $J_0$ is the standard complex structure as above. The tangent bundle $T(G/K)$ can be described as 
$$T(G/K)=\{[g,v+\mathrm{Ad}(g)\mathfrak k]\mid g\in G,\;v\in\mathfrak g\},$$
where $(g,v+\mathrm{Ad}(g)\mathfrak k)\sim (gk,v+\mathrm{Ad}(g)\mathfrak k)$. We now write the explicit identification of $T\mathfrak{C}_G$ and $T(G/K)$ given by the differential of $F_{\mathfrak{C}_G}$: 
$$dF_{\mathfrak{C}_G}([g,v+\mathrm{Ad}(g)\mathfrak k])=(J,L),$$ 
where $J:=gJ_0g^{-1}$, $L=vJ-Jv$. Direct computation shows that $(J,L)$ does not depend on the choice of a representative of the class $[g,v+\mathrm{Ad}(g)\mathfrak k]$.

We now describe the complex linear extensions of the operators $L$ and $J$. Let $L\in T_J\mathfrak{C}_G$ for some $J\in\mathfrak{C}_G$. The actions of $J$ and $L$ on $A^2$ extend naturally to the action on $A_\mathbb C^2$. The complex linear extensions of these operators are denoted as usual by $J_{\mathbb C}$ and $L_\mathbb C$. Then $J_\mathbb C$ has two isotropic $A_\mathbb C$-eigenlines $l_\pm\subset A^2_{\mathbb{C}}$ corresponding to eigenvalues $\pm i$ (cf.~\cite[Proposition~5.20]{ABRRW}).

\begin{lem}\label{lem:complexified_complex_structure}
    Let $J\in\mathfrak{C}_G$ and $L\in T_J\mathfrak{C}_G$. Then $J_\mathbb C\circ L_\mathbb C|_{l_\pm}=\mp iL_\mathbb C|_{l_\pm}$. In particular, $L_\mathbb C(l_\pm)\subseteq l_\mp$. 
\end{lem}

\begin{proof}
    Since for all $L\in T_J\mathfrak{C}_G$, it is $LJ=-JL$, we get for all $v_\pm\in l_\pm$:
    $$J_\mathbb C(L_\mathbb C(v_\pm))=- L_\mathbb C(J_\mathbb C(v_\pm))=\mp iL_\mathbb C(v_\pm).$$
\end{proof}

\begin{rem}
    The map $L_\mathbb C$ gives rise to two maps $L_\pm\colon l_\pm\to A_\mathbb C^2/l_\pm$ mapping $v_\pm\mapsto Lv_\pm + l_\pm$. 
\end{rem}

Notice that $l_\pm=\bar l_\mp$ because $J$ is a real operator. Let $v_\pm$ be a generator of $l_\pm$ such that $v_-=i\bar v_+$ and $\omega_\mathbb C(v_+,v_-)=1$. Then there exist $a_\pm\in A_\mathbb C$ such that $L_\mathbb C v_\pm=v_\mp a_\pm$. Further, since $\omega_{\mathbb C}(L_\mathbb C v_\pm,v_\pm)\in A^{\sigma_\mathbb C}_\mathbb C$, we obtain 
$$\omega(L_\mathbb C v_\pm,v_\pm)=\omega(v_\mp a_\pm,v_\pm)=\mp\sigma_{\mathbb C}(a_\pm),$$
i.e. $a_\pm\in A^{\sigma_\mathbb C}_\mathbb C$. Moreover, 
$$v_+\bar a_+=i\overline{v_- a_+}=i\overline{L_\mathbb C(v_+)}=\overline{L_\mathbb C(-iv_+)}=\overline{L_\mathbb C(\bar v_-)}=L_\mathbb C(v_-)=v_+ a_-,$$ 
i.e. $a_\mp=\bar a_\pm$. Thus, in the basis $v=(v_+,v_-)$, we obtain:

\begin{equation}\label{eqn:form of J and L}
J_\mathbb C=\begin{bmatrix}
i & 0 \\
0 & -i
\end{bmatrix}_v,\; L_\mathbb C=\begin{bmatrix}
    0 & \bar a_+ \\
    a_+ & 0
\end{bmatrix}_v,\text{ where }a_+\in A_\mathbb C^{\sigma_\mathbb C}.
\end{equation}
Notice, that since the basis $(v_+,v_-)$ is not uniquely defined, the same holds for $a_+$. After complexification of $T\mathfrak{C}_G$, the variables $a_+$ and $\bar a_+$ become elements of $A_\mathbb C^{\sigma_{\mathbb C\{i\}}}\otimes_\mathbb R\mathbb C\{I\}$. Then the linear map 
\begin{equation}\label{splitting_lin_map}L':=\frac{1}{2}(L_\mathbb C+\theta_i(L_\mathbb C)-iL_\mathbb C+i\theta_i(L_\mathbb C))\in\mathrm{Mat}_2(A_{\mathbb C\{I\}})
\end{equation}
is antidiagonal with $A_\mathbb{C}^{\sigma_\mathbb C}$-elements on the antidiagonal which are independent. After identification of the imaginary units $i$ and $I$, we obtain the following descriptions:

\noindent For $J\in\mathfrak{C}_G$, 
\begin{align}
 T^{\mathbb C}_J\mathfrak{C}_G& =\left\{L'\colon A_\mathbb C^2\to A_\mathbb C^2 \mid L'\text{ is $\mathbb C$-linear}, L'J_\mathbb C+J_\mathbb CL'=0,\;h_{L'}\text{ is $\sigma_\mathbb C$-symmetric}\right\}  \nonumber \\
  & \cong \left\{(a_+,a_-)\in A_\mathbb C^{\sigma_\mathbb C}\times A_\mathbb C^{\sigma_\mathbb C}\right\}, \text{ and }\label{formula:complexified tangent bundle_cx str model} \\
 T^{\mathbb C}\mathfrak{C}_G& \cong\left\{(J,a_+,a_-)\mid J\in\mathfrak{C}_G,\;a_+,a_-\in A_\mathbb C^{\sigma_\mathbb C}\right\}. \nonumber
\end{align}

\subsection{Projective space model}\label{sec_proj-sp-mod}
To introduce the projective space model for $\mathcal{X}_{G}$, one first needs to consider the complexified algebras. Denote by $(A_\mathbb{C},\sigma_\mathbb{C})$ the complexification of the Hermitian and semisimple $\mathbb{R}$-algebra $(A,\sigma)$, that is, $A_\mathbb{C}=A\otimes_\mathbb{R}\mathbb{C}$, and the anti-involution $\sigma_\mathbb{C}$ is the $\mathbb{C}$-linear extension of $\sigma$ as follows
$$
\sigma_\mathbb{C}(X+iY):=\sigma(X)+\sigma(Y)i.
$$
Then $(A_\mathbb{C},\sigma_\mathbb{C})$ is never a Hermitian algebra, since
\begin{align*}
    \sigma_\mathbb{C}(1)\cdot 1+\sigma_\mathbb{C}(i)\cdot i=0.
\end{align*}
Note that $\sigma$ also extends $\mathbb{C}$-anti-linearly to an anti-involution $\bar \sigma_\mathbb{C}$ on $A_\mathbb{C}$ as
$$
\bar \sigma_\mathbb{C}(X+iY):=\sigma(X)-\sigma(Y)i;
$$
this makes $(A_\mathbb{C},\bar\sigma_\mathbb{C})$ a Hermitian and semisimple $\mathbb{C}$-algebra (see {\cite[Corollary 2.79]{ABRRW}}). 

The standard symplectic form $\omega$ on $A^2$ also extends to a symplectic form on $A_\mathbb{C}^2$ as
\begin{align}\label{com_symp}
    \omega_\mathbb{C}(x,y)=\sigma_\mathbb{C}(x)^t\begin{pmatrix}
        0&1\\
        -1&0
    \end{pmatrix}y.
\end{align}
This symplectic form $\omega_\mathbb{C}$ induces an indefinite form $h\colon A_\mathbb{C}^2\times A_\mathbb{C}^2\to A_\mathbb{C}$ as
\begin{align*}
    h(x,y):=i\omega_\mathbb{C}(\bar{x},y),
\end{align*}
which is easily seen to be $\bar\sigma_\mathbb{C}$-sesquilinear in the sense that
\begin{align*}
    h(y,x)=\bar\sigma_\mathbb{C}(h(x,y)).
\end{align*}

The action of $G$ on $A_\mathbb{C}^2$ preserves both the symplectic form $\omega_\mathbb{C}$ and the indefinite form $h$, which is proportional to the standard symplectic form for the algebra $(A_\mathbb C,\bar\sigma_\mathbb C)$. In fact, this means 
$$G = \mathrm{Sp}_{2}(A_\mathbb C,\sigma_\mathbb C)\cap\mathrm{Sp}_{2}(A_\mathbb C,\bar\sigma_\mathbb C).$$
However, as we have seen in Section~\ref{sec:incarnations}, the group $\mathrm{Sp}_{2}(A_\mathbb C,\bar\sigma_\mathbb C)$ is isomorphic to $\mathrm{O}_{(1,1)}(A_\mathbb C,\bar\sigma_\mathbb C)$. By Section~\ref{sec:projective_O}, the projective space models for its symmetric space are 
\begin{align*}
    \mathfrak{P}^\pm_{\mathrm{Sp}_{2}(A_\mathbb C,\bar\sigma_\mathbb C)}:=\left\{xA_\mathbb C\midwd\; x\in A_\mathbb C^{2}\text{ regular such that }h(x,x)\in (A_\mathbb C^{\bar\sigma_\mathbb C})_{\pm} \right\}.
\end{align*}

The intersection with the space of isotropic lines
\begin{align*}
   \mathfrak{P}_G^\pm=\mathfrak{P}^\pm_{\mathrm{Sp}_{2}(A_\mathbb C,\bar\sigma_\mathbb C)}\cap \mathbb{P}\mathrm{Is}\left( {{\omega}_{\mathbb{C}}} \right)
\end{align*}
gives models of the symmetric space $G/K$ (see \cite[Corollary 5.21]{ABRRW}). We thus have the following:

\begin{defn}
    Either space $\mathfrak{P}^+_G$ or $\mathfrak{P}^-_G$ is called the \emph{projective space model}  for the symmetric space $G/K$.
\end{defn}

\begin{prop}[{\cite[Proposition~5.18]{ABRRW}}]
    The linear action of $G<\mathrm{GL}_2(A_\mathbb C)$ on $A_\mathbb C^2$ extends naturally to the action to $\mathbb{P}\mathrm{Is}(\omega_\mathbb{C})$ which preserves $\mathfrak{P}_G^+$ and $\mathfrak{P}_G^-$. This action restricted to $\mathfrak{P}_G^+$ and $\mathfrak{P}_G^-$ is transitive. The stabilizer of $(i,1)^tA_\mathbb C$ agrees with $K$.
\end{prop}

The spaces $\mathfrak{P}_G^\pm$ are open in the space of isotropic lines $\mathbb{P}\mathrm{Is}\left( {{\omega }_{\mathbb{C}}}\right)$, which is a compact manifold. Therefore, the closure of $\mathfrak{P}_G^\pm$ which agrees with 
\begin{align*}
   \overline {\mathfrak{P}}_G^\pm=\overline{\mathfrak{P}}^\pm_{\mathrm{Sp}_{2}(A_\mathbb C,\bar\sigma_\mathbb C)}\cap \mathbb{P}\mathrm{Is}\left( {{\omega}_{\mathbb{C}}} \right)
\end{align*}
 is compact.

We can apply Proposition~\ref{prop:tangent_gen_omega}, Remark~\ref{prop:tangent_skew},  and Corollary~\ref{cor:tangent_skew_invariant} to compute the tangent spaces and the action of $G$ on them:

\begin{prop}\label{prop:tangent_pm}
    Let $l\in \mathfrak{P}_G^\pm$ and $x$ be a generator of $l$. Then the tangent space $T_l\mathfrak{P}_G^\pm$ is isomorphic to both
    \begin{equation}\label{first_descr_tangent_proj}
    {\left\{ [x,v]\mid v\in A^{2}_{\mathbb C}/l,\; \omega_\mathbb C(v,x)\in A_\mathbb C^{{\sigma_{\mathbb C}}} \right\}}\cong x^{\perp_{\omega_\mathbb C}}/l,
    \end{equation}
    and
    \begin{equation}\label{second_descr_tangent_proj}
    \left\{Q\in\mathrm{Hom}(l,A_\mathbb C^2/l)\mid \omega_\mathbb C(Q(x),x)\in A^{\sigma_\mathbb C}_\mathbb C\right\}.
    \end{equation}
    As before, the isomorphism between these two spaces is given by $[x,v]=[Q\colon xa\mapsto va]$.

\noindent The action of $\mathrm{Sp}_2(A,\sigma)$ on the first realization (\ref{first_descr_tangent_proj}) of the tangent bundle is linear. The action on the second realization (\ref{second_descr_tangent_proj}) is by conjugation (cf. Proposition \ref{prop:tangent_gen_omega}).
\end{prop}

\begin{rem}
    As before, we identify
    $$T_l\mathfrak{P}_G^\pm=\left\{Q\in\mathrm{Hom}(l,A_\mathbb C^2/l)\mid \omega_\mathbb C(Q(x),x)\in A^{\sigma_\mathbb C}_\mathbb C\right\}$$
    since it is a representative-independent description.
\end{rem}

The complex structures model and the projective space model are related by the following isometry: 
$$
\begin{array}{lccr}
    F_{\mathfrak{C}_G,\mathfrak{P}_G^\pm}\colon &  \mathfrak{C}_G & \to & \mathfrak{P}_G^\pm\\
     & J & \mapsto & l_J^\pm,
\end{array}
$$
where the isotropic line $l_J^\pm$ is the $(\mp i)$-eigenline of $J_\mathbb C$, that is, $J_\mathbb C(x)=\mp ix$, for all $x\in l_J^\pm$. This isometry preserves the structure of Hermitian symmetric spaces and is $G$-equivariant.

\begin{prop}\label{prop:differential_F}
    Let $J\in\mathfrak{C}_G$, $L\in T_J\mathfrak{C}_G$ and $x\in l_J^\pm$ be a generator of $l_J^\pm$. Then
    $$d_JF_{\mathfrak{C}_G,\mathfrak{P}_G^\pm}(L)=[x,\pm\frac{i}{2}L_\mathbb C(x)+l_J^\pm]=\pm\frac{i}{2}[L_\mathbb C]\in T_{l_J^\pm}\mathfrak{P}_G^\pm,$$ 
    where $[L_\mathbb C]\colon l_J^\pm\to A_\mathbb C^2/l_J^\pm$ is the natural projection of the linear operator $L_\mathbb C|_{l_J^\pm}\colon l_J^\pm\to A_\mathbb C^2$.
\end{prop}

\begin{proof}
    Let $J(t)$ be a smooth path in $\mathfrak{C}_G$ such that $J(0)=J$ and $\dot J(0)=L$. Let $x(t)$ be a smooth path in $A^2_\mathbb C$ such that $x(t)$ generates $l_{J(t)}^\pm$ and $x(0)=x$. Then $J_\mathbb C(t)x(t)=\mp ix(t)$, for all $t$. Differentiating this equality at $t=0$, we obtain 
    \begin{align*}
    L_\mathbb C(x)+J_\mathbb C(\dot x(0)) & =\mp i \dot x(0),\\
    (J_\mathbb C\pm i\mathrm{Id})\dot x(0) & =-L_\mathbb C(x).
    \end{align*}
    Notice that, if $\dot x(0)$ satisfies the equality above, then for any $a\in A_\mathbb C$, $\dot x(0)+xa$ satisfies it too. Now it is enough to check that $\pm\frac{i}{2}L_\mathbb C(x)$ satisfies it:
    \[(J_\mathbb C\pm i\mathrm{Id})\left(\pm\frac{i}{2}L_\mathbb C(x)\right)
    =\frac{i}{2}(iL_\mathbb C(x) + iL_\mathbb C(x))=-L_\mathbb C(x).\]
    This shows that 
    $$d_JF_{\mathfrak{C}_G,\mathfrak{P}_G^\pm}(L)=[x,\pm\frac{i}{2}L(x)+l_J^\pm]=\pm\frac{i}{2}[L_\mathbb C].$$
\end{proof}

Let $l\in\mathfrak{P}_G^\pm$; its complexification is denoted by $l_\mathbb C:=l\otimes_\mathbb R\mathbb C\subset A_\mathbb C^2\otimes_\mathbb R \mathbb C$. The complexification of the tangent space of $\mathfrak P_G^\pm$ at $l$ is given by the space of all $A_\mathbb C\otimes_\mathbb R \mathbb C$-linear morphisms:
$$T_l^\mathbb C\mathfrak{P}_G^\pm=\left\{Q\in\mathrm{Hom}(l_\mathbb C,A_\mathbb C^2\otimes_\mathbb R\mathbb C/l_\mathbb C)\mid \omega_{\mathbb C\mathbb C}(Q(x),x)\in A^{\sigma_\mathbb C}_\mathbb C\otimes_\mathbb R\mathbb C\right\},$$
where $\omega_{\mathbb C\mathbb C}$ denotes the $\mathbb C$-linear extension of $\omega_\mathbb C$. Using the identification~\eqref{formula:complexified tangent bundle_cx str model}, we obtain another description of $T_l^\mathbb C\mathfrak{P}_G^\pm$ which does not require double complexification:
\begin{align}\label{eq:compl.tangent.P.Sp}
    \begin{aligned}
    T_l^\mathbb C\mathfrak{P}_G^\pm&=\left\{(Q,Q')\in\mathrm{Hom}(l,A_\mathbb C^2/l)\times \mathrm{Hom}(\bar l,A_\mathbb C^2/\bar l)\midwd
    \begin{aligned}
        &\ \omega_{\mathbb C}(Q(x),x)\in A^{\sigma_\mathbb C}_\mathbb C,\\
        &\ \omega_{\mathbb C}(Q'(x),x)\in A^{\sigma_\mathbb C}_\mathbb C
    \end{aligned}\right\}\\
    &=T_l\mathfrak{P}_G^\pm\oplus T_{\bar l}\mathfrak{P}_G^\mp. 
    \end{aligned}
\end{align}
Notice that if $l\in\mathfrak{P}_G^\pm$, then $\bar l\in\mathfrak{P}_G^\mp$.

The spaces $\mathfrak{P}_G^\pm$ are open domains in  the projective space $\mathbb{P}\mathrm{Is}(\omega_\mathbb C)$. Later on we will construct affine charts for these spaces. 

\subsection{Half-space models}\label{sec:upperhalf_real}

We next discuss the half-space models of the symmetric space $\mathcal X_{G}$. We denote as before by $A_\mathbb C$ the complexification of $A$, $A_\mathbb{C}:=A\otimes_\mathbb R\mathbb C$ and extend $\sigma$ to $A_\mathbb C$ complex linearly, that is, $\sigma_\mathbb C(x+yi):=\sigma(x)+\sigma(y)i$.
Every element of $z\in A_\mathbb{C}^{\sigma_\mathbb{C}}$ can be uniquely written as $z=x+yi$ where $x,y\in A^\sigma$. We denote by $x:=\mathrm{Re}(z)$, $y:=\mathrm{Im}(z)$. We also have a complex conjugation on $A_\mathbb{C}$ given by $\bar z=x-yi$.

\begin{defn}
    The \emph{upper (resp. lower) half-space model} is defined as the space
    $$\mathfrak U_G^\pm:=\{z\in A_\mathbb{C}^{\sigma_\mathbb{C}}\mid \mathrm{Im}(z)\in A^\sigma_\pm\}.$$
\end{defn}

\noindent In the following, we will only discuss the upper half-space model $\mathfrak U_G^+$. However, everything we prove for the upper half-space model can be appropriately adapted to the case of the lower half-space model as well. To simplify notation, we denote $\mathfrak U_G:=\mathfrak U_G^+$ (as is also done in~\cite{ABRRW}).  First of all, this is indeed a model of the symmetric space $\mathcal X_{G}$:

\begin{prop}[{\cite[Proposition~5.8]{ABRRW}}]
    The group $G$ acts smoothly and transitively on $\mathfrak{U}_G$ by M\"obius transformations, i.e. 
    $$g.z=(az+b)(cz+d)^{-1},$$ 
    where $g=\begin{pmatrix}
        a & b \\
        c & d
    \end{pmatrix}$. The stabilizer of $i\in\mathfrak{U}_G$ agrees with $K$.
\end{prop}

Since $\mathfrak U_G$ is an open subspace of the vector space $A_\mathbb{C}^{\sigma_\mathbb{C}}$, then for every $z\in \mathfrak U_G$, $T_z\mathfrak U_G = A_\mathbb{C}^{\sigma_\mathbb{C}}$. The action of $G$ extends naturally to the action on $T\mathfrak U_G=\mathfrak U_G\times A^{\sigma_\mathbb C}_\mathbb C$. For $g=\begin{pmatrix}
        a & b \\
        c & d
    \end{pmatrix}$
and $(z,v)\in \mathfrak U_G\times A^{\sigma_\mathbb C}_\mathbb C$, the action is given by 
\begin{equation}\label{eq:diff_moebius}
    g.(z,v)=\left((az+b)(cz+d)^{-1}, av(cz+d)^{-1}-(az+b)(cz+d)^{-1}cv(cz+d)^{-1}\right).  
\end{equation}

For every $z\in\mathfrak U_G$, the stabilizer $\mathrm{Stab}_G(z)$ is conjugate to the stabilizer of $z=i$ which is $K$. We are now going to understand the action of $K$ on $T_i\mathfrak U_G\cong A_\mathbb C^{\sigma_\mathbb C}$. By~\eqref{eq:diff_moebius}, for $g=\begin{pmatrix}
    a & b \\
    -b & a
\end{pmatrix}\in K$ and $v\in T_i\mathfrak U_G$,
$$g(v)
=(a+bi)v(-bi+a)^{-1}
=(a+bi)v(\sigma(a)+\sigma(b)i)=qv\sigma_\mathbb C(q)$$
because $(-bi+a)^{-1}=\sigma(a)+\sigma(b)i$ and where $q=a+bi\in\mathrm O(A_\mathbb C,\bar\sigma_\mathbb C)\cong K$ (cf. \eqref{eq:conj_symp_real}). Let $v=v_1+v_2i$, where $v_i\in A^\sigma$, then
$$g.v
=av_1\sigma(a)-bv_1\sigma(b)-av_2\sigma(b)-bv_2\sigma(a)
+(av_1\sigma(b)+bv_1\sigma(a)+av_2\sigma(a)-bv_2\sigma(b))i.$$

Note that $A_\mathbb C^{\sigma_\mathbb C}$ can be identified with $\mathfrak m$ via the map $v_1+v_2i\mapsto \begin{pmatrix}
    v_2 & v_1 \\
    v_1 & -v_2
\end{pmatrix}$. This action agrees with the isotropy action:
\begin{align*}
& \begin{pmatrix}
    a & b \\
    -b & a
\end{pmatrix}\begin{pmatrix}
    v_2 & v_1 \\
    v_1 & -v_2
\end{pmatrix}\begin{pmatrix}
    a & b \\
    -b & a
\end{pmatrix}^{-1}=\begin{pmatrix}
    a & b \\
    -b & a
\end{pmatrix}\begin{pmatrix}
    v_2 & v_1 \\
    v_1 & -v_2
\end{pmatrix}\begin{pmatrix}
    \sigma(a) & -\sigma(b) \\
    \sigma(b) & \sigma(a)
\end{pmatrix}\\
= & \begin{pmatrix}
    av_2+bv_1 & av_1-bv_2 \\
    -bv_2+av_1 & -bv_1-av_2
\end{pmatrix}\begin{pmatrix}
    \sigma(a) & -\sigma(b) \\
    \sigma(b) & \sigma(a)
\end{pmatrix} = \begin{pmatrix} \alpha & \beta \\ \gamma & \delta \end{pmatrix},
\end{align*}
where 
\begin{align*}
& \alpha:= av_2\sigma(a)+bv_1\sigma(a)+av_1\sigma(b)-bv_2\sigma(b),\\ 
& \beta := -av_2\sigma(b)-bv_1\sigma(b)+av_1\sigma(a)-bv_2\sigma(a), \\
& \gamma:= -av_2\sigma(b)-bv_1\sigma(b)+av_1\sigma(a)-bv_2\sigma(a),\\
& \delta:= -av_2\sigma(a)-bv_1\sigma(a)-av_1\sigma(b)+bv_2\sigma(b).    
\end{align*}

The projective space models and the half-space models are related by the following isometries: 
$$\begin{matrix}
F_{\mathfrak{P}_G^\pm,\mathfrak U_G^\pm} \colon & \mathfrak{P}_G^\pm & \to & \mathfrak U_G^\pm\\
 & (x_1,x_2)^tA_\mathbb C & \mapsto & x_1x_2^{-1},
\end{matrix}$$
which are $G$-equivariant. The differentials of these maps can be computed immediately:

\begin{prop}
    Let $x=(x_1,x_2)^t\in\mathrm{Is}(\omega_\mathbb C)$ such that $l=xA_\mathbb C\in\mathfrak{P}_G^\pm$. Let $v=[x,(v_1,v_2)^t+l]\in T_{l}\mathfrak{P}_G^\pm$. Then 
    $$d_lF_{\mathfrak{P}_G^\pm,\mathfrak U_G^\pm}(v)=(v_1-F_{\mathfrak{P}_G^\pm,\mathfrak U_G^\pm}(x) v_2)x_2^{-1}.$$ 
\end{prop}

\subsection{Precompact model}\label{sec:precompact_real}

Finally, we discuss the precompact model of the symmetric space of $G$. In this section, we denote
    $G':=T^{-1}GT$ where $T=\frac{1}{\sqrt{2}}\begin{pmatrix}
        1 & i\\
        i & 1
    \end{pmatrix}$.
We also denote $K'=T^{-1}KT$ a maximal compact subgroup of $G'$ (cf. Section~\ref{sec:incarnations}). Since $G$ and $G'$ are conjugate, their symmetric spaces are isometric Hermitian manifolds.

\begin{defn}
    The space
    $$\mathfrak{B}_G:=\{z\in A^{\sigma_\mathbb C}_\mathbb C\mid 1-\bar zz\in (A^{\bar\sigma_\mathbb C}_\mathbb  C)_+\}$$
    is called the \emph{precompact model}.
\end{defn}

\noindent This is also a model of the symmetric space $\mathcal X_{G}$:

\begin{prop}[{\cite[Proposition~5.23]{ABRRW}}]
    The group $G'$ acts smoothly and transitively on $\mathfrak{B}_G$ by M\"obius transformations as
    \[g.z=(az+b)(cz+d)^{-1},\] 
    where $T^{-1}gT=\begin{pmatrix}
        a & b \\
        c & d
    \end{pmatrix}$ and $T=\frac{1}{\sqrt{2}}\begin{pmatrix}
        1 & i\\
        i & 1
    \end{pmatrix}$. The stabilizer of $0\in\mathfrak{B}_G$ agrees with $K'$.
\end{prop}

The closure of the space $\mathfrak{B}_G$ is compact:
$$\overline{\mathfrak{B}}_G:=\{z\in A^{\sigma_\mathbb C}_\mathbb C\mid 1-\bar zz\in (A^{\bar\sigma_\mathbb C}_\mathbb  C)_{\geq 0}\}.$$

Since $\mathfrak{B}_G$ is an open subspace of the vector space $A_\mathbb{C}^{\sigma_\mathbb{C}}$, then for every $z\in \mathfrak{B}_G$, $T_z\mathfrak{B}_G = A_\mathbb{C}^{\sigma_\mathbb{C}}$.  The~action of $G'$ extends naturally to the action on the tangent bundle $T\mathfrak{B}_G=\mathfrak{B}_G\times A^{\sigma_\mathbb C}_\mathbb C$. For 
    $g=\begin{pmatrix}
        a & b \\
        c & d
    \end{pmatrix}$
and $(z,v)\in \mathfrak{B}_G\times A^{\sigma_\mathbb C}_\mathbb C$, the action is given by
$$g.(z,v)=\left((az+b)(cz+d)^{-1}, av(cz+d)^{-1}-(az+b)(cz+d)^{-1}cv(cz+d)^{-1}\right).$$

For every $z\in\mathfrak B_G$, the stabilizer $\mathrm{Stab}_{G'}(z)$ is conjugate to the stabilizer of $x=0$ which is $K'$. We are now going to understand the action of $K'$ on $T_0\mathfrak B_G= A_\mathbb{C}^{\sigma_\mathbb{C}}$. For $u\in \mathrm O(A_\mathbb C,\bar\sigma_\mathbb C)$,
    $g=\begin{pmatrix}
        u & 0 \\
        0 & \bar u
    \end{pmatrix}\in K'$ 
(cf.~\eqref{eq:conj_symp_real}) and $v\in T_0\mathfrak B_G$, then
$$g.v=uv\bar u^{-1}=uv\sigma(u).$$

The projective space model and the precompact model are related by the following $G$-equivariant isometry: 
$$\begin{matrix}
F_{\mathfrak{P}_G^+,\mathfrak{B}_G} \colon & \mathfrak{P}_G^+ & \to & \mathfrak{B}_G\\
 & (x_1,x_2)^tA_\mathbb C & \mapsto & y_1y_2^{-1},
\end{matrix}$$
where $y=\begin{pmatrix}y_1 \\y_2\end{pmatrix}=\frac{1}{\sqrt{2}}\begin{pmatrix}
1 & -i \\
-i & 1
\end{pmatrix}\begin{pmatrix}x_1 \\x_2\end{pmatrix}$. Its differential can be obtained as follows:

\begin{prop}
    Let $x=(x_1,x_2)^t\in\mathrm{Is}(\omega_\mathbb C)$ such that $l=xA_\mathbb C\in\mathfrak{P}_G^+$, and $v=[x,(v_1,v_2)^t+l]\in T_l\mathfrak{P}_G^+$. Then 
    $$d_{l}F_{\mathfrak{P}_G^+,\mathfrak{B}_G}(v)=(w_1-F_{\mathfrak{P}_G^+,\mathfrak{B}_G}(x) w_2)y_2^{-1},$$
    where $y=\begin{pmatrix}y_1 \\y_2\end{pmatrix}=\frac{1}{\sqrt{2}}\begin{pmatrix}
1 & -i \\
-i & 1
\end{pmatrix}\begin{pmatrix}x_1 \\x_2\end{pmatrix}$, and  $w=\begin{pmatrix}w_1 \\w_2\end{pmatrix}=\frac{1}{\sqrt{2}}\begin{pmatrix}
1 & -i \\
-i & 1
\end{pmatrix}\begin{pmatrix}v_1 \\v_2\end{pmatrix}$.
\end{prop}

\subsection{Riemannian metric on the models}

In terms of the upper half-space model, it is particularly easy to write the Riemannian metric on the Riemannian symmetric space of $\mathrm{Sp}_2(A,\sigma)$ (cf.~\cite[Exercise~8, Chapter~XIII]{FK}). Namely, for $z=x+yi\in \mathfrak U^+_G$, $v=v_x+v_yi\in T_z\mathfrak U^+_G=A^{\sigma_\mathbb C}_\mathbb C$,  the associated norm of this metric is given by:
\begin{align*}
    g_z(v,v):=&\mathrm{tr}\left(y^{-\frac{1}{2}}\bar vy^{-1}vy^{-\frac{1}{2}}\right)\\
    =&\mathrm{tr}\left(y^{-\frac{1}{2}}(v_xy^{-1}v_x+v_yy^{-1}v_y)y^{-\frac{1}{2}}\right).    
\end{align*}
This metric is $\mathrm{Sp}_2(A,\sigma)$-invariant. Using the polarization, one can also write the expression of $g_z(v,w)$, for $v,w\in T_z\mathfrak U^+_G$. The isomorphisms between models allow us to express this metric in terms of all four models of the symmetric space of $\mathrm{Sp}_2(A,\sigma)$.

Notice that if $z=i$, then the metric becomes particularly easy: 
$$g_i(v,w)=\mathrm{tr}(v_x\circ w_x+v_y\circ w_y).$$

\subsection{Example: \texorpdfstring{$A=\mathrm{Mat}_n(\mathbb R)$}{}}\label{sec:exam_A_real_mat}
We now present all four models, their tangent spaces and the differential of the diffeomorphisms between them explicitly in a certain important case. Let $A=\mathrm{Mat}_n(\mathbb R)$ and $\sigma$ be the transposition. Then $G:=\mathrm{Sp}_2(A,\sigma)$ is isomorphic to the real symplectic group $\mathrm{Sp}_{2n}(\mathbb R)$ which is the group of isometries of the symplectic space $(\mathbb R^{2n}, \Omega)$, where $\Omega=\begin{pmatrix} 0 & \mathrm{Id} \\ -\mathrm{Id} & 0 \end{pmatrix}$ and $\mathrm{Id}$ denotes the $n\times n$ identity matrix.

    As shown in~\cite{ABRRW}, the space of complex structures $\mathfrak{C}_G$ can be seen as the space of all linear operators $J\colon \mathbb R^{2n}\to\mathbb R^{2n}$ such that $J^2=-\mathrm{Id}$ and $J^t\Omega$ is a symmetric positive definite matrix. Then the tangent space $T_J\mathfrak{C}_G$ can be identified with the space of all linear maps $L\colon \mathbb R^{2n}\to\mathbb R^{2n}$ such that $JL=-LJ$ and $L^t\Omega$ is a symmetric matrix.

    The projective space models $\mathfrak{P}_G^\pm$ in this case can be seen as the space of all complex Lagrangian subspaces of $(\mathbb C^{2n},\omega)$, that is, $n$-dimensional subspaces $L$ such that $\omega|_L\equiv 0$ and the sesquilinear form $h(x,y):=i\omega(\bar x,y)$, for $x,y\in L$, is positive-, resp. negative-definite. Its compactification agrees with the space of all complex Lagrangian subspaces of $(\mathbb C^{2n},\omega)$ such that the sesquilinear form $h$ is non-negative, resp. non-positive definite. The tangent space $T_L\mathfrak{P}_G^\pm$ can be then identified with the space of all linear maps $Q\colon L\to \mathbb{C}^{2n}/L$ such that the bilinear form $\omega(Q(\cdot),\cdot)$ on $L$ is symmetric. 

    The upper half-space models $\mathfrak U_G^\pm$ in this case are the well-known Siegel half spaces, i.e. 
    $$\mathfrak U_G^\pm=\{z\in\mathrm{Sym}_n(\mathbb C)\mid\mathrm{Im}(z)\in\mathrm{Sym}^\pm_n(\mathbb R)\},$$
    and their tangent spaces are $T_z\mathfrak U_G^\pm=\mathrm{Sym}_n(\mathbb C)$, for every $z\in\mathfrak U_G^\pm$.

    Finally, the precompact model $\mathfrak{B}_G$ in this case is described by
    $$\mathfrak{B}_G=\{z\in\mathrm{Sym}_n(\mathbb C)\mid \mathrm{Id}-\bar zz\in\mathrm{Herm}^+_n(\mathbb C)\}.$$
    Its compactification is 
    $$\overline{\mathfrak{B}}_G=\{z\in\mathrm{Sym}_n(\mathbb C)\mid \mathrm{Id}-\bar zz\in\mathrm{Herm}^{\geq 0}_n(\mathbb C)\}.$$
    The tangent space at every $z\in\mathfrak B_G$ is then $T_z\mathfrak B_G=\mathrm{Sym}_n(\mathbb C)$.
    
    The map $F_{\mathfrak{C}_G,\mathfrak{P}_G^\pm}$ maps a complex structure $J\colon \mathbb R^{2n}\to \mathbb R^{2n}$ to the complex Lagrangian subspace $L_J$ which is the $(\mp i)$-eigenspace of $J_\mathbb C$. If $M\colon \mathbb R^{2n}\to \mathbb R^{2n}$ is a tangent vector at $J$, then 
    \[d_JF_{\mathfrak{C}_G,\mathfrak{P}_G^\pm}(M)=\pm\frac{i}{2}[M_\mathbb C],\] 
    where $[M_\mathbb C]\colon L_J\to \mathbb C^{2n}/L_J$, which is the natural projection of $M_\mathbb C|_{L_J}$.

    Let now $L\in\mathfrak{P}_G^\pm$ and $Q\colon L\to \mathbb C^{2n}/L$ be a tangent vector at $L$. Then $F_{\mathfrak{P}_G^\pm,\mathfrak U_G^\pm}$ maps $L$ to the following symmetric matrix. We first fix a basis of $L$ and write it as a $2n\times n$-matrix $A$; this matrix contains an upper $(n\times n)$-block $A_1$ and a lower one, say $A_2$, which is invertible because $L\in\mathfrak{P}_G^\pm$. Then $F_{\mathfrak{P}_G^\pm,\mathfrak U_G^\pm}(L)=A_1A_2^{-1}$ and this expression is clearly independent of the chosen basis. Applying $Q$ to the vectors of the basis, we obtain a family of $2n\times n$ matrices $B+Ax$, where $x\in\mathrm{Mat}_n(\mathbb R)$. We choose one of them, say for $x=0$ we obtain $B$, then we also obtain in a similar vain two $n\times n$ matrices $B_1$ and $B_2$ which are upper and lower submatrices of $B$. Thus, 
    \[d_LF_{\mathfrak{P}_G^\pm,\mathfrak U_G^\pm}(Q)=(B_1-A_1A_2^{-1}B_2)A_2^{-1}.\] Easy calculation shows that this expression does not depend on any of the choices made.

    One similarly describes the map $F_{\mathfrak{P}_G^+,\mathfrak{B}_G}$: Let $T:=\frac{1}{\sqrt{2}}\begin{pmatrix}
    \mathrm{Id} & -i\mathrm{Id}\\
    -i\mathrm{Id} & \mathrm{Id}
    \end{pmatrix}$, denote $A':=TA$, $B':=TB$, and let $A'_1$ and $A'_2$ be the lower and the upper $n\times n$ submatrices of $A'$ respectively, and $B'_1$ and $B'_2$ be the lower and the upper $n\times n$ submatrices of $B'$ respectively. Then, 
    \[F_{\mathfrak{P}_G^+,\mathfrak{B}_G}(L)=A_1'(A_2')^{-1}\] 
    and \[d_LF_{\mathfrak{P}_G^+,\mathfrak B_G}(Q)=(B'_1-A'_1(A'_2)^{-1}B'_2)(A'_2)^{-1}.\]

\section{Models for the symmetric space of \texorpdfstring{$\mathrm{Sp}_2(A_{\mathbb{C}}, \sigma_{\mathbb{C}})$}{} and their tangent spaces}\label{sec:complex_models}

Let $(A,\sigma)$ be a Hermitian algebra. We denote by $G:=\mathrm{Sp}_2(A_{\mathbb{C}},\sigma_\mathbb C)$ and $K^c:=\mathrm{KSp}_2(A_{\mathbb{C}},\sigma_\mathbb C)$. In this section, we discuss models of the symmetric space $G/K^c$ introduced in~\cite{ABRRW}, compute the tangent space for each model and describe equivariant diffeomorphisms between them.

\subsection{Space of quaternionic structures}\label{sec:quatstructures}
A \emph{quaternionic structure} on $A_\mathbb{C}^2$ is an additive map $J\colon A_\mathbb{C}^2\to A_\mathbb{C}^2$ such that $J^2=-\mathrm{Id}$ and $J(xa)=J(x)\bar{a}$, for all $x\in A_\mathbb{C}^2$ and $a\in A_\mathbb{C}$. Each quaternionic structure $J$ defines a $\bar\sigma_\mathbb{C}$-sesquilinear form on $A_{\mathbb{C}}^2$ by
$$
h_J(\bullet,\bullet):=\omega_\mathbb{C}(J(\bullet),\bullet),
$$ 
where $\omega_\mathbb{C}$ is the standard symplectic form on $A_\mathbb{C}^2$ given in Section~\ref{sec:incarnations}.

\begin{defn}
    The \emph{space of quaternionic structures} on $A_\mathbb{C}^2$ is the space
\begin{align*}
    \mathfrak{C}_G:=\{ J \ \text{is a quaternionic structure on}\ A_{\mathbb{C}}^2\ |\ h_J \ \text{is a }\ \bar\sigma_{\mathbb{C}}\text{-inner product}\}.
\end{align*}
\end{defn}
In particular, $J_0\in\mathfrak{C}_G$, where $J_0(x,y):=(\bar y,-\bar x)$ is the standard quaternionic structure on $A_{\mathbb{C}}^2$.  The group $G$ acts on $\mathfrak{C}_G$ by conjugation, that is, $g.J= gJg^{-1}$ smoothly and transitively, and the stabilizer of $J_0$ agrees with $K^c$. Thus, $\mathfrak C_G$ is a model of the symmetric space $G/K^c$ (cf. \cite[Theorem~6.7]{ABRRW}).

\begin{prop}\label{prop:tangent_sqs}
    Let $J\in\mathfrak{C}_G$, then the tangent space of $\mathfrak{C}_G$ at $J$ is described by
    $$T_J\mathfrak{C}_G=\left\{L\colon A_\mathbb{C}^2\to A_\mathbb{C}^2 \midwd \begin{aligned}
    &\ L(xa)=L(x)\bar a\;\;\text{for $x\in A_\mathbb{C}^2$, $a\in A_\mathbb{C}$},\\
    &\ \text{$h_L$ is $\bar\sigma_{\mathbb C}$-symmetric},\text{ and }\\
    &\ LJ+JL=0
    \end{aligned}\right\}.$$
    The group $\mathrm{Sp}_2(A_\mathbb C,\sigma_\mathbb C)$ acts on $T\mathfrak{C}_G$  by conjugation, i.e. 
    $$g.(J,L)=(gJg^{-1},gLg^{-1}).$$ 
    In particular, 
    $$T_{gJg^{-1}}\mathfrak{C}_G = g T_J\mathfrak{C}_Gg^{-1}.$$
\end{prop}

\begin{proof} The proof is exactly the same as the proof of Proposition~\ref{prop:tangent_scs}.
\end{proof}

\begin{rem}
    Notice that the condition for a quaternionic structure $J$ that $h_J$ is $\bar\sigma_\mathbb C$-symmetric is equivalent to the condition that $J$ is an anti-symplectic operator, i.e. $\omega(J\cdot,J\cdot)=\overline{\omega(\cdot,\cdot)}$.
\end{rem}

We now consider the complexification $A_{\mathbb{CC}}=A_\mathbb C\otimes_\mathbb R \mathbb C$ of the algebra $A_{\mathbb{C}}$. To clarify the notation, we denote by $i$ the imaginary unit in $A_\mathbb C$ and by $I$  the imaginary unit in $\mathbb C$. Notice that $A_\mathbb C\otimes \mathbb C$ is naturally isomorphic to the algebra $A_\mathbb C\times A_\mathbb C$. This isomorphism maps $a_1+a_2I$ to $(a_1+a_2i,a_1-a_2i)$, where $a_1,a_2\in A_\mathbb C$ (for more details we refer to~\cite[Appendix~B.1.1]{ABRRW}). The real locus $A_{\mathbb C}$ of $A_\mathbb{CC}$ maps under this natural isomorphism to the diagonal $\{(a,a)\mid a\in A_\mathbb C\}$, in particular the image of $i$ is $(i,i)$. The image of $I$ is $(i,-i)$. 

There are two complex conjugation maps on $A_\mathbb {CC}$: $\theta_i$ and $\theta_I$ (cf. the notation of Section~\ref{sec:conjugations}). Under the identification by the isomorphism above, $\theta_i$ maps to the conjugation in both factors composed with the permutation of the components; we shall denote this conjugation by $\theta_{(i,i)}$. The map $\theta_I$ can be also identified with the permutation of components; this shall be denoted as usual by $\theta_{(i,-i)}$. 

The anti-involution $\sigma_\mathbb C$ on $A_\mathbb C$ extends to the following anti-involution $\sigma_0\colon A_{\mathbb C}\times A_{\mathbb C} \to A_{\mathbb C}\times A_{\mathbb C}$ with 
$$\sigma_0\left((m_1,m_2)\right)=(\sigma_\mathbb C(m_1),\sigma_\mathbb C(m_2)),$$
for $m_1,m_2\in A_\mathbb C$.  On the other hand, the anti-involution $\bar\sigma_\mathbb C$ on $A_\mathbb C$ extends to the following anti-involution $\sigma_1\colon A_{\mathbb C}\times A_{\mathbb C} \to A_{\mathbb C}\times A_{\mathbb C}$ with $$\sigma_1\left((m_1,m_2)\right)= (\bar\sigma_\mathbb C(m_2),\bar\sigma_\mathbb C(m_1)),$$ for $m_1,m_2\in A_\mathbb C$. 

Remember from Section~\ref{sec:quatstructures}, the description of the tangent of the quaternionic model space at a~quaternionic structure $J\in\mathfrak{C}_G$,  given by 
\begin{align*}
T_J\mathfrak{C}_G=\left\{L\colon A_\mathbb{C}^2\to A_\mathbb{C}^2  \midwd
\begin{aligned} 
&\ L(xa)=L(x)\bar a \text{ for } x\in A_\mathbb{C}^2,\;a\in A_\mathbb{C}, \\
&\ h_L \text{ is }\bar\sigma_{\mathbb C}\text{-symmetric},\text{ and}\\
&\ LJ+JL=0
\end{aligned}
\right\}.
\end{align*}
Using the identifications introduced above, we may now give the complexification of this tangent space at $J$ as 
\begin{align}\label{eq:complex_tang_quater_str_model} 
T^\mathbb C_J\mathfrak{C}_G=\left\{L\colon (A_{\mathbb C}\times A_{\mathbb C})^2 \to (A_{\mathbb C}\times A_{\mathbb C})^2 \midwd 
\begin{aligned}  
&\ L(xa)=L(x)\theta_{(i,i)}(a),\\ &\text{ for } x\in (A_{\mathbb C}\times A_{\mathbb C})^2 \text{ and } a\in A_{\mathbb C}\times A_{\mathbb C},\\
&\ h_L \text{ is } \sigma^{(i,i)}\text{-symmetric}, \text{ and } \\
&\ LJ+JL=0
\end{aligned}
\right\}.
\end{align}

\subsection{Projective space model}\label{sec:projstructure_cpx} To introduce the projective space model for $\mathcal{X}_{\mathrm{Sp}_{2}(A_{\mathbb{C}},\sigma_{\mathbb{C}})}$, one first needs to consider quaternionified algebras. We shall denote the complexified algebra by $A_\mathbb{C}=A\otimes_\mathbb{R}\mathbb{C}\{i\}$, where $i$ is put here to emphasize the imaginary unit of $\mathbb{C}$. Now we consider the quaternionification $A_\mathbb{H}:=A\otimes_\mathbb{R}\mathbb{H}\{i,j,k\}$, where $i,j,k$ are the imaginary units of the quaternion skew-field $\mathbb{H}$ that fit into the relation $ij=-ji=k$. 

Given the classification of anti-involutions on $\mathbb{H}$, from \cite[Proposition A.37]{ABRRW} we may obtain the following anti-involutions on $A_{\mathbb{H}}$ (though note that the following does not provide a classification of anti-involutions of $A_{\mathbb{H}}$): 
\begin{enumerate}
        \item $ [\sigma = {\bar{\cdot}} ] \qquad \sigma_1(x+yj)=\bar{\sigma}_{\mathbb{C}}(x)-\sigma_{\mathbb{C}}(y)j$ 
        \item $ [v=k] \qquad \sigma_0(x+yj)={\sigma}_{\mathbb{C}}(x)+\bar{\sigma}_{\mathbb{C}}(y)j$
        \item $ [v=i] \qquad \sigma_2(x+yj)=\bar{\sigma}_{\mathbb{C}}(x)+\sigma_{\mathbb{C}}(y)j$
        \item $[v=j] \qquad \sigma_3(x+yj)={\sigma}_{\mathbb{C}}(x)-\bar{\sigma}_{\mathbb{C}}(y)j$.
\end{enumerate}

Note that since $({{A}_{\mathbb{C}}},\bar{\sigma}_\mathbb{C})$ is Hermitian, thus the algebras $(\mathrm{Mat}_2(A_\mathbb{C}),\bar\sigma_\mathbb{C}^t)$ and $({{A}_{\mathbb{H}}},{{\sigma }_{1}})$ are Hermitian as well.

The anti-involutions $\sigma_0$ and $\sigma_3$ on $A_\mathbb H$ extend $\sigma_\mathbb C$, so we may extend on $A_{\mathbb{H}}^{2}$ the standard symplectic form $\omega_\mathbb{C}$ on $A_{\mathbb{C}}^{2}$ in two ways as
\begin{align}
    {{\omega }_{\mathbb{H}}}(x,y):={{\sigma }_{0}}{(x)}^t\left( \begin{matrix}
   0 & 1  \\
   -1 & 0  \\
\end{matrix} \right)y,
\end{align}
and as
\begin{align}
    {\omega'_{\mathbb{H}}}(x,y):={{\sigma }_{3}}{(x)}^t\left( \begin{matrix}
   0 & 1  \\
   -1 & 0  \\
\end{matrix} \right)y.
\end{align}
We also consider the following indefinite form on $A_{\mathbb{H}}^{2}$:
\[h(x,y):={{\sigma }_{1}}{{(x)}^{t}}\left( \begin{matrix}
   0 & j  \\
   -j & 0  \\
\end{matrix} \right)y.\]
As we have seen in Proposition~\ref{prop:incarn}, $\mathrm{Sp}_2(A_{\mathbb C},\sigma_{\mathbb C})=\mathrm{Sp}_2(A_\mathbb H,\sigma_0)\cap\mathrm{Aut}(h)$ and $\mathrm{Sp}_2(A_\mathbb H,\sigma_3)= \mathrm{Aut}(h)$.

We define the following two spaces 
\begin{align*}
    \mathfrak{P}_G^\pm:={{\mathfrak{P}}_{\mathrm{Aut}(h)}^\pm}\cap \mathbb{P}\mathrm{Is}\left( {{\omega }_{\mathbb{H}}} \right),
\end{align*}
where $\mathfrak{P}_{\mathrm{Aut}(h)}^\pm$ are projective space models of the symmetric space for the indefinite orthogonal group $\mathrm{Aut}(h)$ (cf.~Section~\ref{sec:projective_O}). These spaces are models of the symmetric space $\mathcal{X}_{\mathrm{Sp}_{2}(A_{\mathbb{C}},\sigma_{\mathbb{C}})}$. We thus have the following definition:

\begin{defn}\label{proj_model_complex}
Let $\mathrm{Sp}_{2}(A_{\mathbb{C}},\sigma_{\mathbb{C}})$ be the symplectic group for the complexification $(A_{\mathbb{C}},\sigma_{\mathbb{C}})$ of a real Hermitian algebra. Either model $\mathfrak{P}_G^+$ or $\mathfrak{P}_G^-$ as introduced above is called the \emph{projective space model} for the symmetric space $\mathcal{X}_{\mathrm{Sp}_{2}(A_{\mathbb{C}},\sigma_{\mathbb{C}})}$.
\end{defn}

The spaces $\mathfrak{P}_G^\pm$ are open domains of the compact manifold $\mathbb P\mathrm{Is}(\omega_\mathbb H)$. Thus, their closures are compact:
\begin{align*}
    \overline{\mathfrak{P}}_G^\pm:={\overline{\mathfrak{P}}_{\mathrm{Aut}(h)}^\pm}\cap \mathbb{P}\mathrm{Is}\left( {{\omega }_{\mathbb{H}}} \right).
\end{align*}

\begin{prop}[{\cite[Proposition~5.18]{ABRRW}}]
    The action of $\mathrm{Sp}_2(A_\mathbb C,\sigma_\mathbb C)<\mathrm{GL}_2(A_\mathbb H)$ on $A_\mathbb H^2$ extends naturally to the action on $\mathbb{P}\mathrm{Is}(\omega_\mathbb{H})$ which preserves $\mathfrak{P}_G^+$ and $\mathfrak{P}_G^-$. The restriction of this action to $\mathfrak{P}_G^+$ and $\mathfrak{P}_G^-$ is transitive. The stabilizer of $(j,1)^tA_\mathbb H$ agrees with $K^c$.
\end{prop}

We now study the tangent $T\mathfrak{P}_G^+$ of the projective space model more explicitly. Since $\mathfrak{P}_G^+$ is an open subspace of $\mathbb P\mathrm{Is}\left( {{\omega }_{\mathbb{H}}} \right)$, we can apply Remark~\ref{prop:tangent_skew} and Corollary~\ref{cor:tangent_skew_invariant} to derive that for every $l\in \mathfrak{P}_G^+$ and $x\in A_\mathbb H^2$ generating $l$, there are two realizations of the tangent space at $l$:
\begin{align*}
    T_{l}\mathfrak{P}_G^+
&=\{Q\in\mathrm{Hom}(l,A_\mathbb H^2/l)\mid \omega_\mathbb H(Q(x),x)\in A^{\sigma_0}_\mathbb H\}\\
&=\{[x,v]\mid x\text{ generates }l_J,\;v\in x^{\perp_{\omega_\mathbb H}}\}\\
&\cong x^{\perp_{\omega_\mathbb H}}/l.
\end{align*}

The action of $\mathrm{Sp}_2(A_\mathbb{C},\sigma_\mathbb{C})$ on the first realization of the tangent bundle is by conjugation. The~action on the second realization is linear.

The quaternionic structures model and the projective space model are related by the following isometry: 
For $l_J$, the $A_{\mathbb{H}}$-line $yA_{\mathbb{H}}$ such that $J_\mathbb{H}(y)=-yj$, there is a well-defined homeomorphism between the two models 
\begin{align}\label{homeo_quatern_proj}
F_{\mathfrak{C}_G,\mathfrak{P}_G^\pm}\colon \mathfrak{C}_G & \to \mathfrak{P}_G^\pm \\
J & \mapsto l_{J}^\pm,
\end{align}
where $l_J^\pm$ is the $A_{\mathbb{H}}$-line such that $l_J^\pm=yA_{\mathbb{H}}$, for a regular $y\in A_\mathbb H^2$ with $J_\mathbb{H}(y)=\mp yj$. This isometry preserves the structures of the Riemannian symmetric spaces and is equivariant under the $\mathrm{Sp}_2 (A_\mathbb C, \sigma_\mathbb C)$-action.
We compute the differential of the map $F_{\mathfrak{C}_G,\mathfrak{P}_G^\pm}$ explicitly:

\begin{prop}\label{prop:differential_F_quat}
    Let $J\in\mathfrak{C}_G$, $L\in T_J\mathfrak{C}_G$ and $x\in l_J$ be a generator of $l_J$. Then
    $$d_{J}F_{\mathfrak{C}_G,\mathfrak{P}_G^\pm}(L)=[x,\pm\frac{1}{2}L_\mathbb H(x)j+l_J^\pm]=\pm\frac{1}{2}[L_\mathbb H(\bullet)]j\in T_{l_J^\pm}\mathfrak{P}_G^\pm,$$
    where $[L_\mathbb H(\bullet)]\colon l_J^\pm\to A^2_\mathbb H/l_J^\pm$ denotes the projection of the restriction map $L_\mathbb H|_{l_J^\pm}\colon l_J^\pm\to A^2_\mathbb H$.
\end{prop}

\begin{proof} The proof is similar to the proof of Proposition~\ref{prop:differential_F}.
\end{proof}

Let $l\in\mathfrak{P}_G^\pm$. Its complexification is denoted by $l_\mathbb C:=l\otimes_\mathbb R\mathbb C\subset A_\mathbb H^2\otimes_\mathbb R \mathbb C$. The complexification of the tangent space of $\mathfrak P_G^\pm$ at $l$ is given by the space of all $A_\mathbb H\otimes_\mathbb R \mathbb C$-linear morphisms:
\begin{align*}
T_l^\mathbb C\mathfrak{P}_G^\pm
&=\left\{Q\in\mathrm{Hom}(l_\mathbb C,A_\mathbb H^2\otimes_\mathbb R\mathbb C/l_\mathbb C)\mid \omega_{\mathbb H\mathbb C}(Q(x),x)\in A^{\sigma_0}_\mathbb H\otimes_\mathbb R\mathbb C\right\}\\
&=\left\{Q\in\mathrm{Hom}(l_\mathbb C,l_\mathbb C^{\perp_{h_\mathbb C}})\mid \omega_{\mathbb H\mathbb C}(Q(x),x)\in A^{\sigma_0}_\mathbb H\otimes_\mathbb R\mathbb C\right\},
\end{align*}
where $\omega_{\mathbb H\mathbb C}$ denotes the $\mathbb C$-linear extension of $\omega_\mathbb H$, and $h_{\mathbb C}$ denotes the $\mathbb C$-linear extension of $h$.

From Appendix~\ref{sec:cq_to_mat2}, there is a $\mathbb C$-algebra isomorphism $\Phi\colon A_\mathbb H\otimes_\mathbb R\mathbb C\to \mathrm{Mat}_2(A_\mathbb C)=:\hat A$. Under this isomorphism $\sigma_0\otimes\mathrm{Id}$ identifies with the map 
$$\hat\sigma_0\colon a\mapsto \sigma_\mathbb C(a)^t,$$
and $\sigma_1\otimes\mathrm{Id}$ identifies with the map 
$$\hat\sigma_1\colon a\mapsto \begin{pmatrix}
0 & 1 \\
-1 & 0
\end{pmatrix}^{-1}\sigma_\mathbb C(a)^t\begin{pmatrix}
0 & 1 \\
-1 & 0
\end{pmatrix}.$$
This allows us to identify a right $A_\mathbb H\otimes_\mathbb R\mathbb C$-module of rank 2 with a right $A_\mathbb C$-module of rank $4$ equipped with two $\sigma_\mathbb C$-sesquilinear symplectic forms $\hat\omega$ and $\hat h$ defined as follows: for $x,y\in A^4$,
$$\hat\omega(x,y)=\sigma_\mathbb C(x)^t\begin{pmatrix}
    0 & 0 & 1 & 0\\
    0 & 0 & 0 & 1\\
    -1 & 0 & 0 & 0\\
    0 & -1 & 0 & 0
\end{pmatrix}y,$$
$$\hat h(x,y)=\sigma_\mathbb C(x)^t\begin{pmatrix}
    0 & 1 & 0 & 0\\
    -1 &  0 & 0 & 0\\
    0 & 0 & 0 & 1\\
    0 & 0 & -1 & 0
\end{pmatrix}\begin{pmatrix}
    0 & 0 & 0 & i\\
    0 & 0 & i & 0\\
    0 & -i & 0 & 0\\
    -i & 0 & 0 & 0
\end{pmatrix}y=\sigma_\mathbb C(x)^t\begin{pmatrix}
    0 & 0 & i & 0\\
    0 & 0 & 0 & -i\\
    -i & 0 & 0 & 0\\
    0 & i & 0 & 0
\end{pmatrix}y.$$

Let $\hat l$ be an $A_\mathbb C^2$-submodule of $A^4$, which is the image of $l_\mathbb C$ under $\Phi$, for $l\in \mathfrak{P}_G^\pm$. Then, we can describe the complexified tangent space at $l$ of $\mathfrak{P}_G^\pm$ as follows
\begin{align}
\begin{aligned}\label{eq:spc_proj_comp.tangent}
    T_l^\mathbb C\mathfrak{P}_G^\pm
    &=\{Q\in\mathrm{Hom}(\hat l,\hat A^4/\hat l)\mid \hat \omega(Q(x),x)+\hat \omega(x,Q(x))=0,\;x\in\hat l\}\\
    &=\left\{Q\in\mathrm{Hom}(\hat l,\hat l)\mid \hat h(Q(x),x)+ \hat h(x,Q(x))=0,\;x\in\hat l\right\}.    
\end{aligned}
\end{align}
The last equality is given by the fact that $\hat\omega$ and $\hat h$ allow us to identify canonically $A^4/\hat l$ and $\hat l$ as follows: for $x+\hat l\in A_\mathbb C^4/\hat l$, we associate the unique $y\in \hat l$ such that $\hat\omega(z,x)=\hat h(z,y)$ for all $y\in \hat l$. This identification is clearly $A_\mathbb C$-linear.

The spaces $\mathfrak{P}_G^\pm$ are open domains in  the projective space $\mathbb{P}\mathrm{Is}(\omega_\mathbb H)$. Later on we will construct affine charts for these spaces. 

\subsection{Half-space models}\label{sec:upperhalf_complex}

The next model of the symmetric space of $\mathcal X_{\mathrm{Sp}_2(A_\mathbb{C},\sigma_\mathbb{C})}$ that we are going to discuss, is the upper half-space model. We denote as before by $A_\mathbb H$ the quaternionification of $A$, that is, $A_\mathbb{H}:=A\otimes_\mathbb R\mathbb H\{i,j,k\}$.

Every element of $z\in A_\mathbb{H}$ can be uniquely written as $z=x+yj$ where $x,y\in A_\mathbb C$. We denote by $\mathrm{Re}_\mathbb C(z):=x$, $\mathrm{Im}_\mathbb C(z):=y$. Moreover, if $z\in A_\mathbb{H}^{\sigma_0}$, then $x\in A_\mathbb C^{\sigma_\mathbb C}$ and $y\in A_{\mathbb C}^{\bar\sigma_\mathbb C}$.

\begin{defn}
    The \emph{upper (resp. lower) half-space models} are defined as
    $$\mathfrak{U}_G^\pm:=\left\{z\in A_\mathbb H^{\sigma_0}\mid \mathrm{Im}_\mathbb C(z)\in (A_\mathbb C^{\bar\sigma_\mathbb C})_\pm\right\}.$$
\end{defn}

\noindent They are models of the symmetric space  $\mathcal X_{\mathrm{Sp}_2(A_\mathbb{C},\sigma_\mathbb{C})}$ as seen by the following:

\begin{prop}[{\cite[Proposition~6.10]{ABRRW}}]
    The group $\mathrm{Sp}_2(A_\mathbb C,\sigma_\mathbb C)$ acts smoothly and transitively on $\mathfrak{U}_G^\pm$ by M\"obius transformations, 
    $$g.z=(az+b)(cz+d)^{-1},$$ 
    where $g=\begin{pmatrix}
        a & b \\
        c & d
    \end{pmatrix}$. The stabilizer of $\pm j\in\mathfrak{U}_G^\pm$ agrees with $K^c$.
\end{prop}

Since $\mathfrak{U}_G^\pm$ is an open subspace of the vector space $A_\mathbb{H}^{\sigma_0}$, then for every $z\in \mathfrak{U}_G^\pm$, $T_z\mathfrak{U}_G^\pm = A_\mathbb{H}^{\sigma_0}$.  The action of $\mathrm{Sp}_2(A_\mathbb C,\sigma_\mathbb C)$ extends naturally to the action on $T\mathfrak{U}_G^\pm=\mathfrak{U}_G^\pm\times A^{\sigma_0}_\mathbb H$. For $g=\begin{pmatrix}
        a & b \\
        c & d
    \end{pmatrix}$
and $(z,v)\in \mathfrak{U}_G^\pm\times A^{\sigma_0}_\mathbb H$, the action is given by 
\begin{equation}\label{eq:diff_moebius_c}
    g.(z,v)=\left((az+b)(cz+d)^{-1}, av(cz+d)^{-1}-(az+b)(cz+d)^{-1}cv(cz+d)^{-1}\right).    
\end{equation}

For every $z\in\mathfrak{U}_G^\pm$, the stabilizer $\mathrm{Stab}_{G}(z)$ is conjugated to the stabilizer of $z=j$ which is $K^c$. We are now going to understand the action of $K^c$ on $T_j\mathfrak{U}_G^+\cong A_\mathbb H^{\sigma_0}$. By~\eqref{eq:diff_moebius_c}, for $g=\begin{pmatrix}
    a & b \\
    -\bar b & \bar a
\end{pmatrix}\in K^c$ and $v\in T_j\mathfrak{U}_G^+$, we have that
$$g.v
=(a+bj)v(-\bar bj+\bar a)^{-1}
=(a+bj)v(\sigma_{\mathbb C}(a)+\bar\sigma_{\mathbb C}(b)j)=qv\sigma_0(q),$$
because $(-\bar bj+\bar a)^{-1}=\sigma_{\mathbb C}(a)+\bar\sigma_{\mathbb C}(b)j$ and where $q=a+bj\in\mathrm O(A_\mathbb H,\sigma_1)\cong K^c$ (cf.~\eqref{eq:conj_symp_cpx}). 

Let $v=v_1+v_2j\in A_\mathbb H^{\sigma_0}$, then $v_1\in A^{\sigma_\mathbb C}_\mathbb C$, $v_2\in A^{\bar \sigma_\mathbb C}_\mathbb C$, and
\begin{align*}
g.v
=&av_1\sigma_{\mathbb C}(a)-b\bar v_1\sigma_{\mathbb C}(b)-av_2\sigma_{\mathbb C}(b) -b\bar v_2\sigma_{\mathbb C}(a)
\\
& +(av_1\bar\sigma_{\mathbb C}(b)+b\bar v_1\bar\sigma_{\mathbb C}(a)+av_2\bar\sigma_{\mathbb C}(a)-b\bar v_2\bar\sigma_{\mathbb C}(b))j.
\end{align*}
Notice that $A_\mathbb H^{\sigma_0}$ can be identified with $\mathfrak m^c$ via the map: $v_1+v_2j\mapsto \begin{pmatrix}
    v_2 & v_1 \\
    \bar v_1 & -\bar v_2
\end{pmatrix}$. This action agrees with the isotropy action:
\begin{align*}
& \begin{pmatrix}
    a & b \\
    -\bar b & \bar a
\end{pmatrix}\begin{pmatrix}
    v_2 & v_1 \\
    \bar v_1 & -\bar v_2
\end{pmatrix}\begin{pmatrix}
    a & b \\
    -\bar b & \bar a
\end{pmatrix}^{-1}=\begin{pmatrix}
    a & b \\
    -\bar b & \bar a
\end{pmatrix}\begin{pmatrix}
    v_2 & v_1 \\
    \bar v_1 & -\bar v_2
\end{pmatrix}\begin{pmatrix}
    \bar\sigma_{\mathbb C}(a) & -\sigma_{\mathbb C}(b) \\
    \bar\sigma_{\mathbb C}(b) & \sigma_{\mathbb C}(a)
\end{pmatrix}\\
= & \begin{pmatrix}
    av_2+b\bar v_1 & av_1-b\bar v_2 \\
    -\bar bv_2+\bar a\bar v_1 & -\bar bv_1-\bar a\bar v_2
\end{pmatrix}\begin{pmatrix}
    \bar \sigma_{\mathbb C}(a) & -\sigma_{\mathbb C}(b) \\
    \bar \sigma_{\mathbb C}(b) & \sigma_{\mathbb C}(a)
\end{pmatrix} = \begin{pmatrix}
    \alpha & \beta \\
    \gamma & \delta
\end{pmatrix},
\end{align*}
where
\begin{align*}
\alpha & := av_2\bar\sigma_{\mathbb C}(a)+b\bar v_1\bar\sigma_{\mathbb C}(a)+av_1\bar\sigma_{\mathbb C}(b)-b\bar v_2\bar\sigma_{\mathbb C}(b),\\
\beta & := -av_2\sigma_{\mathbb C}(b)-b\bar v_1\sigma_{\mathbb C}(b)+av_1\sigma_{\mathbb C}(a)-b\bar v_2\sigma_{\mathbb C}(a), \\
\gamma & := -av_2\sigma_{\mathbb C}(b)-b\bar v_1\sigma_{\mathbb C}(b)+av_1\sigma_{\mathbb C}(a)-b\bar v_2\sigma_{\mathbb C}(a),\\
\delta & := -av_2\bar\sigma_{\mathbb C}(a)-b\bar v_1\bar\sigma_{\mathbb C}(a)-av_1\bar\sigma_{\mathbb C}(b)+b\bar v_2\bar\sigma_{\mathbb C}(b).
\end{align*}

Further, the map $m\mapsto im$ provides an $\mathbb R$-linear isomorphism between $\mathfrak{m}$ and $\mathfrak k$. Under this identification, the action of $K$ on $T_j\mathfrak{U}_G^+$ by M\"obius transformations agrees with the adjoint action of $K$ on its Lie algebra $\mathfrak{k}$.

The projective space models and the half-space models are related by the following $G$-equivariant isometries: 
$$\begin{matrix}
F_{\mathfrak{P}_G^\pm,\mathfrak{U}_G^\pm} \colon & \mathfrak{P}_G^\pm & \to & \mathfrak{U}_G^\pm\\
 & (x_1,x_2)^tA_\mathbb H & \mapsto & x_1x_2^{-1}.
\end{matrix}$$
The differentials of these maps can be computed immediately:

\begin{prop}
    Let $x=(x_1,x_2)^t\in\mathrm{Is}(\omega_\mathbb H)$ such that $l=xA_\mathbb H\in\mathfrak{P}_G^\pm$, and $v=[x,(v_1,v_2)^t+l]\in T_{l}\mathfrak{P}_G^\pm$. Then, 
    $$d_{l}F_{\mathfrak{P}_G^\pm,\mathfrak{U}_G^\pm}(v)=(v_1-F_{\mathfrak{P}_G^\pm,\mathfrak{U}_G^\pm}(x) v_2)x_2^{-1}.$$ 
\end{prop}

\subsection{Precompact model}\label{sec:precompact_complex}

Finally, we discuss the precompact model for the symmetric space of $G=\mathrm{Sp}_2(A_\mathbb{C},\sigma_\mathbb{C})$. In this section, we denote
    $G':=Q^{-1}GQ$ where $Q=\frac{1}{\sqrt{2}}\begin{pmatrix}
        1 & j\\
        j & 1
    \end{pmatrix}$.
We also denote $K'=Q^{-1}K^cQ$ a maximal compact subgroup of $G'$ (cf. Section~\ref{sec:incarnations}). Since $G$ and $G'$ are conjugates, their symmetric spaces are isometric Riemannian manifolds.

\begin{defn}
    The \emph{precompact model} is defined as the space 
    $$\mathfrak{B}_G:=\left\{z\in A^{\sigma_0}_\mathbb{H}\mid 1-\sigma_1(z)z\in (A^{\sigma_1}_\mathbb{H})_+\right\}.$$
\end{defn}

It is a model of the symmetric space $\mathcal X_{\mathrm{Sp}_2(A_\mathbb{C},\sigma_\mathbb{C})}$ as seen by the following:

\begin{prop}[{\cite[Proposition~6.16]{ABRRW}}]
    The group $G'$ acts smoothly and transitively on $\mathfrak{B}_G$ by M\"obius transformations 
    $$g.z=(az+b)(cz+d)^{-1},$$ 
    where $g=\begin{pmatrix}
        a & b \\
        c & d
    \end{pmatrix}$. The stabilizer of $0\in\mathfrak{B}_G$ agrees with $K'$.
\end{prop}

The closure of the space $\mathfrak{B}_G$ is compact:
    $$\overline{\mathfrak{B}}_G:=\left\{z\in A^{\sigma_0}_\mathbb{H}\mid 1-\sigma_1(z)z\in (A^{\sigma_1}_\mathbb{H})_{\geq 0}\right\}.$$

Since $\mathfrak{B}_G$ is an open subspace of the vector space $A_\mathbb{H}^{\sigma_0}$, then for every $z\in \mathfrak{B}_G$, $T_z\mathfrak{B}_G = A_\mathbb{H}^{\sigma_0}$.  The action of $\mathrm{Sp}_2(A_\mathbb C,\sigma_\mathbb C)$ extends naturally to the action on $T\mathfrak{B}_G=\mathfrak{B}_G\times A^{\sigma_0}_\mathbb H$. For $T^{-1}gT=\begin{pmatrix}
        a & b \\
        c & d
    \end{pmatrix}$
and $(z,v)\in \mathfrak{U}_G\times A^{\sigma_0}_\mathbb H$,  the action is given by
$$g.(z,v)=\left((az+b)(cz+d)^{-1}, av(cz+d)^{-1}-(az+b)(cz+d)^{-1}cv(cz+d)^{-1}\right).$$

For every $z\in\mathfrak B_G$, the stabilizer $\mathrm{Stab}_{G'}(z)$ is conjugated to the stabilizer of $x=0$ which is $K'$. We are now going to understand the action of $K'$ on $T_0\mathfrak B_G= A_\mathbb{H}^{\sigma_0}$. For $u\in \mathrm O(A_\mathbb H,\sigma_1)$,
    $g=\begin{pmatrix}
        u & 0 \\
        0 & (ij) u (ij)^{-1}
    \end{pmatrix}\in K'$ 
(cf.~\eqref{eq:conj_symp_cpx}) and $v\in T_0\mathfrak B_G$, then 
$$g.v=uv\sigma_0(u).$$

\begin{rem}\label{rem:not_quatern_module}
    Notice that the tangent space $T_z\mathcal X$ at some point $z$ in the symmetric space $\mathcal{X}$ of $\mathrm{Sp}_2(A_{\mathbb{C}}, \sigma_{\mathbb{C}})$  is isomorphic to $A_\mathbb H^{\sigma_0}$, which is not a quaternionic module. It is isomorphic to $A_\mathbb C^{\sigma_\mathbb C}+ A_\mathbb C^{\bar\sigma_\mathbb C}j$, which is even not a complex vector space but only a real one. However, after complexification, it becomes isomorphic to the complex vector space 
    \[\left(A_{\mathbb{C}}^{\sigma_{\mathbb{C}}} + A_{\mathbb{C}}^{\bar\sigma_{\mathbb{C}}}j\right) \otimes \mathbb{C}\cong (A_{\mathbb{C}}^{\sigma_{\mathbb{C}}})^3\oplus A_{\mathbb{C}}^{-\sigma_{\mathbb{C}}}.\]  
\end{rem}

The projective space model and the precompact model are related by the following $G$-equivariant isometry: 
$$\begin{matrix}
F_{\mathfrak{P}_G^+,\mathfrak{B}_G} \colon & \mathfrak{P}_G^+ & \to & \mathfrak{B}_G\\
 & (x_1,x_2)^tA_\mathbb H & \mapsto & y_1y_2^{-1},
\end{matrix}$$
where $y=\begin{pmatrix}y_1 \\y_2\end{pmatrix}=\frac{1}{\sqrt{2}}\begin{pmatrix}
1 & -j \\
-j & 1
\end{pmatrix}\begin{pmatrix}x_1 \\x_2\end{pmatrix}$. Its differential can be computed as follows:

\begin{prop}
    Let $x=(x_1,x_2)^t\in\mathrm{Is}(\omega_\mathbb H)$ such that $l=xA_\mathbb H\in \mathfrak{P}_G^+$, and $v=[x,(v_1,v_2)^t+l]\in T_{l}\mathfrak{P}_G^+$. Then, 
    $$d_{l}F_{\mathfrak{P}_G^+,\mathfrak{B}_G}(v)=(w_1-F_{\mathfrak{P}_G^+,\mathfrak{B}_G}(x) w_2)y_2^{-1},$$
    where $y=\begin{pmatrix}y_1 \\y_2\end{pmatrix}=\frac{1}{\sqrt{2}}\begin{pmatrix}
1 & -j \\
-j & 1
\end{pmatrix}\begin{pmatrix}x_1 \\x_2\end{pmatrix}$, and $w=\begin{pmatrix}w_1 \\w_2\end{pmatrix}=\frac{1}{\sqrt{2}}\begin{pmatrix}
1 & -j \\
-j & 1
\end{pmatrix}\begin{pmatrix}v_1 \\v_2\end{pmatrix}$.
\end{prop}

\subsection{Riemannian metric on the models}

In terms of the upper half-space model, it is particularly easy to write the Riemannian metric on the Riemannian symmetric space of $\mathrm{Sp}_2(A_\mathbb C,\sigma_\mathbb C)$. Namely, for $z=x+yj\in \mathfrak U^+_G$, $v\in T_z\mathfrak U^+_G=A^{\sigma_0}_\mathbb H$, the associated norm of this metric is given by:
\begin{align*}
    g_z(v,v):=&\mathrm{tr}\left(y^{-\frac{1}{2}}\sigma_1(v)y^{-1}vy^{-\frac{1}{2}}\right).    
\end{align*}
As in the case of the group $\mathrm{Sp}_2(A,\sigma)$, one can check that this metric is $\mathrm{Sp}_2(A_\mathbb C,\sigma_\mathbb C)$-invariant. Using the polarization, one can also write the expression of $g_z(v,w)$, for $v,w\in T_z\mathfrak U^+_G$. The isomorphisms between models allow us to express this metric in terms of all four models of the symmetric space of $\mathrm{Sp}_2(A_\mathbb C,\sigma_\mathbb C)$.

Notice that if $z=j$, then the metric becomes particularly easy: 
$$g_j(v,w)=\mathrm{tr}(\sigma_1(v)w+\sigma_1(w)v).$$

\subsection{Example: \texorpdfstring{$A=\mathrm{Mat}_n(\mathbb R)$}{}}
In parallel to Section~\ref{sec:exam_A_real_mat}, we demonstrate all four models of the symmetric space of $\mathrm{Sp}_2(A_\mathbb C,\sigma_\mathbb C)$, their tangent spaces and the differentials of the diffeomorphisms between them in a certain important example.
    Let $A=\mathrm{Mat}_n(\mathbb R)$ and $\sigma$ be the transposition. Then $G=\mathrm{Sp}_2(A_\mathbb C,\sigma_\mathbb C)$ is isomorphic to the complex symplectic group $\mathrm{Sp}_{2n}(\mathbb C)$ which is the group of isometries of the symplectic space $(\mathbb C^{2n}, \Omega)$ where $\Omega=\begin{pmatrix} 0 & \mathrm{Id} \\ -\mathrm{Id} & 0 \end{pmatrix}$ and $\mathrm{Id}$ denotes the $n\times n$ identity matrix.
    
    As shown in~\cite{ABRRW}, $\mathfrak{C}_G$ can be seen as the space of all complex anti-linear operators $J\colon \mathbb C^{2n}\to\mathbb C^{2n}$ such that $J\bar J=-\mathrm{Id}$ and $J^t\Omega$ is a Hermitian positive definite matrix. Then the tangent space $T_J\mathfrak{C}_G$ can be identified with the space of all complex anti-linear maps $L\colon \mathbb C^{2n}\to\mathbb C^{2n}$ such that $J\bar L=-L\bar J$ and $L^t\Omega$ is a Hermitian matrix.

    The projective space models $\mathfrak{P}_G^\pm$ in this case can be seen as the space of all quaternionic Lagrangian subspaces of $(\mathbb H^{2n},\omega)$, i.e. $n$-dimensional subspaces $L$ such that $\omega|_L\equiv 0$ where $\omega (x,y)=\sigma_0(x)^t\Omega y$, and the sesquilinear form $h(x,y):=\sigma_1(x)^t j \Omega y$, where $x,y\in L$, is positive-, resp. negative-definite.  Its compact closure agrees with the space of all quaternionic Lagrangian subspaces of $(\mathbb H^{2n},\omega)$ such that the sesquilinear form $h$ is non-negative, resp. non-positive definite. The tangent space $T_L\mathfrak{P}_G^\pm$ can be then identified with the space of all linear maps $Q\colon L\to \mathbb{H}^{2n}/L$ such that the bilinear form $\omega(Q(\cdot),\cdot)$ on $L$ is $\sigma_0$-symmetric, i.e. $\omega(Q(x),y)=\sigma_0(\omega(Q(y),x))$, for all $x,y\in L$.

    The upper half-space models in this case are given by
    $$\mathfrak{U}_G^\pm=\{z_1+z_2j\in\mathrm{Mat}_n(\mathbb H)\mid z_1\in\mathrm{Sym}_n(\mathbb C),\;z_2\in\mathrm{Herm}^\pm_n(\mathbb C)\},$$
     and the tangent space is then 
     \[T_z\mathfrak{U}_G^\pm=\{v_1+v_2j\in\mathrm{Mat}_n(\mathbb H)\mid v_1\in\mathrm{Sym}_n(\mathbb C),\;v_2\in\mathrm{Herm}_n(\mathbb C)\},\] 
     for every $z\in\mathfrak{U}_G^\pm$.

    The precompact model in this case is given by
    $$\mathfrak{B}_G=\left\{z_1+z_2j\in\mathrm{Mat}_n(\mathbb H)\midwd \begin{array}{l}z_1\in\mathrm{Sym}_n(\mathbb C),\;z_2\in\mathrm{Herm}_n(\mathbb C),\\
    \mathrm{Id}-(\bar z_1-\bar z_2j)(z_1+z_2j)\in \mathrm{Herm}^+_n(\mathbb H)
    \end{array}\right\}.$$
    Its compact closure is
    $$\overline{\mathfrak{B}}_G=\left\{z_1+z_2j\in\mathrm{Mat}_n(\mathbb H)\midwd \begin{array}{l}z_1\in\mathrm{Sym}_n(\mathbb C),\;z_2\in\mathrm{Herm}_n(\mathbb C),\\
    \mathrm{Id}-(\bar z_1-\bar z_2j)(z_1+z_2j)\in \mathrm{Herm}^{\geq 0}_n(\mathbb H)\end{array}\right\}.$$
    The tangent space at every $z\in\mathfrak{B}_G$ is \[T_z\mathfrak{B}_G=\{v_1+v_2j\in\mathrm{Mat}_n(\mathbb H)\mid v_1\in\mathrm{Sym}_n(\mathbb C),\;v_2\in\mathrm{Herm}_n(\mathbb C)\}.\]

    The map $F_{\mathfrak{C}_G,\mathfrak{P}_G^\pm}$ maps a quaternionic structure $J\colon \mathbb C^{2n}\to \mathbb C^{2n}$ to the quaternionic Lagrangian subspace $L_J$ which is the $(\mp j)$-eigenspace of $J_\mathbb H$. If $M\colon \mathbb C^{2n}\to \mathbb C^{2n}$ is a tangent vector at $J$, then 
    \[d_JF_{\mathfrak{C}_G,\mathfrak{P}_G^\pm}(M)=\pm\frac{1}{2}[M_\mathbb H(\bullet)]j,\] 
    where $[M_\mathbb H]\colon L_J\to \mathbb H^{2n}/L_J$ which is the natural projection of $M_\mathbb H|_{L_J}$.

    Let now $L\in\mathfrak{P}_G^\pm$ and $Q\colon L\to \mathbb H^{2n}/L$ be a tangent vector at $L$. Then $F_{\mathfrak{P}_G^\pm,\mathfrak{U}_G^\pm}$ maps $L$ to the following symmetric matrix. We first fix a basis of $L$ and write it as a $2n\times n$-matrix $A$. This matrix contains an upper $(n\times n)$-block $A_1$ and a lower one, say  $A_2$, which is invertible because $L\in\mathfrak{P}_G^\pm$. Then $F_{\mathfrak{P}_G^\pm,\mathfrak{U}_G^\pm}(L)=A_1A_2^{-1}$ and this expression is clearly independent of the chosen basis. Applying $Q$ to the vectors of the basis, we obtain a family of $2n\times n$ matrices $B+Ax$ where $x\in\mathrm{Mat}_n(\mathbb R)$. We choose one of them, say for $x=0$ we obtain $B$, then we also obtain in a similar vain two $n\times n$-matrices $B_1$ and $B_2$ which are upper and lower submatrices of $B$. Then,
    \[d_LF_{\mathfrak{P}_G^\pm,\mathfrak U_G^\pm}(Q)=(B_1-A_1A_2^{-1}B_2)A_2^{-1}.\] 
    Easy calculations show that this expression does not depend on any of the choices made.

    One similarly describes the map $F_{\mathfrak{P}_G^+,\mathfrak{B}_G}$. Let $T:=\frac{1}{\sqrt{2}}\begin{pmatrix}
    \mathrm{Id} & -j\mathrm{Id}\\
    -j\mathrm{Id} & \mathrm{Id}
    \end{pmatrix}$, denote $A':=TA$, $B':=TB$, and let $A'_1$ and $A'_2$ be the lower and upper $n\times n$ submatrices of $A'$ respectively, and $B'_1$ and $B'_2$ be the lower and upper $n\times n$ submatrices of $B'$ respectively. Then, 
    \[F_{\mathfrak{P}_G^+,\mathfrak{B}_G}(L)=A_1'(A_2')^{-1}\] and 
    \[d_LF_{\mathfrak{P}_G^+,\mathfrak B_G}(Q)=(B'_1-A'_1(A'_2)^{-1}B'_2)(A'_2)^{-1}.\]

\section{Projective space model for some compact groups}\label{sec:models for compact}

In this section, we assume $(A,\sigma)$ to be a Hermitian algebra. We discuss the projective space model of symmetric spaces of the following compact groups $G$ with respect to the following compact subgroups $K$ (embedded diagonally): 
\begin{enumerate}
    \item $K=\mathrm{O}(A,\sigma)\times\mathrm{O}(A,\sigma)<G=\mathrm{O}_2(A,\sigma)$;
    \item $K=\mathrm{O}(A,\sigma)<G=\mathrm{KO}_{(1,1)}(A,\sigma)\cong \mathrm{O}(A,\sigma)\times\mathrm{O}(A,\sigma)$;
    \item $K=\mathrm{O}(A,\sigma)<G=\mathrm{KSp}_{2}(A,\sigma)\cong\mathrm{O}(A_\mathbb C,\bar\sigma_\mathbb C)$;
    \item $K=\mathrm{O}(A_\mathbb C,\bar\sigma_\mathbb C)<G=\mathrm{KSp}^c_{2}(A_\mathbb C,\sigma_\mathbb C)\cong\mathrm{O}(A_\mathbb H,\sigma_1)$.
\end{enumerate}

\begin{thm}$ $
    \begin{enumerate}
        \item The space $\mathbb P(A^2)$ is a model for the symmetric space  
        $\mathrm{O}_2(A,\sigma)/(\mathrm{O}(A,\sigma)\times\mathrm{O}(A,\sigma))$;
        \item The space $\mathbb P\mathrm{Is}(h)$ is a model for the symmetric space $\mathrm{KO}_{(1,1)}(A,\sigma)/\mathrm{O}(A,\sigma)$, where $h$ is the standard indefinite orthogonal form on $A^2$;
        \item The space $\mathbb P\mathrm{Is}(\omega)$ is a model for the symmetric space $\mathrm{KSp}_2(A,\sigma)/\mathrm{O}(A,\sigma)$, where $\omega$ is the standard symplectic form on $A^2$;
        \item The space $\mathbb P\mathrm{Is}(\omega_\mathbb C)$ is a model for the symmetric space $\mathrm{KSp}^c_2(A_\mathbb C,\sigma_\mathbb C)/\mathrm{O}(A_\mathbb C,\bar\sigma_\mathbb C)$, where $\omega_\mathbb C$ is the complex linear extension of the standard symplectic form $\omega$ on $A^2$.
    \end{enumerate}    
\end{thm}

\begin{proof}
    The cases (1), (3), and (4) are proven in~\cite[Corollary 4.2 and 4.3]{ABRRW}. To see (2), it is enough to check that $\mathrm{KO}_{(1,1)}(A,\sigma)$ acts transitively on $\mathbb P\mathrm{Is}(h)$ and the stabilizer of the line $(1,0)^tA$ agrees with $\mathrm{O}(A,\sigma)$. By Gram--Schmidt orthogonalization (cf.~\cite[Proposition~3.16]{ABRRW}), for every $h$-isotropic vector $x=(x_1,x_2)^t$ such that $\sigma(x_1)x_1+\sigma(x_2)x_2=1$, there exist $y=(y_1,y_2)^t$ such that $\sigma(y_1)y_1+\sigma(y_2)y_2=1$ and $\sigma(x_1)y_1+\sigma(x_2)y_2=0$. Then $y$ is $h$-isotropic as well, and moreover, since $h$ is non-degenerate, then $h(x,y)\in A^\times$. Therefore, $y$ can be normalized so that $h(x,y)=1$. Then the matrix 
    \[g:=\begin{pmatrix}
        x_1 & y_1 \\
        x_2 & y_2
        \end{pmatrix}\in \mathrm{KO}_{(1,1)}(A,\sigma)=\mathrm{O}_{(1,1)}(A,\sigma)\cap \mathrm{O}_{2}(A,\sigma),\] 
    where we identify  $\mathrm{O}_{2}(A,\sigma)$ with $\mathrm{Aut}(b)$, where $b(x,y)=\sigma(x_1)y_1+\sigma(x_2)y_2$. This matrix maps the line $(1,0)^tA$ to the line spanned by $x$, i.e. $\mathrm{KO}_{(1,1)}(A,\sigma)$ acts transitively on $\mathbb P\mathrm{Is}(h)$. The stabilizer of the line $(1,0)^tA$ is clearly $\mathrm{O}(A,\sigma)$ diagonally embedded into $\mathrm{KO}_{(1,1)}(A,\sigma)$.  
\end{proof}

The precompact and unbounded models, which we defined for non-compact Lie groups as affine charts of the corresponding projective space models, cannot exist for these symmetric spaces because symmetric spaces for compact groups are closed manifolds. However, the latter 3 spaces are parts of the boundary of the corresponding noncompact symmetric spaces for $\mathrm{O}_{(1,1)}(A,\sigma)$, $\mathrm{Sp}_{2}(A,\sigma)$, and $\mathrm{Sp}_{2}(A_\mathbb C,\sigma_\mathbb C)$, thus the precompact models can be used to provide their embeddings into certain vector spaces:

\begin{thm}$ $
    \begin{enumerate}
        \item The space 
        $$\mathfrak{B}_{\mathrm{KO}_{(1,1)}(A,\sigma)}:=\{a\in A\mid 1-\sigma(a)a=0\}$$ is a model for the symmetric space $\mathrm{KO}_{(1,1)}(A,\sigma)/\mathrm{O}(A,\sigma)$;
        \item The space 
        $$\mathfrak{B}_{\mathrm{KSp}_2(A,\sigma)}:=\{z\in A^{\sigma_\mathbb C}_\mathbb C\mid 1-\bar zz=0\}$$ is a model for the symmetric space $\mathrm{KSp}_2(A,\sigma)/\mathrm{O}(A,\sigma)$;
        \item The space 
        $$\mathfrak{B}_{\mathrm{KSp}^c_2(A_\mathbb C,\sigma_\mathbb C)}:=\left\{z\in A^{\sigma_0}_\mathbb{H}\mid 1-\sigma_1(z)z=0\right\}$$ is a model for the symmetric space $\mathrm{KSp}^c_2(A_\mathbb C,\sigma_\mathbb C)/\mathrm{O}(A_\mathbb C,\bar\sigma_\mathbb C)$.
    \end{enumerate}
\noindent The action of the group $G$ on its symmetric space is given by the same formulas as in each Section on the corresponding non-compact symmetric spaces.
\end{thm}

\newpage 

\part{Applications to Higgs bundles}
In this second part of the article, we give three applications to the theory of Higgs bundles.  First, we show that each geometric model for the Riemannian symmetric space we described in the first part provides a new geometric interpretation of the $G$-Higgs bundle data, where $G$ is a symplectic group or an indefinite orthogonal group over an involutive,
possibly noncommutative, Hermitian algebra, or its complexification. 

Second, we give an exact component count for the moduli spaces of $\mathrm{Sp}_2(A_{\mathbb{C}}, \sigma_{\mathbb{C}})$-Higgs bundles and of $\mathrm{O}(A_{\mathbb{C}}, \sigma_{\mathbb{C}})$-Higgs bundles which does not rely on  Morse--Bott theoretic techniques. 

As a third application, we construct a factorization of the
Hitchin morphism for $\mathrm{Sp}_2(A_{\mathbb C},\sigma_{\mathbb C})$-Higgs bundles, together with analogous factorizations for the real groups $\mathrm{Sp}_2(A,\sigma)$ and $\mathrm{O}_{(1,1)}(A,\sigma)$. The existence of such factorization allows for a more concrete understanding of the image of the Hitchin morphism.

\section{Higgs bundle data}\label{sec:Higgs-Data}

In this section, we use the geometric models of the symmetric spaces constructed in Part 1 in order to describe the corresponding $G$-Higgs bundle data, and in particular to give different geometric realizations.  

Let $X$ be a compact Riemann surface, let $G$ be one of the Lie groups considered in Sections \ref{sec:ort_models}-\ref{sec:complex_models}, and let $K\subseteq G$ be a maximal compact subgroup. Given a reductive representation
$\rho:\pi_1(X)\longrightarrow G
$, let $
f_{G/K}:\widetilde X\longrightarrow G/K
$ be a $\rho$-equivariant harmonic map, whose existence is guaranteed by Corlette’s theorem (see Appendix \ref{sec:Corlette_Hadamard}). Since the isometries $G/K \cong \mathcal{X}$ we studied in the first part are $G$-equivariant, this yields $\rho$-equivariant harmonic maps $
f_{\mathcal X}:\widetilde X\longrightarrow \mathcal{X}
$, for each geometric model $\mathcal X$ of the symmetric space $G/K$.

The map $f_{\mathcal{X}}$ then determines a $K$-reduction of the flat principal $G$-bundle $(E_G,\nabla)$ corresponding to $\rho$, that is a smooth principal $K$-bundle $E_K$ with $E_K\times_KG\cong E_G$. After complexification, this reduction gives a principal $K^{\mathbb C}$-bundle 
$E_{K^{\mathbb C}} \to X$ equipped with a holomorphic structure $\bar{\partial}_{E_{K^{\mathbb{C}}}}$.  More precisely, the harmonic reduction decomposes the flat connection as
\[
\nabla=\nabla_K+\Psi,
\qquad
\Psi\in\Omega^1(X,E_K(\mathfrak m)).
\]
The $(0,1)$-part of $\nabla_K$ defines a holomorphic structure $\bar{\partial}_{E_{K^{\mathbb{C}}}}$ on $E_{K^{\mathbb C}}$, while $\varphi:=\Psi^{1,0}$
defines a holomorphic section
$\varphi\in H^0(X,E_{K^{\mathbb C}}(\mathfrak m^{\mathbb C})\otimes K_X)$.
Thus, one obtains the \(G\)-Higgs bundle associated with \(\rho\).

Different geometric models $\mathcal X$ of $G/K$ lead to different descriptions of the
$G$-Higgs bundle data. Therefore, all realizations for the symmetric space define the same algebraic moduli problem of a holomorphic principal $K^{\mathbb C}$-bundle together with a holomorphic section of an associated bundle; the existence of the moduli space
of polystable $G$-Higgs bundles follows from Schmitt's general construction  (cf.~\cite[Section 2.7]{Schmitt}). We
denote this moduli space by
$\mathcal M_{\mathrm{Higgs}}(X,G)$.

\subsection{Motivating example}
In order to motivate our approach, we illustrate the classical description of the $G$-Higgs bundle data in the case of the group $G=\mathrm{Sp}_{2n}(\mathbb{R})$. A maximal compact subgroup $K$ of $G$ can be identified with the unitary group $\mathrm{U}(n)$ whose complexification is isomorphic to $K^\mathbb C
\cong \mathrm{GL}_n(\mathbb C)$.

The complexified Lie algebra
\[{{\mathfrak{g}}^{\mathbb{C}}}=\mathfrak{sp}_{2n}\left(\mathbb{C} \right)=\left\{ \left( \begin{matrix}
   A & B  \\
   C & -{{A}^{t}}  \\
\end{matrix} \right)\midwd A,B,C\in {{\mathsf{\mathrm{Mat}}}_{n}}\left( \mathbb{C} \right);{{B}^{t}}=B,{{C}^{t}}=C \right\}\] has split real form $\mathfrak{sp}_{2n}\left(\mathbb{R} \right)$ and compact real form $\mathfrak{sp}(n)$.

The Cartan involution $\theta :\mathfrak{sp}_{2n}\left(\mathbb{C} \right)\to \mathfrak{sp}_{2n}\left(\mathbb{C} \right)$ with $\theta \left( X \right)=-{{X}^{t}}$ determines a Cartan decomposition for a choice of maximal compact subgroup  $K\cong  \mathrm{U}(n)\subset \mathrm{Sp}_{2n}( \mathbb{R})$ as
	\[\mathfrak{sp}_{2n}\left(\mathbb{R} \right)=\mathfrak{u}\left(n \right)\oplus \mathfrak{m}\]
	with complexification 
    \begin{equation*}\label{Cartan_complex_sp2n}\mathfrak{sp}_{2n}\left( \mathbb{C} \right)=\mathfrak{gl}_n\left(\mathbb{C} \right)\oplus {{\mathfrak{m}}^{\mathbb{C}}}.
    \end{equation*}
Classically, an $\mathrm{Sp}_{2n}(\mathbb R)$-Higgs bundle is defined as follows:

\begin{defn}
A \emph{$\mathrm{Sp}_{2n}(\mathbb R)$-Higgs bundle} over $X$ is a pair $\left( E,\varphi  \right)$ where
\begin{itemize}
  \item $E$ is a principal holomorphic $\mathrm{GL}_n(\mathbb C)$-bundle over $X$ and
  \item a \emph{Higgs field} $\varphi \in H^0(X,E\left( {{\mathfrak{m}}^{\mathbb{C}}} \right)\otimes K_X)$.
\end{itemize}
\end{defn}

One translates this principal-bundle definition into a vector-bundle one as follows:  
Applying the change-of-basis matrix $T=\left( \begin{matrix}
   I_n & iI_n  \\
   I_n & -iI_n  \\
\end{matrix} \right)$ on ${{\mathbb{C}}^{2n}}$, we can identify the summands in the Cartan decomposition of $\mathfrak{sp}_{2n}\left(\mathbb{C} \right)\subset \mathfrak{sl}_{2n}\left(\mathbb{C} \right)$ as:
\begin{align*}
\mathfrak{gl}_n\left(\mathbb{C} \right)& =\left\{ \left( \begin{matrix}
   Z & 0  \\
   0 & -{{Z}^{t}}  \\
\end{matrix} \right)\midwd Z\in {{\mathsf{\mathrm{Mat}}}_{n}}\left( \mathbb{C} \right) \right\},\\
{{\mathfrak{m}}^{\mathbb{C}}}& =\left\{ \left( \begin{matrix}
   0 & \beta   \\
   \gamma  & 0  \\
\end{matrix} \right)\midwd \,\beta ,\gamma \in {{\mathsf{\mathrm{Mat}}}_{n}}\left( \mathbb{C} \right);{{\beta }^{t}}=\beta ,{{\gamma }^{t}}=\gamma   \right\}.
\end{align*}

The Cartan decomposition and the choice of basis on $\mathbb{C}^{2n}$ by $T$ above, induce a splitting of $\mathbb{C}^{2n}$ into two copies of $\mathbb{C}^n$. Then, the symplectic form $\omega$ provides an identification of one copy $\mathbb{C}^n$ with the dual vector space $(\mathbb{C}^n)^*$ of the second copy. Therefore, a pair $(\beta, \gamma)$ in $\mathfrak{m}^{\mathbb{C}}$ as above can be regarded as an element of the space   
\[\mathrm{Sym}^{2}\left( {{\mathbb{C}}^{n}} \right)\oplus \mathrm{Sym}^{2}\left( {{\left( {{\mathbb{C}}^{n}} \right)}^{*}} \right).\]

Let $V$ denote the rank $n$ vector bundle associated to a holomorphic principal $\mathrm{GL}_n( \mathbb{C})$-bundle $E$ via the standard representation. Then from the Cartan decomposition for the Lie algebra $\mathfrak{sp}_{2n}\left(\mathbb{C} \right)$ we can identify
	\[E\left( {{\mathfrak{m}}^{\mathbb{C}}} \right)={\mathrm{Sym}}^{2}\left( V \right)\oplus {\mathrm{Sym}}^{2}\left( {{V}^{*}} \right)\]
and so the general definition for an $\mathrm{Sp}_{2n}( \mathbb{R})$-Higgs bundle specializes to the following:

\begin{defn}\label{defn:classical_Sp2nR}
An $\mathrm{Sp}_{2n}( \mathbb{R})$-Higgs bundle over $X$ is defined by a triple $\left( V,\beta ,\gamma  \right)$, where $V$ is a rank $n$ holomorphic vector bundle over $X$ and $\beta , \gamma$ are symmetric homomorphisms
	\[\beta :{{V}^{*}}\to V\otimes K_X\text{  and  }\gamma :V\to {{V}^{*}}\otimes K_X.\]
\end{defn}

The embedding $\mathrm{Sp}_{2n}( \mathbb{R})\hookrightarrow \mathrm{SL}_{2n}( \mathbb{C})$ allows one to reinterpret the defining $\mathrm{Sp}_{2n}( \mathbb{R})$-data of a Higgs bundle as special $\mathrm{SL}_{2n}( \mathbb{C})$-data. We can thus consider an $\mathrm{Sp}_{2n}( \mathbb{R})$-Higgs bundle to be defined as a pair $\left( W,\Phi  \right)$, where \begin{enumerate}
\item $W=V\oplus {{V}^{*}}$ is a rank $2n$ holomorphic vector bundle over $X$ and
\item $\Phi :W\to W\otimes K_X$ is a Higgs field with $\Phi = \begin{pmatrix}
   0 & \beta   \\
   \gamma  & 0  \\
\end{pmatrix}$, 
\end{enumerate}
for $V$, $\beta$, $\gamma$ as above.

However, this transition from the principal bundle description to the vector bundle description can be done in a different way using the equivariant map $f_\mathfrak P\colon\tilde X\to \mathfrak P$ where $\mathfrak P$ is the projective space model which is in this case the space of positive complex Lagrangians of $\mathbb C^{2n}$:
$$\mathfrak P=\left\{L\subset \mathbb C^{2n}\midwd\; \dim(L)=n,\; \omega|_L=0,\;i\omega(\bar\cdot,\cdot)\text{ is a positive definite Hermitian form on $L$}\right\}.$$

As $K^\mathbb C<\mathrm{Sp}_{2n}(\mathbb C)$, for a principal $K^\mathbb C$-bundle $E$, we take the associated holomorphic vector bundle of rank $2n$, which, by a slight abuse of notation, we denote by $W$. However, we will see that in fact it is the same vector bundle. For every $x\in\tilde X$, the map $f_\mathfrak P$ provides a splitting $\mathbb C^{2n}=L_x\oplus L_x^*$ where $L_x^*$ is the complement of $L_x$ with respect to the Hermitian form $i\omega(\bar\cdot,\cdot)$, which also identifies as the dual vector space for $L_x$ using the symplectic form $\omega$. This splitting descends to a splitting of the bundles $W=V\oplus V^*$. As the complexified tangent space is
$$T_L^\mathbb C\mathfrak P=\left\{(Q,Q') \in\mathrm{Hom}(L,L^*)\times\mathrm{Hom}(L^*,L)\mid Q,Q'\text{ are symmetric as bilinear forms}\right\},$$
the Higgs field is therefore identified with a pair of symmetric morphisms
\[
\beta:V^*\to V\otimes K_X \qquad \text{and} \qquad 
\gamma:V\to V^*\otimes K_X.
\]
Thus, the Higgs field $\Phi\colon W\to W\otimes K_X$ can be written as an anti-diagonal matrix
$$\Phi=\begin{pmatrix}
   0 & \beta   \\
   \gamma  & 0  \\
\end{pmatrix}.$$
Note that this agrees with the vector bundle description of an $\mathrm{Sp}_{2n}(\mathbb R)$-Higgs bundle above.

In what follows, we describe explicitly the Higgs bundle data for the groups $\mathrm{O}_{(1,1)}(A, \sigma)$ and $\mathrm{Sp}_2(A, \sigma)$ in all four geometric models of the symmetric space. Since the groups $A^{\times}$ and $\mathrm{O}(A_{\mathbb{C}}, \sigma_{\mathbb{C}})$ can be both embedded as subgroups into $\mathrm{O}_{(1,1)}(A,\sigma)$, the $G$-Higgs bundle data in these cases can be induced as special cases of $\mathrm{O}_{(1,1)}(A,\sigma)$-data. For the groups $\mathrm{Sp}_2(A_{\mathbb{C}}, \sigma_{\mathbb{C}})$, we demonstrate the Higgs bundle data only in the quaternionic structures model and the projective space model which are the most interesting cases. Several examples are included in Appendix \ref{Appendix:several_examples}.

\subsection{Higgs bundles for models associated to groups \texorpdfstring{$\mathrm{O}_{(1,1)}(A,\sigma)$}{} }\label{subsec:indefi_defs}

For an indefinite orthogonal group $\mathrm{O}_{(1,1)}(A,\sigma)$ over a Hermitian noncommutative algebra $A$ with an anti-involution $\sigma$, we interpret the Higgs bundle data in terms of each of the four models of the symmetric space $\mathcal X:=\mathcal{X}_{\mathrm{O}_{(1,1)}(A,\sigma)}$ described explicitly in Section \ref{sec:ort_models}. 

In this case, a flat $\mathrm{O}_{(1,1)}(A,\sigma)$-bundle is an $A^2$-bundle with the standard indefinite orthogonal form on $A^2$
\[\omega(x,y):=\sigma(x)^t\Omega y,\]
where $\Omega=\begin{pmatrix}0 & 1 \\ 1 & 0\end{pmatrix}$, for $x,y\in A^2$. Let us denote by $(E, \omega, \nabla)$ this flat bundle.

\subsubsection{Indefinite involutive operators model.}

Recall from Section \ref{sec:indefinite_O} that $\mathfrak{C}_G$ denotes the indefinite involutive operators model of $\mathcal{X}$. For $x\in X$, let $J_x:=f_{\mathfrak{C}_G}(x)\in\mathfrak{C}_G$ be the corresponding indefinite involutive operator. Then we have the identification $\mathrm{Stab}_G(J_x)\cong K$, which acts on the tangent space
$$
T_{J_x}\mathfrak{C}_G=\{\ L\colon A^2\to A^2\mid \text{$L$ is $A$-linear and }LJ_x+J_xL=0\ \};
$$
note that $J_x$ extends $\mathbb{C}$-linearly to an $A_{\mathbb C}$-linear operator $J_{x}^{\mathbb{C}}$ on $A_{\mathbb{C}}^2$. Therefore, the Higgs field in this case is a $(1,0)$-form $\varphi$ with values in $f_{\mathfrak{C}_G}^*T^\mathbb{C}\mathfrak{C}_G$, such that $\bar\partial_E \varphi =0$. More specifically, the Higgs field in this model is interpreted as an endomorphism $ L\colon A_{\mathbb{C}}^2\to A_{\mathbb{C}}^2$ which is $A_{\mathbb{C}}$-linear and satisfies  $LJ_{x}^{\mathbb{C}}+J_{x}^{\mathbb{C}}L=0$.

Therefore, using the indefinite operators model for the group $G=\mathrm{O}_{(1,1)}(A,\sigma)$, a $G$-Higgs bundle can be described as: 
\begin{itemize}
    \item a holomorphic $A^2_{\mathbb{C}}$-bundle $V$ equipped with: 
    \begin{itemize}
        \item an orthogonal structure $h$,
        \item an indefinite involutive endomorphism $J$ of $V$, which is compatible with $h$, i.e. the form $h_J=h(J(\cdot),\cdot)$ is non-degenerate,
    \end{itemize}
    and
    \item a Higgs field $\varphi\in H^0(X,\mathrm{End}_{A_\mathbb C}(V)\otimes K_X)$ such that $\varphi J+J\varphi=0$.
\end{itemize}

\subsubsection{Projective space models.}

Recall from Section \ref{sec:projective_O} that $\mathfrak{P}^{\pm}_G$ denote the two projective space models of $\mathcal{X}$. For $x\in X$, let $l_x:=f_{\mathfrak{P}^{+}_G}(x)\in\mathfrak{P}^+_G$ be the corresponding $A$-line in $\mathbb{P}(A^2)$, and $l_x^{\perp_\omega}\in \mathfrak{P}_G^-$ the $\omega$-orthogonal line which is always transverse to $l_x$. Then we have the identification $\mathrm{Stab}_G(l_x)\cong K$, which acts on the tangent space
$$
T_{l_x}\mathfrak{P}^{\pm}_G=\mathrm{Hom}(l_x,A^2/l_x);
$$
note that the $A$-line $l_x$ extends $\mathbb{C}$-linearly to an $A_{\mathbb C}$-line $l_{x}^{\mathbb{C}}\subset A_{\mathbb{C}}^2$. 

As above, for $x\in X$ we denote $l_x:=f_{\mathfrak{P}_G^+}(x)\in \mathfrak{P}_G^+$. Then, the Higgs field in this case is a $(1,0)$-form $\varphi$ with values in $f_{\mathfrak{P}_G^+}^*T^\mathbb{C}\mathfrak{P}_G^+$, such that $\bar\partial_E \varphi =0$. More specifically, for each $x\in X$, 
the fiber over $x$ of the bundle part of a Higgs bundle consists of two $A_\mathbb C$-lines $l_x^\mathbb C$ and $(l_x^{\perp_{\omega}})^\mathbb C$, and a Higgs field is a homomorphism $l_x^\mathbb C \to (l_x^{\perp_{\omega}})^\mathbb C$, where by a slight abuse of notation we still write $\omega$ to mean $\omega_{\mathbb{C}}$ on the complexified lines..

Therefore, using the projective space model for the group $G=\mathrm{O}_{(1,1)}(A,\sigma)$, a $G$-Higgs bundle can be described as: 
\begin{itemize}
    \item a holomorphic $A^2_{\mathbb{C}}$-bundle $V$ equipped with an orthogonal structure $h$, which splits as a direct sum of two $A_{\mathbb{C}}$-line bundles, $\ell$ and $\ell^{\perp_{h}}$, such that $h$ descends to a non-degenerate orthogonal structure on each of these two line subbundles, and
    \item a Higgs field $\varphi\in H^0(X,\mathrm{Hom}_{A_\mathbb C}(\ell, \ell^{\perp_{h}})\otimes K_X)$. 
\end{itemize}

An analogous definition is obtained for the projective space model $\mathfrak{P}_G^-$ as well.

\subsubsection{Half-space models and precompact model.} Since the half-space models $\mathfrak U^{\pm}_G$ and the precompact model $\mathfrak{B}_G$ are open domains in $A$, the complexification of the tangent space for every point of $\mathfrak U^{\pm}_G$ and $\mathfrak{B}_G$ is naturally isomorphic to  $A_\mathbb C$. Therefore, $f^*_{\mathfrak U^{\pm}_G} T^\mathbb C\mathfrak U^{\pm}_G$ and $f^*_{\mathfrak{B}_G} T^\mathbb C\mathfrak{B}_G$ identify canonically with $\tilde X\times A_\mathbb C$ and a Higgs field $\varphi(x)$ is in both cases a holomorphic 1-form with values in $A_\mathbb C$. The only difference between fibers of $\tilde X\times  A_\mathbb C$ over different points $x\in\tilde X$ is the action of the stabilizer $\mathrm{Stab}_{\mathrm{O}_{(1,1)}(A,\sigma)}(f(x))$, which is described in Section~\ref{sec:halfspace_O} (for the  upper half-space model) and Section~\ref{sec:precompact_O} (for the precompact model). 

We thus conclude that the Higgs field in the upper half-space model $\mathfrak U^+_G$ or in the \emph{precompact model} $\mathfrak{B}_G$ is a $(1,0)$-form $\varphi$ with values in $f_{\mathfrak{U}^+_G}^*T^\mathbb{C}\mathfrak{U}^+_G \cong  f_{\mathfrak{B}_G}^*T^\mathbb{C}\mathfrak{B}_G$, such that $\bar\partial_E \varphi =0$. In other words, the Higgs field in this case is an element in $A_{\mathbb{C}}$.

An analogous geometric realization is obtained for the lower half-space model $\mathfrak{U}_G^-$ as well.

\subsection{Higgs bundles for models associated to groups  \texorpdfstring{$A^{\times}$ and $\mathrm{O}(A_{\mathbb{C}}, \sigma_{\mathbb{C}})$}{}}

The groups $A^{\times}$ and $\mathrm{O}(A_{\mathbb{C}}, \sigma_{\mathbb{C}})$ can be both embedded as subgroups into $\mathrm{O}_{(1,1)}(A,\sigma)$, thus the $G$-Higgs bundle data in these cases are described as special cases of the $\mathrm{O}_{(1,1)}(A,\sigma)$-definitions given above. We next provide these descriptions only for the indefinite involutive operators and the projective space models, while the 
definitions for the half-space models and the precompact model associated to the groups $A^{\times}$ and $\mathrm{O}(A_{\mathbb{C}}, \sigma_{\mathbb{C}})$ are obtained in an exactly analogous way.

\subsubsection{Groups $A^{\times}$.}

As we have seen in Section \ref{sec:A_models}, consider a diagonal copy $\hat{G}$ of the group $A^{\times}$ into $\mathrm{O}_{(1,1)}(A,\sigma)$ and a diagonal copy $\hat{K}$ of a maximal compact subgroup $K=\mathrm{O}(A,\sigma)$. We have the following geometric interpretations:

Using the indefinite involutive operators model for groups $G=A^{\times}$, a $G$-Higgs bundle can be described as: 
\begin{itemize}
    \item a holomorphic $A^2_{\mathbb{C}}$-bundle $V$ equipped with: 
    \begin{itemize}
        \item an orthogonal structure $h$,
        \item a symplectic structure $\omega$,
        \item an indefinite involutive endomorphism $J$ of $V$, which is compatible with $h$ and such that $\omega(J(\cdot),J(\cdot))=-\omega(\cdot,\cdot)$,
    \end{itemize}
    and
    \item a Higgs field $\varphi\in H^0(X,\mathrm{End}_{A_\mathbb C}(V)\otimes K_X)$ such that 
    \begin{itemize}
        \item $\omega(\varphi(\cdot),\cdot)+\omega(\cdot,\varphi(\cdot))=0$, 
        \item $\varphi J+J\varphi=0$.
        \end{itemize}
\end{itemize}

Furthermore, using the projective space model for the group $G=A^\times$, a $G$-Higgs bundle can be described as: 
\begin{itemize}
    \item a holomorphic $A^2_{\mathbb{C}}$-bundle $V$ equipped with:
    \begin{itemize}
    \item an orthogonal structure $h$,
    \item a symplectic structure $\omega$,
    \end{itemize}
    which splits as a direct sum of two $\omega$-isotropic $A_{\mathbb{C}}$-line subbundles, $\ell$ and $\ell^{\perp_{\omega}}$, such that $h$ descends to a non-degenerate orthogonal structure on each of these two line subbundles, and
    \item a Higgs field $\varphi\in H^0(X,\mathrm{End}_{A_\mathbb C}\ell\otimes K_X)$ such that 
    $$h|_{\ell}(\varphi(\cdot),\cdot)=h|_{\ell}(\cdot,\varphi(\cdot)).$$ 
\end{itemize}
From this description, it is obvious that all information about $V$ is contained in the subbundle $\ell$. So we can reformulate this description:
a $G$-Higgs bundle can be described as: 
\begin{itemize}
    \item a holomorphic $A_{\mathbb{C}}$-bundle $\ell$ equipped with an orthogonal structure $h$ and
    \item a Higgs field $\varphi\in H^0(X,\mathrm{End}_{A_\mathbb C}\ell\otimes K_X)$ such that $h(\varphi(\cdot),\cdot)=h(\cdot,\varphi(\cdot)).$ 
\end{itemize}

\subsubsection{Groups $\mathrm{O}(A_{\mathbb{C}}, \sigma_{\mathbb{C}})$}\label{subsec:O(A_C)defs}

We have seen in Section \ref{sec:OC_models} that the group $G:=\mathrm{O}(A_{\mathbb{C}},\sigma_{\mathbb{C}})$ is contained in $A_{\mathbb{C}}^{\times}$, thus $G$ is embedded into $\mathrm{O}_{(1,1)}(A,\bar{\sigma}_{\mathbb{C}})$ as a subgroup $\hat{G}$ of diagonal matrices. Let $\hat{K}$ denote the corresponding diagonal copy of $K:=\mathrm{O}(A, \sigma)$ in $\mathrm{O}_{(1,1)}(A,\bar{\sigma}_{\mathbb{C}})$. As in the previous subsection, we have the following geometric interpretations for the Higgs fields of the $G$-Higgs bundles: 

Using the indefinite involutive operators model for groups $G=\mathrm{O}(A_{\mathbb{C}}, \sigma_{\mathbb{C}})$, a $G$-Higgs bundle can be described as: 
\begin{itemize}
    \item a holomorphic $A^2_{\mathbb{C}}$-bundle $V$ equipped with: 
    \begin{itemize}
        \item two orthogonal structures $h$ and $b$,
        \item an indefinite involutive endomorphism $J$ of $V$, which is compatible with $h$ and such that $b(J(\cdot),J(\cdot))=-b(\cdot,\cdot)$,
    \end{itemize}
    and
    \item a Higgs field $\varphi\in H^0(X,\mathrm{End}_{A_\mathbb C}(V)\otimes K_X)$ such that 
    \begin{itemize}
        \item $b(\varphi(\cdot),\cdot)+b(\cdot,\varphi(\cdot))=0$, 
        \item $\varphi J+J\varphi=0$.
        \end{itemize}
\end{itemize}

Furthermore, using the projective space model for the group $G=\mathrm{O}(A_{\mathbb{C}}, \sigma_{\mathbb{C}})$, a $G$-Higgs bundle can be described as: 
\begin{itemize}
    \item a holomorphic $A^2_{\mathbb{C}}$-bundle $V$ equipped with two orthogonal structures $h$ and $b$,
    which splits as a direct sum of two $b$-isotropic $A_{\mathbb{C}}$-line subbundles, $\ell$ and $\ell^{\perp_{b}}$, such that $h$ descends to a non-degenerate orthogonal structure on each of these two line subbundles, and
    \item a Higgs field $\varphi\in H^0(X,\mathrm{End}_{A_\mathbb C}\ell\otimes K_X)$ such that 
    $$h|_{\ell}(\varphi(\cdot),\cdot)+h|_{\ell}(\cdot,\varphi(\cdot))=0.$$ 
\end{itemize}
From this description, it is obvious that all information about $V$ is contained in the subbundle $\ell$. We can reformulate this description of a $G$-Higgs bundle as follows: 
\begin{itemize}
    \item a holomorphic $A_{\mathbb{C}}$-bundle $\ell$ equipped with an orthogonal structure $h$, and
    \item a Higgs field $\varphi\in H^0(X,\mathrm{End}_{A_\mathbb C}\ell\otimes K_X)$ such that $h(\varphi(\cdot),\cdot)+h(\cdot,\varphi(\cdot))=0.$ 
\end{itemize}

\subsection{Higgs bundles for models associated to groups \texorpdfstring{$\mathrm{Sp}_2(A, \sigma)$}{}}\label{subsec:real_defs}

We consider similarly the Higgs bundle data for symplectic groups $\mathrm{Sp}_{2}(A,\sigma )$ over a Hermitian noncommutative algebra $A$ with an anti-involution $\sigma$. For $G=\mathrm{Sp}_2(A, \sigma)$, $K=\mathrm{KSp}_2(A, \sigma)$, we give next the geometric interpretation of the Higgs field for each model of the symmetric space: 

\subsubsection{Complex structures model.}\label{G-Higgs in cx_str_model}

Recall from Section \ref{sec_cplx-str-mod} that $\mathfrak{C}_G$ denotes the complex structures model of $\mathcal{X}$, for $x\in X$, let $J_x:=f_{\mathfrak{C}_G}(x)\in\mathfrak{C}_G$ be the corresponding complex structure. Then we have the identification $\mathrm{Stab}_G(J_x)\cong K$, which acts on the tangent space
\begin{align*}
        T_{J_x}\mathfrak{C}_G=\left\{
        L\colon A^2\to A^2\midwd 
        \begin{aligned}
            &\ L \text{ is } A\text{-linear},\\
            &\ h_L \text{ is } \sigma\text{-symmetric, and}\\
            &\ LJ_{x}+J_{x}L=0
        \end{aligned}
        \right\};
    \end{align*}
note that $J_x$ extends $\mathbb{C}$-linearly to an $A_{\mathbb C}$-linear operator $J_{x}^{\mathbb{C}}$ on $A_{\mathbb{C}}^2$. Therefore, using the complex structures model for the group $G$, a $G$-Higgs bundle can be described as: 
\begin{itemize}
    \item a holomorphic $A^2_{\mathbb{C}}$-bundle $V$ equipped with: 
    \begin{itemize}
        \item a symplectic structure $\omega$,
        \item an automorphism $J$ of $V$ such that $iJ$ is an indefinite involutive operator on $V$ and the form $h_J=\omega(J(\cdot),\cdot)$ is symmetric and non-degenerate,
    \end{itemize}
    and
    \item a Higgs field $\varphi\in H^0(X,\mathrm{End}_{A_\mathbb C}(V)\otimes K_X)$ such that 
    \begin{itemize}
        \item $\omega(\varphi(\cdot),\cdot)=\omega(\cdot,\varphi(\cdot))$,  
        \item $\varphi J+J\varphi=0$.
        \end{itemize}
\end{itemize}

\subsubsection{Projective space model.}\label{df:SpR-Higgs-proj}

Notice that, as we have already seen in Lemma~\ref{lem:complexified_complex_structure}, any $L \in T^\mathbb C_{J}\mathfrak{C}_G$ has a compatible structure of eigen-directions with $J_\mathbb C$. Namely, let $v_\pm\in A_\mathbb C^2$ be such that $J_\mathbb C v_\pm=\pm i v_\pm$. Since $J$ maps $A^2$ to $A^2$, $v_\pm$ can be chosen so that $v_\pm=\bar v_\mp$ and, moreover, $Lv_\pm=v_\mp a_\pm$ where $a_\pm\in A_\mathbb C^{\sigma_\mathbb C}$. This gives rise to two linear maps $L_\pm\colon l_\pm\to A_\mathbb C^2/l_\pm$ mapping $v_\pm\mapsto Lv_\pm+l_\pm$ where $l_\pm=v_\pm A_\mathbb C=\bar v_\mp A_\mathbb C=\bar l_\mp$. Notice that $L_\pm$ represent tangent vectors of $\mathfrak{P}_G^+$ at $l_\pm$. In this regard, we can identify the complexified tangent bundle  $T^\mathbb C\mathfrak{P}_G^+$ with the space of triples (cf.~\eqref{eq:compl.tangent.P.Sp})
$$\{(l,L_+,L_-)\mid l\in\mathfrak{P}_G^+,\;L_+\in T_l\mathfrak{P}_G^+,\; L_-\in T_{\bar l}\mathfrak{P}_G^-\}.$$

Therefore, using the projective space model for the group $G$, a $G$-Higgs bundle can be described as: 
\begin{itemize}
    \item a holomorphic $A^2_{\mathbb{C}}$-bundle $V$ equipped with a symplectic structure $\omega$, which splits as a direct sum of two $\omega$-isotropic $A_{\mathbb{C}}$-line subbundles, $\ell$ and $\ell'$, and
    \item a Higgs field, which is a pair $(\beta,\gamma)$ where $\beta\in H^0(X,\mathrm{Hom}_{A_\mathbb C}(\ell,\ell')\otimes K_X)$ and $\gamma\in H^0(X,\mathrm{Hom}_{A_\mathbb C}(\ell',\ell)\otimes K_X)$ such that $\omega(\beta(x),y)=\omega(x,\beta(y))$ for all local sections $x$ and $y$ of $\ell$ and $\omega(\gamma(x'),y')=\omega(x',\gamma(y'))$ for all local sections $x'$ and $y'$ of $\ell'$. 
\end{itemize}

\subsubsection{Upper half-space model and precompact model.} Since the upper half-space model $\mathfrak U_G^+$ and the precompact model $\mathfrak{B}_G$ are open domains in $A_\mathbb C^{\sigma_\mathbb C}$, the complexification of the tangent space for every point of $\mathfrak U_G^+$ is naturally isomorphic to  $A_\mathbb C^{\sigma_\mathbb C}\times A_\mathbb C^{\sigma_\mathbb C}$. Therefore, $f^*_{\mathfrak U_G^+} T^\mathbb C\mathfrak U_G^+$ and $f^*_{\mathfrak{B}_G} T^\mathbb C\mathfrak{B}_G$ identify canonically with $\tilde X\times (A_\mathbb C^{\sigma_\mathbb C})^2$ and a Higgs field $\varphi(x)$ is in both cases a holomorphic 1-form with values in $A_\mathbb C^{\sigma_{\mathbb{C}}}\times A_\mathbb C^{\sigma_\mathbb C}$.

The only difference between the fibers of $\tilde X\times (A^{\sigma_\mathbb C}_\mathbb C)^2$ over different points $x\in\tilde X$ is the action of the stabilizer $\mathrm{Stab}_{\mathrm{Sp}_2(A,\sigma)}(f(x))$, which is described in Section~\ref{sec:upperhalf_real} (for the upper half-space model) and Section~\ref{sec:precompact_real} (for the precompact model). Thus, the Higgs field is a $(1,0)$-form $\varphi$ with values in $f_{\mathfrak{U}_G^+}^*T^\mathbb{C}\mathfrak{U}_G^+ \cong  f_{\mathfrak{B}_G}^*T^\mathbb{C}\mathfrak{B}_G$, such that $\bar\partial_E \varphi =0$. In other words, the Higgs field $\varphi(x)$ is an element in $A_\mathbb C^{\sigma_{\mathbb{C}}}\times A_\mathbb C^{\sigma_\mathbb C}$.

\subsection{Higgs bundles for models associated to groups \texorpdfstring{$\mathrm{Sp}_2(A_{\mathbb{C}}, \sigma_{\mathbb{C}})$}{}}\label{subsec:complex_defs}

We give next the description of Higgs bundle data for the quaternionic structures model and the projective space model of the symmetric space of groups $\mathrm{Sp}_2(A_{\mathbb{C}}, \sigma_{\mathbb{C}})$ which are the most interesting cases; descriptions for the rest two models given in Section \ref{sec:complex_models} can be obtained analogously. For $G=\mathrm{Sp}_2(A_{\mathbb{C}}, \sigma_{\mathbb{C}})$, the $G$-Higgs fields enjoy a geometric description as follows: 

\emph{Quaternionic structures model}.
    From the description of the complexified tangent space of the quaternionic structures model in \eqref{eq:complex_tang_quater_str_model}, a $G$-Higgs bundle can be described as: 
    \begin{itemize}
        \item an $A^2_{\mathbb{C}\{i\}\mathbb C\{I\}}$-bundle $V$ which is holomorphic with respect to $\mathbb C\{I\}$ equipped with:
        \begin{itemize}
            \item a symplectic structure $\omega$,
            \item an automorphism $J$ of $V$ which is $A_{\mathbb C\{I\}}$-linear and a quaternionic structure with respect to $i$ and the $\sigma^i$-sesquilinear form $h_J=\omega(J(\cdot),\cdot)$ is $\sigma^i$-symmetric and non-degenerate,
        \end{itemize}
        and
        \item a Higgs field $\varphi\in H^0(X,\mathrm{End}_{A_{\mathbb C\{I\}}}(V)\otimes K_X)$ such that 
        \begin{itemize}
            \item $\varphi(xa)=\varphi(x)\theta_i(a)$ for a local section $x$ of $V$ and $a\in A_{\mathbb{C}\{i\}\mathbb C\{I\}}$
            \item the $\sigma^i$-sesquilinear form $h_\varphi=\omega(\varphi(\cdot),\cdot)$ is $\sigma^i$-symmetric,  
            \item $\varphi J+J\varphi=0$.
        \end{itemize}
\end{itemize}

\emph{Projective space model.}
Using the projective space model for the group $G$, a $G$-Higgs bundle can be described as: 
\begin{itemize}
    \item a holomorphic $A^4_{\mathbb{C}}$-bundle $V$ equipped with two symplectic structures $\omega$ and $h$,
    which splits as a direct sum of two $\omega$-isotropic $A_{\mathbb{C}}^2$-subbundles, $\ell$ and $\ell'$ such that $h$ induces non-degenerate symplectic structures on each of them, and
    \item a Higgs field $\varphi\in H^0(X,\mathrm{Hom}_{A_\mathbb C}(\ell,\ell)\otimes K_X)$ such that $h(\varphi(x),y)+h(x,\varphi(y))=0$ for all local sections $x$ and $y$ of $\ell$. 
\end{itemize}
From this description, it is obvious that all information about $V$ is contained in the subbundle $\ell$. We can reformulate this description of a $G$-Higgs bundle as follows: 
\begin{itemize}
    \item a holomorphic $A_{\mathbb{C}}^2$-bundle $\ell$ equipped with a symplectic structure $h$, and
    \item a Higgs field $\varphi\in H^0(X,\mathrm{End}_{A_\mathbb C}\ell\otimes K_X)$ such that $h(\varphi(\cdot),\cdot)+h(\cdot,\varphi(\cdot))=0.$ 
\end{itemize}

\section{Connected components of the moduli spaces}\label{sec:components}

In this section we will be always considering that the involutive algebra $(A, \sigma)$ is Hermitian and semisimple. From the Artin--Wedderburn Theorem (see Section \ref{sec:incarnations}), it follows that
\begin{equation}\label{Artin_wed_A_1}
A \cong \bigoplus_{i=1}^p \mathrm{Mat}_{l_i}\mathbb{R} \oplus \bigoplus_{j=1}^q \mathrm{Mat}_{m_j}\mathbb{C} \oplus \bigoplus_{k=1}^r \mathrm{Mat}_{n_k}\mathbb{H}.
\end{equation}

\subsection{Moduli space \texorpdfstring{$\mathcal{M}_{\rm Higgs}(X,\mathrm{Sp}_2(A_\mathbb{C}, \sigma_{\mathbb{C}}))$}{for symplectic group}}

\begin{thm}\label{thm:comp_count_identity_Sp}
The moduli space $\mathcal{M}_{\rm Higgs}(X,\mathrm{Sp}_2(A_\mathbb{C}, \sigma_{\mathbb{C}})_0)$ has $2^r$-many connected components.
\end{thm}

\begin{proof}
In \cite[Theorem 0.1]{Li}, Li has shown that for a connected complex semisimple Lie group $G$, the number of connected components of the character variety $\mathcal{R}(G)$ is equal to $\vert\pi_1(G)\vert$. This result was later generalized by Ho--Liu for connected reductive complex Lie groups (cf. \cite[Appendix A]{LR15})). In particular, if $G$ is a connected reductive complex Lie group, then the number of the components of the character variety $\mathcal{R}(G)$ is equal to the cardinality of the fundamental group of the maximal semisimple subgroup of $G$. Via the nonabelian Hodge correspondence for such a complex Lie group $G$, we imply that the moduli space of $G$-Higgs bundles $\mathcal{M}_{\rm Higgs}(X,G)$ has the same number of connected components. 

\noindent Now for $G=\mathrm{Sp}_2(A_\mathbb{C}, \sigma_{\mathbb{C}})_0$, Theorem \ref{prop:max_semisimple_simply_con} provides directly that the moduli space 
\[\mathcal{M}_{\rm Higgs}(X,\mathrm{Sp}_2(A_\mathbb{C}, \sigma_{\mathbb{C}})_0)\]
has $2^r$-many connected components.
\end{proof}

We may now obtain component counts for various examples of the groups $\mathrm{Sp}_2(A_\mathbb{C}, \sigma_{\mathbb{C}}) $ (cf. Section \ref{sec:real_examples}):

\begin{cor}\label{cor:comp_counts_specific} We have the following: 
\begin{enumerate}
 \item The moduli space $\mathcal{M}_{\rm Higgs}(X,\mathrm{Sp}_{2n}(\mathbb{C}))$ is connected.
 \item The moduli space $\mathcal{M}_{\rm Higgs}(X,\mathrm{GL}_{2n}(\mathbb{C}))$ is connected.
 \item The moduli space $\mathcal{M}_{\rm Higgs}(X,\mathrm{SO}_{4n}(\mathbb{C}))$ has two connected components.
\end{enumerate}    
\end{cor}

\begin{proof}  
We show accordingly: 
\begin{enumerate}
\item  When $(A,\sigma)=(\mathrm{Mat}_n(\mathbb{R}),(\cdot)^{t})$, then the group $\mathrm{Sp}_2(A_\mathbb{C}, \sigma_{\mathbb{C}}) $ is the group $\mathrm{Sp}_{2n}(\mathbb{C})$. From \eqref{Artin_wed_A}
we deduce that for the algebra $A$ in this case, we have $p=1$ and $q=r=0$. Theorem \ref{thm:comp_count_identity_Sp} then provides that the moduli space $\mathcal{M}_{\rm Higgs}(X,\mathrm{Sp}_{2n}(\mathbb{C}))$ is connected.

\item Similarly, when $(A,\sigma)=(\mathrm{Mat}_n(\mathbb{C}),(\bar{\cdot})^{t})$, then the group $\mathrm{Sp}_2(A_\mathbb{C}, \sigma_{\mathbb{C}}) $ is the group $\mathrm{GL}_{2n}(\mathbb{C})$. From \eqref{Artin_wed_A}
we deduce that for the algebra $A$ in this case, we have $q=1$ and $p=r=0$. Theorem \ref{thm:comp_count_identity_Sp} then provides that the moduli space $\mathcal{M}_{\rm Higgs}(X,\mathrm{GL}_{2n}(\mathbb{C}))$ is connected.

\item Lastly, when $A=\mathrm{Mat}_n(\mathbb{H})$ and $\sigma(r_0+r_1j)=\bar r_0^t- r_1^tj$, for $r_0,r_1\in\mathrm{Mat}_n(\mathbb C)$, then the group $(\mathrm{Sp}_2(A_\mathbb{C}, \sigma_{\mathbb{C}}))_0 $ is the group $\mathrm{SO}_{4n}(\mathbb{C})$. From \eqref{Artin_wed_A}
we deduce that for the algebra $A$ in this case, we have $r=1$ and $p=q=0$. Theorem \ref{thm:comp_count_identity_Sp} then provides that the moduli space $\mathcal{M}_{\rm Higgs}(X,\mathrm{SO}_{4n}(\mathbb{C}))$ has two connected components.
\end{enumerate}

\end{proof}

\subsection{Moduli space \texorpdfstring{$\mathcal{M}_{\rm Higgs}(X,\mathrm{O}(A_\mathbb{C}, \sigma_{\mathbb{C}}))$}{for orthogonal group}}

We can also obtain a component count for the moduli space $\mathcal{M}_{\rm Higgs}(X,\mathrm{O}(A_\mathbb{C}, \sigma_{\mathbb{C}}))$, which provides slightly more general component counts than those of Corollary \ref{cor:comp_counts_specific}.

\begin{thm}\label{thm:comp_count_identity_O}
The moduli space $\mathcal{M}_{\rm Higgs}(X,\mathrm{O}(A_\mathbb{C}, \sigma_{\mathbb{C}})_0)$ has $2^{p_{\geq3}}$-many connected components, where $p_{\geq3}:=\#\{i\mid l_i\geq3\}$.
\end{thm}

\begin{proof}
Using the same methodology as in the proof of Theorem \ref{thm:comp_count_identity_Sp}, for the group $G=\mathrm{O}(A_\mathbb{C}, \sigma_{\mathbb{C}})_0$, Theorem \ref{prop:max_semisimple_simply_con_O} provides directly that the moduli space \[\mathcal{M}_{\rm Higgs}(X,\mathrm{O}(A_\mathbb{C}, \sigma_{\mathbb{C}})_0)\] has $2^{p_{\geq3}}$-many connected components.
\end{proof}

We may now obtain component counts for various examples of the groups $\mathrm{O}(A_\mathbb{C}, \sigma_{\mathbb{C}}) $ (cf. Section \ref{sec:real_examples}):

\begin{cor}\label{cor:comp_counts_specific_O} We have the following: 
\begin{enumerate}
 \item The moduli space $\mathcal{M}_{\rm Higgs}(X,\mathrm{SO}_n(\mathbb{C}))\ (n\geq3)$ has two connected components.
 \item The moduli space $\mathcal{M}_{\rm Higgs}(X,\mathrm{GL}_n(\mathbb{C}))$ is connected.
 \item The moduli space $\mathcal{M}_{\rm Higgs}(X,\mathrm{Sp}_{2n}(\mathbb{C}))$ is connected.
\end{enumerate}    
\end{cor}

\begin{proof}  
We show accordingly: 
\begin{enumerate}
\item  When $(A,\sigma)=(\mathrm{Mat}_n(\mathbb{R}),(\cdot)^{t})$, then the group $\mathrm{O}(A_\mathbb{C}, \sigma_{\mathbb{C}}) $ is the group $\mathrm{O}_n(\mathbb{C})$, and so $\mathrm{O}(A_\mathbb{C}, \sigma_{\mathbb{C}})_0 $ is the group $\mathrm{SO}_n(\mathbb{C})$. From \eqref{Artin_wed_A}
we deduce that for the algebra $A$ in this case, we have $p=1$ and $q=r=0$. Theorem \ref{thm:comp_count_identity_O} then provides that the moduli space $\mathcal{M}_{\rm Higgs}(X,\mathrm{SO}_n(\mathbb{C}))$ has two connected components when $n\geq3$.

\item Similarly, when $(A,\sigma)=(\mathrm{Mat}_n(\mathbb{C}),(\bar{\cdot})^{t})$, then the group $\mathrm{O}(A_\mathbb{C}, \sigma_{\mathbb{C}}) $ is the group $\mathrm{GL}_n(\mathbb{C})$. From \eqref{Artin_wed_A}
we deduce that for the algebra $A$ in this case, we have $q=1$ and $p=r=0$. Theorem \ref{thm:comp_count_identity_O} then provides that the moduli space $\mathcal{M}_{\rm Higgs}(X,\mathrm{GL}_n(\mathbb{C}))$ is connected.

\item Lastly, when $A=\mathrm{Mat}_n(\mathbb{H})$ and $\sigma(r_0+r_1j)=\bar r_0^t- r_1^tj$, for $r_0,r_1\in\mathrm{Mat}_n(\mathbb C)$, then the group $\mathrm{O}(A_\mathbb{C}, \sigma_{\mathbb{C}})_0 $ is the group $\mathrm{Sp}_{2n}(\mathbb{C})$. From \eqref{Artin_wed_A}
we deduce that for the algebra $A$ in this case, we have $r=1$ and $p=q=0$. Theorem \ref{thm:comp_count_identity_O} then provides that the moduli space $\mathcal{M}_{\rm Higgs}(X,\mathrm{Sp}_{2n}(\mathbb{C}))$ is connected.
\end{enumerate}
\end{proof}

\section{Hitchin morphism and factorization}\label{sec:Hitchin morphism}

Let $G$ be a real reductive Lie group with Lie algebra $\mathfrak g$, and let $K\subseteq G$  be a maximal compact subgroup with Cartan decomposition
\[
\mathfrak g=\mathfrak k\oplus\mathfrak m,
\]
then the affine GIT quotient $\mathfrak{m}^{\mathbb C}/\!\!/K^{\mathbb C}$ induced from the complexification of the isotropy representation $K\to\mathrm{GL}(\mathfrak{m})$ defines the Hitchin base
\[
\mathcal B(X,G):=H^0(X,\mathfrak{m}^{\mathbb C}/\!\!/K^{\mathbb C}\otimes K_X)\cong\bigoplus_{i=1}^{n} H^0(X,K_X^{e_i}),
\]
 where $e_1,\dots,e_n$ are the fundamental degrees of the invariant homogeneous polynomials that generate the invariant algebra $\mathbb{C}[\mathfrak{m}^{\mathbb C}]^{K^{\mathbb C}}$. By evaluating the basic $K^{\mathbb C}$-invariant polynomials on Higgs fields, one obtains the Hitchin morphism
\[
h_G:\mathcal M_{\mathrm{Higgs}}(X,G)\longrightarrow \mathcal B(X,G).
\]
When $G$ is a complex reductive group, regarded as a real reductive group, one has $K^{\mathbb C}\cong G$ and $\mathfrak{m}^{\mathbb C}\cong\mathfrak{g}$,  hence the identification $\mathfrak{m}^{\mathbb C}/\!\!/K^{\mathbb C}\cong\mathfrak g/\!\!/G$. Therefore, the above definition recovers the classical Hitchin base and the classical Hitchin morphism for complex groups. 

In this section, we reinterpret the Hitchin morphism using the geometric models of the corresponding symmetric spaces constructed in Part 1. We first treat the complex group
\[
\mathrm{Sp}_2(A_{\mathbb C},\sigma_{\mathbb C}),
\]
and then the real groups
\[
\mathrm{Sp}_2(A,\sigma) \qquad \text{and} \qquad \mathrm O_{(1,1)}(A,\sigma).
\]
The guiding principle is the same in all cases: rather than working only with the abstract invariant-theoretic quotient, we use concrete realizations of the corresponding representation arising from the tangent spaces of the geometric models. In these realizations, the relevant invariant polynomials are obtained from quadratic norm maps.

More explicitly, let $W_{\mathbb C}$ be a concrete complexified tangent model for interpreting the corresponding Higgs fields, and let $J_{\mathbb C}$ be the complexification of a formally real Jordan algebra naturally associated with the model. The geometric nature of $W_{\mathbb C}$ and $J_{\mathbb C}$ gives rise to a quadratic norm map
\[
\tilde{\mathcal N}_{\mathbb C}: W_{\mathbb{C}}\longrightarrow J_{\mathbb C},
\]
which is $K^{\mathbb{C}}$-equivariant. Let now $H^{\mathbb{C}}$ be the quotient of $K^{\mathbb{C}}$ by the kernel of the action of $K^{\mathbb{C}}$ on $J_{\mathbb{C}}$. In the cases considered below, the spaces $W_{\mathbb C}$, the Jordan algebras $J_{\mathbb C}$, the norm maps, and the groups $K^{\mathbb{C}}$ and $H^{\mathbb{C}}$ are as follows:
\[
\footnotesize
\renewcommand{\arraystretch}{2.1}
\begin{array}{|c|c|c|c|c|c|}
\hline
\text{Group} & W_{\mathbb C} & J_{\mathbb C} & \tilde{\mathcal N}_{\mathbb C}: W_{\mathbb{C}}\rightarrow J_{\mathbb C} & K^{\mathbb C} & H^{\mathbb C} \\
\hline
\mathrm{Sp}_2(A_{\mathbb C},\sigma_{\mathbb C})&A_{\mathbb H}^{\sigma_0}\otimes_{\mathbb R}\mathbb C\{I\}&A_{\mathbb H}^{\sigma_1}\otimes_{\mathbb R}\mathbb C\{I\}&a\longmapsto a\sigma_1(a)&\mathrm{Sp}_2(A_{\mathbb C},\sigma_{\mathbb C})&
\mathrm{Sp}_2(A_{\mathbb C},\sigma_{\mathbb C})
\\
\hline
\mathrm{Sp}_2(A,\sigma)&A_{\mathbb C}^{\sigma_{\mathbb C}}\oplus A_{\mathbb C}^{\sigma_{\mathbb C}}&A_{\mathbb C}&(q_+,q_-)\longmapsto q_+q_-&A_{\mathbb C}^{\times}&A_{\mathbb C}^{\times}
\\
\hline
\mathrm O_{(1,1)}(A,\sigma)&A_{\mathbb C}&A_{\mathbb C}^{\sigma_{\mathbb C}}&q\longmapsto \sigma_{\mathbb C}(q)q&\mathrm O(A_{\mathbb C},\sigma_{\mathbb C})\times\mathrm O(A_{\mathbb C},\sigma_{\mathbb C})&\mathrm O(A_{\mathbb C},\sigma_{\mathbb C})
\\
\hline
\end{array}
\]

Passing the equivariant norm map $\tilde{\mathcal N}_{\mathbb C}: W_{\mathbb{C}}\rightarrow J_{\mathbb C}$ to affine GIT quotients
gives a morphism
\[
\hat{\mathcal{N}}_{\mathbb C}: W_{\mathbb C}/\!\!/K^{\mathbb C}\longrightarrow J_{\mathbb C}/\!\!/H^{\mathbb C}.
\]

Now choose homogeneous generators
\[
p_1,\dots,p_n\in\mathbb C[J_{\mathbb C}]^{H^{\mathbb C}},\qquad \deg(p_i)=d_i.
\]
Then each composite $p_i\circ\widetilde{\mathcal N}_{\mathbb C}$ is a $K^{\mathbb{C}}$-invariant homogeneous polynomial on $W_{\mathbb{C}}$ of degree $2d_i$. In the cases considered below, the quotient identifications show that these composites generate the invariant algebra $\mathbb C[W_{\mathbb C}]^{K^{\mathbb C}}$. Hence, for a $G$-Higgs bundle $(E_{K^{\mathbb{C}}},\varphi)$, whenever the Higgs field is represented in the corresponding geometric model by a section
\[
\varphi\in H^0(X,E_{K^{\mathbb{C}}}(W_{\mathbb C})\otimes K_X),
\] 
we obtain sections
\[
(p_i\comp \tilde{\mathcal{N}}_{\mathbb{C}})(\varphi)\in H^0(X,K_X^{2d_i}).
\]
Therefore, the factorization through the affine quotient $J_{\mathbb C}/\!\!/H^{\mathbb C}$ naturally produces a map
\[
\mathcal M_{\mathrm{Higgs}}(X,G)\longrightarrow \bigoplus_{i=1}^{n} H^0(X,K_X^{2d_i}).
\]
The main point of this section is that, for the groups considered here, this target agrees with the corresponding Hitchin base:
\[
\bigoplus_{i=1}^{n} H^0(X,K_X^{2d_i}) \cong \mathcal B(X,G)=\bigoplus_{i=1}^{n} H^0(X,K_X^{e_i}).
\]

We first illustrate this mechanism in the basic case $\mathrm{SL}_2(\mathbb{C})\cong\mathrm{Sp}_2(\mathbb{C})$, then treat the general complex group 
\[
\mathrm{Sp}_2(A_{\mathbb C},\sigma_{\mathbb C}),
\]
and finally treat the real groups
\[
\mathrm{Sp}_2(A,\sigma)\qquad \text{and} \qquad \mathrm O_{(1,1)}(A,\sigma)
\]
in the different geometric incarnations of their symmetric spaces.

\subsection{The case \texorpdfstring{$G=\mathrm{SL}_2( \mathbb{C})$}{G=SL(2,C)}}
The complex Lie group $G=\mathrm{SL}_2( \mathbb{C})$ can be viewed as a symplectic group $\mathrm{Sp}_2(A_{\mathbb{C}}, \sigma_{\mathbb{C}})$ for $A=\mathbb{R}$ and $\sigma = \mathrm{Id}$. Then, $K= \mathrm{Sp}(1) \cong  \mathrm{SU}(2)$ and $\mathfrak{m}=\mathbb{H}^{\sigma_0}$ with complexification $\mathfrak{m}^{\mathbb{C}}= \mathbb{H}^{\sigma_0} \otimes \mathbb{C}$. There is a map 
\begin{align}\label{map_square_norm}
\mathbb{H}^{\sigma_0} & \to \mathbb{R}^+ \\
q & \mapsto \vert q\vert ^2=q\bar{q} \nonumber
\end{align}
defined as the square of the norm of a quaternionic element. This map is preserved by the action of $\mathrm{SU}(2)$ on $\mathbb{H}^{\sigma_0}$. Therefore, the map above is invariant for this action. Now take the complex extension of the map (\ref{map_square_norm}) given by 
\[\mathbb{H}^{\sigma_0} \otimes_{\mathbb{R}} \mathbb{C}  \to \mathbb{C}\]
and this is invariant under the action of $K^{\mathbb{C}}$, the complexification of $\mathrm{SU}(2)$. 

\vspace{2mm}

\noindent \textit{Quaternionic structures model.} In this model, an $\mathrm{SL}_2( \mathbb{C})$-Higgs field is given by an element $L \in T_J^{\mathbb{C}}\mathfrak{C}_G$ with 
\[JL +LJ =0.\]
Since $L \in \mathrm{Mat}_2(\mathbb{C})$, the coefficient of its invariant polynomial is given by the trace $\mathrm{Tr}(L^2)$. Therefore, in this model, the Hitchin base is no different than the classical description.

\vspace{2mm}

\noindent \textit{Projective space model.} Now we consider $\mathrm{SL}_2(\mathbb{C})$-Higgs bundles in the projective space model. For $l\in \mathfrak{P}_G$ and $Q\in T_l\mathfrak{P}_G$, it is $Q \colon  l \to \theta_j(l)$. The map
    $$Q':=\theta_j\circ Q\circ\theta_j \colon  \theta_j(l) \mapsto l$$
is $\mathbb H$-linear. The pair of lines $(l,\theta_j(l))$ give rise to a symplectic basis $(v,v':=\theta_j(v)j)$ of $\mathbb H^2$ such that $v$ generates $l$ and $v'$ generates $\theta_j(l)$. Then $Q(v)=v'a$ and $Q'(v')=v\theta_i(\theta_j(a))=v\bar a$, for $a\in \mathbb H^{\sigma_0}$, where $\bar a$ is a quaternionic conjugate of $a$. Consider the operator 
\[Q' \circ Q \colon l \mapsto l.\]
Its operator norm 
\[\Vert Q' \circ Q \Vert :=\sup_{x\in l,\,\Vert x \Vert =1 }\Vert (Q' \circ Q)(x)\Vert=\bar a a\in\mathbb R_{\geq 0}\]
does not depend on the choice of the symplectic basis $(v,v')$, in other words, this operator norm is invariant under the action of $\mathrm{SL}_2(\mathbb C)$ on the tangent bundle. In particular, for every $l\in \mathfrak{P}_G$, this is an invariant of the action of $\mathrm{Stab}_{\mathrm{SL}_2(\mathbb C)}(l)\cong\mathrm{SU}(2)$.

Now for $Q\in T_l^\mathbb C\mathfrak{P}_G$, it is $Q \colon  l\otimes \mathbb C\{I\} \mapsto \theta_j(l)\otimes \mathbb C\{I\}$ and $Q':=\theta_j\circ Q\circ\theta_j$. In a symplectic basis $(v,v':=\theta_j(v)j)$ of $\mathbb H^2$ such that $v$ generates $l$ and $v'$ generates $\theta_j(l)$, $Q(v)=v'a$ and $Q'(v')=v\theta_i(\theta_j(a))$, for $a\in \mathbb H^{\sigma_0}\otimes\mathbb C\{I\}$. The assignment
$$\begin{array}{rccl}
\mathcal N_\mathbb C\colon &T^{\mathbb{C}}\mathfrak{P}_G^+ & \to & \mathbb C\{I\}\\
& (l,Q) &\mapsto &\theta_i(\theta_j(a))a
\end{array}$$
gives an invariant under the action of $\mathrm{SL}_2(\mathbb C)$ seen as the complexification of $\mathrm{Stab}_{\mathrm{SL}_2(\mathbb C)}(l)$.

Then this complexified norm map is invariant under the action of $K^{\mathbb{C}}$ by symplectic transformations. Therefore, we get a map
\[
f^{*}_{\mathfrak{P}_G^+}T^{\mathbb{C}}\mathfrak{P}_G^+ \otimes K_X \to H^0(X, K^2_X)\]
which maps a point $(E_{K^{\mathbb{C}}}, \Phi=(l,Q)dz)$ in $\mathcal{M}_{\mathrm{Higgs}}(X,\mathrm{SL}_2(\mathbb{C}))$ to $\mathcal N_\mathbb C(l,Q) dz^2$. 

\vspace{2mm}

\noindent \textit{Upper half-space and precompact model.} We consider an $\mathrm{SL}_2( \mathbb{C})$-Higgs bundle in the upper half-space model or in the precompact model:
\[\mathfrak{U}_G \cong  \mathfrak{B}_G.\]
In these two cases, we have 
\[T^{\mathbb{C}}\mathfrak{U} \cong  T^{\mathbb{C}}\mathfrak{B}_G \cong  A_{\mathbb{H}}^{\sigma_0} \otimes_{\mathbb{R}} \mathbb{C} \cong  \mathbb{H}^{\sigma_0} \otimes_{\mathbb{R}} \mathbb{C}.\]
Moreover, the Higgs field $\Phi$ is a $(1,0)$-form with values in 
\[f^{*}_{\mathfrak{U}}T^{\mathbb{C}}\mathfrak{U} \cong  f^{*}_{\mathfrak{B}_G}T^{\mathbb{C}}\mathfrak{B}_G,\]
which is an adjoint $\mathrm{SU}(2)$-bundle.

We next define a morphism $\mathcal N_\mathbb C\colon T^{\mathbb{C}}\mathfrak{U}\to \mathbb C$. For $(z,v)\in T^{\mathbb{C}}\mathfrak{U}$, there exists $g\in\mathrm{SL}_2(\mathbb C)$ such that $g.(z,v)= (j,v_0)$, for $v_0\in \mathbb{H}^{\sigma_0} \otimes_{\mathbb{R}} \mathbb{C}\{I\}$. Let $\mathcal N_\mathbb C(z,v):=\theta_i(\theta_j(v_0))v_0\in\mathbb C$. This is invariant under the action of the complexified 
$\mathrm{Stab}_{\mathrm{SL}_2(\mathbb C)}(z)$ on $T_z^{\mathbb{C}}\mathfrak{U}$ which is isomorphic to $\mathrm{SL}_2(\mathbb C)$ which is the complexified stabilizer of the point $j$.

The construction for $\mathfrak{B}_G$ is similar. The only difference is that we use the point $0\in\mathfrak{B}_G$ instead of $j\in\mathfrak{U}_G$.

The morphism $\mathcal N_\mathbb C$ introduced above extends to the map 
\[f^{*}_{\mathfrak{U}}T^{\mathbb{C}}\mathfrak{U} \otimes K_X \to H^0(X, K^2_X),\]
which maps a point $(E_{K^{\mathbb{C}}}, \Phi=(z,v)dz)$ in $\mathcal{M}_{\mathrm{Higgs}}(X,\mathrm{SL}_2(\mathbb{C}))$ to $\mathcal N_\mathbb C(z,v) dz^2$. This defines the Hitchin morphism for these two models of the symmetric space.

\subsection{The general case \texorpdfstring{$G=\mathrm{Sp}_2(A_{\mathbb{C}}, \sigma_{\mathbb{C}})$}{}}\label{sec:Hitchin_morphism_Sp_general}
We now give the description of the Hitchin base for general complex symplectic groups, motivated by the descriptions in the case of $\mathrm{SL}_2( \mathbb{C})$ above. 

Before we start discussing the Hitchin morphism in different models, we introduce and discuss some properties of the following $\mathrm{KSp}^c_2(A_\mathbb C,\sigma_{\mathbb{C}})$-equivariant map:
\begin{align*}
\tilde{\mathcal{N}}_\mathbb R\colon A_{\mathbb{H}}^{\sigma_0} & \to (A_{\mathbb{H}}^{\sigma_1})_{\geq 0}\\
q & \mapsto q\sigma_1(q),
\end{align*}
where $\mathrm{KSp}^c_2(A_\mathbb C,\sigma_{\mathbb{C}})$  acts on $A^{\sigma_1}_\mathbb H$ by conjugation preserving $(A_{\mathbb{H}}^{\sigma_1})_{\geq 0}$, and on $A^{\sigma_0}_\mathbb H$ by $\sigma_0$-congruence. More precisely, for $k=\begin{pmatrix}
a & b\\
-\bar{b} & \bar{a}
\end{pmatrix}\in \mathrm{KSp}^c_2(A_\mathbb C,\sigma_{\mathbb{C}})$, which we can also identify with the $a+bj\in A_\mathbb H$ such that  $\bar{\sigma}_{\mathbb{C}}(a)a+\sigma_{\mathbb{C}}(b)\bar{b}=1 \text{ and } \bar{\sigma}_{\mathbb{C}}(a)b-\sigma_{\mathbb{C}}(b)\bar{a}=0, \text{ for } a,b\in A_{\mathbb{C}}$, the action on $A_{\mathbb{H}}^{\sigma_0}$ is given by 
\begin{equation}\label{eq:action_on_A^sigma0}
    k.q = kq\sigma_0(k)=(a+bj)q(-\bar{b}j+\bar{a})^{-1}
\end{equation}
for $q\in A_{\mathbb{H}}^{\sigma_0}$. 
Moreover, it is 
\begin{align}\label{eq:invariance}
\begin{aligned}
 k.q\,\sigma_1(k.q) & = (a+bj)q\left( \sigma_{\mathbb{C}}(a)+\bar{\sigma}_{\mathbb{C}}(b)j\right)(\bar{a} - \bar{b}j) \sigma_1(q) \left( \bar{\sigma}_{\mathbb{C}}(a)-\sigma_{\mathbb{C}}(b)j\right)\\
 & = (a+bj)q\sigma_1(q) \left( \bar{\sigma}_{\mathbb{C}}(a)-\sigma_{\mathbb{C}}(b)j\right)\\
 & =(a+bj)q\sigma_1(q) (a+bj)^{-1}.
 \end{aligned}
\end{align}
This implies that the map is indeed equivariant. 
The complexification of this map then produces the map 
\begin{align*}
\tilde{\mathcal{N}}_\mathbb C\colon A_{\mathbb{H}}^{\sigma_0} \otimes_{\mathbb{R}} \mathbb{C}\{I\}& \to A_{\mathbb{H}}^{\sigma_1} \otimes_{\mathbb{R}} \mathbb{C}\{I\}\\
q & \mapsto q\sigma_1(q).
\end{align*}

We next introduce the quotient map:
\begin{equation}\label{eq:real_norm}
\begin{array}{rccl}
\mathcal N_\mathbb R\colon & A^{\sigma_0}_\mathbb H & \to & A^{\sigma_1}_\mathbb H/\mathrm{KSp}^c_2(A_\mathbb C,\sigma_\mathbb C)\\
& q & \mapsto & [q\sigma_1(q)].
\end{array}    
\end{equation}
After the complexification of $A_\mathbb H$, we can define the corresponding complexified map:
\begin{equation}\label{eq:complex_norm}
\begin{array}{rrcl}
\mathcal N_\mathbb C\colon & A^{\sigma_0}_\mathbb H\otimes_\mathbb R\mathbb C\{I\} & \to & A^{\sigma_1}_\mathbb H\otimes_\mathbb R \mathbb C\{I\}/\!\!/\mathrm{Sp}_2(A_\mathbb C,\sigma_\mathbb C)\\
& q & \mapsto & [q\sigma_1(q)],
\end{array}    
\end{equation}
where we, slightly abusing the notation, denote the $\mathbb C\{I\}$-linear extension of $\sigma_1$ again by $\sigma_1$ and $\mathrm{Sp}_2(A_\mathbb C,\sigma_\mathbb C)$ is seen as the complexification of $\mathrm{KSp}^c_2(A_\mathbb C,\sigma_\mathbb C)$. This map is indeed the quotient map of $\tilde{\mathcal{N}}_\mathbb C$ by the group $\mathrm{Sp}_2(A_\mathbb C,\sigma_\mathbb C)$. This action is, as before, invariant under $\sigma_0$-congruence where we, also slightly abusing the notation, denote the $\mathbb C\{I\}$-linear extension of $\sigma_0$ again by $\sigma_0$. Notice that because the group is not compact, the affine GIT-quotient is needed to provide a quotient which is Hausdorff and, more generally, an algebraic variety. The map $\mathcal N_\mathbb C$ is known as the affine GIT morphism and, furthermore, by the Chevalley restriction theorem, this GIT-quotient can be identified with $\mathfrak{a}^\mathbb C/W(\mathfrak a^\mathbb C)$ where $\mathfrak{a}^\mathbb C$ is a maximal (abelian) subalgebra in $\mathfrak{sp}_2(A_{\mathbb{C}},\sigma_{\mathbb{C}})$, and $W(\mathfrak a^\mathbb C)$ is the Weyl group associated to $\mathfrak a^\mathbb C$ (cf.~\cite[Section 4]{GPR}). 

The invariants of elements in $A^{\sigma_1}_\mathbb H\otimes_\mathbb R \mathbb C\{I\}$ under the action of $\mathrm{Sp}_2(A_\mathbb C,\sigma_\mathbb C)$ are given by coefficients of the characteristic polynomials which are homogeneous polynomials of degrees $d \in \{d_1, ..., d_n\}$ of the complexification of the formally real Jordan algebra $A_{\mathbb{H}}^{\sigma_1}$. Thus, for each $d$ there is a map defined by the polynomial of degree $d$
\[\chi_d\colon A_{\mathbb{H}}^{\sigma_1}\otimes_{\mathbb{R}}\mathbb{C}\{I\} \to \mathbb{C}\{I\},\]
\noindent where the map $\chi_d$ maps each conjugacy class of elements in $A^{\sigma_1}_\mathbb H\otimes_\mathbb R \mathbb C\{I\}$ to the $d$-coefficient of its characteristic polynomial.
The composition of this map $\chi_d$ with the map $\mathcal{N}_{\mathbb{C}}$ defined in \eqref{eq:complex_norm} then gives a map \[\mathcal{N}_{\mathbb{C}}^d\colon A_{\mathbb{H}}^{\sigma_0}\otimes_{\mathbb{R}}\mathbb{C}\{I\} \to \mathbb{C}\{I\},\]
thus fitting into the following diagram
\begin{equation}\label{eq:factorization_char_pol}
    \xymatrix{
&A^{\sigma_1}_\mathbb H\otimes_\mathbb R \mathbb C\{I\}/\!\!/\mathrm{Sp}_2(A_\mathbb C,\sigma_\mathbb C)\ar[dd]^{\chi_d}\\
&&\\
A_{\mathbb{H}}^{\sigma_0} \otimes_{\mathbb{R}} \mathbb{C}\{I\}\ar[uur]^{\mathcal N_\mathbb C}\ar[r]^{\mathcal{N}_{\mathbb{C}}^d}&\mathbb{C}\{I\}.
}
\end{equation}

At the level of moduli spaces of $\mathrm{Sp}_2(A_\mathbb C,\sigma_\mathbb C)$-Higgs bundles, the diagram is respectively of the form 
\begin{equation}\label{hit_mor_fac_sp_general}
    \xymatrix{
&H^0(X,A^{\sigma_1}_\mathbb H\otimes_\mathbb R \mathbb C\{I\}/\!\!/\mathrm{Sp}_2(A_\mathbb C,\sigma_\mathbb C)\otimes K_X^2)\ar[dd]^{\bigoplus\limits_{i=1}^n\chi_{d_i}}\\
&&\\
\mathcal{M}_{\mathrm{Higgs}}(X,\mathrm{Sp}_2(A_{\mathbb{C}}, \sigma_{\mathbb{C}}))\ar[uur]^{\mathcal N_\mathbb C^{\bullet}}\ar[r]^{\hat{\mathcal{N}}_{\mathbb{C}}}&\bigoplus\limits_{i=1}^{n} H^0(X, K_X^{2d_i}),
}
\end{equation}
where $\mathcal N_\mathbb C^{\bullet}$ denotes the map at the level of $\mathrm{Sp}_2(A_\mathbb C,\sigma_\mathbb C)$-Higgs bundles induced by \eqref{eq:complex_norm}, for $\bullet$ a choice of model for the symmetric space $\mathcal{X}$ of $\mathrm{Sp}_2(A_\mathbb C,\sigma_\mathbb C)$. Note that the map $\hat{\mathcal N}_\mathbb C$ is the usual Hitchin morphism. We describe the map $\mathcal N_\mathbb C^{\bullet}$ more explicitly below for each of the four models of the symmetric space. 

\vspace{2mm}

\noindent \textit{Quaternionic structures model.} As seen above in the special case of $\mathrm{SL}_2( \mathbb{C})$ in this model, the coefficient of the invariant polynomial of an element $L$ in the complexified tangent $T^{\mathbb{C}}_J\mathfrak{C}_G$ is just given by the trace $\mathrm{Tr}(L^2)$. For the case when the algebra $A$ is a matrix algebra, we can define analogously a map from general $\mathrm{Sp}_2(A_{\mathbb{C}}, \sigma_{\mathbb{C}})$-Higgs bundles to the invariant polynomials of their Higgs field viewed as an element in $T^{\mathbb{C}}_J\mathfrak{C}_G$ given by the traces of all powers of a point $L$ viewed as a $(2 \times 2)$-matrix over $A_{\mathbb{C}}$. Note that the requirement that $A$ is a matrix algebra is not necessary for the realization of the Hitchin $\mathrm{Sp}_2(A_{\mathbb{C}}, \sigma_{\mathbb{C}})$-morphism in the other models.

\vspace{2mm}

\noindent \textit{Projective space model.} 
Now we consider the general case of $\mathrm{Sp}_2(A_\mathbb{C},\sigma_\mathbb C)$-Higgs bundles in the projective space model. An element $Q\in T_l\mathfrak{P}_G$ of the tangent space at $l\in \mathfrak{P}_G$ can be seen as an $A_\mathbb H$-linear map $Q \colon  l \mapsto \theta_j(l)$. The map
    $$Q':=\theta_j\circ Q\circ\theta_j \colon  \theta_j(l) \mapsto l$$
is $A_\mathbb H$-linear. The pair of lines $(l,\theta_j(l))$ give rise to a symplectic basis $(v,v':=\theta_j(v)j)$ of $A_\mathbb H^2$ such that $v$ generates $l$ and $v'$ generates $\theta_j(l)$. Then $Q(v)=v'a$ and $Q'(v')=v\theta_i(\theta_j(a))=v\sigma_1(a)$, for $a\in A_\mathbb H^{\sigma_0}$. 
We define an $A_\mathbb H^{\sigma_1}$-valued map which is an $A_\mathbb H^{\sigma_1}$-valued norm of $Q' \circ Q$: 
\[\tilde{\mathcal N}_\mathbb R^\mathfrak P(Q) :=a\sigma_1(a)\in (A_\mathbb H^{\sigma_1})_{\geq 0}.\]
This is equivariant under the actions of $\mathrm{KSp}^c_2(A_\mathbb C,\sigma_\mathbb C)$ on $A_\mathbb H^{\sigma_1}$ by conjugation and 
as 
\[\mathrm{Stab}_{\mathrm{Sp}_2(A_\mathbb C,\sigma_\mathbb C)}(l)=g\mathrm{KSp}^c_2(A_\mathbb C,\sigma_\mathbb C)g^{-1}\] on $T_l\mathfrak{P}_G$, for some fixed $g\in \mathrm{Sp}_2(A_\mathbb C,\sigma_\mathbb C)$ such that $g((j,1)^tA_\mathbb H)=l$. Thus, it descends (cf.~\eqref{eq:real_norm}) to the map \[{\mathcal N}_\mathbb R^\mathfrak P\colon T_l \mathfrak{P}_G\to A^{\sigma_1}_\mathbb H/\mathrm{KSp}^c_2(A_\mathbb C,\sigma_{\mathbb C}).\]

Now, for $Q\in T_l^\mathbb C\mathfrak{P}_G$, it is $Q \colon  l\otimes \mathbb C\{I\} \mapsto \theta_j(l)\otimes \mathbb C\{I\}$ and $Q':=\theta_j\circ Q\circ\theta_j$. In a symplectic basis $(v,v':=\theta_j(v)j)$ of $A_\mathbb H^2$ such that $v$ generates $l$ and $v'$ generates $\theta_j(l)$, $Q(v)=v'a$ and $Q'(v')=v\theta_i(\theta_j(a))=v\sigma_1(a)$, for $a\in A_\mathbb H^{\sigma_0}\otimes\mathbb C\{I\}$. The assignment
$$\begin{array}{rccl}
\mathcal N_\mathbb C^{\mathfrak P}\colon &T^{\mathbb{C}}\mathfrak{P}_G & \to & (A_\mathbb H^{\sigma_1})\otimes_\mathbb R\mathbb C\{I\}/\!\!/\mathrm{Sp}_2(A_\mathbb C,\sigma_\mathbb C)\\
& (l,Q) &\mapsto & [a\sigma_1(a)]
\end{array}$$
gives an invariant under the action of $\mathrm{Sp}_2(A_\mathbb C,\sigma_\mathbb C)$. Tensoring $T^{\mathbb{C}}\mathfrak{P}_G$ with $K_X$, we obtain the map
$$
T^{\mathbb{C}}\mathfrak{P}_G\otimes K_X  \to  (A_\mathbb H^{\sigma_1})\otimes_\mathbb R\mathbb C\{I\}/\!\!/\mathrm{Sp}_2(A_\mathbb C,\sigma_\mathbb C)\otimes K_X^2,
$$
such  a map is still denoted as $\mathcal N_\mathbb C^{\mathfrak P}$ when there is no ambiguity. Applying $\chi_d$ to $(A_\mathbb H^{\sigma_1})\otimes_\mathbb R\mathbb C\{I\}/\!\!/\mathrm{Sp}_2(A_\mathbb C,\sigma_\mathbb C)\otimes K_X^2$ we get
$$
(A_\mathbb H^{\sigma_1})\otimes_\mathbb R\mathbb C\{I\}/\!\!/\mathrm{Sp}_2(A_\mathbb C,\sigma_\mathbb C)\otimes K_X^2 \longrightarrow K_X^{2d},
$$
which we still employ $\chi_d$ to denote. Thus, we can consider the composition map
$$
\mathcal N_\mathbb C^{\mathfrak P,d}:=\chi_d\comp\mathcal N_\mathbb C^{\mathfrak P}: T^{\mathbb{C}}\mathfrak{P}_G\otimes K_X  \longrightarrow  K_X^{2d}.
$$
Therefore, for all $d\in\{d_1,\dots,d_n\}$ as above, we get a map
\[
\hat{\mathcal N}^d_\mathbb C\colon \mathcal{M}_{\mathrm{Higgs}}(X,\mathrm{Sp}_2(A_{\mathbb{C}}, \sigma_{\mathbb{C}})) \to H^0(X, K^{2d}_X),\]
which maps a point $(E_{K^{\mathbb{C}}}, \Phi=(l,Q)dz)$ in $\mathcal{M}_{\mathrm{Higgs}}(X,\mathrm{Sp}_2(A_\mathbb{C},\sigma_\mathbb C))$ to $\mathcal N_\mathbb C^{\mathfrak P,d}((l,Q)dz)$. 
This then finally induces a map
\begin{align*}
 \hat{\mathcal{N}}_{\mathbb{C}}\colon  \mathcal{M}_{\mathrm{Higgs}}(X,\mathrm{Sp}_2(A_{\mathbb{C}}, \sigma_{\mathbb{C}})) & \to \bigoplus_{i=1}^{n} H^0(X, K_X^{2d_i}). \\
 \end{align*}
This map factors as in the diagram \eqref{hit_mor_fac_sp_general}, for $\bullet = \mathfrak{P}.$

\vspace{2mm}

\noindent \textit{Upper half-space and precompact model.} In this case, it is 
$T^{\mathbb{C}}_\bullet\mathfrak{U}_G \cong  T^{\mathbb{C}}_\bullet\mathfrak{B}_G \cong  A_{\mathbb{H}}^{\sigma_0} \otimes_{\mathbb{R}} \mathbb{C}.$ We define the map 
\begin{align*}
\mathcal N_\mathbb C^{\mathfrak U}\colon T^{\mathbb{C}}\mathfrak{U}\xrightarrow{q_1} (A_\mathbb H^{\sigma_0})\otimes_\mathbb R\mathbb C\{I\}/\mathrm{KSp}^c_2(A_\mathbb C,\sigma_\mathbb C) & \xrightarrow{q_2} (A_\mathbb H^{\sigma_1})\otimes_\mathbb R\mathbb C\{I\}/\mathrm{KSp}^c_2(A_\mathbb C,\sigma_\mathbb C) \\ & \xrightarrow{q_3}(A_\mathbb H^{\sigma_1})\otimes_\mathbb R\mathbb C\{I\}/\!\!/\mathrm{Sp}_2(A_\mathbb C,\sigma_\mathbb C)
\end{align*}
as a composition of the following maps: For $(z,v)\in T^{\mathbb{C}}\mathfrak{U}$, there exists $g\in\mathrm{Sp}_2(A_\mathbb C,\sigma_\mathbb C)$ such that $g.(z,v)= (j,v_0)$, for $v_0\in A_\mathbb{H}^{\sigma_0} \otimes_{\mathbb{R}} \mathbb{C}\{I\}$. The element $g$ is uniquely defined up to the left multiplication by an element of $\mathrm{KSp}^c_2(A_\mathbb C,\sigma_\mathbb C)=\mathrm{Stab}_{\mathrm{Sp}_2(A_\mathbb C,\sigma_\mathbb C)}(j)$. Thus, we can identify 
$$T^{\mathbb{C}}\mathfrak{U}/\mathrm{Sp}_2(A_\mathbb C,\sigma_\mathbb C)\cong A_\mathbb{H}^{\sigma_0} \otimes_{\mathbb{R}} \mathbb{C}\{I\}/\mathrm{KSp}^c_2(A_\mathbb C,\sigma_\mathbb C),$$
where the action is given by~\eqref{eq:action_on_A^sigma0}, and the map $q_1$ is defined as the quotient map:
$$q_1\colon T^{\mathbb{C}}\mathfrak{U}\to A_\mathbb{H}^{\sigma_0} \otimes_{\mathbb{R}} \mathbb{C}\{I\}/\mathrm{KSp}^c_2(A_\mathbb C,\sigma_\mathbb C)$$
mapping $(z,v)$ to the class $\mathrm{KSp}_2(A_\mathbb C,\sigma_\mathbb C).v_0$. As the map $a\mapsto a\sigma_1(a)$ maps $A_\mathbb{H}^{\sigma_0} \otimes_{\mathbb{R}} \mathbb{C}\{I\}$ to $A_\mathbb{H}^{\sigma_1} \otimes_{\mathbb{R}} \mathbb{C}\{I\}$ and by~\eqref{eq:invariance}, the following map is well defined:
$$q_2\colon A_\mathbb{H}^{\sigma_0} \otimes_{\mathbb{R}} \mathbb{C}\{I\}/\mathrm{KSp}^c_2(A_\mathbb C,\sigma_\mathbb C)\to A_\mathbb{H}^{\sigma_1} \otimes_{\mathbb{R}} \mathbb{C}\{I\}/\mathrm{KSp}^c_2(A_\mathbb C,\sigma_\mathbb C).$$
The action of $\mathrm{KSp}^c_2(A_\mathbb C,\sigma_\mathbb C)$ on $A_\mathbb{H}^{\sigma_1} \otimes_{\mathbb{R}} \mathbb{C}\{I\}$ is given by $(k,v_0)\mapsto kv_0\sigma_1(k)$ where $v_0\in A_\mathbb{H}^{\sigma_1} \otimes_{\mathbb{R}} \mathbb{C}\{I\}$ and $k\in \mathrm{O}(A_\mathbb H,\sigma_1)\cong \mathrm{KSp}^c_2(A_\mathbb C,\sigma_\mathbb C)$ (cf. identification from Section~\ref{sec:models for compact}).

Finally, since $\mathrm{Sp}_2(A_\mathbb C,\sigma_\mathbb C)$ is the complexification of $\mathrm{KSp}^c_2(A_\mathbb C,\sigma_\mathbb C)$, and $\mathrm{KSp}^c_2(A_\mathbb C,\sigma_\mathbb C)$ acts on $A_\mathbb{H}^{\sigma_1}$, this action extends to the complexified spaces, i.e. $\mathrm{Sp}_2(A_\mathbb C,\sigma_\mathbb C)$ acts on $A_\mathbb{H}^{\sigma_1} \otimes_{\mathbb{R}} \mathbb{C}\{I\}$. Moreover, as a subgroup $\mathrm{KSp}^c_2(A_\mathbb C,\sigma_\mathbb C)$ acts on $A_\mathbb{H}^{\sigma_1} \otimes_{\mathbb{R}} \mathbb{C}\{I\}$ and this action agrees with the action described above. Thus, the following quotient map is well-defined: 
$$q_3\colon A_\mathbb{H}^{\sigma_1} \otimes_{\mathbb{R}} \mathbb{C}\{I\}/\mathrm{KSp}^c_2(A_\mathbb C,\sigma_\mathbb C)\to A_\mathbb{H}^{\sigma_1} \otimes_{\mathbb{R}} \mathbb{C}\{I\}/\!\!/\mathrm{Sp}_2(A_\mathbb C,\sigma_\mathbb C).$$

Let 
$$
\mathcal N^{\mathfrak U}_\mathbb C(z,v):=[v_0\sigma_1(v_0)]=q_3(q_2(q_1(z,v)))\in A_\mathbb{H}^{\sigma_1}\otimes_{\mathbb{R}}\mathbb C\{I\}/\!\!/\mathrm{Sp}_2(A_\mathbb C,\sigma_\mathbb C).
$$ 
This is an invariant under the action of the complexified 
$\mathrm{Stab}_{\mathrm{Sp}_2(A_\mathbb C,\sigma_\mathbb C)}(z)$ on $T_z^{\mathbb{C}}\mathfrak{U}$ which is isomorphic to $\mathrm{Sp}_2(A_\mathbb C,\sigma_\mathbb C)$ which is the complexified stabilizer of the point $j$.

The construction for $\mathfrak{B}_G$ is similar. The only difference is that we use the point $0\in\mathfrak{B}_G$ instead of $j\in\mathfrak{U}_G$.

The invariants of elements in $A^{\sigma_1}_\mathbb H\otimes_\mathbb R \mathbb C\{I\}$ under the action of $\mathrm{Sp}_2(A_\mathbb C,\sigma_{\mathbb{C}})$ are given by coefficients of the characteristic polynomials which are homogeneous polynomials of degrees $d \in \{d_1, ..., d_n\}$ of the complexification of the formally real Jordan algebra $A_{\mathbb{H}}^{\sigma_1}$. Thus, for each $d$ there is a map defined by the polynomial of degree $d$:
\[\mathcal{N}_{\mathbb{C}}^d\colon A_{\mathbb{H}}^{\sigma_0}\otimes_{\mathbb{R}}\mathbb{C} \to \mathbb{C}\]
and now this map for each $d$ induces the map
\begin{align*}
 \hat{\mathcal{N}}_{\mathbb{C}}^d\colon  \mathcal{M}_{\mathrm{Higgs}}(X,\mathrm{Sp}_2(A_{\mathbb{C}}, \sigma_{\mathbb{C}})) & \to H^0(X, K_X^{2d}). 
\end{align*}
This then finally induces a Hitchin morphism
\begin{align*}
 \hat{\mathcal{N}}_{\mathbb{C}}\colon  \mathcal{M}_{\mathrm{Higgs}}(X,\mathrm{Sp}_2(A_{\mathbb{C}}, \sigma_{\mathbb{C}})) & \to \bigoplus_{i=1}^{n} H^0(X, K_X^{2d_i}).
\end{align*}

Similar to the description of the projective space model, we have that the map $\mathcal{N}_{\mathbb{C}}^d$ factors as in~\eqref{eq:factorization_char_pol}
and the Hitchin morphism factors as in \eqref{hit_mor_fac_sp_general}.

Lastly, for a polystable  $\mathrm{Sp}_2(A_\mathbb C,\sigma_\mathbb C)$-Higgs bundle  $(E,\varphi)$ in the upper-half space model (not a gauge equivalence class), we see that the Higgs field $\varphi\in f_\mathfrak U^*(T^\mathbb C\mathfrak U)\otimes K_X$, can be associated to the following two sections: 
\begin{itemize}
\item $\varphi'\in H^0(X,(A_\mathbb{H}^{\sigma_0} \otimes_{\mathbb{R}} \mathbb{C}\{I\}/\mathrm{KSp}^c_2(A_\mathbb C,\sigma_\mathbb C))\otimes K_X)$,
\item $\varphi''\in H^0(X,(A_\mathbb{H}^{\sigma_1} \otimes_{\mathbb{R}} \mathbb{C}\{I\}/\mathrm{KSp}^c_2(A_\mathbb C,\sigma_\mathbb C))\otimes K_X^2)$
\end{itemize}
where $\varphi'=q_1(\varphi)$ and $\varphi''=q_2(\varphi')$. Therefore, denoting the space of such Higgs bundles as ${M}_{\mathrm{Higgs}}(X,\mathrm{Sp}_2(A_\mathbb C,\sigma_\mathbb C))$, then the diagram \eqref{hit_mor_fac_sp_general} further factorizes as 

\begin{equation}\label{hit_mor_fac_sp_general_not_equiv}    
\xymatrix{
H^0(X,(A_\mathbb{H}^{\sigma_0} \otimes_{\mathbb{R}} \mathbb{C}\{I\}/\mathrm{KSp}^c_2(A_\mathbb C,\sigma_\mathbb C))\otimes K_X) \ar[r]^{q_2} & H^0(X,(A_\mathbb{H}^{\sigma_1} \otimes_{\mathbb{R}} \mathbb{C}\{I\}/\mathrm{KSp}^c_2(A_\mathbb C,\sigma_\mathbb C))\otimes K_X^2) \ar[dd]^{q_3}\\
&\\
{M}_{\mathrm{Higgs}}(X,\mathrm{Sp}_2(A_\mathbb C,\sigma_\mathbb C)) \ar[uu]^{q_1} \ar[r]^{\mathcal{N}_{\mathbb{C}}^{\bullet}\ \ \ \ \ \ \ }&   H^0(X,A^{\sigma_1}_\mathbb H\otimes_\mathbb R \mathbb C\{I\}/\!\!/\mathrm{Sp}_2(A_\mathbb C,\sigma_\mathbb C)\otimes K_X^2)
}
\end{equation}

\subsection{The real case \texorpdfstring{$G=\mathrm{Sp}_2(A,\sigma)$}{}}\label{sec:real_Sp_Hitchin_mor}

We now explain the analogous factorization for the real group
\[
G=\mathrm{Sp}_2(A,\sigma),
\]
where $(A,\sigma)$ is a semisimple Hermitian involutive algebra. The main difference from the complex group \( \mathrm{Sp}_2(A_{\mathbb C},\sigma_{\mathbb C}) \)
is that the relevant Hitchin base is not $H^0(X,\mathfrak{g}^{\mathbb{C}}/\!\!/G^{\mathbb{C}}\otimes K_X)$, but rather $H^0(X,\mathfrak{m}^{\mathbb{C}}/\!\!/K^{\mathbb{C}}\otimes K_X)$, where
\[
\mathfrak g=\mathfrak k\oplus\mathfrak m
\]
is the Cartan decomposition of $G$.

We use the projective and half-space models of the symmetric space $\mathcal{X}_G=G/K$ described in Section \ref{sec:real_models}. After the change of basis used there, the
maximal compact subgroup $K=\mathrm{KSp}_2(A,\sigma)$ is identified with
\[
K\cong \mathrm{O}(A_{\mathbb C},\bar\sigma_{\mathbb C}),
\]
and the tangent space of the symmetric space at the base point is identified with $A_{\mathbb C}^{\sigma_{\mathbb C}}$, which is viewed as a real vector space.

The associated real Jordan algebra is $A_{\mathbb C}^{\bar\sigma_{\mathbb C}}$, with Jordan product
\[
x\circ y=\frac{1}{2}(xy+yx).
\]
Since $(A,\sigma)$ is Hermitian, this is a formally real Jordan algebra. The real quadratic norm map is
\begin{align*}
    \tilde{\mathcal{N}}_{\mathbb R}:
A_{\mathbb C}^{\sigma_{\mathbb C}}&\longrightarrow A_{\mathbb C}^{\bar\sigma_{\mathbb C}},\\
q&\longmapsto q\,\bar\sigma_{\mathbb C}(q),
\end{align*}
where $K$ acts on $A_{\mathbb C}^{\bar\sigma_{\mathbb C}}$ by conjugation. This is well-defined because
\[
\bar\sigma_{\mathbb C}\bigl(q\,\bar\sigma_{\mathbb C}(q)\bigr)=\bar\sigma_{\mathbb C}^2(q)\,\bar\sigma_{\mathbb C}(q)=q\,\bar\sigma_{\mathbb C}(q).
\]
Moreover, the norm map is $K$-equivariant. Indeed, if
$k\in \mathrm{O}(A_{\mathbb C},\bar\sigma_{\mathbb C})$,
then the action on $A_{\mathbb C}^{\sigma_{\mathbb C}}$ is by congruence,
\begin{align}\label{eq:action_on_A_C^sigma}
    k.q=kq\sigma_{\mathbb C}(k),
\end{align}
and one has
\[
\tilde{\mathcal{N}}_{\mathbb R}(k.q)=kq\sigma_{\mathbb C}(k)\,\bar\sigma_{\mathbb C}\bigl(kq\sigma_{\mathbb C}(k)\bigr)=kq\bar\sigma_{\mathbb C}(q)k^{-1}.
\]

We now complexify this construction. First we have the complexification
\[
A_{\mathbb C}^{\sigma_{\mathbb C}}\otimes_{\mathbb R}\mathbb{C}\cong 
A_{\mathbb C}^{\sigma_{\mathbb C}}\oplus A_{\mathbb C}^{\sigma_{\mathbb C}}, 
\]
and the complexification of the real Jordan algebra as
\[
A_{\mathbb C}^{\bar\sigma_{\mathbb C}}\otimes_{\mathbb R}\mathbb C
\cong A_{\mathbb C}.
\]
The complexified maximal compact group $K^{\mathbb C}\cong A_{\mathbb C}^{\times}$ acts on $A_{\mathbb C}^{\sigma_{\mathbb C}}\oplus A_{\mathbb C}^{\sigma_{\mathbb C}}$ as 
\[
k.(q_+,q_-)=\bigl(kq_+\sigma_{\mathbb C}(k),
\sigma_{\mathbb C}(k)^{-1}q_-k^{-1}\bigr).
\]

The complexified quadratic norm map is therefore
\begin{align*}
\tilde{\mathcal{N}}_{\mathbb C} :\ \ \ 
A_{\mathbb C}^{\sigma_{\mathbb C}}\oplus A_{\mathbb C}^{\sigma_{\mathbb C}}\  &\ \ \longrightarrow\ \ \ A_{\mathbb C},\\
(q_+,q_-) \ \ &\ \ \longmapsto\ \ q_+q_-.
\end{align*}
It is $K^{\mathbb C}$-equivariant with respect to the conjugation action of $A_{\mathbb C}^{\times}$ on $A_{\mathbb C}$, since
\[
\tilde{\mathcal{N}}_{\mathbb C}\bigl(k.(q_+,q_-)\bigr)=
kq_+\sigma_{\mathbb C}(k)\sigma_{\mathbb C}(k)^{-1}q_-k^{-1}=k(q_+q_-)k^{-1}.
\]
Consequently, it induces a morphism of affine GIT quotients
\[
\hat{\mathcal{N}}_{\mathbb C}: (A_{\mathbb C}^{\sigma_{\mathbb C}}\oplus A_{\mathbb C}^{\sigma_{\mathbb C}})/\!\!/K^{\mathbb C}
\longrightarrow A_{\mathbb C}/\!\!/A_{\mathbb C}^{\times}.
\]

\begin{prop}
The morphism
\[
\hat{\mathcal{N}}_{\mathbb C}: (A_{\mathbb C}^{\sigma_{\mathbb C}}\oplus A_{\mathbb C}^{\sigma_{\mathbb C}})/\!\!/K^{\mathbb C}
\longrightarrow A_{\mathbb C}/\!\!/A_{\mathbb C}^{\times}
\]
is an isomorphism.
\end{prop}

\begin{proof}
Since $A$ is semisimple, the assertion is checked factor by factor using the Artin--Wedderburn decomposition of $A$ (c.f. \eqref{Artin_wed_A}). For a real matrix factor
\[
(A,\sigma)=({\mathrm{Mat}}_n(\mathbb R),(\bullet)^t),
\]
one obtains
\[
A_{\mathbb{C}}^{\sigma_{\mathbb{C}}}\oplus A_{\mathbb{C}}^{\sigma_{\mathbb{C}}}={\mathrm{Sym}}_n(\mathbb{C})\oplus{\mathrm{Sym}}_n(\mathbb{C}),\qquad A_{\mathbb{C}}={\mathrm{Mat}}_n(\mathbb{C}),\qquad K^{\mathbb{C}}={\mathrm{GL}}_n(\mathbb{C}),
\]
and 
\begin{align*}
    \tilde{\mathcal{N}}_{\mathbb{C}}:\ {\mathrm{Sym}}_n(\mathbb{C})\oplus{\mathrm{Sym}}_n(\mathbb{C})&\longrightarrow{\mathrm{Mat}}_n(\mathbb{C}),\\
    (\beta,\gamma)&\longmapsto\beta\gamma.
\end{align*}
The group ${\mathrm{GL}}_n(\mathbb{C})$ acts on ${\mathrm{Sym}}_n(\mathbb{C})\oplus{\mathrm{Sym}}_n(\mathbb{C})$ by
\[
g\cdot(\beta,\gamma)=(g\beta g^t,(g^{-1})^t\gamma g^{-1}),
\]
and acts on ${\mathrm{Mat}}_n(\mathbb{C})$ by conjugation, so $\tilde{\mathcal{N}}_{\mathbb{C}}$ is ${\mathrm{GL}}_n(\mathbb{C})$-equivariant. Thus, the coefficients of the characteristic polynomial of $\beta\gamma$ define ${\mathrm{GL}}_n(\mathbb{C})$-invariant polynomials.

By the Chevalley restriction theorem of Kostant--Rallis (c.f. \cite{KR71} or Section \ref{Hi_Ko_Ra_sec}) for the complexified isotropy representation of $\mathrm{Sp}_{2n}(\mathbb R)$, these polynomials  generate the whole invariant algebra
\[
\mathbb{C}[{\mathrm{Sym}}_n(\mathbb{C})\oplus{\mathrm{Sym}}_n(\mathbb{C})]^{{\mathrm{GL}}_n(\mathbb{C})}.
\]
More precisely,
\[
\mathbb{C}[{\mathrm{Sym}}_n(\mathbb{C})\oplus{\mathrm{Sym}}_n(\mathbb{C})]^{{\mathrm{GL}}_n(\mathbb{C})}=\mathbb{C}[p_1(\beta\gamma),\cdots,p_n(\beta\gamma)],
\]
where 
\[
\det(\lambda\mathrm{I}_n-\beta\gamma)=\lambda^n+p_1(\beta\gamma)\lambda^{n-1}+\cdots+p_n(\beta\gamma).
\]
These generators are precisely the pullbacks of the standard generators of
\[
\mathbb{C}[{\mathrm{Mat}}_n(\mathbb{C})]^{{\mathrm{GL}}_n(\mathbb{C})}.
\]
Therefore, $\tilde{\mathcal{N}}_{\mathbb{C}}$ induces the isomorphism 
\begin{align*}
    \tilde{\mathcal{N}}_{\mathbb{C}}^*: \mathbb{C}[{\mathrm{Mat}}_n(\mathbb{C})]^{{\mathrm{GL}}_n(\mathbb{C})}\xrightarrow{\cong}\mathbb{C}[{\mathrm{Sym}}_n(\mathbb{C})\oplus{\mathrm{Sym}}_n(\mathbb{C})]^{{\mathrm{GL}}_n(\mathbb{C})}.
\end{align*}
The complex and quaternionic simple factors can be treated in the same way, in each case, the invariant polynomials are the coefficients of the corresponding characteristic polynomial of $\beta\gamma$. Finally, taking the product over all simple factors gives the desired isomorphism of affine quotients.
\end{proof}

So we have the quotient map
\begin{align}\label{eq:complex_norm_real-Sp_2}
\begin{aligned}
    \mathcal{N}_{\mathbb C} :\ \ \ 
A_{\mathbb C}^{\sigma_{\mathbb C}}\oplus A_{\mathbb C}^{\sigma_{\mathbb C}}  &\ \ \longrightarrow\ \ \ A_{\mathbb C}/\!\!/A_{\mathbb C}^{\times},\\
(q_+,q_-)\ \  &\ \ \longmapsto\ \ [q_+q_-].
\end{aligned}
\end{align}

Now choose homogeneous generators of the invariant algebra
\[
p_{1},\dots,p_{n}\in \mathbb{C}[A_{\mathbb{C}}]^{A_{\mathbb C}^{\times}},
\qquad \deg(p_i)=d_i.
\]
Equivalently, one may take the coefficients of the characteristic polynomial in the formally real Jordan algebra $A_{\mathbb{C}}^{\bar\sigma_{\mathbb C}}$ and then complexify. For each degree $d$, the corresponding invariant polynomial induces a map
\[
\chi_d: A_{\mathbb C}/\!\!/A_{\mathbb C}^{\times}\longrightarrow \mathbb{C}.
\]
Composing this map with \eqref{eq:complex_norm_real-Sp_2} gives a polynomial map of degree $2d$ (because of the quadratic nature of $\mathcal{N}_{\mathbb{C}}$)
\[
\mathcal{N}_{\mathbb{C}}^{d}: A_{\mathbb C}^{\sigma_{\mathbb C}}\oplus A_{\mathbb C}^{\sigma_{\mathbb C}}\longrightarrow \mathbb{C},
\]
thus fitting into the diagram
\begin{equation}\label{eq::factorization_char_pol_real-Sp_2}
    \xymatrix{
&A_{\mathbb C}/\!\!/A_{\mathbb C}^{\times}\ar[dd]^{\chi_d}\\
&&\\
(A_{\mathbb C}^{\sigma_{\mathbb C}}\oplus A_{\mathbb C}^{\sigma_{\mathbb C}})\ar[uur]^{\mathcal{N}_{\mathbb{C}}}\ar[r]^{\mathcal{N}_{\mathbb{C}}^d}&\mathbb{C}
.}
\end{equation}

At the level of moduli spaces of $\mathrm{Sp}_2(A,\sigma)$-Higgs bundles, the diagram is respectively of the form
\begin{equation}\label{hit_mor_fac_real-Sp_2_general}
    \xymatrix{
&H^0(X,A_{\mathbb C}/\!\!/A_{\mathbb C}^{\times}\otimes K_X^2)\ar[dd]^{\bigoplus\limits_{i=1}^n\chi_{d_i}}\\
&&\\
\mathcal{M}_{\mathrm{Higgs}}(X,\mathrm{Sp}_2(A,\sigma))\ar[uur]^{\mathcal N_\mathbb C^{\bullet}}\ar[r]^{\hat{\mathcal{N}}_{\mathbb{C}}}&\bigoplus\limits_{i=1}^{n} H^0(X, K_X^{2d_i}),
}
\end{equation}
where $\mathcal N_\mathbb C^{\bullet}$ denotes the map at the level of $\mathrm{Sp}_2(A,\sigma)$-Higgs bundles induced by \eqref{eq:complex_norm_real-Sp_2}, for $\bullet$ a choice of model for the symmetric space $\mathcal{X}$ of $\mathrm{Sp}_2(A,\sigma)$. Note that the map $\hat{\mathcal{N}}_{\mathbb{C}}$ above is the usual Hitchin morphism. We describe the map $\mathcal N_\mathbb C^{\bullet}$ more explicitly below for each of the four models of the symmetric space.

\vspace{2mm}

\noindent \textit{Complex structures model.}
For the special case when $A$ is a matrix algebra, we can define analogously a map from general $\mathrm{Sp}_2(A,\sigma)$-Higgs bundles to the coefficients of the characteristic polynomials of their Higgs field viewed as an element in $T^{\mathbb{C}}_J\mathfrak{C}_G$ given by the traces of all powers of the operator $L$ viewed as a $(2 \times 2)$-matrix over $A_{\mathbb{C}}$. Note that the requirement that $A$ is a matrix algebra is not necessary for the realization of the Hitchin $\mathrm{Sp}_2(A,\sigma)$-morphism in the other models.

\vspace{2mm}

\noindent \textit{Projective space model.}
We next describe the same construction in the projective space model. Let $\mathfrak P_G^+$ be the positive projective model of the symmetric space. A point $l\in\mathfrak P_G^+$ is an isotropic $A_{\mathbb C}$-line
in $A_{\mathbb C}^2$, and the tangent space at $l$ is described by
\[
T_l\mathfrak P_G^+ \cong
\operatorname{Hom}_{A_{\mathbb C}}(l,l^{\perp_h}),
\]
with the positivity condition described in Section \ref{sec:real_models}.

The involution $(\bar{\cdot})$ on $A_{\mathbb C}$ sends the positive line $l$ to the opposite line $\bar l\in\mathfrak P_G^-$. Thus the tangent space can be viewed as the space of $A_{\mathbb C}$-linear maps
\[
Q:l\longrightarrow \bar l.
\]
Choose a generator $v$ of $l$, then $\bar v$ generates $\bar l$, and
\[
Q(v)=\bar v\,a\qquad \text{for a unique } a\in A_{\mathbb C}^{\sigma_{\mathbb C}}.
\]
The conjugate map
\[
Q':=(\bar{\cdot})\circ Q\circ (\bar{\cdot}) : \bar l\longrightarrow l
\]
is then given by
\[
Q'(\bar v)=v\,\bar\sigma_{\mathbb C}(a).
\]
Therefore the composition $Q\circ Q'$ is represented by the element
\[
a\bar\sigma_{\mathbb C}(a)\in A_{\mathbb C}^{\bar\sigma_{\mathbb C}}.
\]
This gives the real norm in the projective model:
\[
\widetilde{\mathcal N}_{\mathbb R}^{\mathfrak P}(Q)
:= a\bar\sigma_{\mathbb C}(a) \in A_{\mathbb C}^{\bar\sigma_{\mathbb C}}.
\]
Changing the generator $v$ changes $a$ by the corresponding congruence action by the stabilizer $K=\mathrm{Stab}_{\mathrm{Sp}_2(A,\sigma)}(l)
\cong \mathrm{O}(A_{\mathbb C},\bar\sigma_{\mathbb C})$, and therefore changes $a\bar\sigma_{\mathbb C}(a)$ by conjugation. Hence the class
\[
[\widetilde{\mathcal N}_{\mathbb R}^{\mathfrak P}(Q)]
\in A_{\mathbb C}^{\bar\sigma_{\mathbb C}}/K
\]
is well defined.

Now we complexify, an element of $T_l^{\mathbb C}\mathfrak P_G^+$ can be written as a pair
\[
(Q_+,Q_-),
\]
where, after choosing the pair of opposite lines $(l,\bar l)$, the maps are represented by elements
\[
a_+,a_-\in A_{\mathbb C}^{\sigma_{\mathbb C}}.
\]
That is,
\[
Q_+(v)=\bar v\,a_+,
\qquad
Q_-(\bar v)=v\,a_-.
\]
The complexified quadratic norm map is then
\[
\widetilde{\mathcal N}_{\mathbb C}^{\mathfrak P}(Q_+,Q_-)
:= a_+a_- \in A_{\mathbb C}.
\]
Passing to the affine GIT quotient gives the map
\[
\mathcal N_{\mathbb C}^{\mathfrak P}: T^{\mathbb C}\mathfrak P_G^+ \longrightarrow A_{\mathbb C}/\!\!/A_{\mathbb C}^{\times},
\qquad (l,Q_+,Q_-)\longmapsto [a_+a_-].
\]
This map is independent of all choices. Indeed, changing the adapted basis
or the element $g\in \mathrm{Sp}_2(A,\sigma)$ transporting $l$ to the base point changes $a_+a_-$ by conjugation in $A_{\mathbb C}^{\times}$.

Tensoring with $K_X$, we obtain
\[
T^{\mathbb C}\mathfrak P_G^+\otimes K_X
\longrightarrow
(A_{\mathbb C}/\!\!/A_{\mathbb C}^{\times})\otimes K_X^2.
\]
For a homogeneous invariant polynomial $\chi_d$ of degree $d$, define
\[
\mathcal N_{\mathbb C}^{\mathfrak P,d}:=\chi_d\circ \mathcal N_{\mathbb C}^{\mathfrak P}.
\]
Then
\[
\mathcal N_{\mathbb C}^{\mathfrak P,d}: T^{\mathbb C}\mathfrak P_G^+\otimes K_X\longrightarrow K_X^{2d}.
\]
Thus, for a Higgs bundle $(E_{K^{\mathbb{C}}},\Phi=(l,Q_+,Q_-)\,dz)$ in the projective model, we obtain
\[
\mathcal N_{\mathbb C}^{\mathfrak P,d_i}(\Phi)\in
H^0(X,K_X^{2d_i}), \qquad i=1,\ldots,n,
\]
and hence a map
\[
\widehat{\mathcal N}_{\mathbb C}^{\mathfrak P}:
\mathcal M_{\mathrm{Higgs}}(X,\mathrm{Sp}_2(A,\sigma))
\longrightarrow
\bigoplus_{i=1}^{n}H^0(X,K_X^{2d_i}).
\]
This map fits into the factorization diagram
\[
\xymatrix{
&
H^0\bigl(X,(A_{\mathbb C}/\!\!/A_{\mathbb C}^{\times})\otimes K_X^2\bigr)
\ar[dd]^{\bigoplus_i\chi_{d_i}}
\\
&&\\
\mathcal M_{\mathrm{Higgs}}(X,\mathrm{Sp}_2(A,\sigma))
\ar[uur]^{\mathcal N_{\mathbb C}^{\mathfrak P}}
\ar[r]_{\widehat{\mathcal N}_{\mathbb C}^{\mathfrak P}}
&
\displaystyle\bigoplus_{i=1}^{n}H^0(X,K_X^{2d_i}).
}
\]

\vspace{2mm}

\noindent \textit{Upper half-space and precompact model.}
We now describe the same factorization in the upper half-space model. Recall that
\[
\mathfrak U_G^+=
\left\{
z\in A_{\mathbb C}^{\sigma_{\mathbb C}}
\ \middle|\
\operatorname{Im}(z)\in A^\sigma_+
\right\}.
\]
This is an open subset of the real vector space
$A_{\mathbb C}^{\sigma_{\mathbb C}}$, hence $T_z\mathfrak U_G^+\cong A_{\mathbb C}^{\sigma_{\mathbb C}}$ for all $z\in \mathfrak U_G^+$. Its complexified tangent space is
\[
T_z^{\mathbb C}\mathfrak U_G^+\cong A_{\mathbb C}^{\sigma_{\mathbb C}}\oplus A_{\mathbb C}^{\sigma_{\mathbb C}}.
\]

We define the map 
\begin{align*}
\mathcal N_\mathbb C^{\mathfrak U}\colon T^{\mathbb{C}}\mathfrak{U}\xrightarrow{q_1} (A_{\mathbb C}^{\sigma_{\mathbb C}}\oplus A_{\mathbb C}^{\sigma_{\mathbb C}})/K  \xrightarrow{q_2} A_\mathbb C/K  \xrightarrow{q_3}A_\mathbb C/\!\!/K^\mathbb C
\end{align*}
as a composition of the following maps: Given $(z,v_+,v_-)\in T^{\mathbb C}\mathfrak U_G^+$,
choose $g\in\mathrm{Sp}_2(A,\sigma)$ such that 
\[
g.(z,v_+,v_-)=(i,q_+,q_-),\qquad (q_+,q_-)\in A_{\mathbb C}^{\sigma_{\mathbb C}}
\oplus A_{\mathbb C}^{\sigma_{\mathbb C}}.
\]
The element $g$ is unique up to left multiplication by the stabilizer $K=\mathrm{Stab}_{\mathrm{Sp}_2(A,\sigma)}(i)\cong \mathrm{O}(A_{\mathbb C},\bar\sigma_{\mathbb C})$. Thus, we can identify 
$$
T^{\mathbb{C}}\mathfrak{U}/\mathrm{Sp}_2(A,\sigma)\cong (A_{\mathbb C}^{\sigma_{\mathbb C}}\oplus A_{\mathbb C}^{\sigma_{\mathbb C}})/K,$$
where the action is given by~\eqref{eq:action_on_A_C^sigma}, and the map $q_1$ is defined as the quotient map:
$$
q_1\colon T^{\mathbb{C}}\mathfrak{U}\to (A_{\mathbb C}^{\sigma_{\mathbb C}}\oplus A_{\mathbb C}^{\sigma_{\mathbb C}})/K
$$
mapping $(z,v_+,v_-)$ to the class $K.(v_+,v_-)$. As the map $(q_+,q_-)\mapsto q_+q_-$ maps $A_{\mathbb C}^{\sigma_{\mathbb C}}\oplus A_{\mathbb C}^{\sigma_{\mathbb C}}$ to $A_\mathbb{C}$ and by~\eqref{eq:invariance}, the following map is well defined:
$$
q_2\colon (A_{\mathbb C}^{\sigma_{\mathbb C}}\oplus A_{\mathbb C}^{\sigma_{\mathbb C}})/K\to A_\mathbb{C}/K.
$$
The action of $K$ on $A_\mathbb{C}$ is given by $(k,q)\mapsto kqk^{-1}$. 

Finally, since $K^{\mathbb{C}}\cong A_{\mathbb{C}}^\times$ is the complexification of $K\cong\mathrm{O}(A_{\mathbb C},\bar\sigma_{\mathbb C})$, and $K$ acts on $A_\mathbb{C}^{\bar\sigma_{\mathbb C}}$, this action extends to the complexified spaces, i.e. $K^{\mathbb C}$ acts on $A_\mathbb{C}$. Moreover, as a subgroup $K$ acts on $A_\mathbb{C}$ and this action agrees with the action described above. Thus, the following quotient map is well-defined: 
$$q_3\colon A_\mathbb{C}/K\longrightarrow A_\mathbb{C}/\!\!/K^{\mathbb C}.$$

Let 
$$
\mathcal N^{\mathfrak U}_\mathbb C(z,v_+,v_-):=[q_+q_-]=q_3(q_2(q_1(z,v_+,v_-)))\in A_\mathbb{C}/\!\!/K^{\mathbb C}.
$$ 
This is an invariant under the action of the complexified 
$\mathrm{Stab}_{K^\mathbb C}(z)$ on $T_z^{\mathbb{C}}\mathfrak{U}$ which is isomorphic to $K^\mathbb C$ which is the complexified stabilizer of the point $i$.

The construction for $\mathfrak{B}_G$ is similar. The only difference is that we use the point $0\in\mathfrak{B}_G$ instead of $i\in\mathfrak{U}_G$.

As above, the invariants of elements in $A_\mathbb C$ under the action of $K^{\mathbb{C}}$ are given by coefficients of the characteristic polynomials which are homogeneous polynomials of degrees $d \in \{d_1, ..., d_n\}$ of the complexification of the formally real Jordan algebra $A_{\mathbb C}^{\bar\sigma_{\mathbb C}}$. Thus, for each $d$ there is a map defined by the polynomial of degree $d$:
\[\mathcal{N}_{\mathbb{C}}^d\colon A_{\mathbb{C}} \longrightarrow \mathbb{C}\]
and now this map for each $d$ induces the map
\[
\hat{\mathcal{N}}_{\mathbb{C}}^{d}:\mathcal{M}_{\rm Higgs}\bigl(X,\mathrm{Sp}_2(A,\sigma)\bigr)\longrightarrow H^{0}(X,K_X^{2d}).
\]
This then finally induces the map
\[
\hat{\mathcal{N}}_{\mathbb{C}}:\mathcal{M}_{\rm Higgs}\bigl(X,\mathrm{Sp}_2(A,\sigma)\bigr)
\longrightarrow \bigoplus_{i=1}^{n}H^{0}(X,K_X^{2d_i}).
\]
Similar to the description of the projective space model, we have that the map $\mathcal{N}_{\mathbb{C}}^{d}$ factors as in \eqref{eq::factorization_char_pol_real-Sp_2} and the Hitchin morphism $\hat{\mathcal{N}}_{\mathbb{C}}$ factors as in \eqref{hit_mor_fac_real-Sp_2_general}.

Lastly, for a polystable $\mathrm{Sp}_2(A,\sigma)$-Higgs bundle $(E,\varphi)$ in the upper half-space model (not a gauge equivalence class), we see that the Higgs field $\varphi\in f_{\mathfrak{U}}^{*}(T^{\mathbb{C}}\mathfrak{U}_{G})\otimes K_{X}$, can be associated to the following two sections:
\begin{itemize}
\item $\varphi'\in H^{0}\bigl(((A_{\mathbb C}^{\sigma_{\mathbb C}}\oplus A_{\mathbb C}^{\sigma_{\mathbb C}})/\mathrm{O}(A_{\mathbb C},\bar\sigma_{\mathbb C}))\otimes K_X\bigr)$,
\item $\varphi''\in H^{0}\bigl(X,(A_{\mathbb{C}}/\mathrm{O}(A_{\mathbb C},\bar\sigma_{\mathbb C}))\otimes K_X^{2}\bigr)$,
\end{itemize}
where $\varphi'=q_{1}(\varphi)$ and $\varphi''=q_{2}(\varphi')$. Therefore, denoting the space of such Higgs bundles as $M_{\rm Higgs}(X,\mathrm{Sp}_{2}(A,\sigma))$, the diagram \eqref{hit_mor_fac_real-Sp_2_general} further factorizes as
\begin{equation}
\begin{tikzcd}
H^{0}\bigl(X,((A_{\mathbb C}^{\sigma_{\mathbb C}}\oplus A_{\mathbb C}^{\sigma_{\mathbb C}})/\mathrm{O}(A_{\mathbb C},\bar\sigma_{\mathbb C}))\otimes K_{X}\bigr)
\arrow[r,"q_{2}"]& H^{0}\bigl(X,(A_{\mathbb{C}}/\mathrm{O}(A_{\mathbb C},\bar\sigma_{\mathbb C}))\otimes K_{X}^{2}\bigr)
\arrow[dd,"q_{3}"]
\\
\\
M_{\rm Higgs}\bigl(X,\mathrm{Sp}_2(A,\sigma)\bigr)
\arrow[uu,"q_{1}"]
\arrow[r,"\mathcal{N}_{\mathbb{C}}^{\bullet}"]&H^{0}\bigl(X,(A_{\mathbb{C}}/\!\!/A_{\mathbb{C}}^\times)\otimes K_{X}^{2}\bigr).
\end{tikzcd}
\end{equation}

\subsection{The real case \texorpdfstring{$G=\mathrm{O}_{(1,1)}(A,\sigma)$}{}}

We now explain the analogous factorization for the real group
\[
G=\mathrm{O}_{(1,1)}(A,\sigma),
\]
where $(A,\sigma)$ is a semisimple Hermitian involutive algebra. As in the case of
$\mathrm{Sp}_2(A,\sigma)$, the relevant Hitchin base is $H^0(X,\mathfrak{m}^{\mathbb{C}}/\!\!/K^{\mathbb{C}}\otimes K_X)$ for $\mathfrak g=\mathfrak k\oplus\mathfrak m$ being the Cartan decomposition.

We use the projective and half-space models of the symmetric space $\mathcal X_G=G/K$ described in Section \ref{sec:ort_models}. After the change of basis used there, the maximal compact subgroup $K=\mathrm{KO}_{(1,1)}(A,\sigma)$ is identified with
\[
K\cong \mathrm{O}(A,\sigma)\times \mathrm{O}(A,\sigma),
\]
and the tangent space of the symmetric space at the base point is identified with $A$, which is viewed as a real vector space.

The associated real Jordan algebra is $A^\sigma,$ which is also formally real. The real quadratic norm map is
\begin{align*}
    \tilde{\mathcal N}_{\mathbb R}:
A &\longrightarrow A^\sigma,\\
q &\longmapsto \sigma(q)q,
\end{align*}
where the second factor $\mathrm{O}(A,\sigma)$ of $K=\mathrm{O}(A,\sigma)\times\mathrm{O}(A,\sigma)$ acts on $A^\sigma$ by conjugation. This is well-defined because
\[
\sigma\bigl(\sigma(q)q\bigr)=\sigma(q)\sigma^2(q)=\sigma(q)q.
\]
Moreover, the norm map is $K$-equivariant. Indeed, the action of $K$ on $A$ is given by
\begin{align}\label{eq:action_on_A_O11}
    (k_1,k_2).q=k_1qk_2^{-1},
\end{align}
and one has
\[
\tilde{\mathcal N}_{\mathbb R}\bigl((k_1,k_2).q\bigr)
=
\sigma(k_1qk_2^{-1})k_1qk_2^{-1}
=
k_2\sigma(q)qk_2^{-1}.
\]

We now complexify this construction. We have the complexified maximal compact group 
\[
K^{\mathbb{C}}\cong\mathrm{O}(A_{\mathbb{C}},\sigma_{\mathbb{C}})\times\mathrm{O}(A_{\mathbb{C}},\sigma_{\mathbb{C}}),
\]
and it acts on $A_{\mathbb{C}}:=A\otimes_{\mathbb R}\mathbb{C}$ by
\[
(k_1,k_2).q=k_1qk_2^{-1}.
\]

The complexified quadratic norm map is therefore
\begin{align*}
\tilde{\mathcal N}_{\mathbb{C}}:\ 
A_{\mathbb{C}} &\longrightarrow A_{\mathbb{C}}^{\sigma_{\mathbb{C}}}:=A^\sigma\otimes_{\mathbb R}\mathbb{C},\\
q &\longmapsto \sigma_{\mathbb{C}}(q)q.
\end{align*}
It is $K^{\mathbb{C}}$-equivariant with respect to the conjugation action of the second factor $\mathrm{O}(A_{\mathbb{C}},\sigma_{\mathbb{C}})$ of $K^{\mathbb{C}}$ on $A_{\mathbb{C}}^{\sigma_{\mathbb{C}}}$, since
\[
\tilde{\mathcal N}_{\mathbb{C}}\bigl((k_1,k_2).q\bigr)=\sigma_{\mathbb{C}}(k_1qk_2^{-1})k_1qk_2^{-1}=k_2\sigma_{\mathbb{C}}(q)qk_2^{-1}.
\]
Consequently, it induces a morphism of affine GIT quotients
\[
\hat{\mathcal N}_{\mathbb{C}}: A_{\mathbb{C}}/\!\!/\bigl(\mathrm{O}(A_{\mathbb{C}},\sigma_{\mathbb{C}})\times
\mathrm{O}(A_{\mathbb{C}},\sigma_{\mathbb{C}})\bigr)
\longrightarrow
A_{\mathbb{C}}^{\sigma_{\mathbb{C}}}/\!\!/\mathrm{O}(A_{\mathbb{C}},\sigma_{\mathbb{C}}).
\]

\begin{prop}
The morphism
\[
\hat{\mathcal N}_{\mathbb{C}}: A_{\mathbb{C}}/\!\!/\bigl(\mathrm{O}(A_{\mathbb{C}},\sigma_{\mathbb{C}})\times
\mathrm{O}(A_{\mathbb{C}},\sigma_{\mathbb{C}})\bigr)\longrightarrow A_{\mathbb{C}}^{\sigma_{\mathbb{C}}}/\!\!/\mathrm{O}(A_{\mathbb{C}},\sigma_{\mathbb{C}})
\]
induced by $q\mapsto \sigma_{\mathbb{C}}(q)q$ is an isomorphism.
\end{prop}

\begin{proof}
Since $A$ is semisimple, the assertion is checked factor by factor using the Artin--Wedderburn decomposition of $A$ (cf.~\eqref{Artin_wed_A}). For a real matrix factor
\[
(A,\sigma)=({\mathrm{Mat}}_n(\mathbb R),(\bullet)^t),
\]
one obtains
\[
A_{\mathbb{C}}={\mathrm{Mat}}_n(\mathbb{C}),\qquad A_{\mathbb{C}}^{\sigma_{\mathbb{C}}}={\mathrm{Sym}}_n(\mathbb{C}),\qquad K^{\mathbb{C}}=\mathrm{O}_n(\mathbb{C})\times \mathrm{O}_n(\mathbb{C}).
\]
and 
\begin{align*}
    \tilde{\mathcal{N}}_{\mathbb{C}}:\ {\mathrm{Mat}}_n(\mathbb{C})&\longrightarrow{\mathrm{Sym}}_n(\mathbb{C}),\\
    q&\longmapsto q^tq.
\end{align*}
Similarly, we can see $\tilde{\mathcal{N}}_{\mathbb{C}}$ is equivariant with respect to the conjugation action of the second factor $\mathrm{O}_n(\mathbb{C})$ on ${\mathrm{Sym}}_n(\mathbb{C})$. $ \tilde{\mathcal{N}}_{\mathbb{C}}$ induces the isomorphism
\[
 \tilde{\mathcal{N}}_{\mathbb{C}}^*: \mathbb{C}[{\mathrm{Sym}}_n(\mathbb{C})]^{{\mathrm O}_n(\mathbb{C})}\xrightarrow{\cong}\mathbb{C}[{\mathrm{Mat}}_n(\mathbb{C})]^{{\mathrm O}_n(\mathbb{C})\times {\mathrm O}_n(\mathbb{C})}.
\]

In fact, we can compute $\mathbb{C}[{\mathrm{Mat}}_n(\mathbb{C})]^{{\mathrm O}_n(\mathbb{C})\times {\mathrm O}_n(\mathbb{C})}$ via two-step group actions. Let $G_1=G_2=\mathrm{O}_n(\mathbb{C})$, $G_1$ acts on $\mathrm{Mat}_n(\mathbb{C})$ by left multiplication:
\[
k_1.q:=k_1\cdot q.
\]
By the first fundamental theorem for the orthogonal group (see \cite[Proposition 10.2]{KP96}, or \cite[Corollary 5.2.5]{GW09}), the invariant algebra for the left $G_1$-action is generated by the entries of the Gram matrix $q^tq$. More precisely, write $q=(q_1,\cdots,q_n)\in\mathrm{Mat}_n(\mathbb{C})$, then
\[
\mathbb{C}[{\mathrm{Mat}}_n(\mathbb{C})]^{{\mathrm O}_n(\mathbb{C})}=\mathbb{C}[(q^tq)_{ij}\ |\ 1\leq i,j\leq n].
\]
Thus the map
\begin{align*}
    \pi: {\mathrm{Mat}}_n(\mathbb{C})&\longrightarrow{\mathrm{Sym}}_n(\mathbb{C})\\
    q&\longmapsto q^tq
\end{align*}
is the affine GIT quotient for the left $G_1$-action, namely we have the first identification
\[
{\mathrm{Mat}}_n(\mathbb{C})/\!\!/{\mathrm O}_n(\mathbb{C})\cong{\mathrm{Sym}}_n(\mathbb{C}).
\]
Moreover, the right $G_2$-action descends to the conjugation action on ${\mathrm{Sym}}_n(\mathbb{C})$, since
\[
(qk_2^{-1})^t(qk_2^{-1})=k_2(q^tq)k_2^{-1}.
\]
Therefore, quotienting first by the left factor $G_1$ and then by the right factor $G_2$ gives
\[
\bigl({\mathrm{Mat}}_n(\mathbb{C})/\!\!/{\mathrm O}_n(\mathbb{C})\bigr)/\!\!/{\mathrm O}_n(\mathbb{C})\cong {\mathrm{Sym}}_n(\mathbb{C})/\!\!/{\mathrm O}_n(\mathbb{C}), 
\]
namely we have
\[
{\mathrm{Mat}}_n(\mathbb{C})/\!\!/\bigl({\mathrm O}_n(\mathbb{C})\times {\mathrm O}_n(\mathbb{C})\bigr)\cong{\mathrm{Sym}}_n(\mathbb{C})/\!\!/{\mathrm O}_n(\mathbb{C}).
\]

The complex and quaternionic simple factors can be treated in the same way: in each case, the invariant polynomials are the coefficients of the corresponding characteristic polynomial of $\sigma_{\mathbb{C}}(q)q$ in the complexified Jordan algebra $A_{\mathbb{C}}^{\sigma_{\mathbb{C}}}$. Finally, taking the product over all simple factors gives the desired isomorphism of affine quotients.
\end{proof}

So we have the quotient map
\begin{align}\label{eq:complex_norm_real-O11}
\begin{aligned}
    \mathcal N_{\mathbb{C}}:\ A_{\mathbb{C}}&\longrightarrow A_{\mathbb{C}}^{\sigma_{\mathbb{C}}}/\!\!/\mathrm{O}(A_{\mathbb{C}},\sigma_{\mathbb{C}}),\\
q&\longmapsto [\sigma_{\mathbb{C}}(q)q].
\end{aligned}
\end{align}

Now choose homogeneous generators of the invariant algebra
\[
p_1,\dots,p_n\in\mathbb{C}[A_{\mathbb{C}}^{\sigma_{\mathbb{C}}}]^{\mathrm{O}(A_{\mathbb{C}},\sigma_{\mathbb{C}})},
\qquad \deg(p_i)=d_i.
\]
Equivalently, one may take the coefficients of the characteristic polynomial in the formally real Jordan algebra $A^\sigma$ and then complexify. For each degree $d$, the corresponding invariant polynomial induces a map
\[
\chi_d: A_{\mathbb{C}}^{\sigma_{\mathbb{C}}}/\!\!/\mathrm{O}(A_{\mathbb{C}},\sigma_{\mathbb{C}})
\longrightarrow\mathbb{C}.
\]
Composing this map with \eqref{eq:complex_norm_real-O11} gives a polynomial map of degree $2d$ because of the quadratic nature of $\mathcal N_{\mathbb{C}}$:
\[
\mathcal N_{\mathbb{C}}^d: A_{\mathbb{C}}\longrightarrow \mathbb{C}.
\]
Thus we have the diagram
\begin{equation}\label{eq::factorization_char_pol_real-O11}
    \xymatrix{
&A_{\mathbb{C}}^{\sigma_{\mathbb{C}}}/\!\!/\mathrm{O}(A_{\mathbb{C}},\sigma_{\mathbb{C}})\ar[dd]^{\chi_d}\\
&&\\
A_{\mathbb{C}}\ar[uur]^{\mathcal N_{\mathbb{C}}}\ar[r]^{\mathcal N_{\mathbb{C}}^d}&\mathbb{C}.
}
\end{equation}

At the level of moduli spaces of $\mathrm{O}_{(1,1)}(A,\sigma)$-Higgs bundles, the diagram is respectively of the form
\begin{equation}\label{hit_mor_fac_real-O11_general}
    \xymatrix{
&H^0\bigl(X,(A_{\mathbb{C}}^{\sigma_{\mathbb{C}}}/\!\!/\mathrm{O}(A_{\mathbb{C}},\sigma_{\mathbb{C}}))\otimes K_X^2\bigr)
\ar[dd]^{\bigoplus\limits_{i=1}^n\chi_{d_i}}
\\
&&\\
\mathcal{M}_{\mathrm{Higgs}}(X,\mathrm{O}_{(1,1)}(A,\sigma))\ar[uur]^{\mathcal N_{\mathbb{C}}^{\bullet}}\ar[r]^{\hat{\mathcal N}_{\mathbb{C}}}&\bigoplus\limits_{i=1}^{n} H^0(X,K_X^{2d_i}),
}
\end{equation}
where $\mathcal{N}_{\mathbb{C}}^{\bullet}$ denotes the map at the level of $\mathrm{O}_{(1,1)}(A,\sigma)$-Higgs bundles induced by \eqref{eq:complex_norm_real-O11}, for $\bullet$ a choice of model for the symmetric space $\mathcal X_G$ of $\mathrm{O}_{(1,1)}(A,\sigma)$. Note that the map $\hat{\mathcal N}_{\mathbb{C}}$ above is the usual Hitchin morphism for the real group $\mathrm{O}_{(1,1)}(A,\sigma)$. We describe the map $\mathcal{N}_{\mathbb{C}}^{\bullet}$ more explicitly below for the models of the symmetric space.

\vspace{2mm}

\noindent \textit{Indefinite involutive operators model.}
In the indefinite involutive operators model $\mathfrak{C}_G$, a point of the symmetric space is an involutive operator $J$ on $A^2$ compatible with the standard indefinite orthogonal form. The tangent space at $J$ consists of $A$-linear morphisms $L: A^2\to A^2$ that anti-commuting with $J$, namely $LJ+JL=0$. Transporting $J$ to the base point by an element of $\mathrm{O}_{(1,1)}(A,\sigma)$ identifies $T_J\mathfrak{C}_G$ with $A$. After complexification, one obtains $T_J^{\mathbb{C}}\mathfrak{C}_G\cong A_{\mathbb{C}}$.

 Thus, for an element $L\in T_J^{\mathbb{C}}\mathfrak{C}_G$, after choosing such a transport to the base point, we obtain an element
\[
q\in A_{\mathbb{C}}.
\]
Changing the transport changes $q$ by the action of the complexified stabilizer
\[
K^{\mathbb{C}}\cong\mathrm{O}(A_{\mathbb{C}},\sigma_{\mathbb{C}})\times\mathrm{O}(A_{\mathbb{C}},\sigma_{\mathbb{C}}).
\]
Therefore, the class
\[
[\sigma_{\mathbb{C}}(q)q]\in A_{\mathbb{C}}^{\sigma_{\mathbb{C}}}/\!\!/\mathrm{O}(A_{\mathbb{C}},\sigma_{\mathbb{C}})
\]
is independent of all choices. This gives a well-defined invariant map
\begin{align*}
    \mathcal{N}_{\mathbb{C}}^{\mathfrak C}: T^{\mathbb{C}}\mathfrak{C}_G&\longrightarrow A_{\mathbb{C}}^{\sigma_{\mathbb{C}}}/\!\!/\mathrm{O}(A_{\mathbb{C}},\sigma_{\mathbb{C}}),\\
    (J,L)&\longmapsto[\sigma_{\mathbb{C}}(q)q].
\end{align*}
For a homogeneous invariant polynomial $\chi_d$ of degree $d$, define
\[
\mathcal{N}_{\mathbb{C}}^{\mathfrak C,d} := \chi_d\circ \mathcal N_{\mathbb{C}}^{\mathfrak C}: T^{\mathbb{C}}\mathfrak{C}_G\longrightarrow\mathbb{C}.
\]
After tensoring with $K_X$, this gives
\[
\mathcal{N}_{\mathbb{C}}^{\mathfrak C,d}: T^{\mathbb{C}}\mathfrak C_G\otimes K_X \longrightarrow K_X^{2d}.
\]
Thus, for a Higgs bundle $(E_{K^{\mathbb{C}}},\Phi=(J,L)\,dz)$ in the indefinite involutive operators model, we obtain
\[
\mathcal{N}_{\mathbb{C}}^{\mathfrak C,d_i}(\Phi)\in H^0(X,K_X^{2d_i}), \qquad i=1,\ldots,n,
\]
and hence a map
\[
\hat{\mathcal N}_{\mathbb{C}}^{\mathfrak C}: \mathcal M_{\mathrm{Higgs}}(X,\mathrm{O}_{(1,1)}(A,\sigma))
\longrightarrow \bigoplus_{i=1}^{n}H^0(X,K_X^{2d_i}).
\]

\vspace{2mm}

\noindent \textit{Projective space model.}
Now we consider the case of $\mathrm{O}_{(1,1)}(A,\sigma)$-Higgs bundles in the projective space model. Let $\mathfrak P_G^+$ be the positive projective model of the symmetric space. A point $l\in\mathfrak{P}_G^+$ is a positive $A$-line in $A^2$ for the standard indefinite orthogonal form. Its orthogonal complement $l^{\perp_h}$ is a negative $A$-line, and the tangent space at $l$ is described by
\[
T_l\mathfrak P_G^+\cong\mathrm{Hom}_A(l,l^{\perp_h}).
\]
Choose a generator $v$ of $l$ and a generator $w$ of $l^{\perp_h}$ such that the pair $(v,w)$ is adapted to the indefinite orthogonal form $h$, namely $(v,w)$ is normalized by
\[
h(v,w)=0,\qquad h(v,v)=1,\qquad h(w,w)=-1.
\]
Then an element $Q\in T_l\mathfrak P_G^+$ is represented by a unique element $a\in A$ via
\[
Q(v)=wa.
\]
The adjoint map
\[
Q^\ast:l^{\perp_h}\longrightarrow l
\]
is represented by
\[
Q^\ast(w)=v\sigma(a).
\]
Therefore, the composition $Q^\ast\circ Q$ is represented by the element
\[
\sigma(a)a\in A^\sigma.
\]
This gives the real norm in the projective model:
\[
\widetilde{\mathcal N}_{\mathbb R}^{\mathfrak P}(Q) := \sigma(a)a\in A^\sigma.
\]
Changing the generator $v$ changes $a$ by the left-right action of the stabilizer
\[
K=\mathrm{Stab}_{\mathrm{O}_{(1,1)}(A,\sigma)}(l)\cong \mathrm{O}(A,\sigma)\times\mathrm{O}(A,\sigma),
\]
and therefore changes $\sigma(a)a$ by conjugation in the second factor $\mathrm{O}(A,\sigma)$. Hence the class
\[
[\widetilde{\mathcal N}_{\mathbb R}^{\mathfrak P}(Q)]\in A^\sigma/\mathrm{O}(A,\sigma)
\]
is well defined.

Now we complexify, an element of $T_l^{\mathbb{C}}\mathfrak P_G^+$ is represented by an element $a\in A_{\mathbb{C}}$, and the complexified quadratic norm map is
\[
\widetilde{\mathcal N}_{\mathbb{C}}^{\mathfrak P}(Q) :=\sigma_{\mathbb{C}}(a)a\in A_{\mathbb{C}}^{\sigma_{\mathbb{C}}}.
\]
Passing to the affine GIT quotient gives the map
\begin{align*}
    \mathcal N_{\mathbb{C}}^{\mathfrak P}: T^{\mathbb{C}}\mathfrak P_G^+&\longrightarrow A_{\mathbb{C}}^{\sigma_{\mathbb{C}}}/\!\!/\mathrm{O}(A_{\mathbb{C}},\sigma_{\mathbb{C}}),\\
    (l,Q)&\longmapsto [\sigma_{\mathbb{C}}(a)a].
\end{align*}
This map is independent of all choices. Indeed, changing the adapted basis or the element $g\in \mathrm{O}_{(1,1)}(A,\sigma)$ transporting $l$ to the base point changes $\sigma_{\mathbb{C}}(a)a$ by conjugation in $\mathrm{O}(A_{\mathbb{C}},\sigma_{\mathbb{C}})$.

Tensoring with $K_X$, we obtain
\[
T^{\mathbb{C}}\mathfrak P_G^+\otimes K_X\longrightarrow (A_{\mathbb{C}}^{\sigma_{\mathbb{C}}}/\!\!/\mathrm{O}(A_{\mathbb{C}},\sigma_{\mathbb{C}}))\otimes K_X^2.
\]
For a homogeneous invariant polynomial $\chi_d$ of degree $d$, define
\[
\mathcal{N}_{\mathbb{C}}^{\mathfrak P,d} := \chi_d\circ \mathcal N_{\mathbb{C}}^{\mathfrak P}: T^{\mathbb{C}}\mathfrak P_G^+\longrightarrow\mathbb{C}.
\]
Then
\[
\mathcal{N}_{\mathbb{C}}^{\mathfrak P,d}: T^{\mathbb{C}}\mathfrak P_G^+\otimes K_X\longrightarrow K_X^{2d}.
\]
Thus, for a Higgs bundle $(E_{K^{\mathbb{C}}},\Phi=(l,Q)\,dz)$ in the projective space model, we obtain
\[
\mathcal N_{\mathbb{C}}^{\mathfrak P,d_i}(\Phi)\in H^0(X,K_X^{2d_i}), \qquad i=1,\ldots,n,
\]
and hence a map
\[
\hat{\mathcal N}_{\mathbb{C}}^{\mathfrak P}: \mathcal M_{\mathrm{Higgs}}(X,\mathrm{O}_{(1,1)}(A,\sigma))
\longrightarrow\bigoplus_{i=1}^{n}H^0(X,K_X^{2d_i}).
\]
This map fits into the factorization diagram
\[
\xymatrix{
& H^0\bigl(X,(A_{\mathbb{C}}^{\sigma_{\mathbb{C}}}/\!\!/\mathrm{O}(A_{\mathbb{C}},\sigma_{\mathbb{C}}))
\otimes K_X^2\bigr)\ar[dd]^{\bigoplus_i\chi_{d_i}}
\\
&&\\
\mathcal M_{\mathrm{Higgs}}(X,\mathrm{O}_{(1,1)}(A,\sigma))\ar[uur]^{\mathcal N_{\mathbb{C}}^{\mathfrak P}}
\ar[r]_{\hat{\mathcal N}_{\mathbb{C}}^{\mathfrak P}}
&
\displaystyle\bigoplus_{i=1}^{n}H^0(X,K_X^{2d_i}).
}
\]

\vspace{2mm}

\noindent \textit{Upper half-space and precompact model.}
We now describe the same factorization in the upper half-space model. Recall that the upper half-space model $\mathfrak U_G^+$ is an open subset whose tangent space at each point is naturally identified with $A$. Hence $T_z^{\mathbb{C}}\mathfrak U_G^+\cong A_{\mathbb{C}}$.

We define the map
\[
\mathcal N_{\mathbb{C}}^{\mathfrak U}: T^{\mathbb{C}}\mathfrak U_G^+\xrightarrow{q_1} A_{\mathbb{C}}/\bigl(\mathrm{O}(A,\sigma)\times\mathrm{O}(A,\sigma)\bigr)\xrightarrow{q_2}A_{\mathbb{C}}^{\sigma_{\mathbb{C}}}/\mathrm{O}(A,\sigma)\xrightarrow{q_3} A_{\mathbb{C}}^{\sigma_{\mathbb{C}}}/\!\!/\mathrm{O}(A_{\mathbb{C}},\sigma_{\mathbb{C}})
\]
as a composition of the following maps. Given
\[
(z,v)\in T^{\mathbb{C}}\mathfrak U_G^+,
\]
choose $g\in \mathrm{O}_{(1,1)}(A,\sigma)$ such that
\[
g.(z,v)=(1,q), \qquad q\in A_{\mathbb{C}}.
\]
The element $g$ is unique up to left multiplication by the stabilizer
\[
K=\mathrm{Stab}_{\mathrm{O}_{(1,1)}(A,\sigma)}(1)\cong\mathrm{O}(A,\sigma)\times\mathrm{O}(A,\sigma).
\]
Thus, we can identify
\[
T^{\mathbb{C}}\mathfrak U_G^+/\mathrm{O}_{(1,1)}(A,\sigma)\cong A_{\mathbb{C}}/K,
\]
where the action is given by~\eqref{eq:action_on_A_O11}, and the map $q_1$ is defined as the quotient map
\[
q_1: T^{\mathbb{C}}\mathfrak U_G^+\longrightarrow A_{\mathbb{C}}/K
\]
mapping $(z,v)$ to the class $K.q$. Since the map $q\longmapsto \sigma_{\mathbb{C}}(q)q$ maps $A_{\mathbb{C}}$ to $A_{\mathbb{C}}^{\sigma_{\mathbb{C}}}$ and is equivariant with respect to the action~\eqref{eq:action_on_A_O11}, the following map is well defined:
\[
q_2: A_{\mathbb{C}}/K\longrightarrow A_{\mathbb{C}}^{\sigma_{\mathbb{C}}}/\mathrm{O}(A,\sigma).
\]
The action of $\mathrm{O}(A_{\mathbb{C}},\sigma_{\mathbb{C}})$ on $A_{\mathbb{C}}^{\sigma_{\mathbb{C}}}$ is by conjugation.

Finally, the quotient map
\[
q_3: A_{\mathbb{C}}^{\sigma_{\mathbb{C}}}/\mathrm{O}(A,\sigma)\longrightarrow
A_{\mathbb{C}}^{\sigma_{\mathbb{C}}}/\!\!/\mathrm{O}(A_{\mathbb{C}},\sigma_{\mathbb{C}})
\]
is well defined. We set
\[
\mathcal{N}_{\mathbb{C}}^{\mathfrak U}(z,v) :=[\sigma_{\mathbb{C}}(q)q]= q_3(q_2(q_1(z,v)))\in A_{\mathbb{C}}^{\sigma_{\mathbb{C}}}/\!\!/\mathrm{O}(A_{\mathbb{C}},\sigma_{\mathbb{C}}).
\]
This is an invariant under the action of the complexified stabilizer of $z$ on $T_z^{\mathbb{C}}\mathfrak U_G^+$.

The construction for the precompact model $\mathfrak B_G$ is similar. The only difference is that we use the base point $0\in\mathfrak B_G$ instead of $1\in\mathfrak U_G^+$.

As above, the invariants of elements in $A_{\mathbb{C}}^{\sigma_{\mathbb{C}}}$ under the action of $\mathrm{O}(A_{\mathbb{C}},\sigma_{\mathbb{C}})$ are given by coefficients of the characteristic polynomials, which are homogeneous polynomials of degrees $d\in\{d_1,\dots,d_n\}$ of the complexification of the formally real Jordan algebra $A^\sigma$. Thus, for each $d$ there is a map defined by a polynomial of degree $d$:
\[
\mathcal N_{\mathbb{C}}^{d}:A_{\mathbb{C}}^{\sigma_{\mathbb{C}}}\longrightarrow\mathbb{C}.
\]
Composing with the quadratic norm gives a polynomial of degree $2d$ on $A_{\mathbb{C}}$ and induces the map
\[
\hat{\mathcal N}_{\mathbb{C}}^{d}: \mathcal M_{\mathrm{Higgs}}\bigl(X,\mathrm{O}_{(1,1)}(A,\sigma)\bigr)
\longrightarrow H^0(X,K_X^{2d}).
\]
This then finally induces the map
\[
\hat{\mathcal N}_{\mathbb{C}}: \mathcal M_{\mathrm{Higgs}}\bigl(X,\mathrm{O}_{(1,1)}(A,\sigma)\bigr) \longrightarrow\bigoplus_{i=1}^{n}H^0(X,K_X^{2d_i}).
\]
Similar to the description of the projective space model, the map $\mathcal N_{\mathbb{C}}^{d}$ factors as in
\eqref{eq::factorization_char_pol_real-O11}, and the Hitchin morphism $\hat{\mathcal N}_{\mathbb{C}}$ factors as in
\eqref{hit_mor_fac_real-O11_general}.

Let $(E,\varphi)$ be a polystable $\mathrm{O}_{(1,1)}(A,\sigma)$-Higgs bundle in the upper half-space model, then $\varphi\in f_{\mathfrak U}^{*}(T^{\mathbb{C}}\mathfrak U_G)\otimes K_X$, and we can associate the following two sections:
\begin{itemize}
\item $\varphi'\in H^0\bigl(X,(A_{\mathbb{C}}/(\mathrm{O}(A,\sigma)\times\mathrm{O}(A,\sigma)))\otimes K_X\bigr)$,
\item $\varphi''\in H^0\bigl(X,(A_{\mathbb{C}}^{\sigma_{\mathbb{C}}}/\mathrm{O}(A,\sigma))\otimes K_X^2\bigr)$,
\end{itemize}
where $\varphi'=q_1(\varphi)$ and $\varphi''=q_2(\varphi')$. Therefore, denoting the space of such Higgs bundles by
$M_{\mathrm{Higgs}}(X,\mathrm{O}_{(1,1)}(A,\sigma))$, the diagram \eqref{hit_mor_fac_real-O11_general} further factorizes as
\begin{equation}
\begin{tikzcd}
H^0\bigl(X,(A_{\mathbb{C}}/
(\mathrm{O}(A,\sigma)\times\mathrm{O}(A,\sigma)))\otimes K_X\bigr)
\arrow[r,"q_2"]
&
H^0\bigl(X,(A_{\mathbb{C}}^{\sigma_{\mathbb{C}}}/
\mathrm{O}(A,\sigma))\otimes K_X^2\bigr)
\arrow[dd,"q_3"]
\\
\\
M_{\mathrm{Higgs}}\bigl(X,\mathrm{O}_{(1,1)}(A,\sigma)\bigr)
\arrow[uu,"q_1"]
\arrow[r,"\mathcal N_{\mathbb{C}}^{\bullet}"]
&
H^0\bigl(X,
(A_{\mathbb{C}}^{\sigma_{\mathbb{C}}}/\!\!/
\mathrm{O}(A_{\mathbb{C}},\sigma_{\mathbb{C}}))\otimes K_X^2
\bigr).
\end{tikzcd}
\end{equation}

\subsection{Geometric identification and dimensional efficiency}

The factorizations constructed in \eqref{hit_mor_fac_sp_general}, \eqref{hit_mor_fac_real-Sp_2_general}, and \eqref{hit_mor_fac_real-O11_general} have the same geometric form. A Higgs field is first interpreted in one of the geometric models of the symmetric space, so that its value lies in a concrete complexified tangent model $W_{\mathbb C}$. A quadratic norm map
\[
\tilde{\mathcal N}_{\mathbb C}: W_{\mathbb C}\longrightarrow J_{\mathbb C}
\]
then sends this tangent datum to the complexification $J_{\mathbb C}$ of a formally real Jordan algebra. The map $\tilde{\mathcal N}_{\mathbb C}$ is equivariant with respect to the $K^{\mathbb{C}}$-action on $W_{\mathbb C}$ and the $H^{\mathbb{C}}$-action on $J_{\mathbb C}$, where $H^{\mathbb{C}}$ denotes the group acting on $J_{\mathbb C}$ in the corresponding quotient. Passing to affine GIT quotients gives a factorization of the invariant-theoretic quotient defining the Hitchin base.

For the complex group
\[
G=\mathrm{Sp}_2(A_{\mathbb C},\sigma_{\mathbb C}),
\]
one has $H^{\mathbb{C}}=\mathrm{Sp}_2(A_{\mathbb C},\sigma_{\mathbb C})$, and the factorization takes the form
\[
A_{\mathbb H}^{\sigma_0}\otimes_{\mathbb R}\mathbb C\{I\}\longrightarrow A_{\mathbb H}^{\sigma_1}\otimes_{\mathbb R}\mathbb C\{I\}
\longrightarrow\bigl(A_{\mathbb H}^{\sigma_1}\otimes_{\mathbb R}\mathbb C\{I\}\bigr)/\!\!/\mathrm{Sp}_2(A_{\mathbb C},\sigma_{\mathbb C}).
\]
For the real group
\[
G=\mathrm{Sp}_2(A,\sigma),
\]
one has $H^{\mathbb{C}}=A_{\mathbb C}^{\times}$, and the quotient $\mathfrak m^{\mathbb C}/\!\!/K^{\mathbb C}$ is realized as
\[
(A_{\mathbb C}^{\sigma_{\mathbb C}}\oplus A_{\mathbb C}^{\sigma_{\mathbb C}})/\!\!/K^{\mathbb C}\cong A_{\mathbb C}/\!\!/A_{\mathbb C}^{\times}.
\]
For the real group 
\[
G=\mathrm{O}_{(1,1)}(A,\sigma),
\]
one has $H^{\mathbb{C}}=\mathrm{O}(A_{\mathbb C},\sigma_{\mathbb C})$, and the quotient $\mathfrak m^{\mathbb C}/\!\!/K^{\mathbb C}$ is realized as
\[
A_{\mathbb C}/\!\!/\bigl(\mathrm{O}(A_{\mathbb C},\sigma_{\mathbb C})\times\mathrm{O}(A_{\mathbb C},\sigma_{\mathbb C})\bigr)\cong A_{\mathbb C}^{\sigma_{\mathbb C}}/\!\!/\mathrm{O}(A_{\mathbb C},\sigma_{\mathbb C}).
\]

Thus, in each case, the invariant-theoretic quotient defining the Hitchin base is described through invariant polynomials on an intermediate affine GIT quotient of the complexified Jordan algebra $J_{\mathbb{C}}$. More precisely, choose homogeneous generators
\[
p_1,\dots,p_n\in \mathbb{C}[J_{\mathbb{C}}]^{H^{\mathbb C}}, \qquad \mathrm{deg}(p_i)=d_i
\]
so that
\[
\mathbb{C}[J_{\mathbb{C}}]^{H^{\mathbb C}}=\mathbb{C}[p_1,\cdots,p_n].
\]
Then the Higgs field $\varphi$ is first sent, via the quadratic norm map $\tilde{\mathcal{N}}_{\mathbb{C}}$, to a section with values in $J_{\mathbb{C}}\otimes K_X^2$, and the invariant polynomials $p_i$ are then applied.

Concretely, the quadratic norm map $\tilde{\mathcal{N}}_{\mathbb{C}}$ is
\[
q\longmapsto q\sigma_1(q)
\]
for $G=\mathrm{Sp}_2(A_{\mathbb C},\sigma_{\mathbb C})$; it is
\[
(q_+,q_-)\longmapsto q_+q_-
\]
for $G=\mathrm{Sp}_2(A,\sigma)$; and it is
\[
q\longmapsto \sigma_{\mathbb C}(q)q
\]
for $G=\mathrm{O}_{(1,1)}(A,\sigma)$. The resulting composition is exactly the map $\mathcal{N}_{\mathbb{C}}^{\bullet}$ appearing in the previous subsections, where $\bullet$ denotes the chosen geometric model of the symmetric space.

Since $\tilde{\mathcal{N}}_{\mathbb{C}}$ is quadratic, the composite $p_i\circ \tilde{\mathcal{N}}_{\mathbb{C}}$ is homogeneous of degree $2d_i$ as a polynomial on the tangent model $W_{\mathbb{C}}$. Hence, after applying it to a Higgs field $\varphi$, one obtains a
section
\[
(p_i\circ \tilde{\mathcal{N}}_{\mathbb{C}})(\varphi)\in H^0(X,K_X^{2d_i}).
\]
Therefore, the factorizations of the Hitchin morphism obtained in the previous subsections have target
\[
\bigoplus_{i=1}^{n} H^0(X,K_X^{2d_i}).
\]

Let $e_1,\dots,e_n$ denote the fundamental degrees of the invariant algebra $\mathbb{C}[\mathfrak{m}^{\mathbb C}]^{K^{\mathbb C}}$ (it is $\mathbb{C}[\mathfrak g]^{G}$ when $G$ is complex) defining the corresponding Hitchin base. In our situation, the identification of affine GIT quotients show that the Jordan-theoretic degrees $d_i$ satisfy
\[
e_i=2d_i,\qquad i=1,\dots,n.
\]
Consequently,
\[
\bigoplus_{i=1}^{n} H^0(X,K_X^{2d_i}) = \bigoplus_{i=1}^{n} H^0(X,K_X^{e_i}) = \mathcal B(X,G),
\]
which is precisely the Hitchin base. In other words, the intermediate quotient $J_{\mathbb{C}}/\!\!/{H^{\mathbb C}}$
provides a geometric realization of the Hitchin base through the quadratic norm map arising from the Jordan-algebraic models of the symmetric space.

This identification yields a substantial reduction in the algebraic complexity required to characterize the Hitchin base. For example, in the complex symplectic case $G=\mathrm{Sp}_{2n}(\mathbb{C})$, the Higgs field is represented in the tangent model $W_{\mathbb C}$, but its Hitchin invariants factor through the complexified Jordan algebra
\[
J_{\mathbb C}\cong\mathrm{Herm}_n(\mathbb H)\otimes_{\mathbb{R}}\mathbb{C}.
\]
This Jordan algebra has complex dimension $2n^2-n$, whereas the corresponding Lie algebra $\mathfrak{sp}_{2n}(\mathbb{C})$ has dimension $2n^2+n$. Thus the intermediate quotient is built from a smaller linear space than the full Lie algebra. 

Similarly, in the real indefinite orthogonal case, for the classical example
\[
(A,\sigma)=\bigl(\mathrm{Mat}_n(\mathbb R),(\cdot)^t\bigr),
\]
one has
\[
\mathfrak{m}^{\mathbb C}\cong \mathrm{Mat}_n(\mathbb C),\qquad K^{\mathbb{C}}\cong\mathrm{O}_n(\mathbb{C})\times\mathrm{O}_n(\mathbb{C})
\]
and 
\[
J_{\mathbb{C}}\cong \mathrm{Sym}_n(\mathbb{C}), \qquad H^{\mathbb C}\cong \mathrm{O}_n(\mathbb{C}).
\]
The complexified Jordan algebra $\mathrm{Sym}_n(\mathbb{C})$ has complex dimension $\frac{n(n+1)}{2}$, which is strictly smaller than $\dim_{\mathbb{C}}\bigl(\mathrm{Mat}_n(\mathbb C)\bigr)=n^2$ whenever $n>1$. Moreover, in this example the generators of $\mathbb{C}[\mathrm{Sym}_n(\mathbb{C})]^{\mathrm{O}_n(\mathbb{C})}$ have degrees $1,\dots,n$, whereas the corresponding generators of $\mathbb{C} [\mathrm{Mat}_n(\mathbb C)]^{\mathrm{O}_n(\mathbb{C})\times\mathrm{O}_n(\mathbb{C})}$ have degrees $2,4,\dots,2n$. Equivalently, $e_i=2d_i$.

Thus, the factorization through $J_{\mathbb{C}}/\!\!/H^{\mathbb{C}}$ replaces the full Lie-algebraic description of the Hitchin morphism by lower-dimensional Jordan data, while recovering the same Hitchin base. This gives a more explicit description of the Hitchin morphism, where the coefficients of the base are seen to arise naturally from the internal geometry of the noncommutative models developed in Part I.

\subsection{Hitchin--Kostant--Rallis section for the real groups \texorpdfstring{$\mathrm{Sp}_2(A,\sigma)$ and $\mathrm{O}_{(1,1)}(A,\sigma)$}{}}\label{Hi_Ko_Ra_sec}

We now apply the construction of Hitchin--Kostant--Rallis sections for real reductive groups, as developed in \cite{GPR}, to our specific setting. For the detailed construction, we refer the reader to \cite{GPR}.

Let $G$ be a real reductive group with Cartan involution $\theta$, inducing the Cartan decomposition $\mathfrak{g} = \mathfrak{k} \oplus \mathfrak{m}$ where $\mathfrak{k}$ is the Lie algebra of a maximal compact subgroup $K \subseteq G$. Let $\mathfrak{a} \subseteq \mathfrak{m}$ be a maximal abelian subspace and $W(\mathfrak{a}^\mathbb{C})$ the associated Weyl group. Critically, if $(G, K, \theta)$ is \emph{strongly reductive}, that is, $G$ acts by inner automorphisms on $\mathfrak{g}^{\mathbb{C}}$ via the adjoint representation, there is an identification of the categorical quotient \cite[Section 4]{GPR}:
$$
\mathfrak{m}^{\mathbb{C}} /\!\!/ K^{\mathbb{C}} \cong \mathfrak{a}^{\mathbb{C}} / W(\mathfrak{a}^{\mathbb{C}}).
$$
This identification gives rise to the Chevalley morphism
$$
\chi \colon \mathfrak{m}^{\mathbb{C}}\longrightarrow\mathfrak{a}^{\mathbb{C}}/W(\mathfrak{a}^{\mathbb{C}}).
$$
In this case, the \emph{Hitchin morphism} for the moduli space $\mathcal{M}_{\mathrm{Higgs}}(X,G)$ of polystable $G$-Higgs bundles (of degree 0) is defined as
$$
h_G \colon \mathcal{M}_{\mathrm{Higgs}}(X,G) \longrightarrow B(X,G) := H^0(X, \mathfrak{a}^{\mathbb{C}}/W(\mathfrak{a}^{\mathbb{C}}) \otimes K_X).
$$
Under the strongly reductivity assumption, there exists a Kostant--Rallis section \[s_{\mathrm{KR}} \colon \mathfrak{a}^{\mathbb{C}}/W(\mathfrak{a}^{\mathbb{C}}) \to \mathfrak{m}^{\mathbb{C}}\] through regular elements. Globally, for a fixed square root of $K_X$, this construction yields $N$ many ($N$ here is the number of cosets in $\mathrm{Ad}(G)^\theta / \mathrm{Ad}(K)$) inequivalent \emph{Hitchin--Kostant--Rallis sections} 
$$
s_{\mathrm{HKR}} \colon B(X,G) \longrightarrow \mathcal{M}_{\mathrm{Higgs}}(X,G),
$$ 
each factoring through the moduli space $\mathcal{M}_{\mathrm{Higgs}}(X,\widehat{G}_0)$ that corresponds to the maximal split subgroup $\widehat{G}_0 \subseteq G$.

We emphasize that strong reductivity is not needed to define the Hitchin morphism itself. The Hitchin morphism for a real reductive group is already defined at the beginning of this section using the affine quotient $\mathfrak m_{\mathbb C}/\!\!/K^{\mathbb C}$. Strong reductivity is used to identify this quotient with $\mathfrak a^{\mathbb C}/W(\mathfrak a^{\mathbb C})$ and to construct the Hitchin--Kostant--Rallis section.

The applicability of this theory to groups over Hermitian algebras $(A, \sigma)$ follows from the properties of strongly reductive Lie groups. By \cite[Prop.~1.2.8]{Peon}, any connected reductive group contained in a complex reductive group is strongly reductive. Together with Propositions \ref{prop:O_(1,1)-connected} and \ref{prop:Sp_2-connected}, we therefore obtain:

\begin{prop}
Let $(A, \sigma)$ be a Hermitian algebra. 
\begin{enumerate}
    \item The group $\mathrm{Sp}_2(A, \sigma)$ is strongly reductive. 
    \item If $A^{\times}$ is connected, the group $\mathrm{O}_{(1,1)}(A, \sigma)$ is strongly reductive.
\end{enumerate}
Consequently, both groups admit a Hitchin morphism $h_G$ taking values in $B(X,G)$ and possess Hitchin--Kostant--Rallis sections as described above.
\end{prop}

Specifying the algebra $A$ and the involution $\sigma$ in each case, one can obtain a detailed description of the Hitchin base and the Hitchin--Kostant--Rallis section. We demonstrate this in a basic example below.

\subsection{Example: Hitchin base and Hitchin section for the case \texorpdfstring{$G=\mathrm{Sp}_{4}(\mathbb{R})$}{}.}
We now restrict to the case when $G=\mathrm{Sp}_{4}(\mathbb{R})$ and consider the model $\mathfrak{C}$ of all complex structures on $A^2$, for $A=\mathrm{Mat}_2(\mathbb{R})$, following our analysis in Section \ref{sec:real_Sp_Hitchin_mor} (cf. also Appendix \ref{sec:comparison of defs}).

Remember that in its classical interpretation, for the fibration 
$$h\colon \mathcal{M}_{\mathrm{Higgs}}(X,\mathrm{Sp}_4(\mathbb{R})) \to \mathcal{B},$$ the Hitchin base is the space 
\[\mathcal{B}= H^0(X, K_X^2)\bigoplus H^0(X, K_X^4)\] 
of quadratic and quartic holomorphic differentials over the Riemann surface $X$. Moreover, the points in the Hitchin section for the space $\mathcal{M}_{\mathrm{Higgs}}(X,\mathrm{Sp}_4(\mathbb{R}))$ are triples $(V, \beta, \gamma)$ in the sense of Definition \ref{defn:classical_Sp2nR} with 
\begin{equation}\label{eqn:Higgs_fields_Teichmuller}
V= K^{\frac{3}{2}} \oplus K^{-\frac{1}{2}}, \quad \beta = \begin{pmatrix} q_4 & q_2 \\ q_2& 1 \end{pmatrix} \quad \text{and} \quad \gamma = \begin{pmatrix} 0 & 1 \\ 1& 0 \end{pmatrix},
\end{equation}
where $q_2 \in H^0(X, K_X^2)$ and $q_4 \in H^0(X, K_X^4)$ (see \cite[Table 1]{BGG12}). 

Following Section \ref{sec:real_Sp_Hitchin_mor}, the Hitchin morphism
\begin{align*}
 \hat{\mathcal{N}}_{\mathbb{C}}\colon  \mathcal{M}_{\mathrm{Higgs}}(X,\mathrm{Sp}_4(\mathbb{R})) & \to H^0(X, K_X^2)\bigoplus H^0(X, K_X^4) 
\end{align*}
when using the complex structures model, is a map to the invariant polynomials of $T_{J_{\mathbb{C}}}^{\mathbb{C}}\mathfrak{C}$ given by the traces $\mathrm{Tr}(L^2_{\mathbb{C}})$ and $\mathrm{Tr}(L^4_{\mathbb{C}})$, for the $\mathbb{C}$-linear operator 
$L_\mathbb C \colon A^2_{\mathbb{C}} \to A^2_{\mathbb{C}}$, which in matrix presentation with respect to the basis $v=(v_+, v_-)$  has the form $\begin{pmatrix}
    0 & a_- \\
    a_+ & 0
\end{pmatrix}$. Here the morphisms $a_{+}\colon l_{+} \to l_{-}$ and $a_{-}\colon l_{-} \to l_{+}$ of the isotropic $A_{\mathbb{C}}$-lines $l_{-}$ and $l_{+}$ spanned by the eigenvectors $v_{-}$ and $v_{+}$ respectively form a basis of a complex structure $J_{\mathbb{C}}$ 
(cf. Section \ref{sec:comparison of defs}). 
 
Under the interpretation of $\mathrm{Sp}_4(\mathbb{R})$-Higgs bundles $(V, V^*, \beta, \gamma)$ as tuples $(l_{-}, l_{+}, a_{+}, a_{-})$, we now have that in the Hitchin section in the complex structures model there lie tuples $(l_{-}, l_{+}, a_{+}, a_{-})$, where $l_{+}$ and $l_{-}$ are dual isotropic lines with respect to the form $\omega_{\mathbb{C}}$ (in the sense of Section~\ref{sec:duality}) that correspond to the bundles $V=K^{\frac{3}{2}} \oplus K^{-\frac{1}{2}}$ and $V^*=K^{-\frac{3}{2}} \oplus K^{\frac{1}{2}}$ respectively, while the morphisms $a_{+}$ and $a_{-}$ have matrix presentation
\[a_{-}=\begin{pmatrix} q_4 & q_2 \\ q_2& 1 \end{pmatrix} \quad \text{and} \quad a_{+} = \begin{pmatrix} 0 & 1 \\ 1& 0 \end{pmatrix}.\]

Using the diffeomorphisms $F_{\mathfrak{C},\mathfrak{P}^\pm}$, $F_{\mathfrak{P}^\pm,\mathfrak U^\pm}$ and $F_{\mathfrak{P}^+,\mathfrak{B}}$, one can now obtain the description of the Higgs data in the Hitchin section in the projective space models, the half-space models and the precompact model; these maps were studied in detail in Section \ref{sec:exam_A_real_mat} for the case when $G=\mathrm{Sp}_{2n}(\mathbb{R})$.

\addtocontents{toc}{\protect\vskip8pt}

\newpage
\appendix

\addtocontents{toc}{\protect\vskip6pt}

\section{From representations to Higgs bundles.}\label{sec:Corlette_Hadamard}

\subsection{Corlette's theorem}
Let $f\colon (M_1, g_1) \to (M_2, g_2)$ 
be a $C^2$ map between Riemannian manifolds. Its derivative $df$ defines a section of the endomorphism bundle $\mathcal{E}:=T^*M_1 \otimes f^*TM_2$. Letting $\lVert \cdot \rVert _{\mathcal{E}}$ denote the natural Riemannian metric on $\mathcal{E}$ inherited from the ones on $M_1$ and $M_2$, the \emph{energy density} of $f$ is defined by 
\[e(f)=\frac{1}{2}\lVert df \rVert _{\mathcal{E}}^2=\frac{1}{2} \mathrm{Tr}_{g_1}f^*g_2.\]
The map $f$ is called \emph{harmonic} if for every relatively compact subset $V \subset M_1$, $f$ is a critical point of the energy functional
\[E_V(f)= \int_{V}e(f)dvol_1 , \]
for the volume form $dvol_1$ on $(M_1,g_1)$. Equivalently, $f$ is harmonic if it is a solution of the Euler--Lagrange equation 
\[\mathrm{Tr}_{g_1}\nabla df=0,\]
where $\nabla$ denotes the Levi-Civita connection on $\mathcal{E}$. 

Assume now that $M_1=\tilde{M}_1/\Gamma$, for a discrete group $\Gamma$, and let $\rho\colon \Gamma \to \mathrm{Isom}(M_2,g_2)$ be a representation of $\Gamma$ into the isometry group of $(M_2,g_2)$. For a $C^2$ $\rho$-equivariant map $f\colon \tilde{M}_1 \to M_2$, since $\rho$ is acting by isometries, the energy density $e(f)$ descends to a well-defined function on $M_1$. Thus, all local properties of harmonic maps are inherited by equivariant harmonic maps. 

We now restrict ourselves to a compact Riemann surface $X$ with fundamental group $\pi_1(X)$ and let $G/K$ be a symmetric space of noncompact type, where $G$ is a Lie group with maximal compact subgroup $K$. The Theorem of Corlette \cite{Corlette} asserts that a fundamental group representation $\rho \colon \pi_1(X) \to G$ is reductive if and only if there exists a $\rho$-equivariant harmonic map $f \colon \tilde{X}\to G/K$. Among  numerous important applications and generalizations, this result was recently generalized by Sagman \cite{Sagman} in the case of Hadamard manifolds. A \emph{Hadamard manifold} is a Riemannian manifold $(N,g)$ which is complete, simply connected, and non-positively curved. The following version of Corlette's theorem is the one we are applying for our purposes in this article:

\begin{thm}[{\cite[Theorem 1.1]{Sagman}}]\label{Sagman-harmonic}
    Let $X$ be a compact Riemann surface and $(N,g)$ be a Hadamard manifold. If $\rho \colon \pi_1(X) \to \mathrm{Isom}(N,g)$ is a reductive representation, then there exists a $\rho$-equivariant harmonic map $f \colon \tilde{X}\to N$.
\end{thm}

\subsection{\texorpdfstring{$G$}{G}-holomorphic pairs}\label{sec:general_defs}

In this second part of the Appendix, we see that it is possible to obtain a $G$-holomorphic pair given a principal $K$-bundle over a Riemann surface $X$ and a smooth map $f_{\mathcal{X}}\colon \tilde{X} \to \mathcal{X}:=G/K$, which is not necessarily the classifying map provided by Corlette's theorem for a given reductive fundamental group representation. This is a more general situation that is made possible using the theory of symmetric spaces for groups over involutive algebras and is of independent interest.

Let $G$ be a semisimple Lie group and $K$ be a maximal compact subgroup of $G$. Let $X$ be a compact Riemann surface of negative Euler characteristic, and let $\pi\colon \tilde X\to X$ be its universal covering. 

Let $p\colon E_K\to X$ be a principal $K$-bundle, and let $\tilde E_K=\pi^* E_K\to \tilde X$ be the pull-back bundle of $E_K$ over $\tilde X$. Since $\tilde X$ is a contractible manifold, the bundle $\tilde E_K$ is topologically trivial. For every $\tilde x \in \tilde X$, the fiber $\tilde E_K|_{\tilde x}$ is diffeomorphic to $K$. 

We now choose a smooth map $f_{\mathcal X}\colon \tilde X\to \mathcal X:=G/K$. For every $\tilde{x}\in \mathcal{X}$, we can identify $\tilde E_K|_{\tilde x}$ with $f_{\mathcal{X}}(\tilde{x})=g_{\tilde x}K$ for some $g_{\tilde x}\in G$. For instance, in the classical definition of principal bundles over a contractible space, $\tilde{E}_K \cong \tilde{X} \times K$ and so $f_{\mathcal X}\equiv K$ in this case. Moreover, in this way we can embed $\tilde \iota\colon \tilde E_K \hookrightarrow \tilde E_G:=\tilde X\times G$. Slightly abusing our notation, we denote by $\tilde E_K$ the image $\tilde\iota(\tilde E_K)\subseteq \tilde E_G$. Notice that for a given $f_\mathcal X$, the map  $\tilde\iota$ is in general not unique. It is determined by the lift of the map $f_\mathcal X$ to the group $G$, i.e. a smooth map $f_G\colon \tilde X\to G$ such that $f_G(\tilde x)K=f_\mathcal X(\tilde x)$ for all $\tilde x\in\tilde X$. However, this lift always exists as $\mathcal X$ can be identified with the subspace $\exp(\mathfrak m)\subset G=\exp(\mathfrak m) K$ where $\mathfrak m$ is the orthogonal complement of $\mathfrak k:=\mathrm{Lie}(K)$ with respect to the Killing form on $\mathfrak g:=\mathrm{Lie}(G)$.

The bundle $E_K$ is the quotient bundle of $\tilde E_K$ under certain equivalence relation $\sim$, which identifies for every $\tilde x\in\tilde X$ and every $\gamma\in\pi_1(X)$ a point $(\tilde x,k)$ where $k\in \tilde E_K|_{\tilde x}=f_{\mathcal{X}}(\tilde{x})$ with some point $(\gamma \tilde x,k')$ for a $k'\in \tilde E_K|_{\gamma\tilde x}=f_{\mathcal{X}}(\gamma\tilde{x})$, i.e. $\sim$ provides a smooth map 
\begin{equation*}\label{hom_like_map}
\rho\colon\tilde X\times\pi_1(X)\to G,\text{ such that }\rho(\tilde x,\gamma)k=k'.
\end{equation*}
In particular, $\rho(\tilde x,\gamma)f_{\mathcal{X}}(\tilde{x})=f_{\mathcal{X}}(\gamma\tilde{x})$, i.e., the map $f_{\mathcal X}$ is $\rho$-equivariant. This map is ``homomorphism-like'' in the following sense: 
$$
\rho(\tilde x,1)=1\text{ for all }\tilde x\in\tilde X,
$$
$$
\rho(\tilde x,\gamma_2\gamma_1)=\rho(\gamma_1\tilde x,\gamma_2)\rho(\tilde x,\gamma_1)\text{ for all }\tilde x\in\tilde X, \gamma_1,\gamma_2\in \pi_1(X).
$$

Notice that if $\rho$ does not depend on $\tilde x$, then it is a group homomorphism in the usual sense, thus it canonically defines via the Riemann--Hilbert correspondence a flat $G$-bundle $E_G= (\tilde{X} \times G )/ \sim $ with $E_K \subset E_G$, for $\sim$ the equivalence relation described above. However, the quotient bundle $E_G$ is also well-defined in general by the equivalence relation 
\begin{equation}\label{equiv_rel_rho}(\tilde x,g)\sim (\gamma\tilde x, g'),
\end{equation}
where $g'=\rho(\tilde x,\gamma)g$. It depends on the smooth map $\rho$ and is not flat in general. 

\begin{rem}
Note that $G\to\mathcal{X}$ is a $K$-bundle, thus the pull-back bundle $f_{\mathcal{X}}^*G$ has a structure of $K$-bundle over $\tilde{X}$, and it descends to a $K$-bundle over $X$, which is exactly a $K$-reduction $E_K$ of $E_G$. This $K$-reduction is defined with respect to the map $f_{\mathcal{X}}$.
\end{rem}

We have the following: 

\begin{prop}\label{prop:General_def_G}
The choice of a principal $K$-bundle $E_K$ and the choice of the smooth map $f_{\mathcal X}\colon \tilde{X} \to \mathcal{X}$ thus give a bundle isomorphism for the induced pull-back bundles over $\tilde{X}$
\begin{equation}\label{ident_general_G}
\tilde{E}_K(\mathfrak{m}) \cong f^{*}_{\mathcal{X}} T\mathcal{X},
\end{equation}
which is $\rho$-equivariant. 
\end{prop}

The isomorphism \eqref{ident_general_G} descends to $X$, thus providing an identification between  $E_K(\mathfrak{m})$ and the quotient of $f^{*}_{\mathcal{X}} T\mathcal{X}$ modulo the action by $\rho$. By a slight abuse of notation, we will still denote the latter by $f^{*}_{\mathcal{X}} T\mathcal{X}$ to mean the principal bundle over $X$.

Before we get into the proof of this proposition, we need the following result about a certain correspondence between differential forms over fiber bundles.

Let $H$ be a group and let $\mathrm{pr}\colon E\to M$ be a principal $H$-bundle over a smooth manifold $M$ and let $R_g\colon E\to E$ be the action of an element  $g\in H$ on $E$. The tangent vectors in the kernel of $d\mathrm{pr}\colon TE\to TM$ are called \emph{vertical}.

Now let $V$ be a vector space and $\tau\colon H\to \mathrm{GL}(V)$ be a homomorphism. Let $E(V)$ be the corresponding associate bundle. We remind that a form $\omega\in\Omega^1(E,V)$ is called \emph{horizontal} if the kernel of $\omega$ contains all vertical vectors. A form $\omega$ is \emph{$\tau$-equivariant} if for all $g\in H$ and all $p\in E$, $v\in T_pE$, 
$$\omega_{R(g)p}(dR_g(v))=\tau(g^{-1})\omega_p(v).$$

We have the following
\begin{lem}\label{lem:Maurer-Cartan_G}
There is a canonical 1-1 correspondence between horizontal, $\tau$-equivariant forms on $E$ and all $V$-valued forms on $M$ (i.e. all forms in $\Omega^1(M,E\times_H V)$).  
\end{lem}

\begin{proof}[Proof of Lemma \ref{lem:Maurer-Cartan_G}]
For $p\in E$, $v\in T_pE$, and $\omega\in \Omega^1(E,V)$, let 
\[\alpha^\omega_{\mathrm{pr}(p)}(d\mathrm{pr}(v)):=[p,\omega_p(v)].\] 
The definition of $\alpha^\omega$ above does not depend on any of the choices made. Indeed, let $p_1,p_2\in E$ be such that $p_2:=R(g)p_1$ for a point  $g\in H$, and let $v_1\in T_{p_1}E$, $v_2\in T_{p_2}E$ be such that \[d\mathrm{pr}(v_1)=d\mathrm{pr}(v_2)\in T_{\mathrm{pr}(p_1)}M.\] Then, we see that
\begin{align*}
\alpha^\omega_{\mathrm{pr}(p_2)}(d\mathrm{pr}(v_2)) & = [p_2,\omega_{p_2}(v_2)] \\
& = [p_1g,\tau(g^{-1})\omega_{p_1}(dR_{g^{-1}}v_2)]\\
& = [p_1,\omega_{p_1}(v_1)]\\
& = \alpha^\omega_{\mathrm{pr}(p_1)}(d\mathrm{pr}(v_1)).
\end{align*}
In the last two steps, we used that the form is $H$-equivariant and then that it is horizontal.

Conversely, for a form $\alpha\in \Omega^1(M,E\times_H V)$, we define a horizontal and $\tau$-equivariant element  $\omega^\alpha\in \Omega^1(E,V)$ as follows: for $p\in E$ and $v\in T_pE$, then $w:=\omega^\alpha_p(v) \in V$ is the (unique) element in $V$ such that 
$$[p,w]=\alpha_{\mathrm{pr}(p)}(d\mathrm{pr}(v)).$$
This form is $\tau$-equivariant, since for any $g\in H$, we have that 
\[[R(g)p,\tau(g^{-1})w]=[p,w].\] In particular, 
\[\omega^\alpha_{R(g)p}(dR_g(v))=\tau(g^{-1})\omega^\alpha_p(v).\]
Moreover,  it is horizontal, because by the definition of $\omega^\alpha(v)$, it depends only on the projection of $v\in TE$ to $TM$, i.e. evaluation on vertical directions always gives zero.
\end{proof}

We now apply Lemma \ref{lem:Maurer-Cartan_G} for the case when $E=G$, $H=K$, $M=\mathcal X$, $V=\mathfrak m$ and $\tau=\mathrm{Ad}$, the adjoint action of $K$ restricted to $\mathfrak m$; this will then give a proof of  Proposition \ref{prop:General_def_G}:

\begin{proof}[Proof of Proposition \ref{prop:General_def_G}]
Denote by $G(\mathfrak{m}):=G\times_{\mathrm{Ad}(K)}\mathfrak{m}$, and by $G(\mathfrak{g}/\mathfrak{k}):=G\times_{\mathrm{Ad}(K)}\mathfrak{g}/\mathfrak{k}$ the associate bundle over $\mathcal{X}$, which can be identified with $T\mathcal{X}$ {\cite[Section 1]{BR90}}, since
\begin{equation}\label{identification of TX}
    T\mathcal{X}\cong G(\mathfrak{g}/\mathfrak{k})\cong G(\mathfrak{m}).
\end{equation}

\noindent More precisely, an element of $G(\mathfrak{g}/\mathfrak{k})$ can be identified with the equivalence class of a pair $[g,v+\mathfrak{k}]$ with $g\in G$, $v\in\mathfrak g$ such that $(g,v+\mathfrak{k})\sim (gk,Ad(k^{-1})v+\mathfrak{k})$, for all $k\in K$.
Similarly, elements of $G(\mathfrak{m})$ are the equivalence classes of pairs $[g,m]$ with $g\in G$, $m\in\mathfrak m$ such that $(g,m)\sim (gk,Ad(k^{-1})m)$, for all $k\in K$. 

\noindent The map 
\begin{align*}
    \Phi\colon G(\mathfrak{m})&\to G(\mathfrak{g}/\mathfrak{k})\\
    [g,m]&\mapsto [g,m+\mathfrak k]
\end{align*}
is a $G$-equivariant isomorphism between bundles. Indeed, one checks that 
\begin{align*}
    \Phi([gk,Ad(k^{-1})m])&=[gk,Ad(k^{-1})m+\mathfrak k]\\
    &=[g,m+\mathfrak k]\\
    &=\Phi([g,m]),
\end{align*}
thus $\Phi$ is well-defined. Furthermore, $\Phi$ is an isomorphism because $\mathfrak m$ is a complement of $\mathfrak k$ in $\mathfrak g$. 

\noindent Finally, the space $T\mathcal X$ can be identified with the space of equivalence classes of pairs $[g,v+Ad(g)\mathfrak k]$ with $g\in G$, $v\in\mathfrak g$ such that $(g,v+Ad(g)\mathfrak k)\sim (gk,v+Ad(g)\mathfrak k)$. We identify $T\mathcal X$ and $G(\mathfrak{g}/\mathfrak{k})$ using the bundle isomorphism 
\begin{align*}
\Phi'\colon T\mathcal X & \to G(\mathfrak{g}/\mathfrak{k}),\\ [g,v+Ad(g)\mathfrak k] & \mapsto [g,Ad(g^{-1})v+\mathfrak k].
\end{align*} 
This map is also well-defined since
\begin{align*}
\Phi'([gk,v+Ad(g)\mathfrak k]) & =[gk,Ad((gk)^{-1})v+\mathfrak k]\\
& = [gk,Ad(k^{-1})Ad(g^{-1})v+\mathfrak k]\\
& = [g,Ad(g^{-1})v+\mathfrak k]\\
& = \Phi'([g,v+Ad(g)\mathfrak k]).
\end{align*}
Now since the action of $G$ on $T\mathcal X$ is described by 
\[(g',[g,v+Ad(g)\mathfrak{k}])\mapsto [g'g,Ad(g')v+Ad(g'g)\mathfrak k],\] 
we can define the action of $G$ on the spaces $G(\mathfrak{g}/\mathfrak{k})$ and $G(\mathfrak{m})$ by 
\[(g',[g,v+\mathfrak k])\mapsto[g'g,v+\mathfrak k]\ \mathrm{ and }\ (g',[g,m])\mapsto[g'g,m],\]
respectively. 

\noindent The maps $\Phi$ and $\Phi'$ are both equivariant with respect to these two actions. Indeed, one sees that
\begin{align*}
\Phi'(g'.[g,v+Ad(g)\mathfrak k]) & = \Phi'([g'g,Ad(g')v+Ad(g'g)\mathfrak k])\\
 & = [g'g,Ad(g^{-1})v+\mathfrak k]\\
 & = g'.[g,Ad(g^{-1})v+\mathfrak k]\\
 & = g'.\Phi'([g,v+Ad(g)\mathfrak k]).
 \end{align*}

\noindent Thus, the bundle $E_K(\mathfrak m)\to X$  is now defined by taking the quotient of the pull-back bundle $f^*_\mathcal X(G(\mathfrak m))\to \tilde X$ by the action of the group of deck transformations on $\tilde X$. Similarly, the bundle $E_K(\mathfrak{g}/\mathfrak{k})\to X$ is defined as the quotient of the pull-back bundle $f^*_\mathcal X(G(\mathfrak{g}/\mathfrak{k}))\to \tilde X$ by the equivalence relation $\sim$ defined in \eqref{equiv_rel_rho}. Then, the equivariant isomorphisms $\Phi$ and $\Phi'$ described earlier induce a bundle isomorphism
\[E_K(\mathfrak m) \cong f^{*}_{\mathcal{X}} T\mathcal{X}, \]
which is equivariant under the action of $\pi_1(X)$.

\noindent Now, consider the $K$-bundle given by the natural projection $\mathrm{pr}\colon G\to G/K=\mathcal X$ and remember the Maurer--Cartan form $\omega^{MC}\in\Omega^1(G,\mathfrak g)$ given by the differential of the left-multiplication map $(L_{g^{-1}})_*\colon T_gG\to \mathfrak g$. Then, the bundle identification 
\[\omega_\mathfrak m:=\Phi^{-1}\circ\,\Phi'\colon TG/K\to G\times_{\mathrm{Ad}(K)}\mathfrak m\] can be seen as an element in
\[\Omega^1(G/K,G\times_{\mathrm{Ad}(K)}\mathfrak m),\]
which is naturally identified with the subspace of horizontal, $\mathrm{Ad}(K)$-equivariant forms of $\Omega^1(G,\mathfrak m)$ by Lemma~\ref{lem:Maurer-Cartan_G}.

\noindent  It then follows that the form $\omega_K:=\omega^{MC}-\omega_\mathfrak m\in\Omega^1(G,\mathfrak k)$. The space $\Omega^1(G,\mathfrak k)$ can be naturally identified with $\Omega^1(P_K,\mathfrak k)$, where $P_K$ is the following $K$-subbundle of the trivial $G$-bundle $G/K\times G\to G/K$:
\[P_K=\{(gK,gk)\mid g\in G,\;k\in K\}\subset G/K\times G\to G/K,\] 
and $G\overset{\cong}{\longrightarrow} P_K$ with $g\mapsto (\mathrm{pr}(g),g)$, where $\mathrm{pr}\colon G\to G/K$ is the natural projection as above.

\noindent  Finally, having a $\rho$-equivariant map $f_\mathcal X\colon \tilde X\to G/K$, one obtains the bundle $f_\mathcal X^*(G/K\times G)=E_G$ equipped with the connection $\nabla:=f_\mathcal X^*\omega^{MC}$, and the $K$-subbundle $f^*(P_K)\subset E_G$ with a connection $D:=f_\mathcal X^*\omega_K$. Then, 
\[\Psi:=\nabla-D=f_\mathcal X^*\omega_\mathfrak m=\omega_\mathfrak m\circ df.\] Under the identification $TG/K\cong G\times_{\mathrm{Ad}(K)}\mathfrak m$ which is, as we have seen above, provided by $\omega_\mathfrak m$, we can identify $\Psi$ with $df$ (as is usually seen in the literature in the case when $G=\mathrm{GL}_n(\mathbb{C})$).
\end{proof}

\begin{rem}
The unitary connection $D$ above is essentially the pull-back of the Levi-Civita connection of $\mathcal{X}$; this connection does not depend on the choice of the model $\mathcal{X}$.
\end{rem}

The discussion above leads to the following definition:

\begin{defn}\label{defn:general $G$-Higgs defn_any model}
Let $E_K$ be a principal $K$-bundle over $X$. Let $f_{\mathcal X}\colon \tilde{X} \to \mathcal{X}$ be a smooth map which is equivariant with respect to the homomorphism-like map $\rho\colon \tilde X\times \pi_1(X) \to G$ induced by the principal bundle $E_K$ and the map $f_{\mathcal X}$ as above.  Choose a holomorphic structure $\bar{\partial}_{E_{K^{\mathbb{C}}}}$ to equip the bundle $E_{K^{\mathbb{C}}}$. Then, there is a \emph{$G$-holomorphic pair} $(\mathcal{E}, \varphi)$ where $\mathcal{E}:=\left(E_{K^{\mathbb{C}}}, \bar{\partial}_{E_{K^{\mathbb{C}}}}\right)$ is a holomorphic principal $K^{\mathbb{C}}$-bundle over $X$ and $\varphi$ is a  section of the bundle $f^{*}_{\mathcal{X}} T^{\mathbb{C}}\mathcal{X}\otimes K_X$ with $\bar{\partial}_{E_{K^{\mathbb{C}}}} \varphi =0$. 
\end{defn}

Note that the choice of $f_{\mathcal{X}}$ guarantees the identification  $\tilde{E}_K(\mathfrak{m}) \cong f^{*}_{\mathcal{X}} T\mathcal{X}$ over $\tilde{X}$. Thus, it does not impose any restriction on the holomorphic structure of a $G$-holomorphic pair as above.

\begin{rem} 
In the special case when the homomorphism-like map $\rho$ above is a group homomorphism, then the $G$-holomorphic pair $(E, \varphi)$ is precisely a $G$-Higgs bundle. Moreover, when the classifying map $f_{\mathcal{X}}$ is coming from a reductive fundamental group representation via Corlette's theorem, then such a $G$-Higgs bundle is \emph{polystable}. Note, however, that a complex vector bundle $E_{K^{\mathbb{C}}}$ admits a real structure $E_K$ if and only if $E_{K^{\mathbb{C}}}$ is isomorphic to $\overline{E_{K^{\mathbb{C}}}}$ (see \cite[Section 19.1]{FoFu}).
\end{rem}

\section{Geometric incarnations for certain groups}\label{Appendix:several_examples}
Following the general descriptions for $G$-Higgs bundles for various models of the symmetric space as given in Section \ref{sec:Higgs-Data}, we now illustrate these geometric incarnations of the data in certain examples of classical Lie groups. 
For the reader's convenience we give all the details so that this can be read independently of Section~\ref{sec:Higgs-Data}.

\subsection{\texorpdfstring{$G=\mathrm{Sp}_{2n}(\mathbb{R})$}{G=Sp(2n,R)}}\label{sec:comparison of defs}

Recall the classical definition of an $\mathrm{Sp}_{2n}( \mathbb{R})$-Higgs bundle from \cite[Section 3.1]{GGM}:

\begin{defn}\label{defn:classical_Sp2nR.ap}
An $\mathrm{Sp}_{2n}( \mathbb{R})$-Higgs bundle over $X$ is defined by a triple $\left( V,\beta ,\gamma  \right)$, where $V$ is a rank $n$ holomorphic vector bundle over $X$ and $\beta , \gamma$ are symmetric homomorphisms
	\[\beta :{{V}^{*}}\to V\otimes K_X\text{  and  }\gamma :V\to {{V}^{*}}\otimes K_X.\]
\end{defn}
From (\ref{formula:complexified tangent bundle_cx str model}), the complexified tangent bundle of the space of all complex structures at a point  $J\in\mathfrak{C}_G$ is given by

\begin{align*}
 T^{\mathbb C}_J\mathfrak{C}_G& =\left\{L_\mathbb C\colon A_\mathbb C^2\to A_\mathbb C^2 \midwd
 \begin{aligned}
     &\ L_\mathbb C\text{ is $\mathbb C$-linear},\\
     &\ L_\mathbb CJ_\mathbb C+J_\mathbb CL_\mathbb C=0, \text{ and}\\
     &\ h_{L_\mathbb C}\text{ is $\sigma_\mathbb C$-symmetric}
 \end{aligned}
 \right\} \\
  & \cong \{(a_+,a_-)\in A_\mathbb C^{\sigma_\mathbb C}\times A_\mathbb C^{\sigma_\mathbb C}\}, 
 \end{align*}
and this description depends on the choice of a basis $v=(v_+,v_-)$ of $i$- and $-i$-eigenlines $l_+$ and $l_-$ of $J_\mathbb C$ such that $\omega_\mathbb C(v_+,v_-)=-i$ and $v_+=\bar v_-$. Therefore, using the form $\omega_{\mathbb{C}}$, the isotropic line $l_{+}$ identifies as dual to $l_{-}$ in the sense of Section~\ref{sec:duality}.

For the case $(A,\sigma)=(\mathrm{Mat}_{n}(\mathbb{R}),(\bullet)^t)$ when $\mathrm{Sp}_2(A,\sigma) \cong  \mathrm{Sp}_{2n}(\mathbb{R})$, the anti-involution ${{\sigma }_{\mathbb{C}}}$ is defined as ${{\sigma }_{\mathbb{C}}}\left( a+bi \right)={{a}^{t}}+{b^{t}}i$, and we have 
${A}_{\mathbb{C}}=\left\{ a+bi\mid a,b\in A \right\}$ 
and 
\begin{align*}	
A_{\mathbb{C}}^{{{\sigma }_{\mathbb{C}}}} =\left\{ a+bi \mid a={{a}^{t}},b={{b}^{t}} \right\}=  \mathrm{Sym}_n\left(\mathbb{R} \right)\oplus i\mathrm{Sym}_n(\mathbb{R})=\mathrm{Sym}_n(\mathbb C).
\end{align*}
Moreover, $L_\mathbb C$ in matrix presentation with respect to the basis $v$  has the form $\begin{pmatrix}
    0 & a_- \\
    a_+ & 0
\end{pmatrix}$, 
for the morphisms $a_{-}\colon l_{-} \to l_{+}$ and $a_{+}\colon l_{+} \to l_{-}$. The pullback bundle of $T^{\mathbb C}\mathfrak{C}_G$ via a $\rho$-equivariant map $f$ twisted by the canonical line bundle $K_X$ has sections which are
identified with a morphism 
\[E_{K^{\mathbb{C}}}(A^2_{\mathbb{C}})\to E_{K^{\mathbb{C}}}(A^2_{\mathbb{C}})\otimes K_X,\]
where $E_{K^{\mathbb{C}}}(A^2_{\mathbb{C}})$ denotes the associated bundle to the principal $K^{\mathbb{C}}$-bundle for the maximal compact subgroup $K \subset \mathrm{Sp}_2(A, \sigma)$. An endomorphism of $E_{K^{\mathbb{C}}}(A^2_{\mathbb{C}})$ can be written in matrix presentation with respect to the basis $v$ as $\left( \begin{matrix}
   0 & a_{-}   \\
   a_{+}  & 0  \\
\end{matrix} \right)$. In particular, under this 1-1 correspondence the isotropic line $l_{+}$ corresponds to $V$ from Definition \ref{defn:classical_Sp2nR.ap} and $l_{-}$ corresponds to $V^*$.

\subsection{\texorpdfstring{$G=\mathrm{Sp}_{2n}(\mathbb{C})$}{G=Sp(2n,C)}}

Recall also the classical definition of an $\mathrm{Sp}_{2n}( \mathbb{C})$-Higgs bundle from \cite[Section 3.5]{GGM}:

\begin{defn}\label{defn:class_Sp_2nC}
    An $\mathrm{Sp}_{2n}( \mathbb{C})$-Higgs bundle over $X$ is a pair $(E, \Phi)$ where $E$ is a rank $2n$ holomorphic vector bundle with a symplectic form $\omega (\cdot, \cdot)$, and $\Phi$ is a holomorphic section in $H^0(X, \mathrm{End}(E)\otimes K_X)$ satisfying
    \[\omega(\Phi v,w)=-\omega(v,\Phi w).\]
\end{defn}

We now consider the geometric interpretation of an $\mathrm{Sp}_{2n}( \mathbb{C})$-Higgs bundle for the quaternionic structures model. Let $A=\mathrm{Mat}_n(\mathbb R)$ with $\sigma$ being the transposition, then $A_\mathbb C=\mathrm{Mat}_n(\mathbb C)$. A Higgs field at $x\in X$ is a $1$-form with values in
\begin{align*}
T^\mathbb C_{J_x}\mathfrak{C}_G=\left\{
L\colon A_{\mathbb{C}\mathbb C}^2\to A_{\mathbb{C}\mathbb C}^2 \midwd
\begin{aligned}
    &\ L(xa)=L(x)\theta_{(i,i)}(a)\;\text{ for }x\in A_{\mathbb{C}\mathbb C}^2, a\in A_{\mathbb{C}\mathbb C},\\
    &\ h_L \text{ is } \sigma_1\text{-symmetric}, \text{ and } LJ_x+J_xL=0
\end{aligned}
\right\},
\end{align*}
where $A_{\mathbb C\mathbb C}=A_{\mathbb C}\times A_{\mathbb C}$ and $\sigma_1(z_1,z_2)=(\bar\sigma_\mathbb C(z_2),\bar\sigma_\mathbb C(z_1))$ (cf.~Section~\ref{sec:quatstructures}). We choose a local frame in which $J_x=\Omega=\begin{pmatrix}
0 & \mathrm{Id}_n\\
-\mathrm{Id}_n & 0 
\end{pmatrix}$ and identify in the following the linear maps $J_x$ and $L$ with their matrices with respect to this frame.
The condition that $h_L$ is  $\sigma_1$-symmetric means that $\sigma_1(L)\Omega=\Omega L$. Let $L=\begin{pmatrix}
a & b \\
c & d
\end{pmatrix}$ where $a=(a_1,a_2)$, $b=(b_1,b_2)$, $c=(c_1,c_2)$, $d=(d_1,d_2)$ are elements of $A_{\mathbb C\mathbb C}$. This implies 
$$\bar\sigma_\mathbb C(a_2)=a_1,\,\bar\sigma_\mathbb C(d_2)=d_1, \text{ and }$$ 
$$\bar\sigma_\mathbb C(c_2)=b_1,\,\bar\sigma_\mathbb C(b_2)=c_1.$$
The condition $LJ_x+J_xL=0$ additionally implies:
$$b_1=\bar c_2,\, b_2=\bar c_1,$$
$$a_1=-\bar d_2,\, a_2=-\bar d_1, \text{ and }$$
$$c_1=\sigma_\mathbb C(c_1),\,c_2=\sigma_\mathbb C(c_2).$$
Thus, $L=\begin{pmatrix}
(a_1,\bar\sigma_\mathbb C(a_1)) & (\bar c_2,\bar c_1)\\
(c_1,c_2) & -(\sigma_\mathbb C(a_1),\bar a_1)
\end{pmatrix}$, where $a_1\in A_\mathbb C$, $c_1,c_2\in A_\mathbb C^{\sigma_\mathbb C}$.

In the classical definition of the Higgs field as a $1$-form with values in $\mathfrak m^\mathbb C$, where $\mathfrak m$ is the symmetric part in the Cartan decomposition
$$\mathfrak{sp}_{2n}(\mathbb C)=\mathfrak{k}+\mathfrak m,$$
 $\mathfrak m$ is the space of Hermitian symplectic matrices: 
$$\mathfrak m=\left\{\begin{pmatrix}
a & b \\
\bar b & -\sigma(a)
\end{pmatrix} \midwd a\in A_\mathbb C^{\bar\sigma_\mathbb C},\,b\in A_\mathbb C^{\sigma_\mathbb C}\right\}.$$

Moreover, the complexification of $A^{\bar\sigma_\mathbb C}_{\mathbb{C}}$ is the space of $\sigma_1$-symmetric elements of $A_{\mathbb C\mathbb C}$, i.e. the space of all pairs $(a,\bar\sigma_\mathbb C(a))$ where $a\in A_\mathbb C$. The complexification of $A^{\sigma_\mathbb C}_{\mathbb{C}}$ is just its double $A^{\sigma_\mathbb C}_{\mathbb{C}}\times A^{\sigma_\mathbb C}_{\mathbb{C}}$, thus we obtain

$$\mathfrak m^\mathbb C= \left\{\begin{pmatrix}
(a,\bar\sigma_\mathbb C(a)) & (b_1,b_2)\\
(\bar b_2,\bar b_1) & -(\sigma_\mathbb C(a),\bar a)
\end{pmatrix} \midwd a\in A_\mathbb C,\,b_1,b_2\in A_\mathbb C^{\sigma_\mathbb C}\right\}.$$
Therefore, the description of the Higgs field in terms of the quaternionic structures model agrees with the one in the classical definition.

One can also see that an $\mathrm{Sp}_{2n}(\mathbb{C})$-Higgs bundle viewed as an $\mathrm{Sp}_2(A_{\mathbb{C}}, \sigma_{\mathbb{C}})$-Higgs bundle is exactly the same as an $\mathrm{Sp}_{2n}(\mathbb{C})$-Higgs bundle in the classical sense, when considering the projective space model. Indeed, following the description of the complexified tangent bundle provided in~\eqref{eq:spc_proj_comp.tangent}, we see that a polystable $G=\mathrm{Sp}_{2n}(\mathbb C)$-Higgs bundle in the projective space model $\mathfrak{P}_{G}^+$ that comes from a reductive representation $\rho\colon \pi_1(X)\to G$ consists of a holomorphic $K^\mathbb C$-bundle $(E,\bar\partial_E)$ defined by $f_{\mathfrak{P}_{G}^+}^\mathbb C$ together with a $(1,0)$-form $\varphi$ with values in $f_{\mathfrak{P}_{G}^+}^*T^\mathbb{C}\mathfrak{P}_{G}^+$, such that $\bar\partial_E \varphi =0$. More specifically, for each $x\in \tilde X$, the fiber over $x$ is a $2n$-dimensional complex vector space $L_x^\mathbb C:=f_{\mathfrak{P}_{G}^+}(x)^\mathbb C$ equipped with a symplectic form $\omega$, the Higgs field $\varphi(x)$ is a $(1,0)$-form with values in 
    $$T_{L_x}^\mathbb C\mathfrak{P}_{G}^+=\{Q\in\mathrm{Hom}(L_x^\mathbb C,L_x^\mathbb C)\mid \text{for all $x\in L_x^\mathbb C$, } \omega(Q(x),x)+\omega(x,Q(x))=0\},$$
    such that $\bar\partial_E \varphi =0$.

\subsection{\texorpdfstring{$G=\mathrm{O}(n,n)$}{G=O(n,n)}}
The maximal compact subgroup of $G=\mathrm{O}(n,n)$ is $K:=\mathrm{O}(n)\times\mathrm{O}(n)$. The complexified tangent space in terms of the half-space model or the precompact model is naturally isomorphic to $\mathrm{Mat}_n(\mathbb C)$. We then have the following interpretation:

    An $\mathrm{O}(n,n)$-Higgs bundle over $X$ associated to a reductive representation $\rho\colon \pi_1(X)\to G$ with a $\rho$-equivariant harmonic map $f\colon\tilde X\to \mathfrak B_G\cong G/K$ is a triple $(V_1,V_2,\varphi)$, where $V_1,V_2$ are holomorphic principal $\mathrm{O}_n(\mathbb C)$-bundles over $X$ and $\varphi$ is a holomorphic equivariant section of the pull-back bundle $f^*T^\mathbb C\mathfrak B_G\otimes K_X$ which is a holomorphic bundle over $\tilde X$ with fibers isomorphic to $\mathrm{Mat}_n(\mathbb C)$.

\noindent We can also realize an $\mathrm{O}(n,n)$-Higgs bundle in terms of the projective space model:

    An $\mathrm{O}(n,n)$-Higgs bundle over $X$ associated to a reductive representation $\rho\colon \pi_1(X)\to G$ with a $\rho$-equivariant harmonic map $f\colon\tilde X\to \mathfrak P_G^+\cong G/K$ is a triple $(V_1,V_2,\varphi)$, where $V_1,V_2$ are holomorphic principal $\mathrm{O}_n(\mathbb C)$-bundles over $X$. This can be interpreted as two holomorphic rank $n$ vector bundles equipped with non-degenerate orthogonal forms $\omega_1$, $\omega_2$. The Higgs field in terms of the projective space model is a holomorphic equivariant section of the pull-back bundle $f^*T^\mathbb C\mathfrak P_G^+\otimes K_X$ which can be seen as a holomorphic morphism $\varphi\colon V_1\to V_2\otimes K_X$ (cf.~\eqref{eq:o(1,1)-comp.tang}).

\begin{rem}
    Since the forms $\omega_1$ and $\omega_2$ are non-degenerate, there is a unique morphism $\varphi'\colon V_2\to V_1\otimes K_X$ defined as follows
   \begin{equation}\label{eq:o(n,n)-higgs}
   \omega_1(v_1,\varphi'(v_2))=\omega_2(\varphi(v_1),v_2),   
   \end{equation}
   for $v_1\in V_1$, $v_2\in V_2$. 
    If we choose local orthonormal bases of $V_1$ and $V_2$, then the matrices $\Phi$ and $\Phi'$ corresponding to $\varphi$ and $\varphi'$ in these bases are related as follows: $\Phi'=\Phi^t$ and $\Phi\in\mathrm{Mat}_n(\mathbb C)$. This corresponds (fiberwise) to the element in $\mathrm{Mat}_n(\mathbb C)$ above.
\end{rem}

The above description is equivalent to the definition of general $\mathrm{O}(n,n)$-Higgs bundles in the classical sense (see \cite[Section 2]{ABCGPGO}), here ``general'' means Higgs bundles not necessarily coming from reductive representations, and ``in the classical sense'' means coming from the complexification of the Cartan decomposition; we recall this next:

\begin{defn}
    An \emph{$\mathrm{O}(n,n)$-Higgs bundle} over $X$ is a tuple $(V_1,\omega_1,V_2,\omega_2,\varphi)$ for which $(V_1,\omega_1)$ and $(V_2,\omega_2)$ are rank $n$ holomorphic orthogonal bundles, and $\varphi: V_1\to V_2\otimes K_X$ is a holomorphic morphism.
\end{defn}

\subsection{\texorpdfstring{$G=\mathrm{GL}_n(\mathbb{R})$}{G=GL(n,R)} and \texorpdfstring{$\mathrm{O}_n(\mathbb C)$}{G=O(n,C)}}
In terms of the precompact model, the definition of $G$-Higgs bundles in these cases looks very similar. Embedded into $\mathrm{O}(n,n)$, the groups have the same maximal compact subgroup $K=\mathrm{O}(n)$ (embedded diagonally), so the principal $K^\mathbb C$-bundle for both groups is the same. The only difference appears in the Higgs field. In terms of precompact model, the complexified tangent space at every point of the symmetric space is naturally isomorphic either to $\mathrm{Sym}_n(\mathbb C)$ or $\mathrm{Skew}_n(\mathbb C)$. We thus have the following:

A $\mathrm{GL}_n(\mathbb R)$-Higgs bundle (resp. $\mathrm{O}_n(\mathbb C)$-Higgs bundle) over $X$ associated to a reductive representation $\rho\colon \pi_1(X)\to G$ with a $\rho$-equivariant harmonic map $f\colon\tilde X\to \mathfrak B_G\cong G/K$ is a pair $(V,\varphi)$, where $V$ is a holomorphic principal $\mathrm{O}_n(\mathbb C)$-bundle over $X$ and $\varphi$ is a holomorphic equivariant section of the pull-back bundle $f^*T^\mathbb C\mathfrak B_G\otimes K_X$ which is a holomorphic bundle over $\tilde X$ with fibers isomorphic to  $\mathrm{Sym}_n(\mathbb C)$ (resp. $\mathrm{Skew}_n(\mathbb C)$).

    \vspace{2mm}

    \noindent This can be interpreted as a holomorphic rank $n$ vector bundle equipped with a non-degenerate orthogonal form $\omega$, and $\varphi\colon V\to V\otimes K_X$ is a holomorphic morphism which is symmetric (resp. anti-symmetric) with respect to $\omega$.

\subsection{\texorpdfstring{$G=\mathrm{U}(n,n)$}{G=U(n,n)}}
We start by reviewing the definition of $\mathrm{U}(n,n)$-Higgs bundles in the classical sense (see \cite{MaXi}). The maximal compact subgroup of $G$ is $K:=\mathrm{U}(n)\times\mathrm{U}(n)$, with complexification $K^{\mathbb{C}}=\mathrm{GL}_n(\mathbb{C})\times\mathrm{GL}_n(\mathbb{C})$. Note that the complexification of the Cartan decomposition of $G$ is
$$
\mathfrak{gl}_{2n}(\mathbb{C})=\Big(\mathfrak{gl}_{n}(\mathbb{C})\oplus\mathfrak{gl}_{n}(\mathbb{C})\Big)\oplus\mathfrak{m}^{\mathbb{C}}.
$$
If we embed $\mathfrak{gl}_{n}(\mathbb{C})\oplus\mathfrak{gl}_{n}(\mathbb{C})$ into $\mathfrak{gl}_{2n}(\mathbb{C})$ as block diagonal matrices, then 
$$\mathfrak{m}^{\mathbb{C}}=\bigg\{\begin{pmatrix}
    0&*\\
    *&0
\end{pmatrix}\in\mathfrak{gl}_{2n}(\mathbb{C})\bigg\},$$ 
i.e. it consists of the off-diagonal matrices. This leads to the following definition of general $\mathrm{U}(n,n)$-Higgs bundles in the classical sense:

\begin{defn}\label{def:u(n,n)-classical}
    A \emph{$\mathrm{U}(n,n)$-Higgs bundle} consists of a quadruple $(V_1,V_2,\beta,\gamma)$ so that $V_1$ and $V_2$ are rank $n$ holomorphic vector bundles and $\beta: V_2\to V_1\otimes K_X$, $\gamma: V_1\to V_2\otimes K_X$ are holomorphic morphisms.

    \vspace{2mm}

    \noindent Note that $(V_1,V_2,\beta,\gamma)$ embeds as a $\mathrm{GL}_{2n}(\mathbb{C})$-Higgs bundle $(E=V_1\oplus V_2, \varphi=\begin{pmatrix}
    0&\beta\\
    \gamma&0
\end{pmatrix})$.
\end{defn}

Now we define $\mathrm{U}(n,n)$-Higgs bundles in the sense of $\mathrm{Sp}_2(A,\bar\sigma)$ and $\mathrm{O}_{(1,1)}(A,\bar\sigma)$, where $A=\mathrm{Mat}_n(\mathbb{C})$ and $\bar\sigma(X):=\bar{X}^t$. Remember from (\ref{eq:U(n,n)}) that $\mathrm{Sp}_2(A,\bar\sigma)\cong\mathrm{O}_{(1,1)}(A,\bar\sigma)\cong\mathrm{U}(n,n)$ with a maximal compact subgroup $\mathrm{O}(A,\bar\sigma)\times\mathrm{O}(A,\bar\sigma)\cong \mathrm{U}(n)\times\mathrm{U}(n)$. The complexification of this group is isomorphic to $\mathrm{GL}_n(\mathbb C)\times\mathrm{GL}_n(\mathbb C)$. However, to provide a description of a $\mathrm{U}(n,n)$-Higgs bundle in terms of models, we need to understand this isomorphism more precisely. The complexification $A\otimes_\mathbb R\mathbb C\{I\}$ of $A$ can be identified with $A\times A$ using the isomorphism from~\cite[Appendix B.1.1]{ABRRW} mapping $a_0+a_1I$ to $(a_0+a_1i,a_0-a_1i)$. Under this identification the elements of $a\in A\otimes_\mathbb R\mathbb C\{I\}$ such that $\bar\sigma_\mathbb{C}(a)a=1$ where $\bar\sigma_\mathbb{C}$ is the $\mathbb C\{I\}$-linear extension of $\bar\sigma$ are mapped to pairs $(x,\bar\sigma(x)^{-1})$, i.e. the image of $\mathrm{O}(A\otimes_\mathbb R\mathbb C\{I\},\bar\sigma_\mathbb C)$ is $\{(x,\bar\sigma(x)^{-1})\mid x\in A^\times\}$. The projection to the first component gives a group isomorphism between  $\mathrm{O}(A\otimes_\mathbb R\mathbb C\{I\},\bar\sigma_\mathbb C)$ and $A^\times$.

Regarding $(A,\bar\sigma)$ as a real Hermitian algebra, then for the group $G=\mathrm{O}_{(1,1)}(A,\bar\sigma)$, from Section \ref{sec:ort_models} it follows that the complexified tangent space in terms of the half-space or the precompact model is naturally identified with $A_{\mathbb{C}}=\mathrm{Mat}_n(\mathbb{C})\otimes\mathbb{C}\{I\}$ which is isomorphic to $ \mathrm{Mat}_n(\mathbb{C})\times\mathrm{Mat}_n(\mathbb{C})$ under the isomorphism described above. We thus get to the following interpretation:

    A \emph{$\mathrm{U}(n,n)$-Higgs bundle} over $X$ associated to a reductive representation $\rho\colon \pi_1(X)\to G=\mathrm{O}_{(1,1)}(A,\bar\sigma)$ with a $\rho$-equivariant harmonic map $f\colon\tilde X\to \mathfrak B_G\cong G/K$ is a holomorphic $\mathrm{GL}_n(\mathbb C)\times\mathrm{GL}_n(\mathbb C)$-principal bundle over $X$ which can be interpreted as a pair of holomorphic $\mathrm{GL}_n(\mathbb C)$-principal bundles over $X$ and $\varphi$ is a holomorphic equivariant section of the pull-back bundle $f^*T^\mathbb C\mathfrak B_G\otimes K_X$, which is a holomorphic bundle over $\tilde X$ with fibers isomorphic to  $\mathrm{Mat}_n(\mathbb C)\times \mathrm{Mat}_n(\mathbb{C})$.

For a description of the $\mathrm{U}(n,n)$-Higgs bundle data in terms of the projective space model for the group $\mathrm{O}_{(1,1)}(A,\bar\sigma)$, we first notice that the principal $\mathrm{GL}_n(\mathbb C)\times\mathrm{GL}_n(\mathbb C)$-bundle can be interpreted as a bundle whose fiber is a pair of two $A_\mathbb{C}=A\otimes_\mathbb R\mathbb C$-lines $l_1^\mathbb C$ and $l_2^\mathbb C$ in $A_\mathbb{C}^2$ such that $l_i^\mathbb C=l_i\otimes_\mathbb R\mathbb C$ where $l_i$ are $A$-lines in $(A^2,\omega)$ where $\omega$ is the standard indefinite orthogonal form. Moreover, $l_2=l_1^{\perp_{\omega}}$ and $\omega|_{l_1}$ is positive definite. The Higgs field can be then interpreted as a morphism $\varphi\colon l_1^\mathbb C\to l_2^\mathbb C\otimes K_X$ without any additional conditions. It also defines a unique morphism $\varphi'\colon l_2^\mathbb C\to l_1^\mathbb C\otimes K_X$ by the condition $\omega_\mathbb C(v_1,\varphi'(v_2))=\omega_\mathbb C(\varphi(v_1),v_2)$. However, in this case $A_\mathbb C$ is isomorphic to $A\times A$ and the real locus $A\subset A_\mathbb C$ maps to the diagonal $\{(a,a)\mid a\in A\}\in A\times A$ under this isomorphism, therefore, the lines $l_i^\mathbb C$ can be written as $l_i\times l_i$. Moreover, the $A\times A$-morphism $\varphi\colon l_1\times l_1\to (l_2\times l_2)\otimes K_X$ can be written as a pair of $A$-morphisms $\varphi_1\colon l_1\to l_2\otimes K_X$ and $\varphi_2\colon l_1\to l_2\otimes K_X$. Analogously, the $A\times A$-morphism $\varphi'\colon l_1\times l_1\to (l_2\times l_2)\otimes K_X$ can be written as a pair of $A$-morphisms $\varphi_1'\colon l_2\to l_1\otimes K_X$ and $\varphi_2'\colon l_2\to l_1\otimes K_X$.

\begin{prop}
    The morphisms $\varphi_1$ and $\varphi_1'$ determine the morphism $\varphi$ uniquely.
\end{prop}

\begin{proof}
    We fix a basis $(v_1,v_2)$ of $A^2$ such that $\omega(v_1,v_1)=-\omega(v_2,v_2)=1$ and $l_1=v_1A$, $l_2=v_2A$, then \[l_1^\mathbb C\cong l_1\times l_1=v_1(A\times A), \quad l_2^\mathbb C\cong l_2\times l_2=v_2(A\times A),\] and 
    \[\varphi(v_1,v_1)=v_2(a_1,a_2)=(\varphi_1(v_1),\varphi_2(v_1)),\] for some $a_1,a_2\in A$, and 
    $$\varphi'(v_2,v_2)=v_1(a'_1,a'_2)=(\varphi'_1(v_2),\varphi'_2(v_2)).$$ 
    To prove the proposition, we need to show that $a_2$ is uniquely determined by $a_1$ and $a_1'$. Further,
    \begin{align*}
    -(\bar\sigma(a_2),\bar\sigma(a_1))=\omega_\mathbb C(\varphi(v_1),v_2)=\omega_\mathbb C(v_1,\varphi'(v_2))=(a'_1,a'_2).
    \end{align*}
    Therefore, $a_2=-\bar\sigma(a_1')$.
\end{proof}

The proposition shows that we can use maps $\varphi_1$ and $\varphi_1'$ instead of $\varphi$ to define a Higgs field because they contain exactly the same information. Now we can describe a $\mathrm U(n,n)$-Higgs bundle in terms of the projective space model as follows:

    A $\mathrm{U}(n,n)$-Higgs bundle over $X$ associated to a reductive representation $\rho\colon \pi_1(X)\to G$ with a $\rho$-equivariant harmonic map $f\colon\tilde X\to \mathfrak P_G^+\cong G/K$ is a holomorphic $\mathrm{GL}_n(\mathbb C)\times\mathrm{GL}_n(\mathbb C)$-principal bundle over $X$ which can be seen as a pair $V_1,V_2$ of rank $n$ holomorphic vector bundles over $X$. The Higgs field in terms of the projective space model is a holomorphic equivariant section of the pull-back bundle $f^*T^\mathbb C\mathfrak P_G^+\otimes K_X$ which can be seen as a pair of holomorphic morphisms 
    $\varphi\colon V_1\to V_2\otimes K_X$ and $\varphi'\colon V_2\to V_1\otimes K_X$.

\begin{rem}
    Notice that the description above agrees with the definition of a $\mathrm{U}(n,n)$-Higgs bundle seen as an $\mathrm{Sp}_2(A,\sigma)$, for $(A,\sigma)=(\mathrm{Mat}_n(\mathbb C),\overline\bullet^t)$ (cf. Section \ref{df:SpR-Higgs-proj}).
\end{rem}

\subsection{\texorpdfstring{$G=\mathrm{Sp}(n,n)$}{G=Sp(n,n)}}

As seen in (\ref{eq:Sp(n,n)}), the group $\mathrm{Sp}(n,n)$ can be interpreted as the indefinite orthogonal group $\mathrm{O}_{(1,1)}(A:=\mathrm{Mat}_n(\mathbb{H}),\sigma_1)$. From Section \ref{sec:ort_models}, the complexified tangent space of the projective space model at a point $l\in\mathfrak{P}_G^\pm$ is given by
\begin{align*}
    T_l^\mathbb{C}\mathfrak{P}_G^\pm&=\{Q\colon l_\mathbb C\to l_\mathbb C^\perp\}.
\end{align*}

The stabilizer of any $l\in \mathfrak{P}_G^\pm$  is a maximal compact subgroup which is conjugated to  $\mathrm{Sp}(n)\times\mathrm{Sp}(n)=\mathrm{O}(A,\sigma_1)\times\mathrm{O}(A,\sigma_1)$, where $\mathrm{O}(A,\sigma_1)=\{ a\in A^\times\mid \sigma_1(a)a=\mathrm{Id}\}$ is the orthogonal group of $A$. Its complexification is $\mathrm{Sp}_{2n}(\mathbb C)\times\mathrm{Sp}_{2n}(\mathbb C)$. Using the geometric interpretation of polystable $\mathrm{O}_{(1,1)(A, \sigma)}$-Higgs bundles in the projective space model from Section \ref{subsec:indefi_defs}, we can directly give an interpretation of a $\mathrm{Sp}(n,n)$-Higgs bundle associated to a reductive representation $\rho\colon\pi_1(S)\to\mathrm{Sp}(n,n)$:

    An $\mathrm{Sp}(n,n)$-Higgs bundle over $X$ is a triple $(V_1,V_2,\varphi)$, where $V_i$ are two holomorphic vector bundles of rank $2n$ equipped with symplectic forms $\omega_i$ over $X$, and $\varphi\colon V_1\to V_2\otimes K_X$ is a holomorphic morphism.

\begin{rem}
    Since the forms $\omega_1$ and $\omega_2$ are non-degenerate, there is a unique morphism $\varphi'\colon V_2\to V_1\otimes K_X$ defined by   \begin{equation}\label{eq:sp(n,n)-higgs}
   \omega_1(v_1,\varphi'(v_2))=\omega_2(\varphi(v_1),v_2),   \end{equation}
   for $v_1\in V_1$, $v_2\in V_2$. 
    
\end{rem}

\subsection{\texorpdfstring{$G=\mathrm{SO}^*(4n)$}{G=SO*(4n)}}

We realize the group $\mathrm{SO}^*(4n)$ as an $\mathrm{Sp}_2(A,\sigma)$, for $A:=\mathrm{Mat}_n(\mathbb{H})$ with the Hermitian  anti-involution $\sigma:=\sigma_1$. The complexified algebra $A_\mathbb C$ is isomorphic to $\mathrm{Mat}_{2n}(\mathbb C)$, the complex linear extension $\sigma_\mathbb C$ of $\sigma$ corresponds to the following map: 
$$m\mapsto -\begin{pmatrix}
0 & \mathrm{Id}\\
-\mathrm{Id} & 0
\end{pmatrix} m^t \begin{pmatrix}
0 & \mathrm{Id}\\
-\mathrm{Id} & 0
\end{pmatrix}\ \ (\text{cf.~\cite[Appendix B.1.2]{ABRRW}}).$$ In particular, \[A_\mathbb C^{\sigma_\mathbb C}=\left\{\begin{pmatrix} a & b \\ c& a^t\end{pmatrix}\midwd a\in\mathrm{Mat}_{n}(\mathbb C),\; b,c\in\mathrm{Skew}_{n}(\mathbb C)\right\}.\]

By \eqref{eq:compl.tangent.P.Sp}, the complexified tangent space of the projective space model at a point $l\in\mathfrak{P}_G^\pm$ is given by
\begin{align*}
    T_l^\mathbb C\mathfrak{P}_G^\pm&=\left\{(Q,Q')\in\mathrm{Hom}(l,\bar l)\times \mathrm{Hom}(\bar l,l)\midwd
    \begin{aligned}
       &\ \omega_{\mathbb C}(Q(x),x)\in A^{\sigma_\mathbb C}_\mathbb C,\\ 
       &\ \omega_{\mathbb C}(Q'(x),x)\in A^{\sigma_\mathbb C}_\mathbb C
    \end{aligned}\right\} \\
    &=T_l\mathfrak{P}_G^\pm\oplus T_{\bar l}\mathfrak{P}_G^\mp.
\end{align*}

Notice that the condition $\omega_{\mathbb C}(x,y)=\sigma_\mathbb C(x)^t\begin{pmatrix}
0 & \mathrm{Id} \\
-\mathrm{Id} & 0
\end{pmatrix}y \in A^{\sigma_\mathbb C}_\mathbb C$, where $x,y\in A_\mathbb C^2$, is equivalent to 
$$\omega'(x,y)=x^t\begin{pmatrix}
0 & 0 & 0 & \mathrm{Id} \\ 
0 & 0 & -\mathrm{Id} & 0 \\ 
0 & -\mathrm{Id} & 0 & 0 \\ 
\mathrm{Id} & 0 & 0 & 0 \\ \end{pmatrix}y\in\mathrm{Skew}_{2n}(\mathbb C).$$

The stabilizer of any $l\in \mathfrak{P}_G^\pm$  is a maximal compact subgroup which is conjugated to  $\mathrm{U}(2n)$. Its complexification is $\mathrm{GL}_{2n}(\mathbb C)$. From the description in Section \ref{df:SpR-Higgs-proj}, we can directly understand  a $\mathrm{SO}^*(4n)$-Higgs bundle associated to a reductive representation $\rho\colon\pi_1(S)\to\mathrm{SO}^*(4n)$. Namely: 

    An $\mathrm{SO}^*(4n)$-Higgs bundle over $X$ is a tuple $(V_1,V_2,\omega,\varphi_1,\varphi_2)$, where $V_1$, $V_2$ are holomorphic vector bundles of rank $2n$ over $X$, $\omega$ is a non-degenerate pairing between $V_1$ and $V_2$, i.e. $\omega\colon V_1\otimes V_2\to\mathbb C$, and $\varphi_1\colon V_1\to V_2\otimes K_X$ and $\varphi_2\colon V_2\to V_1\otimes K_X$ are two holomorphic $\omega$-skew-symmetric morphisms, i.e. forms $\omega(\cdot,\varphi_2(\cdot))$ and $\omega(\varphi_1(\cdot),\cdot)$ are skew-symmetric on $V_1$ resp. $V_2$.

\begin{rem}
    Since $\omega$ provides the bundle isomorphism $V_2\cong V_1^*$, the above description is equivalent to the classical one given in \cite[Section 4.2]{BGG06}, that is,  an $\mathrm{SO}^*(4n)$-Higgs bundle over $X$ is a triple $(V,\varphi_1,\varphi_2)$ with $V$ a holomorphic bundle of rank $2n$, and $\varphi_1: V\to V^*\otimes K_X$ and $\varphi_2: V^*\to V\otimes K_X$ are two skew-symmetric morphisms.
\end{rem}

\subsection{\texorpdfstring{$G=\mathrm{SO}_{0}(2,n)$}{G=SO(2,n)} and \texorpdfstring{$G=E_{7{(-25)}}$}{G=E7}}

The groups $G = \mathrm{SO}_{0}(2,n)$ and $G = E_{7(-25)}$ are also Hermitian of tube type, and classical Higgs bundle definitions for these have been given in \cite{BGG06} and \cite{BGR}, respectively. Note, however, that these groups cannot be described as groups of the form $\mathrm{Sp}_2(A, \sigma)$ for a Hermitian algebra $(A, \sigma)$. The group $\mathrm{Spin}_{0}(2,n)$, which is a double cover of $\mathrm{SO}_{0}(2,n)$, does admit a realization as a symplectic group $\mathrm{Sp}_2$ over a certain Hermitian submonoid of the Clifford algebra of signature $(1,n-1)$. Models of the symmetric space for $\mathrm{Spin}_{0}(2,n)$ still exist and have been described in \cite[Section 5]{R25}. However, the description of the complexified tangent bundle of these symmetric spaces lies beyond the scope of this article. As for the group $E_{7(-25)}$, no realization in terms of $\mathrm{Sp}_2$ has been provided in the literature to date (cf.~\cite[Section 3.2.1]{R25}).

\section{Complexified quaternions and the three standard involutions}\label{sec:cq_to_mat2}

Let $(A,\sigma)$ be a real involutive algebra. We denote $A_\mathbb H:=A\otimes_\mathbb R\mathbb H$, where \(\mathbb H\) is the real quaternion algebra with generators
\[
\mathbf i^2=\mathbf j^2=\mathbf k^2=-1,
\qquad
\mathbf i\mathbf j=\mathbf k,
\qquad
\mathbf j\mathbf i=-\mathbf k.
\] We also extend the anti-involution $\sigma$ in two ways 
$$\sigma_0(a_0+a_1\mathbf i+a_2\mathbf j+a_3\mathbf k):=\sigma(a_0)+\sigma(a_1)\mathbf i+\sigma(a_2)\mathbf j+\sigma(a_3)\mathbf k$$
and 
$$\sigma_1(a_0+a_1\mathbf i+a_2\mathbf j+a_3\mathbf k):=\sigma(a_0)-\sigma(a_1)\mathbf i-\sigma(a_2)\mathbf j-\sigma(a_3)\mathbf k.$$

We denote by \(I\) the imaginary unit in the scalar field
\(\mathbb C\), in order to distinguish it from the quaternionic generator
\(\mathbf i\).  Put
\[
A_{\mathbb{HC}}:=A_\mathbb H\otimes_{\mathbb R}\mathbb C.
\]
Slightly abusing our notation, let \(\sigma_i\) be the \(\mathbb C\)-linear extensions of $\sigma_i$ to
\(A_{\mathbb{HC}}\), and \(\bar\sigma_i\) be the \(\mathbb C\)-anti-linear extensions of $\sigma_i$ to
\(A_{\mathbb{HC}}\).

Define a \(A_\mathbb C\)-algebra homomorphism
\[
\Phi:A_{\mathbb{HC}}\longrightarrow \operatorname{Mat}_2(A_\mathbb C)
\]
by
\[
\Phi(\mathbf i)=
\begin{pmatrix}
I&0\\
0&-I
\end{pmatrix},
\qquad
\Phi(\mathbf j)=
\begin{pmatrix}
0&I\\
I&0
\end{pmatrix},
\qquad
\Phi(\mathbf k)=
\begin{pmatrix}
0&-1\\
1&0
\end{pmatrix}.
\]
These matrices satisfy the quaternionic relations
\[
\Phi(\mathbf i)^2=\Phi(\mathbf j)^2=\Phi(\mathbf k)^2=-\mathrm{Id},
\qquad
\Phi(\mathbf i)\Phi(\mathbf j)=\Phi(\mathbf k),
\qquad
\Phi(\mathbf j)\Phi(\mathbf i)=-\Phi(\mathbf k).
\]
Hence \(\Phi\) is an isomorphism of complex algebras
\[
A_{\mathbb{HC}}\cong \operatorname{Mat}_2(A_\mathbb C).
\]
Explicitly, for
\[
x=z_0+z_1\mathbf i+z_2\mathbf j+z_3\mathbf k,
\qquad z_\ell\in A\otimes_\mathbb R\mathbb C,
\]
one has
\[
\Phi(x)=
\begin{pmatrix}
z_0+ z_1I &  z_2I-z_3\\
 z_2I+z_3 & z_0- z_1I
\end{pmatrix}.
\]

Under this identification, the anti-involution \(\sigma_0\) becomes
ordinary transposition composed with $\sigma$ applied componentwise:
\[
\Phi(\sigma_0(x))=\sigma(\Phi(x)^t)
\qquad
\text{for all }x\in A_{\mathbb{HC}}.
\]
Indeed,
\[
\Phi(\mathbf i)^t=\Phi(\mathbf i),
\qquad
\Phi(\mathbf j)^t=\Phi(\mathbf j),
\qquad
\Phi(\mathbf k)^t=-\Phi(\mathbf k),
\]
which is exactly the defining action of \(\sigma_0\) on
\(\mathbf i,\mathbf j,\mathbf k\).

Further:
\[
\Phi(\sigma_1(x))
=
\sigma(K\,\Phi(x)^t\,K^{-1})
\qquad
\text{for all }x\in A_{\mathbb{HC}}.
\]
where
\[
K=\Phi(\mathbf k)=
\begin{pmatrix}
0&-1\\
1&0
\end{pmatrix},
\qquad
K^{-1}=-K.
\]
Thus, for
\[
M=
\begin{pmatrix}
a&b\\
c&d
\end{pmatrix},
\]
the corresponding matrix anti-involution is
\[
M\longmapsto K M^tK^{-1}
=
\begin{pmatrix}
\sigma(d)&-\sigma(b)\\
-\sigma(c)&\sigma(a)
\end{pmatrix}.
\]

Finally:
\[
\Phi(\bar\sigma_1(x))
=
\bar\sigma\!\left({\Phi(x)}^{t}\right)
\qquad
\text{for all }x\in A_{\mathbb{HC}}.
\]

\bigskip

\noindent\small{\textsc{School of Mathematics, Nanjing University}\\
		 Nanjing 210093, China}\\
\emph{E-mail address}:  \texttt{pfhwangmath@gmail.com}

\medskip

\noindent\small{\textsc{Department of Mathematics, University of Patras}\\
  University Campus, Patras 26504, Greece}\\
\emph{E-mail address}:  \texttt{gkydonakis@math.upatras.gr}

\medskip

\noindent\small{\textsc{School of Mathematics, Korea Institute for Advanced Study}\\
  85 Hoegi-ro, Dongdaemun-gu, Seoul 02455, Republic of Korea}\\
\emph{E-mail address}:  \texttt{erogozinnikov@gmail.com}

\medskip

\noindent\small{\textsc{Max Planck Institute for Mathematics in the Sciences}\\
  Inselstraße 22, 04103 Leipzig, Germany}\\
\emph{E-mail address}:  \texttt{wienhard@mis.mpg.de}

\bigskip

\end{document}